\documentclass[12pt,twoside]{amsart}
\usepackage{amssymb,amsmath,amsthm, amscd, enumerate, mathrsfs}
\usepackage{graphicx, hhline}
\usepackage[all]{xy}
\usepackage[usenames]{color}
\usepackage{hyperref}
\usepackage{fancyhdr}
\usepackage[colorinlistoftodos]{todonotes}
\usepackage{bbm}
\usepackage[top=30truemm,bottom=30truemm,left=30truemm,right=30truemm]{geometry}

\usepackage{comment}

\hypersetup{colorlinks=true}

\title[K-moduli of quasimaps]{K-moduli of quasimaps and on quasi-projectivity of K-moduli of Calabi-Yau fibrations over curves}
\author{Kenta Hashizume and Masafumi Hattori}
\date{\today}
\address{Kenta Hashizume \\ Department of 
Mathematics, Faculty of Science, Niigata University, Niigata 950-2181, Japan}
\address{Institute for Research Administration, Niigata University, Niigata 950-2181, Japan}
\email{hkenta@math.sc.niigata-u.ac.jp}
\address{Masafumi Hattori \\ School of 
Mathematical Sciences, University of Nottingham, Nottingham, England}
\email{masafumi.hattori@nottingham.ac.uk}

\newtheorem{thm}{Theorem}[section]
\newtheorem{thm2}[thm]{Theorem-Definition}

\newtheorem{lem}[thm]{Lemma}
\newtheorem{cor}[thm]{Corollary}
\newtheorem{prop}[thm]{Proposition}
\newtheorem{conj}[thm]{Conjecture}

\theoremstyle{definition}
\newtheorem{defn}[thm]{Definition}

\newtheorem{ex}[thm]{Example}
\newtheorem{rem}[thm]{Remark}
\newtheorem{note}[thm]{Notation}
\newtheorem*{ack}{Acknowledgments}

\newtheorem{setup}[thm]{Setup}

\newtheorem*{claim*}{Claim}
\usepackage[colorinlistoftodos]{todonotes}

\begin{document}
\begin{abstract}
  We construct a projective K-moduli space of quasimaps with a certain log Fano condition, which is regarded as a rational map from $\mathbb{P}^1$ to a projective space.  
  Moreover, we investigate relationships between the K-moduli of quasimaps and the K-moduli of Calabi-Yau fibrations over curves of negative Kodaira dimension constructed by the authors when general fibers are Abelian varieties or irreducible holomorphic symplectic manifolds.
  As an application, we obtain the entire quasi-projectivity of the seminormalization and the ampleness of the CM line bundle on the normalization of the K-moduli space of Calabi-Yau fibrations in this case.
\end{abstract}
\maketitle

\setcounter{tocdepth}{3}
\tableofcontents

\section{Introduction}

%Updated proofs and statements after realizing that seminormalization is required; some results are weaker than in v1.

The aim of this paper is to address the problem of quasi-projectivity of the K-moduli of \cite{HH} in the negative Kodaira dimension case.

\subsection{Positivity of CM line bundles and K-moduli conjecture}

{\it K-stability}, originally introduced by Tian \cite{Tia97} for Fano manifolds and later reformulated by Donaldson \cite{Dn2} for polarized varieties, is expected to detect the existence of constant scalar curvature K\"ahler ({\it cscK}) metrics, which is known as the Yau-Tian-Donaldson conjecture.
On the other hand, 
K-stability is also considered to be closely related to moduli theory. Inspired by \cite{KSB,FS,O}, Odaka \cite{Oconj} proposed the following:

\begin{conj}[K-moduli conjecture]\label{conj--k-moduli} With fixed numerical data, there exists a moduli space of K-polystable polarized varieties 
$(X,L)$ as a quasi-projective scheme, and the CM line bundle on the moduli is ample. 
\end{conj}

Here, the CM line bundle, introduced by Tian \cite{Tia97} and also considered in Fujiki--Schumacher \cite{FS}, is naturally defined on schemes parameterizing polarized varieties by using the Knudsen-Mumford expansion (\cite{PT,FR}). 
Fujiki--Schumacher \cite{FS} constructed a moduli space of non-uniruled projective manifolds with unique cscK metrics using an analytic method, and they showed the existence of a generalized Weil--Petersson metric on the moduli. 

Conjecture \ref{conj--k-moduli} has been settled in a few cases, including the canonically polarized case (cf.~\cite{kollar-moduli}) and the log Fano case (cf.~\cite{Xu}) due to dedications of many researchers. 
However, the conjecture is still open for general polarized log pairs.
We note that K-moduli in the sense of Conjecture \ref{conj--k-moduli} are not necessarily proper, in particular for the Calabi--Yau case.
See Subsection \ref{sec--related--works} for more details.

The authors \cite{HH} constructed K-moduli of {\it uniformly adiabatically K-stable klt-trivial fibrations over curves}. Here, klt-trivial fibrations are kinds of Calabi-Yau fibrations admitting the structure of Kawamata log terminal (klt) pairs (see Definition \ref{defn--can-bundle-formula}). Roughly speaking, uniform adiabatic K-stability is ``uniform" K-stability when the polarization is very close to that of the base, introduced by \cite{fibst} (see Definitions \ref{defn--unif--ad--Kst} and \ref{defn--unif--ad--Kst2}). Note that the moduli space of \cite{HH} is non-proper in general.
The second author \cite{Hat23} showed the positivity of the CM line bundle on proper subspaces of the K-moduli space of \cite{HH}.  The approach in \cite{Hat23} heavily relies on the Nakai--Moishezon criterion \cite{kollar-moduli-stable-surface-proj,FM3} for algebraic spaces, and that approach cannot deduce the entire positivity.

To overcome this non-properness issue, we first construct a projective K‑moduli space of base curves with additonal structure (Theorem \ref{main--thm--i}), and then relate it to the K‑moduli space of Calabi–Yau fibrations constructed in \cite{HH} (Theorem \ref{main--thm--ii}).

\subsection{Moduli of maps and fibered varieties}
The moduli space of maps from curves 
$C$ to a geometric object 
$X$, with additional information, is a significant topic in algebraic geometry because this notion is closely related to fibered varieties over curves. 
Abramovich--Vistoli \cite{AV} extend the notion of stable maps \cite{K,FP} to {\it twisted stable maps}, i.e. maps from stacky curves to Deligne--Mumford stacks.
In \cite{AV2}, they also applied the moduli space of twisted stable maps with fixed degree to construct a moduli space of surfaces fibered by curves of general type such that the degree of the map from the base curve to the moduli of fibers is fixed.
For rational elliptic surfaces, the degree of the moduli map can vary between 0 and 12 in irreducible families \cite[Table 6.1]{Mi}, illustrating the obstruction to fixed‑degree map moduli.

\subsection{Main results and key idea: Quasimaps}
The notion of {\it quasimaps} was defined by Givental \cite{givental} and also studied by Morison--Plesser \cite{DP2} as linear systems defining rational maps from $\mathbb{P}^1$ to toric varieties, and their compact moduli was constructed. 
Quasimaps absorb degree‑jumps of moduli maps by allowing basepoints; the jump corresponds to the change of the base locus of the associated linear system. We apply this idea—apparently new in the context of Calabi–Yau fibrations—to produce K‑semistable log Fano quasimaps and a projective K‑moduli that we compare to the CM positivity on the Calabi–Yau fibration side.

   In this paper, for any closed immersion $\iota\colon X\hookrightarrow \mathbb{P}^N$, we regard quasimaps from $\mathbb{P}^1$ to $X$ as $N$-dimensional linear systems defining rational maps to $X$ (see Definition \ref{defn-qmaps} for the precise definition). 
Then we define K-{\it semistable log Fano quasimaps} as Definition \ref{de--kst--qmaps}.

\begin{thm}
[{$=$ Theorem \ref{thm--K-moduli-of-log-Fano-quasimaps}}]\label{main--thm--i}
Fix a closed immersion $\iota\colon X\hookrightarrow \mathbb{P}^N$.
Then a moduli stack $\mathcal{M}^{\mathrm{Kss,qmaps}}_\iota$ of K-semistable log Fano quasimaps from $\mathbb{P}^1$ to $X$ with fixed numerical data is an Artin stack of finite type.
Furthermore, $\mathcal{M}^{\mathrm{Kss,qmaps}}_\iota$ has the projective good moduli space ${M}^{\mathrm{Kps,qmaps}}_\iota$ with the ample CM line bundle $\Lambda^{\mathrm{qmaps}}_{\mathrm{CM}}$.
\end{thm}

We call $M_\iota^{\mathrm{Kps,qmaps}}$ the K-{\it moduli space of log Fano quasimaps}.
Note that we can define the CM line bundle for families of log Fano quasimaps (see Definition \ref{defn--qmaps--CM--linebundle}).
To produce the moduli space, we apply the good moduli theory \cite{AHLH}.

Let $M^{\mathrm{CYfib}}$ be the moduli space of polarized klt-trivial fibrations $f\colon (X,A)\to \mathbb{P}^1$ constructed in \cite{HH} with several fixed invariants (see Definition \ref{defn--HH-moduli} for the precise definition).
Let $\Lambda_{\mathrm{CM},t}$ the CM line bundle with respect to the polarization $A-rtK_{X}$ (see Definition \ref{defn--limit--CM} for more details) for some $r>0$.
Then we obtain the following.

\begin{thm}[{cf.~Theorems \ref{thm--quasi--map--construction--moduli--map}, \ref{thm--quasi-finiteness--of--two--moduli} and Corollary \ref{cor--final}}]\label{main--thm--ii}
Let $V$ be the coarse moduli space of polarized Abelian varieties or symplectic varieties smoothable to projective irreducible holomorphic symplectic manifolds, and let $M^{\mathrm{CYfib}}_{V}\subset M^{\mathrm{CYfib}}$ be the open and closed subspace corresponding to fibrations whose general fibers belong to $V$.

Then there exists a quasi-finite morphism 
    \[
   \alpha\colon (M^{\mathrm{CYfib}}_{V})^{\mathrm{sn}}\to {M}^{\mathrm{Kps.qmaps}}_{\iota},
    \]
    where $\iota\colon \overline{V}^{\mathrm{BB}}\hookrightarrow\mathbb{P}^L$ is a certain closed immersion of the Baily--Borel compactification of $V$.
    Furthermore, $\nu^*\Lambda_{\mathrm{CM},t}$ converges to $\mu\beta^*\alpha^*\Lambda_{\mathrm{CM}}^{\mathrm{qmaps}}$ in $\mathrm{Pic}_{\mathbb{Q}}((M^{\mathrm{CYfib}}_{V})^\nu)$ for some $\mu\in\mathbb{Q}_{>0}$, where $\beta\colon (M^{\mathrm{CYfib}}_{V})^{\nu}\to(M^{\mathrm{CYfib}}_{V})^{\mathrm{sn}}$ and $\nu\colon(M^{\mathrm{CYfib}}_{V})^\nu\to  M^{\mathrm{CYfib}}_{V}$ are the normalization morphisms.

    In particular, the seminormalization $(M^{\mathrm{CYfib}}_{V})^{\mathrm{sn}}$ is quasi-projective and $\nu^*\Lambda_{\mathrm{CM,t}}$ is ample for any sufficiently large $t$.
\end{thm}

As explained in Definition \ref{def-seminormal}, the seminormalization preserves the topological structure of the space.

\begin{rem}
    The quasi-projectivity of $(M^{\mathrm{CYfib}}_{V})^\nu$ or $(M^{\mathrm{CYfib}}_{V})^{\mathrm{sn}}$ does not necessarily imply the quasi-projectivity of the original $M^{\mathrm{CYfib}}_{V}$ a priori, as there exists an example for a non-ample line bundle on a non-proper scheme whose pullback to the normalization is ample \cite[Proposition 2]{Koex}.
    See Remark \ref{rem-Koex} for more details.
    
    We cannot expect the quasi-finiteness of $\alpha$ in Theorem \ref{main--thm--ii} for the moduli constructed in \cite{HH} in the case where the canonical divisors of klt--trivial fibrations are nef.
For details, see Remark \ref{rem--final}.
After this paper was put on arXiv, the second author showed the quasi-projectivity of the seminormalization of the K-moduli space in the case where the canonical divisors are nef in \cite{Hat26}.

In the proof of Theorem \ref{main--thm--ii}, we make use of the proof of the $\mathbf{B}$-semiampleness conjecture for klt--trivial fibrations with fibers Abelian varieties and primitive symplectic varieties by Fujino \cite{fujino--canonical-certain} and Kim \cite{kim-can-bundle-formula}.
\end{rem}

We explain how to associate K-stable quasimaps with uniformly adiabatically K-stable klt--trivial fibrations over curves $f\colon X\to \mathbb{P}^1$ of negative Kodaira dimension.
First of all, $f$ admits a moduli map $p\colon \mathbb{P}^1\to \overline{V}^{\mathrm{BB}}$. 
The associated maps $p$ do not behave well in deformation of $f$ because the degree of $p$ might jump as we explained earlier.
This phenomenon is not allowed in the moduli of stable maps \cite{K,FP} or twisted stable maps \cite{AV2,AV}. 
On the other hand, if you regard $p$ as a quasimap $q$, that is, a rational map with the fixed part being some multiple of the discriminant divisor $B:=\sum_{c\in\mathbb{P}^1}(1-\mathrm{lct}(X;f^{-1}(c)))[c]$, then the quasimaps $q$ behave well.
We can explain this phenomenon using the following example.

\begin{ex}
Let $f\colon X\to \mathbb{P}^1$ be a rational elliptic fibered surface with only canonical singularities.
Suppose that there exists a section $S$ of $f$ such that $S$ is $f$-ample.
As explained in \cite{Mi}, $X$ is a closed subscheme of $\mathbb{P}_{\mathbb{P}^1}(\mathcal{O}(-2)\oplus\mathcal{O}(-3)\oplus\mathcal{O})$ that is defined by the equation $y^2z=x^3+Axz^2+Bz^3$ for some $A\in H^0(\mathcal{O}_{\mathbb{P}^1}(4))$ and $B\in H^0(\mathcal{O}_{\mathbb{P}^1}(6))$.
Then we can define the moduli map $p\colon \mathbb{P}^1\to\mathbb{P}^1$ as $[A^3:4A^3+27B^2]$, eliminating the common zeros of $A^3$ and $4A^3+27B^2$, where $\mathbb{P}^1$ is the Baily--Borel compactification of the moduli of elliptic curves.
As \cite[Table 6.1]{Mi}, degree of $p$ varies from $0$ to $12$ in the irreducible family of $X$.
Hence, we cannot associate any family of maps in the sense of \cite{K,FP,AV} because they assumed the degree of maps to be fixed.

On the other hand, we can associate a family of quasimaps defined by $A^3$ and $4A^3+27B^2$ with the family of $f$. 
\end{ex}

Since $K_X$ is not nef, we see that the associated quasimap $q$ is log Fano. 
Moreover, we set the K-stability of $q$ so that it coincides with the uniform adiabatic K-stability of $f$ by \cite[Theorem 1.1]{Hat}.
This is the main philosophy of Theorem \ref{main--thm--ii}.

\begin{comment}

The {\it log--twisted K-stability}, which is a log generalization of twisted K-stability of Dervan \cite{De2}, is expected to be equivalent to the existence of special K\"ahler currents as the YTD conjecture.
In \cite{DS,RS,BBJ}, this problem was addressed in the log--twisted Fano case.
On the other hand, it is difficult to formulate K-moduli for log--twisted pairs.
In Example \ref{ex--S-compl--log--twisted}, we explain that there exists a family of K-semistable log-twisted Fano pairs that do not satisfy 
$\mathsf{S}$-completeness, which is necessary to apply good moduli theory (cf.~\cite{AHLH}).
This shows that we cannot expect to produce the K-moduli space of log--twisted Fano pairs in the usual sense as K-moduli (see \cite{Xu}).
Our K-moduli of quasimaps bypasses this subtlety.

We also note that the quasi-finiteness in Theorem \ref{main--thm--ii} is also a technical part. Although our proof uses the techniques of the minimal model program after \cite{BCHM}, this is inspired by Kodaira's result \cite{kodaira} that the structure of an arbitrary elliptic surface with a section is completely determined by the monodromy action and the moduli map from the base curve to $\mathbb{P}^1$.
This implies that elliptic surfaces with sections are determined by the moduli map up to discrete data.
\end{comment}

\subsection{Related works}\label{sec--related--works}

After the works in \cite{givental,DP2},
the notion of quasimaps has been generalized and studied by \cite{CFK,CFKM,CCFK} as a map to a GIT stack.
We do not deal with such generalized quasimaps in this paper.

Inspired by the works of \cite{KSB,Al} in the surface case, the moduli theory of canonically polarized schemes with only semi-log canonical singularities has been studied and established due to many dedications (see \cite{kollar-moduli}). This is one of the K-moduli by \cite{O}.
The ampleness of the CM line bundle was shown by Patakfalvi--Xu \cite{PX} by combining the works \cite{kollar-moduli-stable-surface-proj,fujino-semi-positivity,KP} on the projectivity of the moduli space.

The K-moduli space for log Fano pairs is constructed by the series of works \cite{BX,BLX,BHLLX,LXZ} etc, as a proper algebraic space.
The positivity of the CM line bundle was proved by Xu--Zhuang \cite{XZ1} after the work of \cite{CP,P} on the projectivity of proper subspaces in the uniform K-stable locus.

On the other hand, Viehweg \cite{viehweg95} constructed a moduli space of varieties with semiample canonical divisor and only canonical singularities as a quasi-projective scheme.
His moduli space can be regarded as one of the K-moduli spaces (\cite[Corollary 3.5]{Hat}).
Note that the ample line bundle, which he constructed, is not a CM line bundle if the log canonical divisor is not numerically trivial.

We note that K-moduli spaces of log Calabi--Yau pairs are not proper in general because log Calabi--Yau pairs with a semi-log canonical singularity are not K-polystable.
\cite{ABB+,BL24} constructed a projective moduli space of log Calabi--Yau surfaces with some conditions.
These moduli spaces are not K-moduli in the sense of Conjecture \ref{conj--k-moduli}.

On the K\"ahler geometric side, \cite{DN} extended the result of \cite{FS} to general K\"ahler manifolds with a cscK metric.

On the other hand, Ortu \cite{Ortb} constructed the moduli space of smooth fibrations with optimal symplectic connections, which was studied by \cite{DS1,DS2,Ort}.
Her moduli space also admits a positive current, and she did not put any assumption on the dimension and the K-stability on the base.
Note that \cite{Ortb} does not cover our case because our fibrations have singular fibers in general.

\subsection{Structure of the paper}
In Section \ref{sec--pre}, we collect fundamental notions of birational geometry, moduli theory, K-stability and klt--trivial fibration.
In Subsection \ref{subsec--moduli-pol}, we discuss on Abelian varieties and irreducible holomorphic symplectic manifolds.
We put fundamental results on deformation theory, period mapping theory, the Baily--Borel compactification of the moduli of them and the relationship between the moduli part and the Hodge line bundle on the Baily--Borel compactification, which is important to correspond a family of klt--trivial fibrations to the associated family of quasimaps.
On the other hand, we discuss the CM line bundle on the moduli of uniformly adiabatically K-stable klt--trivial fibrations over curves in Subsection \ref{subsec--k-moduli--klt--trivial}.
We show that if we take the polarization closer to a line bundle of the base, then the CM line bundle on a family over a normal base converges to some $\mathbb{Q}$-line bundle.
We define this as the {\em adiabatic limit of the CM line bundle}.

In Section \ref{sec3}, we define K-stability of log Fano quasimaps.
First, we observe that if the weight is sufficiently small and we regard a quasimap as a curve with a linear system, K-stability of the log Fano quasimap is equivalent to K-stability of the curve with the base point and the section.
By using this fact, we can detect K-stability of log Fano quasimaps by using the $\delta$-invariant.
On the other hand, we show that if the weight is large, the moduli stack of K-semistable log Fano quasimaps is embedded into the moduli stack with a sufficiently small weight.
By this fact, we show that the K-moduli theory of log Fano quasimaps is unconditional, and we show Theorem \ref{main--thm--i}.

Finally, in section \ref{sec4}, we show Theorem \ref{main--thm--ii}.
We can divide its proof into two parts, one part is construction of the correspondence and the other part is the quasi-finiteness.
For the first part (Theorem \ref{thm--quasi--map--construction--moduli--map}), we will use the period mapping theory explained in Section \ref{sec--pre}.
For the second part (Theorem \ref{thm--quasi-finiteness--of--two--moduli}), we apply stack theory and MMP.
\begin{ack} 
We thank Harold Blum, Makoto Enokizono, Osamu Fujino, Dai Imaike, Eiji Inoue, Kentaro Inoue, Yuki Koto, and Daisuke Matsushita for valuable discussions.
We are especially grateful to Yuki Koto for letting us know about quasimaps.
We thank Yuji Odaka for kindly informing us that Yoshiki Oshima and he showed a statement similar to Theorem \ref{thm--period--mapping--HK}. 
The second author also thanks Hyunsuk Kim for answering his questions.
We are also grateful to Tetsushi Ito for organizing schools ``Moduli of Hyper-K\"ahler manifolds and associated problems'', which gave us opportunities to learn irreducible holomorphic symplectic manifolds.
A large part of the work on this paper was completed when M.H.~was visiting Chalmers University of Technology.
This visit was partially supported by JSPS Overseas Challenge Program for Young Researchers.
 He thanks the university for the wonderful environment and Lars Martin Sektnan for his hospitality.

K.H. was partially supported by JSPS KAKENHI Grant Number JP22K13887. 
    M.H. was partially supported by JSPS KAKENHI 	22J20059  
(Grant-in-Aid for JSPS Fellows PD) and Royal Society International Collaboration Award ICA\textbackslash1\textbackslash231019.
\end{ack}
\section{Preliminaries}\label{sec--pre}

We work over an algebraically closed field $\mathbbm{k}$ of characteristic zero throughout this paper unless otherwise stated.

\subsection*{Notations and conventions} We collect the following terminologies. 
\begin{enumerate}
\item In this paper, we assume that all schemes are locally Noetherian and defined over $\mathbbm{k}$.
If we say that $S$ is a variety, we assume $S$ to be separated, of finite type over $\mathbbm{k}$, irreducible and reduced.
If we say that $S$ is a curve, we further assume $S$ to be one-dimensional.
\item Let $S$ and $T$ be schemes over $\mathbbm{k}$.
We write $S\times_{\mathrm{Spec}\mathbbm{k}}T$ as $S\times T$ for simplicity. 
We denote $\mathbb{A}^N_{\mathrm{Spec}\mathbbm{k}}\times S$ by $\mathbb{A}^N_{S}$. 
When $S=\mathrm{Spec}R$, we write $\mathbb{A}^N_{\mathrm{Spec}\mathbbm{k}}\times S$ as $\mathbb{A}^N_{R}$.
For $\mathbb{P}^N_{\mathrm{Spec}\mathbbm{k}}$ and $(\mathbb{G}_m)_{\mathrm{Spec}\mathbbm{k}}$, we will write $\mathbb{P}^N_{R}$,  $\mathbb{P}^N_{T}$, $\mathbb{G}_{m,R}$ and $\mathbb{G}_{m,T}$ in the same way.
If $R=\mathbbm{k}$ and there is no fear of confusion, we will omit $\mathbbm{k}$ and we write $\mathbb{A}^N$, $\mathbb{P}^N$, and so on.
    \item Let $S$ be a scheme and let $\mathsf{Sch}_S$ denote the category of all locally Noetherian schemes over $S$.
    Let $\mathsf{Sets}$ denote the category of sets.
    Let $S_{\mathrm{red}}$ denote the reduced scheme structure of $S$.
    Let $s\in S$ be a point.
    Then, let $\kappa(s)$ denote the residue field $\mathcal{O}_{S,s}/\mathfrak{m}$, where $\mathfrak{m}$ is the maximal ideal.
    \item Let $S$ be a scheme. Then a morphism $\mathrm{Spec}(\Omega)\to S$ from a spectrum of an algebraically closed field over $\mathbbm{k}$ is called a {\it geometric point}.
    If the image is $s\in S$, we write this as $\bar{s}\in S$ for simplicity.
    \item $\mathrm{Art}_{\mathbbm{k}}$ is the category of all Artinian local rings over $\mathbbm{k}$ with the residue field $\mathbbm{k}$. 
    \item Let $f\colon X\to Y$ be a morphism of schemes over $S$ and $\varphi\colon T\to S$ a morphism of schemes. 
    Let $L$ be a Cartier divisor on $X$. 
    Then, we will write the induced morphism $X\times_ST\to Y\times_{S}T$ from $f$ and $\varphi$ as $f_T\colon X_T\to Y_T$ and let $L_S$ denote the pullback of $L$ to $X_T$.
    If $s\in S$ is a point and $\varphi$ is $\mathrm{Spec}(\kappa(s))\to S$, then we write $f_T$ as $f_s\colon X_s\to Y_s$ and $L_T$ as $L_s$.
    If $\bar{s}$ is a geometric point of $S$ induced by $T=\mathrm{Spec}(\overline{\kappa(s)})$, then we will write $f_T$ as $f_{\bar{s}}\colon X_{\bar{s}}\to Y_{\bar{s}}$ and $L_T$ as $L_{\bar{s}}$.
    \item We say that a morphism $f\colon X\to Y$ is a {\it contraction} if $f$ is projective and $f_*\mathcal{O}_X\cong\mathcal{O}_Y$.
    We say that a birational map $f\colon X\dashrightarrow Y$ is a {\it birational contraction} if there is no $f^{-1}$-exceptional divisor.
    \item Let $f\colon X\to S$ be a morphism of schemes.
    Let $D_1$ and $D_2$ be $\mathbb{R}$-Weil divisors.
    If $D_1-D_2$ is an $\mathbb{R}$-linear (resp.~$\mathbb{Q}$-linear, $\mathbb{Z}$-linear) combination of principal divisors, then we say that $D_1$ and $D_2$ are $\mathbb{R}$ (resp.~$\mathbb{Q}$, $\mathbb{Z}$)-linearly equivalent over $S$.
For simplicity, if $D_1$ and $D_2$ are $\mathbb{Z}$-linearly equivalent, then we say that $D_1$ and $D_2$ are linearly equivalent.
We write them as $D_1\sim_{\mathbb{R},S}D_2$, $D_1\sim_{\mathbb{Q},S}D_2$ and $D_1\sim_{\mathbb{Z},S}D_2$ or simply $D_1\sim_{S}D_2$.
If $S=\mathrm{Spec}\mathbbm{k}$, we omit $S$.
On the other hand, if $f$ is projective, $D_1-D_2$ is $\mathbb{R}$-Cartier and $(D_1-D_2)\cdot C=0$ for any proper curve $C\subset X$ such that $f(C)$ is a point, we say that $D_1$ and $D_2$ are numerically equivalent over $S$ and write $D_1\equiv_S D_2$.
If $S=\mathrm{Spec}\mathbbm{k}$, then we say that $D_1$ and $D_2$ are numerically equivalent and write $D_1\equiv D_2$. 
\end{enumerate}

\subsection{Log pairs and klt-trivial fibration}

We collect fundamental terminologies of log pairs.
\begin{defn}[Singularities of log pairs]
    A {\it sublog pair} $(X,\Delta)$ consists of a normal variety $X$ and a $\mathbb{Q}$-divisor $\Delta$ on $X$ such that $K_X+\Delta$ is $\mathbb{Q}$-Cartier.
    If $\Delta$ is effective, we say that $(X,\Delta)$ is a {\em log pair}.

Let $(X,\Delta)$ be a log pair.
We say that $F$ is a {\em prime divisor over $X$} if there exists a projective birational morphism $\pi\colon Y\to X$ such that $F$ is a prime divisor on $Y$.

    Let $(X,\Delta)$ be a sublog pair.
    For any prime divisor $F$ over $X$, set the {\it log discrepancy} with respect to $F$ as 
    \[
    A_{(X,\Delta)}(F):=1+\mathrm{ord}_F(K_Y-\pi^*(K_X+\Delta)),
    \]
    where $\pi\colon Y\to X$ a projective birational morphism and $\mathrm{ord}_F$ is the divisorial valuation on $X$ associated with $F$.
    It is easy to see that $A_{(X,\Delta)}(F)$ is independent of the choice of $\pi\colon Y\to X$. 

    If $A_{(X,\Delta)}(F)>0$ (resp.~$\ge0$) for any $F$, we say that $(X,\Delta)$ is {\it subklt} (resp.~{\it sublc}). 
    If $\Delta$ is further effective, we say that $(X,\Delta)$ is  {\it Kawamata log terminal (klt)} (resp.~{\it log canonical (lc)}).
    We say that $(X,\Delta)$  has only terminal (resp.~canonical) singularities if $A_{(X,\Delta)}(F)>1$ (resp.~$\ge1$) for any prime divisor $F$ over $X$ but not on $X$.
    For more details, see \cite{KM}.
\end{defn}

\begin{defn}[Relative Mumford divisor \cite{kollar-moduli}]
    Let $f\colon X\to S$ be a flat projective morphism such that any geometric fiber is a deminormal scheme of dimension $n$.
    A subscheme $D$ is called an {\it effective relative Mumford divisor} if there exists an open subset $U\subset X$ that satisfies the following.
    \begin{enumerate}
        \item for any geometric point $\bar{s}\in S$, $\mathrm{codim}_{X_{\bar{s}}}(X_{\bar{s}}\setminus U)\ge2$,
        \item $D|_U$ is a relative Cartier divisor (i.e.~$D|_U$ is a Cartier divisor flat over $S$),
  \item $D$ is the closure of $D|_U$,
  \item $X_{\bar{s}}$ is smooth at any generic point of $D_{\bar{s}}$.
  \end{enumerate}
  If $\Delta$ is a $\mathbb{Z}$-linear combination of effective relative Mumford divisors, then we also say that $\Delta$ is a {\it relative Mumford divisor}.
  If $\Delta$ is a $\mathbb{Q}$-linear combination of relative Mumford divisors, we say that $\Delta$ is a relative Mumford $\mathbb{Q}$-divisor.
We note here that if $S$ is reduced and $K_{X/S}+\Delta$ is $\mathbb{Q}$-Cartier, for any morphism $T\to S$, we define a relative Mumford $\mathbb{Q}$-divisor $\Delta_T$ on $X\times_ST$ such that $\Delta_T|_{U\times_ST}$ is the pullback of $\Delta|_U$ as a $\mathbb{Q}$-Cartier divisor, where $U\subset X$ is an open subset such that for any geometric point $\bar{s}\in S$, $\mathrm{codim}_{X_{\bar{s}}}(X_{\bar{s}}\setminus U)\ge2$ and $\Delta|_U$ is a $\mathbb{Q}$-Cartier divisor (cf.~\cite[4.3]{kollar-moduli}).
\end{defn}

\begin{defn}[Log canonical threshold]
    Let $(X,\Delta)$ be a sublc pair and $D$ an effective $\mathbb{Q}$-Cartier $\mathbb{Q}$-divisor on $X$.
    We define the {\it log canonical threshold} of $(X,\Delta)$ with respect to $D$ by
    \[
    \mathrm{lct}(X,\Delta;D):=
    {\rm inf}\{t \in \mathbb{R}_{\geq 0} \, | \, \text{$(X,\Delta + tD)$ is sublc}\}.
    \]
\end{defn}

\begin{defn}[Log minimal model, log canonincal model]
Let $X \to Z$ be a projective morphism from a normal variety to a variety, and let $(X,\Delta)$ be an lc pair. 
Let $X' \to Z$ be a projective morphism from a normal variety $X'$ and let $\phi \colon X \dashrightarrow X'$ be a birational contraction over $Z$. 
Put $\Delta'=\phi_{*}\Delta$. 
Then $(X',\Delta')$ is called a {\em log minimal model} of $(X,\Delta)$ over $Z$ if 
\begin{itemize}
\item
$K_{X'}+\Delta'$ is $\mathbb{R}$-Cartier and nef over $Z$, and 
\item
for any $\phi$-exceptional prime divisor $E$ on $X$, we have $A_{(X,\Delta)}(E)<A_{(X',\Delta')}(E)$.
\end{itemize}
We say that $(X',\Delta')$ is a {\em log canonical model} of $(X,\Delta)$ over $Z$ if 
\begin{itemize}
\item
$K_{X'}+\Delta'$ is $\mathbb{R}$-Cartier and ample over $Z$, and 
\item
for any $\phi$-exceptional prime divisor $E$ on $X$, we have $A_{(X,\Delta)}(E)\leq A_{(X',\Delta')}(E)$.
\end{itemize}
For more details of the minimal model program (MMP), see \cite{BCHM}.
\end{defn}

The following lemma states that the log canonical models depend only on the numerical class of $K_X+\Delta$.
We will apply this in Section \ref{sec4}.

\begin{lem}\label{lem--lc-model}
Let $X \to Z$ be a projective morphism of normal quasi-projective varieties. 
Let $(X,\Delta_{1})$ and $(X,\Delta_{2})$ be lc pairs such that $K_{X}+\Delta_{1} \equiv_{Z} K_{X}+\Delta_{2}$. 
Suppose that $(X,\Delta_{1})$ has the log canonical model $(Y_{1},\Gamma_{1})$ over $Z$ (see \cite[Definition 3.50]{KM}). Then $(X,\Delta_{2})$ has the log canonical model $(Y_{2},\Gamma_{2})$ over $Z$ and $Y_{2}$ is isomorphic to $Y_{1}$ over $Z$. 
\end{lem}

\begin{proof}
By the definition of log canonical model, $K_{X}+\Delta_{1}$ is big over $Z$, and therefore $K_{X}+\Delta_{2}$ is big over $Z$. 
By \cite[Theorem 1.7]{hashizumehu}, there exists a sequence of steps of a $(K_{X}+\Delta_{1})$-MMP over $Z$
$$(X,\Delta_{1})\dashrightarrow (X',\Delta'_{1})$$
over $Z$ to a good minimal model $(X',\Delta'_{1})$ over $Z$. 
Then there is a birational morphism $\phi \colon X' \to Y_{1}$ such that $K_{X'}+\Delta'_{1}=\phi^{*}(K_{Y_{1}}+\Gamma_{1})$. 
Let $\Delta'_{2}$ be the birational transform of $\Delta_{2}$ on $X'$. 
By construction of MMP (see also \cite[Remark 6.1 (2)]{hashizumehu}), we see that $K_{X'}+\Delta'_{2}$ is $\mathbb{R}$-Cartier, and we have $K_{X'}+\Delta'_{1} \equiv_{Z} K_{X'}+\Delta'_{2}$. 
Then 
$$K_{X'}+\Delta'_{2}\equiv_{Z}K_{X'}+\Delta'_{1}=\phi^{*}(K_{Y_{1}}+\Gamma_{1})\equiv_{Y_{1}}0.$$
Hence, $K_{X'}+\Delta'_{2}$ is semi-ample over $Y_{1}$, in particular, $\phi_{*}(K_{X'}+\Delta'_{2})$ is $\mathbb{R}$-Cartier and $K_{X'}+\Delta'_{2}=\phi^{*}\phi_{*}(K_{X'}+\Delta'_{2})$. 
Then $\phi_{*}(K_{X'}+\Delta'_{2})\equiv_{Z}\phi_{*}(K_{X'}+\Delta'_{1})=K_{Y_{1}}+\Gamma_{1}$. 
Thus, $\phi_{*}(K_{X'}+\Delta'_{2})$ is ample over $Z$, and therefore, $K_{X'}+\Delta'_{2}$ is semi-ample over $Z$ and $Y_{1}$ is the contraction induced by $K_{X'}+\Delta'_{2}$. 
By construction, $(Y_{1},\phi_{*}\Delta'_{2})$ is the log canonical model of $(X,\Delta_{2})$ over $Z$. 
Therefore, Lemma \ref{lem--lc-model} holds. 
\end{proof}

We also give a remark on relative log MMP over more general rings.

\begin{rem}[Relative log MMP over a Dedekind domain]\label{rem-mmp-over-local-rings}
Let $R$ be a Dedekind domain essentially of finite type over $\mathbbm{k}$, and set $Z= \mathrm{Spec}\, R$. 
Let $\pi \colon X \to Z$ be a projective morphism from a normal variety $X$, and let $(X,\Delta)$ be an lc pair. 
Note that $X$ is not necessarily of finite type over $\mathbbm{k}$. 
We discuss how to run a $(K_{X}+\Delta)$-MMP over $Z$. 

Since $R$ is a regular ring of dimension one, there exists a smooth variety $V$ and a morphism $Z \to V$ corresponding to the localization of $V$ at a $({\rm dim}\,V-1)$-dimensional point. 
We fix an embedding $X \hookrightarrow \mathbb{P}^{N}_{Z}$ for some $N$ and the defining ideal sheaf $\mathcal{I}_{X}$ of $X$ in $\mathbb{P}^{N}_{Z}$. 
By considering the generators of $\mathcal{I}_{X}$, after shrinking $V$, we may find an ideal sheaf $\mathcal{J} \subset \mathcal{O}_{\mathbb{P}^{N}_{V}}$ such that the pullback of $\mathcal{J}$ by $\mathbb{P}^{N}_{Z} \to \mathbb{P}^{N}_{V}$ is equal to $\mathcal{I}_{X}$. 
By defining $Y$ to be the scheme defined by $\mathcal{J}$, we get a projective morphism $Y \to V$ such that the base change of the morphism by $Z \to V$ is $\pi \colon X \to Z$. 
By shrinking $V$, we may assume that $Y$ is a normal variety. 

Let $\Delta=\sum_{i}d_{i}D_{i}$ be the irreducible decomposition. 
By the same argument as above, after shrinking $V$, for every $i$ we get a prime divisor $G_{i}$ whose pullback by $X \to Y$ as a divisorial scheme is $D_{i}$. 
Set $\Gamma=\sum_{i}d_{i}G_{i}$. 
By shrinking $V$, we may assume that $K_{Y}+\Gamma$ is $\mathbb{R}$-Cartier, and the pullback of $K_{Y}+\Gamma$ by $X \to Y$ is $K_{X}+\Delta$. 
Since $(X,\Delta)$ is lc, shrinking $Z$ again, we may assume that $(Y,\Gamma)$ is lc. 

Now we have a projective morphism $(Y,\Gamma) \to V$ whose pullback by $Z \to V$ is equal to $(X,\Delta) \to Z$. 
For any sequence of steps of a $(K_{Y}+\Gamma)$-MMP over $V$, the base change of the MMP by $Z \to V$ is a sequence of steps of a $(K_{X}+\Delta)$-MMP over $Z$.  
\end{rem}

Now, we recall some notions concerned with the canonical bundle formula.
\begin{defn}\label{defn--can-bundle-formula}
    Let $(X,\Delta)$ be a sublc pair and $f\colon  X\to S$ a contraction of normal varieties.
    We define {\em the discriminant $\mathbb{Q}$-divisor} $B$ on $S$ with respect to $f\colon (X,\Delta)\to S$ as follows.
    For any point $\eta\in S$ of codimension one, let $F_{\eta}$ be the Zariski closure of $\eta$ in $S$, which is a prime divisor, and let $\gamma_\eta$ be the log canonical threshold of $f^{*}F_{\eta}$ with respect to $(X,\Delta)$ over $\eta$. 
Note that $\gamma_{\eta}$ is well defined because $F_{\eta}$ is Cartier on an open neighborhood of $\eta$.   
    Then, the discriminant $\mathbb{Q}$-divisor $B$ is defined by 
    \[
    B:=\sum_{\eta}(1-\gamma_\eta)F_{\eta},
    \]
    where $\eta$ runs over all points $\eta\in S$ of codimension one.
    We note that the above is a finite sum.

  A {\em sublc--trivial fibration}, denoted by $f\colon (X,\Delta)\to S$, consists of a sublc pair $(X,\Delta)$ and a projective contraction $f \colon X \to S$ satisfying the following conditions.
    \begin{enumerate}
        \item $K_X+\Delta\sim_{\mathbb{Q}}0$, and
        \item $\mathrm{rank}\,f_*\mathcal{O}_X(\lceil \mathbf{A}^*(X,\Delta)\rceil)=1$.
    \end{enumerate}
    Here, the coherent sheaf $\mathcal{O}_X(\lceil \mathbf{A}^*(X,\Delta)\rceil)$ is defined as in \cite[Lemma 3.22]{fujino-bpf}. 
    In other words, for any log resolution $\pi\colon Y\to X$ of $(X,\Delta)$, if $K_Y=\pi^*(K_X+\Delta)+\sum_ia_iE_i$ and $\pi_*(\sum_ia_iE_i)=-\Delta$, then we set
 \[
\mathcal{O}_X(\lceil \mathbf{A}^*(X,\Delta)\rceil)=\pi_*\mathcal{O}_Y\left(\sum_{a_i\ne-1}\lceil a_i\rceil E_i\right).
\]
The coherent sheaf $\mathcal{O}_X(\lceil \mathbf{A}^*(X,\Delta)\rceil)$ is independent of $\pi$. 
The definition is the same as that of lc-trivial fibrations in \cite[Section 2]{FG} (see also \cite{A}). 
When $(X,\Delta)$ is further lc (resp.~klt, subklt), we say that $f$ is an {\em lc-trivial fibration} (resp.~{\em klt-trivial fibration}, {\em subklt--trivial fibration}). 
 Note that the base varieties of sublc--trivial fibrations are not assumed to be projective in this paper. 
 A {\em polarized sublc-trivial fibration} is a pair of a sublc-trivial fibration $f\colon (X,\Delta)\to S$ and an $f$-ample $\mathbb{Q}$-divisor (or an $f$-ample $\mathbb{Q}$-line bundle) $A$ on $X$, and we denote it by $f\colon (X,\Delta,A)\to S$. 
Similarly, we define a {\em polarized lc-trivial fibration}, and so on.  

Let $f\colon (X,\Delta)\to S$ be a sublc--trivial fibration and $b\in\mathbb{Z}_{>0}$ the smallest positive integer such that $H^0(F,\mathcal{O}_F(b(K_F+\Delta_F)))\ne0$, where $(F,\Delta_F)$ is the geometric generic fiber of $f$.
Then, we define the {\em moduli $\mathbb{Q}$-divisor} $M$ on $S$ so that the following so-called {\em canonical bundle formula}
\[
K_X+\Delta+\frac{1}{b}{\rm div}(\varphi)=f^*(K_S+M+B),
\]
 where ${\rm div}(\varphi)$ is a principal divisor determined by a rational function $\varphi$ of $X$.
 
Let $f\colon (X,\Delta)\to S$ be a sublc-trivial fibration. 
Fix a $\mathbb{Q}$-Cartier divisor $D$ on $S$ such that $K_{X}+\Delta \sim_{\mathbb{Q}}f^{*}D$.  
Let $S' \to S$ be a projective birational morphism from a normal variety $S'$, let $X' \to X\times_{S}S'$ be a resolution of the main component of $X \times_{S}S'$, and let $\Delta'$ be an $\mathbb{R}$-divisor on $X'$ such that $K_{X'}+\Delta'$ is the pullback of $K_{X}+\Delta$ to $X'$. 
Then $(X',\Delta') \to S'$ is a sublc-trivial fibration, and therefore we may define the associated discriminant $\mathbb{Q}$-divisor $B'$ and moduli $\mathbb{Q}$-divisor $M'$. 
Since we fix $D$, by considering all projective birational morphism to $S$ we can define the {\em discriminant $\mathbb{Q}$-b-divisor} $\boldsymbol{\rm B}$ and the {\em moduli $\mathbb{Q}$-b-divisor} $\boldsymbol{\rm M}$ (for the definition of $b$-divisor, see \cite{A}). 
By \cite[Theorem 1.2]{fujino-slc-trivial}, there exists a projective birational morphism $S_{0} \to S$ such that for any projective birational morphism $\phi \colon S'' \to S_{0}$, the following holds: Putting $B_{0}$, $M_{0}$, $B''$, and $M''$ as the traces of $\boldsymbol{\rm B}$ and $\boldsymbol{\rm M}$ on $S_{0}$ and $S''$, then $K_{S''}+B''=\phi^{*}(K_{S_{0}}+B_{0})$, and there exists an open immersion $S_{0} \hookrightarrow \overline{S}_{0}$ and a nef $\mathbb{Q}$-divisor $\overline{M}_{0}$ on $\overline{S}_{0}$ such that $M_{0}=\overline{M}_{0}|_{S_{0}}$. 
We call such $S_{0}$ an {\em Ambro model}. 
We note that Ambro model is not unique. 
\end{defn}

\begin{defn}[Generalized log pair, {\cite{bz}}]
 A {\em generalized log pair} $(X,B+M)$ consists of 
\begin{itemize}
\item
a normal variety $X$ equipped with a projective morphism $X \to Z$, 
\item
an $\mathbb{R}$-divisor $B \geq 0$ on $X$, and 
\item
an $\mathbb{R}$-Cartier divisor $M'$ on a normal variety $X'$ with a projective birational morphism $\phi \colon X'\to X$
\end{itemize}
such that $M'$ is nef over $Z$ and $K_{X}+B+M$ is $\mathbb{R}$-Cartier, where $M:=\phi_{*}M'$. 
When $Z$ is a point, we drop it and we say that $(X,B+M)$ is a projective generalized log pair. 
When $M'=0$, we may identify $(X,B+M)$ with the log pair. 

Although we do not use in this paper, singularities of generalized log pairs can be defined as in the case of usual pairs. For details, see \cite[Section 4]{bz}. 
\end{defn}

\begin{ex}
Let $f\colon (X,\Delta)\to S$ be an lc-trivial fibration and $S \to Z$ a projective morphism. 
It is known that the discriminant $\mathbb{Q}$-b-divisor $\boldsymbol{\rm B}$ and the moduli $\mathbb{Q}$-b-divisor $\boldsymbol{\rm M}$ form a generalized log pair $(S, \boldsymbol{\rm B}_{S} + \boldsymbol{\rm M}_{S})$ with the data $S_{0} \to S \to Z$, where $S_{0}$ is an Ambro model.  
\end{ex}

\begin{lem}\label{lem--compactification-lc-trivial-fib}
Let $f\colon (X,\Delta,A)\to S$ be a polarized lc-trivial fibration such that $S$ is quasi-projective, $A$ is an effective $\mathbb{Q}$-Cartier Weil divisor, and the pair $(X,\Delta+tA)$ is lc for some $t \in \mathbb{R}_{>0}$. 
Let $S \hookrightarrow \overline{S}$ be an open immersion to a normal projective variety $\overline{S}$. 
Then there exist an open immersion $X \hookrightarrow \overline{X}'$ to a normal projective variety $\overline{X}'$, a polarized lc-trivial fibration $\overline{f}'\colon (\overline{X}',\overline{\Delta}',\overline{A}') \to \overline{S}'$ to a normal projective variety $\overline{S}'$, and a birational morphism $\overline{S}' \to \overline{S}$ such that $\overline{S}' \to \overline{S}$ is an isomorphism over $S \subset \overline{S}$, $\overline{f}'|_{X}=f$, $\overline{\Delta}'|_{X}=\Delta$, $\overline{A}'$ is an effective Weil divisor and $\overline{A}'|_{X}=A$, and the pair $(\overline{X}',\overline{\Delta}'+t\overline{A}')$ is lc for some $t \in \mathbb{R}_{>0}$.  
In particular, we have the following diagram:
$$
\xymatrix
{
(X,\Delta) \ar[d]_{f} \ar@{^{(}->}[r] & (\overline{X}',\overline{\Delta}') \ar[d]^{\overline{f}'} 
\\
S \ar@{^{(}->}[r] & \overline{S}'. 
}
$$
Furthermore, any lc center of $(\overline{X}',\overline{\Delta}')$ and component of $\overline{\Delta}'$ intersects $X$. 
\end{lem} 

\begin{proof}
By \cite[Corollary 1.3]{has-mmp}, there exist an open immersion $X \hookrightarrow X^{c}$ to a normal projective variety $X^{c}$, an lc pair $(X^{c},\Delta^{c})$, and a projective morphism $f^{c} \colon (X^{c},\Delta^{c}) \to \overline{S}$ such that $f^{c}|_{X}=f$, and $\Delta^{c}|_{X}=\Delta$, and any lc center of $(X^{c},\Delta^{c})$ intersects $X$. 
By construction of $(X^{c},\Delta^{c})$ (see \cite[Proof of Corollary 1.3]{has-mmp}),  any component of $\Delta^{c}$ also intersects $X$. 
Since $K_{X}+\Delta \sim_{\mathbb{R},\,S}0$, the lc pair $(X,\Delta)$ itself is a good minimal model of $(X,\Delta)$ over $S$. 
By \cite[Theorem 1.2]{has-mmp} and running a $(K_{X^{c}}+\Delta^{c})$-MMP over $S$, we get a good minimal model $(\overline{X},\overline{\Delta})$ of $(X^{c},\Delta^{c})$ over $\overline{S}$. 
Since $(K_{\overline{X}}+\overline{\Delta})|_{X}$ is in particular nef over $S$, the MMP does not modify $f^{c}|_{X}=f$. 
Let $\overline{f} \colon \overline{X} \to \overline{S}'$ be the contraction over $\overline{S}$ defined by $K_{\overline{X}}+\overline{\Delta}$. 
Then $\overline{f}|_{X}=f$ since $(K_{\overline{X}}+\overline{\Delta})|_{X} \sim_{\mathbb{R},\,S}0$. 
By construction, it is easy to check that the lc-trivial fibration $\overline{f}\colon (\overline{X},\overline{\Delta}) \to \overline{S}'$ and the morphism $\overline{S}' \to \overline{S}$ satisfy the conditions of Lemma \ref{rem--compactification-lc-trivial-fib}. 

Let $g \colon (\tilde{Y},\tilde{\Gamma}) \to (\overline{X},\overline{\Delta})$ be a dlt blow-up of $(\overline{X},\overline{\Delta})$, and put $Y = g^{-1}(X)$. 
Let $\tilde{A}$ be the closure of $(g|_{Y})^{*}A$. 
Since $(X,\Delta+tA)$ is lc for some $t>0$ and any lc center of $(\overline{X},\overline{\Delta})$ intersects $X$, we can find $\epsilon>0$ such that $(\tilde{Y},\tilde{\Gamma}+\epsilon \tilde{A})$ is an lc pair whose lc centers intersect $Y$.  
By construction, $(K_{\tilde{Y}}+\tilde{\Gamma}+\epsilon \tilde{A})|_{Y}$ is big and semi-ample over $S$. 
Hence, \cite[Theorem 1.2]{has-mmp} implies that $(\tilde{Y},\tilde{\Gamma}+\epsilon \tilde{A})$ has the lc model $(\overline{X}',\overline{\Delta}'+\epsilon\overline{A}')$ over $\overline{S}'$. 
Let $\overline{f}'\colon \overline{X}' \to \overline{S}'$ be the induced morphism. 
Since $K_{\overline{X}}+\overline{\Delta} \sim_{\mathbb{R},\,\overline{S}'}0$, we have $K_{\tilde{Y}}+\tilde{\Gamma} \sim_{\mathbb{R},\,\overline{S}'}0$, and therefore $K_{\overline{X}'}+\overline{\Delta}' \sim_{\mathbb{R},\,\overline{S}'}0$. 
From this, we see that $\overline{f}'\colon (\overline{X}',\overline{\Delta}') \to \overline{S}'$ is an lc-trivial fibration. 
Since $K_{X}+\Delta+\epsilon A$ is ample over $S$, it follows that $(X,\Delta+\epsilon A)$ is the lc model of $(Y, (\tilde{\Gamma}+\epsilon \tilde{A})|_{Y})$ over $S$. 
This fact and the construction of $(\overline{X}',\overline{\Delta}'+\epsilon\overline{A}') \to \overline{S}'$ show that $X \subset \overline{X}'$, $\overline{\Delta}'|_{X}=\Delta$, and $\epsilon \overline{A}'|_{X}= \epsilon A$. 
We can easily check that $\overline{f}'\colon (\overline{X}',\overline{\Delta}',\overline{A}') \to \overline{S}'$ and $\overline{S}' \to \overline{S}$ are what we wanted to construct. 
\end{proof}

The following corollary is used in the proof of Theorem \ref{thm--Fujino's--period--mapping--theory}. 
The difference between Lemma \ref{lem--compactification-lc-trivial-fib} and Corollary \ref{cor--compactification-lc-trivial-fib-2} below is that $A$ in Corollary \ref{cor--compactification-lc-trivial-fib-2} does not have any condition but $S$ is assumed to be smooth. 

\begin{cor}\label{cor--compactification-lc-trivial-fib-2}
Let $f\colon (X,\Delta,A)\to S$ be a polarized lc-trivial fibration such that $S$ is smooth and quasi-projective. 
Let $S \hookrightarrow \overline{S}$ be an open immersion to a smooth projective variety $\overline{S}$. 
Then there exist an open immersion $X \hookrightarrow \overline{X}'$ to a normal projective variety $\overline{X}'$, a polarized lc-trivial fibraton $\overline{f}'\colon (\overline{X}',\overline{\Delta}', \overline{A}') \to \overline{S}'$ to a normal projective variety $\overline{S}'$, and a birational morphism $\overline{S}' \to \overline{S}$ such that $\overline{S}' \to \overline{S}$ is an isomorphism over $S \subset \overline{S}$, $\overline{f}'|_{X}=f$, $\overline{\Delta}'|_{X}=\Delta$, and $\overline{A}'|_{X}=A$.  
Furthermore, any lc center of $(\overline{X}',\overline{\Delta}')$ and component of $\overline{\Delta}'$ intersects $X$. 
\end{cor}

\begin{proof}
Since $A$ is $f$-ample, there exist a rational number $q$, a rational function $\psi$ on $X$, and a $\mathbb{Q}$-Cartier divisor $H$ on $S$ such that
$$E:=A+f^{*}H+q \cdot {\rm div}(\psi)$$
is effective and $(X,\Delta+tE)$ is lc for some $t>0$. 
By Lemma \ref{lem--compactification-lc-trivial-fib}, we can find an open immersion $X \hookrightarrow \overline{X}'$ to a normal projective variety $\overline{X}'$, a polarized lc-trivial fibraton $\overline{f}'\colon (\overline{X}',\overline{\Delta}', \overline{E}') \to \overline{S}'$ to a normal projective variety $\overline{S}'$, and a projective birational morphism $g \colon \overline{S}' \to \overline{S}$ such that $g \colon \overline{S}' \to \overline{S}$ is an isomorphism over $S \subset \overline{S}$, $\overline{f}'|_{X}=f$, $\overline{\Delta}'|_{X}=\Delta$, and $\overline{E}'|_X=E$. 
Let $\overline{H}$ be the closure of $H$ in $\overline{S}$, and we put $\overline{H}':=g^{*}\overline{H}$. 
Note that $\overline{H}$ is $\mathbb{Q}$-Cartier since $\overline{S}$ is smooth. 
By thinking $\psi$ of a rational function on $\overline{X}'$, we define
$$\overline{A}':=\overline{E}'-\overline{f}'^{*}\overline{H}'-q \cdot {\rm div}(\psi).$$
Then $\overline{A}'$ is $\mathbb{Q}$-Cartier and $\overline{f}'$-ample, and $\overline{A}'|_{X}=E-f^{*}H-q \cdot {\rm div}(\psi)=A$ by the fact $\overline{f}'|_{X}=f$ and the definition of $\overline{H}'$. 
The polarized lc-trivial fibration $\overline{f}'\colon (\overline{X}',\overline{\Delta}', \overline{A}') \to \overline{S}'$ and the morphism $\overline{S}' \to \overline{S}$ are what we wanted.
\end{proof}

\begin{rem}\label{rem--compactification-lc-trivial-fib}
Let $f\colon (X,\Delta, A)\to S$ be a polarized lc-trivial fibration such that $S$ is quasi-projective, and let   
$$
\xymatrix
{
(X,\Delta) \ar[d]_{f} \ar@{^{(}->}[r] & (\overline{X}',\overline{\Delta}') \ar[d]^{\overline{f}'} 
\\
S \ar@{^{(}->}[r] & \overline{S}'. 
}
$$
be the diagram in Lemma \ref{lem--compactification-lc-trivial-fib} or Corollary \ref{cor--compactification-lc-trivial-fib-2}. 
Let $\boldsymbol{\rm B}$ and $\boldsymbol{\rm M}$ (resp.~$\overline{\boldsymbol{\rm B}}'$ and $\overline{\boldsymbol{\rm M}}'$) be the discriminant $\mathbb{Q}$-b-divisor and the moduli $\mathbb{Q}$-b-divisor with respect to $f$ (resp.~$\overline{f}'$). 
Then $\overline{\boldsymbol{\rm B}}'|_{S}=\boldsymbol{\rm B}$ and $\overline{\boldsymbol{\rm M}}'|_{S}=\boldsymbol{\rm M}$ by definition.
\end{rem}

We can easily check that Lemma \ref{lem--compactification-lc-trivial-fib}, Corollary \ref{cor--compactification-lc-trivial-fib-2}, and Remark \ref{rem--compactification-lc-trivial-fib} hold even if $\Delta$ and $A$ are not necessarily a $\mathbb{Q}$-divisor. 

\subsection{K-stability}

We say that a triple $(X,\Delta,L)$ is a {\it polarized normal pair} if $(X,\Delta)$ is a log pair and $L$ is an ample $\mathbb{Q}$-line bundle on $X$.
The following concept is necessary to define K-stability.
\begin{defn}[Test configuration]
    We say that $(\mathcal{X},\mathcal{L})$ is a {\it (semi)ample test configuration} for $(X,\Delta,L)$ if $(\mathcal{X},\mathcal{L})$ satisfies the following conditions.
    \begin{enumerate}
        \item There exists a projective flat morphism $\pi\colon \mathcal{X}\to\mathbb{A}^1$ and $\mathcal{L}$ is a $\pi$-(semi)ample $\mathbb{Q}$-line bundle on $\mathcal{X}$,
        \item $(\mathcal{X},\mathcal{L})$ admits a $\mathbb{G}_m$-action such that $\pi$ is $\mathbb{G}_m$-equivariant, where $\mathbb{A}^1$ admits the natural action by multiplication,
        \item $(\pi^{-1}(1),\mathcal{L}|_{\pi^{-1}(1)})$ is isomorphic to $(X,L)$.
    \end{enumerate}
    If $\mathcal{X}\cong X\times\mathbb{A}^1$, we say that $\mathcal{X}$ is {\it of product type}.
    If $\mathcal{X}\cong X\times\mathbb{A}^1$ with the trivial $\mathbb{G}_m$-action on the first component $X$, then we say that $\mathcal{X}$ is {\it trivial}.
    $X_{\mathbb{A}^1}$ denotes the trivial test configuration and $L_{\mathbb{A}^1}:=p^*L$, where $p\colon X_{\mathbb{A}^1}\to X$ is the natural projection.
    If $\mathcal{X}$ is further normal, we say that $(\mathcal{X},\mathcal{L})$ is a {\it normal (semi)ample test configuration}.
    Then, gluing $(X_{\mathbb{A}^1},L_{\mathbb{A}^1})$ with $(\mathcal{X},\mathcal{L})$, we obtain the natural $\mathbb{G}_m$-equivariant compactification $(\overline{\mathcal{X}},\overline{\mathcal{L}})$ of $(\mathcal{X},\mathcal{L})$ over $\mathbb{P}^1$ such that the restriction of $(\overline{\mathcal{X}},\overline{\mathcal{L}})$ to $\mathbb{P}^1\setminus\{0\}$ is isomorphic to $(X_{\mathbb{A}^1},L_{\mathbb{A}^1})$.
    Let $\mathrm{dim}\,X=d$.
    If $\mathcal{X}$ is normal, we set the {\it Donaldson--Futaki invariant} as
    \[
    \mathrm{DF}_{\Delta}(\mathcal{X},\mathcal{L}):=\frac{1}{L^d}\left((K_{\overline{\mathcal{X}}/\mathbb{P}^1}+\mathcal{D})\cdot \overline{\mathcal{L}}^{d}-\frac{d(K_{X}+\Delta)\cdot L^{d-1}}{(d+1)L^{d}}\cdot \overline{\mathcal{L}}^{d+1}\right),
    \]
    where $\mathcal{D}$ is the $\mathbb{Q}$-divisor, which is obtained as the Zariski closure of $\Delta\times \mathbb{G}_m$ in $\overline{\mathcal{X}}$.

    Furthermore, we set the following invariant.
    We note that there exists the canonical $\mathbb{G}_m$-equivariant birational map $\mathcal{X}\dashrightarrow X_{\mathbb{A}^1}$.
    Consider the normalization $\mathcal{Y}$ of the birational map.
    If we let $\sigma\colon \mathcal{Y}\to\mathcal{X}$ be the natural projection, then $(\mathcal{Y},\sigma^*\mathcal{L})$ is also a normal semiample test configuration.
    Let $q\colon \overline{\mathcal{Y}}\to X$ be the natural projection.
    Then, we set
    \[
    J^{\mathrm{NA}}(\mathcal{X},\mathcal{L})=\frac{1}{L^d}(\overline{\sigma^*\mathcal{L}}-q^*L)\cdot \overline{\sigma^*\mathcal{L}}^d-\frac{1}{(d+1)L^{d}}\cdot \overline{\mathcal{L}}^{d+1}.
    \]
    We note that $J^{\mathrm{NA}}(\mathcal{X},\mathcal{L})>0$ for all ample normal test configuration except trivial test configurations by \cite[Proposition 7.8]{BHJ}.
    On the other hand, we can set the following invariant for any $\mathbb{Q}$-line bundle $T$ on $X$:
    \begin{equation*}
        \mathcal{J}^{T,\mathrm{NA}}(\mathcal{X},\mathcal{L}):=\frac{1}{L^d}\left(q^*L\cdot\overline{\sigma^*\mathcal{L}}^d-\frac{dT\cdot L^{d-1}}{(d+1)L^d}\overline{\sigma^*\mathcal{L}}^{d+1} \right),
    \end{equation*}
    which we call the non--Archimedean $\mathrm{J}^T$-functional.
    For more details, see \cite{BHJ}.
\end{defn}
Now, we can define K-stability for polarized log pairs.
\begin{defn}[K-stability]
    Let $(X,\Delta,L)$ be a polarized log pair.
    We say that $(X,\Delta,L)$ is {\it K-semistable} if $\mathrm{DF}_{\Delta}(\mathcal{X},\mathcal{L})\ge0$ for any ample test configuration.
    $(X,\Delta,L)$ is called {\it uniformly K-stable} if there exists $\delta>0$ such that $\mathrm{DF}_{\Delta}(\mathcal{X},\mathcal{L})\ge\delta J^{\mathrm{NA}}(\mathcal{X},\mathcal{L})$ for any ample normal test configuration.
    Furthermore, $(X,\Delta,L)$ is called {\it K-stable} if $\mathrm{DF}_{\Delta}(\mathcal{X},\mathcal{L})>0$ for any ample non trivial normal test configuration.
    Finally, we say that $(X,\Delta,L)$ is {\it K-polystable} if $(X,\Delta,L)$ is K-semistable and for any ample normal test configuration $(\mathcal{X},\mathcal{L})$, $\mathrm{DF}_{\Delta}(\mathcal{X},\mathcal{L})=0$ if and only if $\mathcal{X}$ is of product type.
\end{defn}

We note that if $(X,\Delta,L)$ is K-semistable, then $(X,\Delta)$ has only lc singularities by \cite{GIToda} and \cite[Theorem 6.1]{OS}.
If it further holds that $L=-K_X-\Delta$, then $(X,\Delta)$ is klt.
We say that $(X,\Delta,L)$ is a {\it log Fano pair} if $L=-K_X-\Delta$ and $(X,\Delta)$ is klt.

Now, we set the $\delta$-invariant, which is first introduced by \cite{FO} and detects completely K-stability of log Fano pairs. 
\begin{defn}\label{defn--delta}
    Let $(X,\Delta,L)$ be  a polarized klt pair of dimension $d$.
    For any prime divisor $E$ over $X$, we choose a projective birational morphism $\pi\colon Y\to X$ from a smooth variety such that $E$ is a prime divisor on $Y$.
    Then, we set
    \[
    S_L(E):=\frac{1}{L^d}\int^{\infty}_0\mathrm{vol}_Y(\pi^*L-xE)dx.
    \]
    The above invariant does not depend on the choice of $\pi$.
    Now, we set the {\it $\delta$-invariant} as
    \[
    \delta(X,\Delta,L):=\inf_{E}\frac{A_{(X,\Delta)}(E)}{S_L(E)},
    \]
    where $E$ runs over all prime divisors over $X$.
\end{defn}

\begin{thm}[{\cite[Theorem B]{BlJ}}]
Let $(X,\Delta,-K_X-\Delta)$ be a log Fano pair.
Then the following hold.
\begin{enumerate}
    \item $(X,\Delta,-K_X-\Delta)$ is K-semistable if and only if $\delta(X,\Delta,-K_X-\Delta)\ge1$, and 
    \item $(X,\Delta,-K_X-\Delta)$ is uniformly K-stable if and only if $\delta(X,\Delta,-K_X-\Delta)>1$.
\end{enumerate}
\end{thm}

\begin{rem}
The $\delta$-invariant is first introduced by Fujita--Odaka \cite{FO} and Blum--Jonsson reformulated it into the form as in Definition \ref{defn--delta} by showing \cite[Theorem 4.4]{BlJ}.
    For log Fano pairs, the K-stability and the uniform K-stability are equivalent by \cite[Theorem 1.6]{LXZ}.
\end{rem}

We introduce the following K-stability for klt--trivial fibrations.

\begin{defn}\label{defn--unif--ad--Kst}
    Let $f\colon (X,\Delta)\to S$ be a klt--trivial fibration of proper normal varieties.
    Let $A$ be an $f$-ample $\mathbb{Q}$-line bundle on $X$ and $L$ an ample $\mathbb{Q}$-line bundle on $S$.
    We call $f\colon (X,\Delta,A)\to (S,L)$ a polarized klt--trivial fibration.
    If there exist real numbers $\delta>0$ and $\epsilon_0>0$ such that  
    \[
    \mathrm{DF}_{\Delta}(\mathcal{X},\mathcal{M})\ge\delta J^{\mathrm{NA}}(\mathcal{X},\mathcal{M})
    \]
    for any $\epsilon\in(0,\epsilon_0)\cap\mathbb{Q}$ and ample normal test configuration $(\mathcal{X},\mathcal{M})$ for $(X,\epsilon A+f^*L)$, then we say that $f\colon (X,\Delta,A)\to (S,L)$ is {\em uniformly adiabatically K-stable}.
\end{defn}

If $S=\mathbb{P}^1$ and $\Delta=0$, we choose $L$ to be $\mathcal{O}(1)$ and we simply write $f\colon (X,A)\to \mathbb{P}^1$. 

\begin{thm}[{\cite[Theorem 1.1]{Hat}}]
    Let $f\colon (X,A)\to \mathbb{P}^1$ be a polarized klt--trivial fibration over $\mathbb{C}$ such that $-K_X$ is nef but not numerically trivial.
    Then, $f$ is uniformly adiabatically K-stable if and only if $\delta(\mathbb{P}^1,B,\mathcal{O}(1))>\mathrm{deg}(-K_{\mathbb{P}^1}-B-M)$, where $B$ and $M$ are the discriminant and the moduli $\mathbb{Q}$-divisors respectively.
\end{thm}

The latter condition is nothing but the K-stability of the log Fano pair $(\mathbb{P}^1,B+D,-K_{\mathbb{P}^1}-B-D)$, where $D$ is a general effective $\mathbb{Q}$-divisor $\mathbb{Q}$-linearly equivalent to $M$.
Here, we note that $M$ is semiample by \cite[Theorem 0.1]{A}.

\begin{defn}\label{defn--unif--ad--Kst2}
For general $\mathbbm{k}$, we say that a polarized klt--trivial fibration $f\colon (X,A)\to\mathbb{P}^1$ such that $-K_X$ is not nef, is {\it uniformly adiabatically K-stable} if $\delta(\mathbb{P}^1,B,\mathcal{O}(1))>\mathrm{deg}(-K_{\mathbb{P}^1}-B-M)$, where $B$ and $M$ are the discriminant and the moduli $\mathbb{Q}$-divisors respectively.   
\end{defn}

\subsection{Moduli schemes and functors}\label{Subsection--Hilb}

We collect the following notions about moduli schemes.

Let $f\colon \mathscr{X}\to S$ be a projective morphism with an $f$-ample line bundle $\mathscr{A}$.
Fix a polynomial $p$ with rational coefficients and a coherent sheaf $\mathscr{F}$ on $\mathscr{X}$.
Consider the following functor $\mathfrak{Quot}^{p,\mathscr{A}}_{\mathscr{X}/S,\mathscr{F}}\colon(\mathsf{Sch}_S)^{\mathrm{op}}\to \mathsf{Sets}$: \[
\mathfrak{Quot}^{p,\mathscr{A}}_{\mathscr{X}/S,\mathscr{F}}(T):=\left\{
\begin{array}{l}
q\colon \mathscr{F}_T\to\mathcal{G}
\end{array}
\middle|
\begin{array}{l}
\text{$\mathcal{G}$ is a coherent sheaf of $\mathscr{X}_T$ flat over $T$}\\
\text{whose fibers have the same Hilbert}\\
\text{polynomial $p$ with respect to $\mathscr{A}$ and $q$}\\
\text{is a surjective $\mathcal{O}_{\mathscr{X}_T}$-linear map.}
\end{array}\right\}/\sim,
\]
where $\sim$ is an equivalence relation such that $q\sim q'$ if $\mathrm{Ker}\,q=\mathrm{Ker}\,q'$.
If $\mathscr{A}$ is $f$-very ample, this functor is represented by the scheme $\mathbf{Quot}^{p,\mathscr{A}}_{\mathscr{X}/S,\mathscr{F}}$ projective over $S$ unique up to isomorphism. 
Consider the following functor:
\[
\mathfrak{Hilb}^{p,\mathscr{A}}_{\mathscr{X}/S}(T):=\left\{
\begin{array}{l}
\mathcal{Z}\subset \mathscr{X}_T
\end{array}
\middle|
\begin{array}{l}
\text{$\mathcal{Z}$ is a closed subscheme of $\mathscr{X}_T$ flat over}\\
\text{$T$ whose fibers have the same Hilbert}\\
\text{polynomial $p$ with respect to $\mathscr{A}$}
\end{array}\right\}.
\]
This functor is represented by the scheme $\mathbf{Hilb}^{p,\mathscr{A}}_{\mathscr{X}/S}$ isomorphic to $\mathbf{Quot}^{p,\mathscr{A}}_{\mathscr{X}/S,\mathcal{O}_{\mathscr{X}}}$ by \cite[Theorem 5.20]{FGA}. See \cite[\S5]{FGA} for more details.
If $S=\mathrm{Spec}\mathbbm{k}$, then we will write $\mathbf{Hilb}^{p,\mathscr{A}}_{\mathscr{X}}$ for simplicity.
Next, consider a subfunctor $\mathfrak{Div}^{p,\mathscr{A}}_{\mathscr{X}/S}$ of $\mathfrak{Hilb}^{p,\mathscr{A}}_{\mathscr{X}/S}$ as follows:
\[
\mathfrak{Div}^{p,\mathscr{A}}_{\mathscr{X}/S}(T):=\left\{
\begin{array}{l}
\mathcal{Z}\subset \mathscr{X}_T
\end{array}
\middle|
\begin{array}{l}
\text{$\mathcal{Z}$ is an effective relative Cartier divisor of $\mathscr{X}_T$}\\
\text{flat over $T$ whose fibers have the same Hilbert}\\
\text{polynomial $p$ with respect to $\mathscr{A}$}
\end{array}\right\}.
\]
We say that $D$ is an {\em effective relative Cartier divisor on $\mathscr{X}$ over $S$} if $D$ is an effective Cartier divisor that is flat over $S$. 
If $\mathscr{X}$ is flat over $S$, then the functor $\mathfrak{Div}^{p,\mathscr{A}}_{\mathscr{X}/S}$ is represented by an open subscheme $\mathbf{Div}^{p,\mathscr{A}}_{\mathscr{X}/S}$ of $\mathbf{Hilb}^{p,\mathscr{A}}_{\mathscr{X}/S}$ by \cite[Theorem 9.3.7]{FGA}.
We set $\mathbf{Hilb}_{\mathscr{X}/S}=\sqcup_{p}\mathbf{Hilb}^{p,\mathscr{A}}_{\mathscr{X}/S}$ and $\mathbf{Div}_{\mathscr{X}/S}=\sqcup_{p}\mathbf{Div}^{p,\mathscr{A}}_{\mathscr{X}/S}$.
Let $g\colon\mathscr{Y}\to S$ be another projective morphism with a $g$-ample line bundle $\mathscr{B}$.
We set the functor
\[
\mathfrak{Isom}_S(\mathscr{X},\mathscr{Y})(T):=\left\{
\begin{array}{l}
\varphi\colon \mathscr{X}_T\to\mathscr{Y}_T
\end{array}
\middle|
\begin{array}{l}
\text{$\varphi$ is an isomorphism over $T$}
\end{array}\right\}.
\]
If $f$ is flat, this functor is represented by a scheme $\mathbf{Isom}_S(\mathscr{X},\mathscr{Y})$ (cf.~\cite[Theorem 5.23]{FGA}).
Moreover, we define the following subfunctor
\[
\mathfrak{Isom}_S((\mathscr{X},\mathscr{A}),(\mathscr{Y},\mathscr{B}))(T):=\left\{
\begin{array}{l}
\varphi\colon \mathscr{X}_T\to\mathscr{Y}_T
\end{array}
\middle|
\begin{array}{l}
\text{$\varphi$ is an isomorphism over}\\
\text{$T$ such that $\varphi^*\mathscr{B}\sim\mathscr{A}$}
\end{array}\right\},
\]
which is represented by a closed scheme $\mathbf{Isom}_S((\mathscr{X},\mathscr{A}),(\mathscr{Y},\mathscr{B}))$ of $\mathbf{Isom}_S(\mathscr{X},\mathscr{Y})$ (see \cite[\S2.2]{HH}).
For a linear algebraic group $G$ and principal $G$-bundles $\mathscr{Z}$ and $\mathscr{W}$ over $S$, we define
\[
\mathfrak{Isom}_S^G(\mathscr{Z},\mathscr{W})(T):=\left\{
\begin{array}{l}
\varphi\colon \mathscr{Z}_T\to\mathscr{W}_T
\end{array}
\middle|
\begin{array}{l}
\text{$\varphi$ is an isomorphism over $T$}\\
\text{preserving $G_T$-bundle structure}
\end{array}\right\}.
\]
This functor is also represented by a principal $G$-bundle $\mathbf{Isom}^G_S(\mathscr{Z},\mathscr{W})$.

Next, we define the Picard scheme.
Let $f\colon \mathscr{X}\to S$ be a projective morphism with an $f$-very ample line bundle $\mathscr{A}$. 
Suppose that every geometric fiber of $f$ is irreducible and reduced.
We set $\mathfrak{Pic}_{\mathscr{X}/S}$ as follows:  
\[
\mathfrak{Pic}_{\mathscr{X}/S}(T)=\{\text{line bundle on $\mathscr{X}_T$}\}/\sim_T,
\]
where we take the quotient by the relative linear equivalence relation.
Note that if there exists a section $\iota\colon S\to \mathscr{X}$, then $\mathfrak{Pic}_{\mathscr{X}/S}$ is an \'etale sheaf.
It is known that the \'etale sheafification of $\mathfrak{Pic}_{\mathscr{X}/S}$ is represented by a separated scheme $\mathbf{Pic}_{\mathscr{X}/S}$ locally of finite type over $S$ (cf.~\cite[\S9]{FGA}).

Let $f\colon X\to S$ be a projective morphism of locally Noetherian schemes and $\mathcal{E}$ and $\mathcal{F}$ coherent sheaves on $X$.
Then, there exists a contravariant functor $\mathfrak{Hom}(\mathcal{E},\mathcal{F})$  from $\mathsf{Sch}_S$ to the category of sets such that for any $S$-scheme $T\to S$, the functor associates $\mathrm{Hom}_{X_T}(\mathcal{E}_T,\mathcal{F}_T)$.
Suppose that $\mathcal{F}$ is flat.
As in \cite[Theorem 5.8]{FGA}, there exists the scheme $\mathbf{Hom}_S(\mathcal{E},\mathcal{F})$ representing $\mathfrak{Hom}(\mathcal{E},\mathcal{F})$.
Take a element $g\in \mathrm{Hom}_{X}(\mathcal{E},\mathcal{F})$.
As remarked in \cite[Remark 5.9]{FGA}, there exists a closed subscheme $S'\subset S$ such that for any morphism $\mu\colon T\to S$ from a scheme, $\mu$ factors through $S'$ if and only if $g_{T}=0$.

The following lemmas are used in Subsection \ref{subsec3.1}.
\begin{lem}\label{lem--closed--cone}
    Let $f\colon X\to S$ be a projective and flat morphism of schemes of finite type over $\mathbbm{k}$ and $\mathscr{L}$ a line bundle on $X$.
    Let $Z\subset \mathbb{A}^{N+1}_{\mathbbm{k}}$ be a closed subscheme whose ideal is generated by finitely many homogeneous polynomials $g_1,\ldots, g_l$.
    Suppose that there exists a $\mathcal{O}_{X}$-linear morphism $h\colon\mathcal{O}_{X}^{\oplus N+1}\to \mathscr{L}$.
    Let $h'\colon \mathbb{A}_{X}(\mathscr{L})\to \mathbb{A}_{X}^{N+1}$ be the morphism induced by $h$.

    Then, there exists a closed subscheme $S'\subset S$ satisfying the following.
For any morphism $\mu\colon T\to S$ from a scheme, $\mu$ factors through $S'$ if and only if $h'_T$ factors through $Z_T$.
\end{lem}
\begin{proof}
    We first note that $h$ induces the canonical morphism $h_m\colon\mathbf{Sym}^m(\mathcal{O}_{X}^{\oplus N+1})\to \mathscr{L}^{\otimes m}$.
    For any morphism $\mu\colon T\to S$ of schemes, $h'_T$ factors through $Z_T$ if and only if $h_{m_i,T}(g_i)=0$ for all $g_i$, where $m_i=\mathrm{deg}\,g_i$.
    On the other hand, we can regard $h_{m_i}(g_i)$ as a section of $\mathbf{Hom}_S(\mathcal{O}_X,\mathscr{L}^{\otimes m_i})\to S$.
    Applying \cite[Remark 5.9]{FGA} to $h_{m_i}(g_i)$, we obtain the desired closed subscheme $S'$.
\end{proof}

\begin{lem}\label{lem--openness--criterion}
Let $S$ be a scheme of finite type over $\mathbbm{k}$ and $Z$ a subset of $S$.
Then the following are equivalent.
\begin{enumerate}
    \item $Z$ is open,
    \item the following two conditions are satisfied,
    \begin{enumerate}
        \item let $S'$ be a closed irreducible subscheme in $S$ and assume that the generic point $g_{S'}$ of $S'$ is contained in $Z$. Then, there exists an open subset $V$ of $S'$ such that $V$ contains $g_{S'}$ and is contained in $Z$,
        \item let $R$ be a discrete valuation ring essentially of finite type and $\phi\colon\mathrm{Spec}\, R\to S$ a morphism.
        Let $0\in \mathrm{Spec}\, R$ be the closed point and suppose that $\phi(0)\in Z$.
        Then, $\phi(\mathrm{Spec}\, R)\subset Z$.
    \end{enumerate}
\end{enumerate}
\end{lem}

\begin{proof}
    This immediately follows from \cite[10.14]{kollar-moduli}.
\end{proof}

\subsection{Artin stacks and good moduli spaces}

We refer to \cite{Ols} and \cite{HH} for fundamental notations of algebraic spaces, Deligne--Mumford stacks, Artin stacks, and coarse moduli spaces.
See also \cite[Remark 2.18]{HH}.
Let $X$ be a scheme of finite type over $S$ and $G$ a smooth group scheme over $S$.
Then, let $[X/G]$ denote the {\it quotient stack}.
Note that $[X/G]$ is an Artin stack.

\begin{defn}
Let $\mathcal{M}$ be an Artin stack of finite type over $\mathbbm{k}$ and $\pi\colon \mathcal{M}\to M$ a morphism to an algebraic space over $\mathbbm{k}$.
We say that $M$ is a {\em good moduli space} and $\pi$ is the {\em canonical morphism} of the good moduli space if the following conditions hold.
\begin{enumerate}
    \item $\pi$ is {\it cohomologically affine}, i.e., $\pi$ is quasi-compact and the induced functor $\pi_*\colon \mathrm{QCoh}(\mathcal{M})\to\mathrm{QCoh}(M)$ of the categories of quasi-coherent sheaves is exact.
    \item the natural morphism $\mathcal{O}_{M}\to\pi_*\mathcal{O}_{\mathcal{M}}$ is an isomorphism.
\end{enumerate}
We note that this $\pi$ is unique up to isomorphisms if it exists (\cite[Theorem 6.6]{alper}).
Since we work over an algebraically closed field of characteristic zero, the coarse moduli spaces of Deligne--Mumford stacks are good moduli spaces (see \cite[Definition 11.1.1]{Ols}).
See \cite{alper} for more details of good moduli spaces.
\end{defn}

We prove the following lemma, which we use to deduce the positivity of the CM line bundle of the K-moduli scheme of log Fano quasimaps in Section \ref{subsection--projectivity}.

\begin{lem}\label{lem--good-moduli-lift}
    Let $\mathcal{M}$ be an Artin stack of finite type with affine diagonal over $\mathbbm{k}$.
    Suppose that there exists a good moduli space $\pi\colon\mathcal{M}\to M$ of $\mathcal{M}$ such that $M$ is proper over $\mathbbm{k}$.
    Let $g\colon S\to M$ be a morphism from a projective normal variety.
    Then, there exist a generically finite morphism $h\colon T\to S$ from a projective normal variety and a non-empty open subset $T^\circ\subset T$ such that $g\circ h|_{T^\circ}$ lifts up to $f^\circ\colon T^\circ\to \mathcal{M}$ with respect to $\pi$ such that $f^\circ$ maps every geometric point $\bar{t}\in T^\circ$ to the unique closed point in $\pi^{-1}(g\circ h(\bar{t}))$. 

    If $S$ is further a curve, then we can choose $T^\circ$ as above to be $T$.
\end{lem}

\begin{proof}
    We only need to treat the first assertion because the second assertion follows from \cite[Theorem A.8]{AHLH}.
    Let $\eta$ be the generic point of $S$. 
    Take the fiber product $\mathcal{M}\times_M\mathrm{Spec}\,\overline{\kappa(\eta)}$ of $\pi$ and $g$. 
    By \cite[Proposition 9.1]{alper}, there exists a unique closed point in $\mathcal{M}\times_M\mathrm{Spec}\,\overline{\kappa(\eta)}$.
    By the limit argument, we can find an affine scheme $U$, a generically finite morphism $h_U\colon U\to S$, and a morphism $f_U\colon U\to \mathcal{M}$ such that $g\circ h_U=\pi\circ f_U$.
   Then $U$ can be embedded into a normal variety $T$ that is projective and generically finite over $S$.
    Let $h\colon T\to S$ be the induced morphism and $\eta'$ the generic point of $T$.
    Then $f_U(\eta')$ is the unique closed point of $\mathcal{M}\times_M\mathrm{Spec}\,\kappa(\eta')$.
    By \cite[Proposition 2.10]{AHLH} for $\mathcal{M}$ and the fact that the closed point of $\mathcal{M}\times_M\mathrm{Spec}\,\kappa(\eta')$ has a reductive stabilizer (cf.~\cite[Proposition 12.14]{alper}), there exists an \'{e}tale neighborhood $\mathrm{Spec}\,B\to M$ of $g(\eta)$ such that $\mathrm{Spec}\,B\times_M\mathcal{M}\cong [\mathrm{Spec}\,A/G]$ for some $\mathbbm{k}$-algebra $A$ of finite type and reductive group $G$ acting on $A$ satisfying $B=A^G$, where $A^G$ denotes the $G$-invariant subring. 
    By \cite[Remark 5 of \S0.2, Amplification 1.3, and the argument in p.~37]{GIT},  the closed orbits of $\mathrm{Spec}\,A$ under $G$ form a constructible subset $Z$. 
    Let $\mu\colon \mathrm{Spec}\,A\to[\mathrm{Spec}\,A/G]$ be the canonical morphism and let $\mu(Z)$ be the image.
    Note that $\mu(Z)$ is constructible.
    Since $f_U^{-1}(\mu(Z))$ is also constructible and dense, $f_U^{-1}(\mu(Z))$ contains a non-empty open subset $S^\circ$.
    Then, we see that $f_U(u)$ is the unique closed point of $\mathrm{Spec}\,\kappa(g\circ h_U(u))\times_{M}\mathcal{M}$ for any point $u\in S^\circ$.
    We complete the proof of the first assertion.
\end{proof}

\subsection{Moduli stacks and spaces of polarized varieties}\label{subsec--moduli-pol}

In this subsection, we discuss some moduli stacks and their basic properties. 

In the rest of this paper, we deal with the following moduli functors, which turn out to be stacks over the base field. 

\begin{itemize}
\item
$\mathfrak{Pol}$, which is the pseudo-functor of polarized normal varieties (Definition \ref{defn--pol}). 
The pseudo-functor $\mathfrak{Pol}$ is a stack over the base field (\cite[Lemma 2.15]{HH}). 

\item
$\mathcal{M}^{\text{klt,CY}}_{d,v}$, which is the moduli stack of polarized klt Calabi--Yau varieties $(X,L)$ of dimension $d$ and $\mathrm{vol}(L)=v$ (Theorem  \ref{thm--klt--CY--moduli}). 
By Theorem  \ref{thm--klt--CY--moduli}, $\mathcal{M}^{\text{klt,CY}}_{d,v}$ is a substack of $\mathfrak{Pol}$. 

\item
$\mathcal{M}^{\text{klt,CY}}_{d,v,P}$, which is the moduli stack of polarized klt Calabi--Yau varieties $(X,L)$ of dimension $d$, $\mathrm{vol}(L)=v$, and the Hilbert polynomial is $P$.

\item
$\mathcal{M}^{\mathrm{Ab}}_{d,v}$, which is a substack of $\mathcal{M}^{\text{klt,CY}}_{d,v}$, is the moduli stack of polarized Abelian varieties (Proposition \ref{prop--smoothness--of--Abelian--moduli}). 

\item
$\mathcal{M}^{\mathrm{Symp}}_{d,v}$, which is a substack of $\mathcal{M}^{\text{klt,CY}}_{d,v}$, is the moduli stack of polarized symplectic varieties smoothable to polarized irreducible holomorphic symplectic manifolds whose second Betti number is at least five (Definition \ref{defn--symp-moduli}). 

\item
$\mathcal{M}^{\mathrm{Kss}}_{c,d,v}$, which is the moduli stack of K-semistable log Fano pairs with some fixed invariants (Definition \ref{defn--Kss-moduli}). 

\item
$\mathcal{M}_{d, v, u, r}$, which is the moduli stack of polarized klt-trivial fibrations over curves with several fixed invariants (Definition \ref{defn--HH-moduli}, \cite[Theorem 1.3]{HH}). 
\end{itemize}

The following is the first moduli functor in the above list.
\begin{defn}\label{defn--pol}
Let $\mathfrak{Pol}$ be the category fibered in groupoids over $\mathsf{Sch}_{\mathbbm{k}}$ such that the collection of objects is
\[
\left\{
\begin{array}{l}
f\colon(\mathcal{X},\mathscr{A})\to S
\end{array}
\;\middle|
\begin{array}{rl}
\bullet\!\!&\text{$f$ is a surjective proper flat morphism of schemes whose}\\
&\text{fibers are geometrically normal and connected,}\\
\bullet\!\!&\text{$\mathscr{A}\in \mathbf{Pic}_{\mathcal{X}/S}(S)$ such that there exists an \'{e}tale covering}\\
&\text{$S'\to S$ by which the pullback of $\mathscr{A}$ to $\mathcal{X}\times_SS'$ is}\\
&\text{represented by a relatively ample line bundle over $S'$}
\end{array}\right\},
\]
and an arrow $(g,\alpha)\colon(f\colon(\mathcal{X},\mathscr{A})\to S)\to (f'\colon(\mathcal{X}',\mathscr{A}')\to S')$ is defined in the way that $\alpha\colon S\to S'$ is a morphism and $g\colon\mathcal{X}\to \mathcal{X}'\times_{S'}S$ is an isomorphism such that $g^*\alpha_{\mathcal{X}'}^*\mathscr{A}'=\mathscr{A}$ as elements of $\mathbf{Pic}_{\mathcal{X}'/S'}(S)$. 
\end{defn}

We note that $\mathfrak{Pol}$ is indeed a stack over $\mathbbm{k}$ by \cite[Lemma 2.15]{HH}.
We will express moduli stacks of certain polarized varieties as substacks of $\mathfrak{Pol}$.

\subsubsection{CM line bundle}

In this subsubsection, we define CM line bundles on certain moduli spaces. 

Let $f\colon X\to S$ be a projective flat morphism such that any geometric fiber of $f$ is a normal variety of dimension $n$.
Let $L$ be an $f$-ample line bundle on $X$.
By \cite[Theorem 4]{KnMu}, there exist line bundles $\lambda_{i}$ for $0\le i\le n$ such that the following holds for any sufficiently large integer $m>0$.
\[
\mathrm{det}f_*(\mathcal{O}_X(mL))\sim \bigotimes_{i=0}^{n+1}\lambda_i^{\otimes \binom{m}{i}}.
\]
We call this the {\it Knudsen--Mumford expansion}.
Such line bundles $\lambda_i$ are uniquely determined only by $f$ and $L$ up to linear equivalence.
Hence, for any morphism $g\colon T\to S$ of schemes and the Knudsen--Mumford expansion
\[
\mathrm{det}f_{T*}(\mathcal{O}_{X_T}(mL_T))\sim \bigotimes_{i=0}^{n+1}\lambda_{i,T}^{\otimes \binom{m}{i}}, 
\]
we have $g^*\lambda_i\sim \lambda_{i,T}$. 
Furthermore, the isomorphisms $\alpha_i\colon g^*\lambda_i\to \lambda_{i,T}$ that induce the canonical isomorphism $g^*\mathrm{det}f_*(\mathcal{O}_X(mL))\overset{\cong}{\longrightarrow}\mathrm{det}f_{T*}(\mathcal{O}_{X_T}(mL_T))$ for any $m \gg 0$ are uniquely determined.
Indeed, if $g$ is the identity, then we can regard $\alpha_i\in H^0(S,\mathcal{O}_S^\times)$, and $\alpha_i$ should satisfy
\[
\prod_{i=0}^{n+1}\alpha_i^{\binom{m}{i}}=1
\]
for any sufficiently large $m$ and hence $\alpha_i=1$ for any $i$.
We always fix such $\alpha_i$ and identify $g^*\lambda_i\sim \lambda_{i,T}$.
This shows that if $G$ is a linear algebraic group that acts on $(X,L)\to S$ equivariantly, then $\lambda_i$ also admits the natural $G$-linearlization for any $i$.

\begin{defn}\label{defn--CM--line--bundle}
    Let $f\colon X\to S$ be a projective flat morphism such that any geometric fiber of $f$ is a normal variety of dimension $n$.
Let $L$ be an $f$-ample line bundle on $X$ and 
\[
\mathrm{det}f_*(\mathcal{O}_X(mL))\sim \bigotimes_{i=0}^{n+1}\lambda_i^{\otimes \binom{m}{i}}
\] 
the Knudsen--Mumford expansion.
Let $p(m):=\chi(X_t,\mathcal{O}_{X_t}(mL_t))$ be the Hilbert polynomial for general $t\in S$. 
We write 
\[
p(m)=a_0m^n+a_1m^{n-1}+O(m^{n-2}).
\]
for some rational numbers $a_{0}$ and $a_{1}$. 
Then, we define a $\mathbb{Q}$-line bundle, called the {\em CM line bundle}, to be
$$\lambda_{\mathrm{CM},f,(X,L)}:=\lambda_{n+1}^{\otimes\frac{\mu+n(n+1)}{a_0(n+1)!}}\otimes\lambda_n^{\otimes -\frac{2}{a_0n!}},$$ where $\mu:=\frac{2a_1}{a_0}$.
\end{defn}

When $\omega_{X/S}\sim_{\mathbb{Q},S}0$, the CM line bundle is called the {\em Hodge line bundle}.

\subsubsection{Klt Calabi-Yau varieties}

In this paper, let {\it klt Calabi-Yau variety} $X$ only mean that $X$ is a projective klt variety with $\mathcal{O}_X\sim K_X$.
We note that such varieties have only canonical singularities.
Define the following set for any $d,v\in\mathbb{Z}_{>0}$;
$$\mathfrak{F}^{\text{klt,CY}}_{d,v}:=\left\{
 (X,L)
\;\middle|
\begin{array}{l}
\text{$X$ is a klt projective variety of dimension $d$ such that}\\
\text{$K_X\sim 0$ and $L$ is an ample line bundle with $\mathrm{vol}(L)=v$.}
\end{array}\right\}.$$

Then, the following result is well-known. 

\begin{thm}\label{thm--klt--CY--moduli}
    Let $\mathcal{M}^{\mathrm{klt,CY}}_{d,v}$ be a substack of $\mathfrak{Pol}$ such that the objects $\mathcal{M}^{\mathrm{klt,CY}}_{d,v}(S)$ for any scheme $S$ is the following collection.
    \[
\left\{
 f\colon(\mathcal{X},\mathcal{L})\to S
\;\middle|
\begin{array}{l}
\text{$f$ is a projective and flat family such}\\
\text{that $(\mathcal{X}_{\bar{s}},\mathcal{L}_{\bar{s}})$ belongs to $\mathfrak{F}^{\text{klt,CY}}_{d,v}$ for any}\\
\text{geometric point $\bar{s}\in S$}
\end{array}\right\}.
    \]
    Then $\mathcal{M}^{\mathrm{klt,CY}}_{d,v}$ is a separated Deligne--Mumford stack of finite type over $\mathbbm{k}$.  

    Furthermore, for the canonical morphism $\beta\colon \mathcal{M}^{\mathrm{klt,CY}}_{d,v}\to M^{\mathrm{klt,CY}}_{d,v}$ to the coarse moduli space $M^{\mathrm{klt,CY}}_{d,v}$ of $\mathcal{M}^{\mathrm{klt,CY}}_{d,v}$, there exists an ample $\mathbb{Q}$-line bundle $\Lambda_{\mathrm{Hodge}}$ on $M^{\mathrm{klt,CY}}_{d,v}$ such that $\beta^*\Lambda_{\mathrm{Hodge}}\sim_{\mathbb{Q}}\lambda_{\mathrm{Hodge}}$ unique up to $\mathbb{Q}$-linear equivalence, where $\lambda_{\mathrm{Hodge}}$ is the Hodge line bundle canonically defined on $\mathcal{M}^{\mathrm{klt,CY}}_{d,v}$. 
\end{thm}

We briefly recall how to construct this moduli stack. 
For any member $(X,L)$ of $\mathfrak{F}^{\text{klt,CY}}_{d,v}$, there exists a positive number $m\in\mathbb{Z}_{>0}$ such that $\mathcal{O}_X(mL)$ and $\mathcal{O}_{X}((m+1)L)$ are very ample.
This follows from \cite[Theorem 1.1]{kollar-eff-basepoint-free} and \cite[Lemma 7.1]{fujino-eff-slc}.
Furthermore, we see that there are only finitely many polynomials $P_1,\ldots,P_k$ such that for any object $(X,L)\in \mathcal{M}^{\textrm{klt,CY}}_{d,v}$, there exists $i\in\{1,\ldots,k\}$ such that $P_i(l)=\chi(X_{\bar{s}},L^{\otimes l})$.
Fix $P=P_i$ for some $i$ and consider a substack $\mathcal{M}^{\textrm{klt,CY}}_{d,v,P}$ of $\mathcal{M}^{\textrm{klt,CY}}_{d,v}$ such that the collection of objects of $\mathcal{M}^{\textrm{klt,CY}}_{d,v,P}(S)$ is the following for any scheme $S$:
 \[
\left\{
 f\colon(\mathcal{X},\mathcal{L})\to S
\;\middle|
\begin{array}{l}
\text{$f$ is a projective and flat family such that $(\mathcal{X}_{\bar{s}},\mathcal{L}_{\bar{s}})$}\\
\text{belongs to $\mathfrak{F}^{\text{klt,CY}}_{d,v}$ and $\chi(\mathcal{X}_{\bar{s}},\mathcal{O}_{\mathcal{X}_{\bar{s}}}(l\mathcal{L}_{\bar{s}}))=P(l)$}\\
\text{for any geometric  point $\bar{s}\in S$ and $l\in\mathbb{Z}$.}
\end{array}\right\}.
    \]
It is easy to see that $\mathcal{M}^{\textrm{klt,CY}}_{d,v,P}$ is an open and closed substack of $\mathcal{M}^{\textrm{klt,CY}}_{d,v}$.
Thus, it suffices to show how to construct $\mathcal{M}^{\textrm{klt,CY}}_{d,v,P}$ as a quotient stack.
Consider $\mathbf{Hilb}_{\mathbb{P}^{P(m)-1}\times\mathbb{P}^{P(m+1)-1}/\mathbbm{k}}^{Q,p_1^*\mathcal{O}(1)\otimes p_2^*\mathcal{O}(1)}$, where $p_1\colon \mathbb{P}^{P(m)-1}\times\mathbb{P}^{P(m+1)-1}\to \mathbb{P}^{P(m)-1}$ and $p_2\colon \mathbb{P}^{P(m)-1}\times\mathbb{P}^{P(m+1)-1}\to \mathbb{P}^{P(m+1)-1}$ are natural projections and $Q(k):=P((2m+1)k)$.
Let 
$$\iota\colon \mathcal{U}\hookrightarrow \mathbb{P}^{P(m)-1}\times\mathbb{P}^{P(m+1)-1}\times \mathbf{Hilb}_{\mathbb{P}^{P(m)-1}\times\mathbb{P}^{P(m+1)-1}/\mathbbm{k}}^{Q,p_1^*\mathcal{O}(1)\otimes p_2^*\mathcal{O}(1)}$$
 be the universal object, and we put $\mathcal{A}:=p_1^*\mathcal{O}(-1)\otimes p_2^*\mathcal{O}(1)|_{\mathcal{U}}$.
By the same argument as in the proof of \cite[Proposition 5.6]{HH}, there exists a locally closed subscheme $Z_1 \subset \mathbf{Hilb}_{\mathbb{P}^{P(m)-1}\times\mathbb{P}^{P(m+1)-1}/\mathbbm{k}}^{Q,p_1^*\mathcal{O}(1)\otimes p_2^*\mathcal{O}(1)}$ such that for any morphism $T\to \mathbf{Hilb}_{\mathbb{P}^{P(m)-1}\times\mathbb{P}^{P(m+1)-1}/\mathbbm{k}}^{Q,p_1^*\mathcal{O}(1)\otimes p_2^*\mathcal{O}(1)}$ from a scheme, $T$ factors through $Z_1$ if and only if the following conditions are satisfied.
\begin{enumerate}
    \item For every geometric point $\bar{t}\in T$, $\mathcal{U}_{\bar{t}}$ is normal, connected and Cohen--Macaulay,
    \item \label{item--Z--required}$\mathcal{A}^{\otimes m}\sim_Tp_{1,T}^*\mathcal{O}(1)|_{\mathcal{U}_T}$ and $\mathcal{A}^{\otimes m+1}\sim_Tp_{2,T}^*\mathcal{O}(1)|_{\mathcal{U}_T}$,
    \item for any point $t\in T$, $p_{1,t}\circ\iota_t$ and $p_{2,t}\circ\iota_t$ are closed immersion,
    \item for every point $t\in T$, the natural morphisms $\mathcal{O}_{T}^{\oplus  P(m)}\to H^0(\mathcal{U}_t,\mathcal{A}^{\otimes m}_{t})$ and $\mathcal{O}_{T}^{\oplus P(m+1)}\to H^0(\mathcal{U}_t,\mathcal{A}^{\otimes P(m+1)}_{t})$ induced by (\ref{item--Z--required}) are surjective, and
    \item for every point $t\in T$, $\mathcal{A}_t$ is ample and the equalities $h^0(\mathcal{U}_t,\mathcal{A}^{\otimes m}_{t})=P(m)$, $h^0(\mathcal{U}_t,\mathcal{A}^{\otimes m+1}_{t})=P(m+1)$ and $H^i(\mathcal{U}_t,\mathcal{A}^{\otimes j}_{t})=0$ hold for any $i,j>0$.
\end{enumerate} 
Let $f\colon \mathcal{U}_{Z_1}\to Z_1$ be the natural morphism.
Since every fiber of $f$ is Cohen--Macaulay, \cite[2.68.5]{kollar-moduli} implies that $\omega_{\mathcal{U}_{Z_1}/Z_1}$ is flat over $Z_1$, and furthermore, for any morphism $h\colon T\to Z_1$ from a scheme, there exists a natural isomorphism $h_{Z_1}^*\omega_{\mathcal{U}_{Z_1}/Z_1}\cong \omega_{\mathcal{U}_{T}/T}$.
Thus, we conclude that there exists an open subset $Z_2\subset Z_1$ such that for any morphism $T\to Z_1$ from a scheme, $T$ factors through $Z_2$ if and only if $\omega_{\mathcal{U}_{T}/T}$ is locally free. 
We note that there exists a locally closed subscheme $Z_3\subset Z_2$ such that for any morphism $g\colon T\to Z_2$, $g$ factors through $Z_3$ such that $\omega_{\mathcal{U}_T/T}\sim_T0$ by \cite[Proposition 9.42]{kollar-moduli}.
Thus, the subset
\[
Z:=\{t\in Z_3| \text{ $\mathcal{U}_{\bar{t}}$ is klt}\}
\]
is open (cf.~\cite[Corollary 4.10]{kollar-mmp}).
By the same argument of the proof of \cite[Theorem 5.1, Proposition 5.6]{HH}, we have
\[
\mathcal{M}^{\textrm{klt,CY}}_{d,v,P}\cong [Z/PGL(P(m))\times PGL(P(m+1))].
\]
By \cite[Theorem 4.2]{O2} and the argument as in the proof of \cite[Theorem 5.1]{HH}, we see that $\mathcal{M}^{\textrm{klt,CY}}_{d,v}$ is a separated Deligne--Mumford stack of finite type.
By \cite{KeM}, it follows that $\mathcal{M}^{\textrm{klt,CY}}_{d,v}$ admits the coarse moduli space. 
We can define $\lambda_{\mathrm{Hodge}}$ as the argument of \cite[Subsection 3.2]{Hat23}.
The existence of $\Lambda_{\mathrm{Hodge}}$ in Theorem \ref{thm--klt--CY--moduli} follows from the standard argument (see also the proof of Lemma \ref{lem--CM--descent} in Section \ref{sec3}).
Finally, the ampleness of $\Lambda_{\mathrm{Hodge}}$ follows from \cite[Theorem 8.23]{viehweg95}.

\subsubsection{Abelian varieties}

To begin with, we state notions of the deformation theory.

\begin{defn}[Prorepresentability]
For any complete local $\mathbbm{k}$-algebra $R$ with the residue field $\mathbbm{k}$, let $h_R\colon \mathrm{Art}_{\mathbbm{k}}\to\mathsf{Sets}$ denote the following functor $$h_R(A):=\mathrm{Hom}_{\mathrm{loc.}\mathbbm{k}\textrm{-alg.}}(R,A),$$ where $\mathrm{Hom}_{\mathrm{loc.}\mathbbm{k}\textrm{-alg.}}(R,A)$ denotes the set of all homomorphisms of the local $\mathbbm{k}$-algebras.
 Let $F$ and $G\colon \mathrm{Art}_{\mathbbm{k}}\to\mathsf{Sets}$ be functors.
 We say that a morphism $\eta\colon F\to G$ is {\it formally smooth} if $F(B)\to F(A)\times_{G(A)}G(B)$ for any surjection $B\to A$ in $\mathrm{Art}_{\mathbbm{k}}$.
 If $F$ or $h_R$ is formally smooth over $h_{\mathbbm{k}}$, we simply say that $F$ or $R$ is formally smooth.
 We say that a complete local $\mathbbm{k}$-algebra $R$ {\em prorepresents} a functor $F$ if there exists an isomorphism $F\cong h_R$.
 In this case, we also say that $F$ is {\it universal} or {\it prorepresentable} by $R$.
 Furthermore, let $R$ be a complete local $\mathbbm{k}$-algebra with the maximal ideal $\mathfrak{m}$.
 Take an object $\xi\in \varprojlim_n F(R/\mathfrak{m}^n)$.
 We say that a pair $(R,\xi)$ is a {\it hull} of $F$ if $\xi$ induces a formally smooth morphism $h_R\to F$ such that $h_R(\mathbbm{k}[\varepsilon]/(\varepsilon^2))\to F(\mathbbm{k}[\varepsilon]/(\varepsilon^2))$ is bijective.
 In the case when such $(R,\xi)$ exists, we say that $F$ is {\it semiversal} or has a {\it prorepresentable hull} $R$.
We set the deformation functor of $X$ as
\[
    \mathrm{Def}(X)(A):=
    \left\{
\begin{array}{l}
\text{isomorphism classes of $\mathcal{X}$ with $\iota\colon X\to \mathcal{X}$, where}\\
\text{$\mathcal{X}$ is a deformation of $X$ and $\iota$ induces an }\\
\text{isomorphism of $\mathcal{X}\times_{\mathrm{Spec}\,A}\mathrm{Spec}\,\kappa(A)$ and $X$}
\end{array}\right\}
    \]
    for $A\in\mathrm{Art}_{\mathbbm{k}}$ with the residue field $\kappa(A)$.
    $\mathrm{Def}(X)$ is semiversal if $X$ is projective by \cite[Proposition 3.10]{Schlessinger}.
\end{defn}

\begin{defn}[Abelian variety]
We call a projective variety $X$ an {\em Abelian variety} if $X$ has a group scheme structure.
\end{defn}
If $X$ is an Abelian variety, it is well known that $X$ is smooth and $K_X\sim\mathcal{O}_X$.
For more details, refer to \cite{Ab}.
The following fact is well known to experts.

\begin{lem}\label{lem--deformation--invariant--abelian}
    Let $\pi\colon\mathcal{X}\to C$ be a projective morphism of normal varieties over a smooth curve $C$.
    Fix a closed point $c\in C$.
    Suppose that $K_{\mathcal{X}}\sim_{C,\mathbb{Q}}0$ and for some closed point $t\in C\setminus\{c\}$, $\mathcal{X}_t$ is an Abelian variety.
If $\mathcal{X}_c$ is a normal variety with only klt singularities, then $\mathcal{X}_c$ is also an Abelian variety.
\end{lem}

\begin{proof}
Let $f\colon X\to\mathcal{X}_c$ be a resolution of singularities. 
Let $\mu\colon X\to \mathrm{Alb}(X)$ be a natural morphism to the Albanese variety of $X$.
    Since $\mathcal{X}_c$ has only klt singularities and hence only rational singularities, it follows from the proof of \cite[Lemma 8.1]{kawamata} that there exists a unique morphism $\alpha\colon \mathcal{X}_c\to \mathrm{Alb}(X)$ such that $\alpha\circ f=\mu$.
    By \cite[Corollary 2.64]{kollar-moduli} and the assumption, we see that $\mathrm{dim}(\mathcal{X}_c)=h^1(\mathcal{X}_c,\mathcal{O}_{\mathcal{X}_c})$.
    Since $\mathcal{X}_c$ has only rational singularites, we have that $h^1(\mathcal{X}_c,\mathcal{O}_{\mathcal{X}_c})=h^1(X,\mathcal{O}_{X})$ by using the following Leray spectral sequence
    \[
    E_2^{ij}=H^j(\mathcal{X}_c,R^if_*\mathcal{O}_{X})\Rightarrow H^{i+j}(X,\mathcal{O}_{X})
    \]
    and $R^if_*\mathcal{O}_{X}=0$ for any $i>0$.
    By \cite[Corollary 2]{kawamata-chara} and the fact that $h^1(X,\mathcal{O}_{X})=h^0(X,\Omega_X^1)$, we see that $\mu$ is birational.
    Therefore, $\alpha$ is also birational.
    Since $\mathrm{Alb}(X)$ is smooth and $K_{\mathcal{X}_c}\sim_{\mathbb{Q}}0$, we have $K_{\mathcal{X}_c}=\alpha^*K_{\mathrm{Alb}(X)}$ and $\alpha$ is isomorphic in codimension one.
    Thus, it follows from \cite[Lemma 2.62]{KM} that $\alpha$ is an isomorphism.
\end{proof}

By Lemma \ref{lem--deformation--invariant--abelian}, we see that any klt Calabi--Yau variety deformation equivalent to an Abelian variety is also an Abelian variety.

Now, we describe the moduli space of polarized Abelian varieties by using the period mapping theory.
Our moduli space is slightly different from the usual one but we show that the two coarse moduli spaces coincide with each other.
Before this, we explain the following object.
\begin{defn}[Abelian scheme]\label{defn--abelscheme}
    Let $\pi\colon X\to S$ be a proper smooth morphism such that all geometric fibers are connected, and let $\varepsilon \colon S\to X$ be a section.
   We say that $(\pi\colon X\to S,\varepsilon)$ is an {\em Abelian scheme} if $\pi$ is a group scheme with a unit section $\varepsilon$.
    Let $\mu\colon X\times_SX\to X$ be the multiplication and $p_1,p_2$ the first and second projection, respectively.
    Fix a $\pi$-ample line bundle $L$ on $X$, and we consider the following line bundle
    \[
    M:=\mu^*L\otimes p_1^*L^{-1}\otimes p_2^*L^{-1}.
    \]
    $M$ is algebraically equivalent to zero, and thus $M$ defines the natural $S$-morphism
    \[
    \Lambda(L)\colon X\to \mathbf{Pic}^0_{X/S}.
    \]
    We note that $\mathbf{Pic}^0_{X/S}$ exists as $\mathbf{Pic}^0_{X/S}=\mathbf{Pic}^\tau_{X/S}$ and it is shown by \cite[Corollary 6.8]{GIT} that $\mathbf{Pic}^0_{X/S}$ is also an Abelian scheme over $S$ in a natural way. For the definitions of $\mathbf{Pic}^0_{X/S}$ and $\mathbf{Pic}^\tau_{X/S}$, see \cite[Proposition 9.5.20 and Definition 9.6.8]{FGA}

    We say that an $S$-morphism $\lambda\colon X\to \mathbf{Pic}^0_{X/S}$ is a {\em map polarization} if for any geometric point $\bar{s}\in S$, we have $\lambda_{\bar{s}}=\Lambda(L)$ for some ample line bundle $L$ on $X_{\bar{s}}$. 
    When $\lambda$ is a map polarization, $(X,\varepsilon,\lambda)$ is called a {\em map polarized Abelian scheme} over $S$ (with a level $1$-structure).
     We define the {\em degree} $r$ of $\lambda$ to be the positive integer such that $\lambda$ is a finite morphism of degree $r^2$ (cf.~\cite[Lemma 6.12]{GIT}).
     If $X$ is an Abelian variety and $L$ is an ample line bundle such that $\Lambda(L)$ is of degree $r$, then $\mathrm{dim}\,H^0(X,L)=r$ but $H^i(X,L)=0$ for any $i>0$ (cf.~\cite[Proposition 6.13]{GIT}).
     Combining this with \cite[Proposition 6.10]{GIT}, we obtain
     \[
     \chi(X,L^{\otimes k})=r\cdot k^g
     \]
     for any $k\in\mathbb{Z}$.

Suppose that $\mathrm{dim}\,X_{\bar{s}}=g$ for any geometric point $\bar{s}\in S$.
For an $n\in\mathbb{Z}_{>1}$, we say that $2g$-sections $\varepsilon_1,\ldots,\varepsilon_{2g}$ form a {\em level $n$-structure} of $(X,\varepsilon)$ if $\varepsilon_1,\ldots,\varepsilon_{2g}$ generate the kernel of $\mu_n\colon X\to X$, which is the multiplication by $n$.
Note that the kernel of $\mu_n$ is isomorphic to $(\mathbb{Z}/n\mathbb{Z})^{2g}$ as an abstract group.
In that situation, $(X,\varepsilon,\lambda,\varepsilon_1,\ldots,\varepsilon_{2g})$ is called a {\em map polarized Abelian scheme with a level $n$-structure}.
\end{defn}

\begin{thm2}[{\cite{GIT}}]\label{thm--abel-scheme}
Fix $n,g,d\in\mathbb{Z}_{>0}$.
We set the functor 
$$\mathscr{A}_{g,d,n}\colon(\mathsf{Sch}_\mathbbm{k})^{\mathrm{op}}\to\mathsf{Sets}$$ such that $\mathscr{A}_{g,d,n}(S)$ is the set of isomorphic classes of polarized Abelian schemes with level $n$-structures of dimension $g$ and degree $d$.
Then, for any $n$, there exists a coarse moduli scheme $A_{g,d,n}$ of $\mathscr{A}_{g,d,n}$ (see \cite[Definition 5.6]{GIT}), which is quasi-projective over $\mathbbm{k}$.
Furthermore, if $n\ge 3$, then $A_{g,d,n}$ represents $\mathscr{A}_{g,d,n}$. 
\end{thm2}

\begin{proof}
    This immediately follows from \cite[Theorem 7.10]{GIT} and \cite[p.~191, Theorem 5]{Ab}.
\end{proof}

It is shown in \cite[Propositions 2.2.3.7 and 2.2.4.4]{Lan} that the deformation functor of map polarized Abelian schemes is prorepresented and formally smooth over $\mathbb{C}$.
These facts imply the smoothness of $A_{g,d,n}$ for any $n\ge3$ and  the normality of $A_{g,d,1}$ (see the argument of the proof of Proposition \ref{prop--smoothness--of--Abelian--moduli}).
This is the well--known fact and we leave its proof to the reader.

Let $\pi\colon X\to S$ be an Abelian scheme with the section $\varepsilon$.
 Let $\mathcal{L}$ be the universal line bundle on $X\times_S\mathbf{Pic}^0_{X/S}$.
    By replacing $\mathcal{L}$ with 
    $$\mathcal{L}\otimes(\pi\times\mathrm{id}_{\mathbf{Pic}^0_{X/S}})^*(\varepsilon\times\mathrm{id}_{\mathbf{Pic}^0_{X/S}})^*\mathcal{L}^{-1},$$ we assume $(\varepsilon\times\mathrm{id}_{\mathbf{Pic}^0_{X/S}})^*\mathcal{L}\cong\mathcal{O}_{\mathbf{Pic}^0_{X/S}}$.
    We set $L^\Delta(\lambda)$ as the pullback of $\mathcal{L}$ under the morphism $(\mathrm{id}_{X},\lambda)\colon X\to X\times_S \mathbf{Pic}^0_{X/S}$.
Let $\sigma\colon X\to \mathbf{Pic}_{X/S}$ be the morphism corresponding to $L^\Delta(\lambda)$ and let $\psi_2\colon \mathbf{Pic}_{X/S}\to \mathbf{Pic}_{X/S}$ be the mulitiplication of two.
Then $\psi_2$ is an \'etale finite morphism.
Indeed, it is easy to see that $\psi_2$ is quasi-finite and proper.
Therefore, $\psi_2$ is finite.
It is easy to see that $\psi_2$ is \'etale by \cite[Theorem 3.1 (iii)]{SGA} and the following fact:
Let $A$ be an Artinian local ring finite over $\mathbbm{k}$ with the maximal ideal $\mathfrak{m}$ and an ideal $I$ such that $I\cdot \mathfrak{m}=0$. 
Let $b\colon\mathrm{Spec}(A/I)\to \mathbf{Pic}_{X/S}$ be a morphism and suppose that $\psi_2\circ b$ is extended to $b'\colon \mathrm{Spec}(A)\to \mathbf{Pic}_{X/S}$.
Then, there exists a morphism $\hat{b}\colon \mathrm{Spec}(A)\to \mathbf{Pic}_{X/S}$ that is an extension of $b$ such that $\psi_2\circ \hat{b}=b'$.
This follows from the fact that if we let $\mathscr{L}$ be a line bundle of $X\times_S\mathrm{Spec}(A/I)$ and if the deformation obstruction of $\mathscr{L}^{\otimes 2}$ vanishes, then the obstruction of $\mathscr{L}$ also vanishes.

Let $\varepsilon'\colon S\to \mathbf{Pic}_{X/S}$ be the zero section.
Then, $\psi_2^{-1}(\varepsilon'(S))$ is isomorphic to $S\times (\mathbb{Z}/2\mathbb{Z})^{\oplus 2g}$ as group schemes with the identity section $\varepsilon$.
Via this identification, we see that $\psi_2^{-1}(\sigma(S))$ is a principal $(\mathbb{Z}/2\mathbb{Z})^{\oplus 2g}$-bundle over $S$.
We call $\psi_2^{-1}(\sigma(S))$ constructed as above the {\it principal $(\mathbb{Z}/2\mathbb{Z})^{\oplus 2g}$-bundle over $S$ associated with the Abelian scheme $\pi$ and the map polarization $\lambda$}.     
By using this notion, we show the following.

\begin{prop}\label{prop--smoothness--of--Abelian--moduli}
    Let $\mathcal{M}^{\mathrm{Ab}}_{d,v}\subset\mathcal{M}^{\mathrm{klt,CY}}_{d,v}$ be a substack such that 
     \[\mathcal{M}^{\mathrm{Ab}}_{d,v}(S)=
\left\{
 f\colon(\mathcal{X},\mathcal{L})\to S
\;\middle|
\begin{array}{l}
\text{$f\in\mathcal{M}^{\mathrm{klt,CY}}_{d,v}$ and for any geometric point $\bar{s}\in S$,}\\
\text{$\mathcal{X}_{\bar{s}}$ is an Abelian variety.}
\end{array}\right\}.
    \]
    Then, $\mathcal{M}_{d,v}^{\mathrm{Ab}}\subset\mathcal{M}^{\mathrm{klt,CY}}_{d,v}$ is smooth and an open and closed substack of $\mathcal{M}^{\mathrm{klt,CY}}_{d,v}$.

    In particular, the coarse moduli space $M_{d,v}^{\mathrm{Ab}}$ is normal and has only quotient singularities.
\end{prop}

\begin{proof}
By Lemma \ref{lem--deformation--invariant--abelian} and \cite[Corollary 2]{kawamata-chara}, we can check that $\mathcal{M}_{d,v}^{\mathrm{Ab}}\subset\mathcal{M}^{\text{klt,CY}}_{d,v}$ is an open and closed substack of $\mathcal{M}^{\text{klt,CY}}_{d,v}$.
Therefore, it suffices to prove the smoothness of $\mathcal{M}^{\text{klt,CY}}_{d,v}$. 
    Put $\mathcal{M}^\mathrm{Ab}_{d,v,P}=\mathcal{M}^{\text{klt,CY}}_{d,v,P}\cap\mathcal{M}^\mathrm{Ab}_{d,v}$ for a polynomial $P$.
    It suffices to show the smoothness of $\mathcal{M}^\mathrm{Ab}_{d,v,P}$ assuming $\mathcal{M}^\mathrm{Ab}_{d,v,P}\ne\emptyset$. 
    Take $Z$ as just after Theorem \ref{thm--klt--CY--moduli} and $Z_{\mathrm{Ab}}$ as the open subset of $Z$, which is the inverse image of $\mathcal{M}^\mathrm{Ab}_{d,v,P}$ under the natural morphism $Z\to\mathcal{M}^{\text{klt,CY}}_{d,v,P}$.
    Since the natual morphism $Z\to\mathcal{M}^{\text{klt,CY}}_{d,v,P}$ is smooth, it suffices to show that $Z_{\mathrm{Ab}}$ is smooth.
    For this, take an Artinian local ring $A$ with the maximal ideal $\mathfrak{m}$ and an ideal $I$ such that $I\cdot \mathfrak{m}=0$.
    To show that $Z_{\mathrm{Ab}}$ is smooth, we only need to show that the natural map
    \[
    Z_{\mathrm{Ab}}(\mathrm{Spec}(A))\to Z_{\mathrm{Ab}}(\mathrm{Spec}(A/I))
    \]
    is surjective by \cite[Theorem 3.1 (iii)]{SGA}.

    Let 
    $$\iota_0\colon(\mathcal{X},\mathcal{L})\hookrightarrow \mathbb{P}^{P(m)-1}\times\mathbb{P}^{P(m+1)-1}\times \mathrm{Spec}(A/I)$$
     be a family of embedded polarized Abelian varieties.
    Since there is no obstruction for the deformation of a point of smooth varieties, there exists a section $\varepsilon\colon \mathrm{Spec}(A/I)\to \mathcal{X}$. 
    By \cite[Proposition 6.15]{GIT}, we can regard $\mathcal{X}$ as a projective smooth group scheme over $\mathrm{Spec}(A/I)$ with the unit section $\varepsilon$. 
    Now we consider the natural map $\Lambda(\mathcal{L})$ in Definition \ref{defn--abelscheme}.
    By \cite[Proposition 2.2.4.4]{Lan}, there exists a deformation $(\bar{\pi}\colon \overline{\mathcal{X}}\to \mathrm{Spec}(A),\overline{\varepsilon},\lambda)$ of a map polarized Abelian scheme $(\pi\colon X\to \mathrm{Spec}(A/I),\varepsilon,\Lambda(\mathcal{L}))$.
    Consider the principal $(\mathbb{Z}/2\mathbb{Z})^{\oplus 2g}$-bundle $S$ over $\mathrm{Spec}(A)$ associated to the Abelian scheme $\overline{\pi}$ and the map polarization $\lambda$.
    By the Henselian lemma, $S$ is trivial as a principal bundle.
    Therefore, there exists a line bundle $\overline{\mathcal{N}}$ on $\overline{\mathcal{X}}$ such that $\overline{\mathcal{N}}^{\otimes 2}\sim L^\Delta(\lambda)$.
    This means $\Lambda(\overline{\mathcal{N}})=\lambda$.
    Let $\mathcal{N}$ be the restriction of $\overline{\mathcal{N}}$ to $\mathcal{X}$.
    By the proof of \cite[Proposition 6.10]{GIT} and the fact that numerically trivial line bundles on an Abelian variety are algebraically equivalent to the trivial line bundle, we see that $\mathcal{N}$ and $\mathcal{L}$ are algebraically equivalent to each other.
    By the Yoneda lemma, we get a morphism of Abelian schemes
    \[
    \psi_{\mathcal{L}}\colon \mathcal{X}\to \mathbf{Pic}^0_{\mathcal{X}/\mathrm{Spec}(A/I)}
    \]
    such that for any morphism $S\to \mathrm{Spec}(A/I)$, $\psi$ induces a natural map of the sets of $S$-valued points
    \[
    \mathcal{X}(S)\ni a\longmapsto t_a^*\mathcal{L}_S\otimes\mathcal{L}_S^{-1}\in\mathbf{Pic}^0_{\mathcal{X}/\mathrm{Spec}(A/I)}(S),
    \]
    where $t_a$ is the translation defined by $a$.
    It is well-known that this morphism is an \'etale surjection (cf.~the proof of \cite[Lemma 6.2]{Hat23}).
    Therefore, by the deformation theory and the property of \'etale morphisms, there exists a section $a\colon \mathrm{Spec}(A/I)\to \mathcal{X}$ such that $t_a^*\mathcal{L}\cong\mathcal{N}$.
    Let $\bar{a}\colon \mathrm{Spec}(A)\to \overline{\mathcal{X}}$ be a section that is an extension of $a$.
    Then, $\overline{\mathcal{L}}:=t_{\bar{a}}^{-1*}\overline{\mathcal{N}}$ is an extension of $\mathcal{L}$.

We note that $\overline{\mathcal{L}}^{\otimes m}$ and $\overline{\mathcal{L}}^{\otimes m+1}$ are very ample by the assumption on $m$.
Since $A$ is an Artinian local ring, by using the two very ample line bundles, we obtain an embedding $\bar{\iota}\colon\overline{\mathcal{X}}\subset \mathbb{P}^{P(m)-1}\times\mathbb{P}^{P(m+1)-1}\times\mathrm{Spec}(A)$.
Let $\iota$ be the restriction of $\bar{\iota}$ to $\mathcal{X}$. 
Then there exists an element $\gamma\in PGL(P(m))\times PGL(P(m+1))(\mathrm{Spec}(A/I))$ such that $\iota_0=\gamma\cdot \iota$.
Since $PGL(P(m))\times PGL(P(m+1))$ is smooth, there exists an extension $\bar{\gamma}\in PGL(P(m))\times PGL(P(m+1))(\mathrm{Spec}(A))$ of $\gamma$.
Set $\overline{\iota_0}:=\bar{\gamma}\cdot \bar{\iota}$.
Then, $\overline{\iota_0}\colon(\overline{\mathcal{X}},\overline{\mathcal{L}})\hookrightarrow\mathbb{P}^{P(m)-1}\times\mathbb{P}^{P(m+1)-1}\times \mathrm{Spec}(A)$ defines an element of $Z_{\mathrm{Ab}}(\mathrm{Spec}(A))$, which is mapped to the element of $Z_{\mathrm{Ab}}(\mathrm{Spec}(A/I))$ defined by $\iota_0$.
This shows that $Z_{\mathrm{Ab}}$ is smooth.

The last part of the assertion follows from \cite[Theorem 11.3.1]{Ols}.
We complete the proof.
\end{proof}

The following theorem unifies the notion of moduli space of map polarized Abelian varieties in \cite{GIT} into our notion of $M^{\text{klt,CY}}_{d,v}$.

\begin{thm}\label{thm--abelian--analytic}
Fix $d,g\in\mathbb{Z}_{>0}$.
Let $M^{\mathrm{Ab}}_{g,g!\cdot d}$ be the coarse moduli space of $\mathcal{M}^{\mathrm{Ab}}_{g,g!\cdot d}$ as in Proposition \ref{prop--smoothness--of--Abelian--moduli}.
Then there exists a natural isomorphism $M^{\mathrm{Ab}}_{g,g!\cdot d}\cong A_{g,d,1}$ that maps a $\mathbbm{k}$-valued point of $M^{\mathrm{Ab}}_{g,g!\cdot d}$ corresponding to $(X,A)$ to a point of $A_{g,d,1}$ corresponding to $(X,\Lambda(A),p)$, where $X$ is an Abelian variety, $A$ is an ample line bundle on $X$, and $p\in X$ is a closed point.
\end{thm}

\begin{proof}
For $n\ge3$, we take the universal Abelian scheme $\pi\colon \mathcal{U}\to A_{g,d,n}$ and the universal map polarization $\lambda$.
Let $S$ be the principal $(\mathbb{Z}/2\mathbb{Z})^{\oplus 2g}$-bundle over $A_{g,d,n}$ associated with the Abelian scheme $\pi$ and the map polarization $\lambda$.
Let $\lambda'$ be the map polarization of $\mathcal{U}\times_{A_{g,d,n}}S$ defined by using the base change of $\lambda$ by $S\to A_{g,d,n}$.
By the definition of $S$, we can take a line bundle $\mathcal{N}$ such that $L^\Delta(\lambda')=\mathcal{N}^{\otimes 2}$.
This shows $\Lambda(\mathcal{N})=\lambda'$.
By using the data $(\mathcal{U}\times_{A_{g,d,n}}S\to S, \mathcal{N})$, we obtain the corresponding morphism $\tilde{\mu}\colon A_{g,d,n}\to M^{\mathrm{Ab}}_{g,g!\cdot d}$.
By \cite[Lemma 6.2]{Hat23}, for any two closed points $s_1,s_2\in S$, if $\mathcal{U}_{s_1}=\mathcal{U}_{s_2}$ and $\mathcal{N}_{s_1}\equiv \mathcal{N}_{s_2}$ then $\tilde{\mu}(s_1)=\tilde{\mu}(s_2)$.
Thus, $\tilde{\mu}$ is $(\mathbb{Z}/2\mathbb{Z})^{\oplus 2g}$-invariant and it induces $\hat{\mu}\colon A_{g,d,n}\to M^{\mathrm{Ab}}_{g,g!\cdot d}$.
It is well known that the forgetful functor $A_{g,d,n}\to A_{g,d,1}$ is a geometric quotient by $GL(2g,\mathbb{Z}/n\mathbb{Z})$ (cf.~\cite[Lemmas 7.11 and 7.12]{GIT}).
It is easy to see that $\hat{\mu}$ is $GL(2g,\mathbb{Z}/n\mathbb{Z})$-invariant and induces $\mu\colon A_{g,d,1}\to M^{\mathrm{Ab}}_{g,g!\cdot d}$. 
Now recall the fact that if two line bundles $A_1$ and $A_2$ over an Abelian variety $X$ are numerically equivalent, then $\Lambda(A_1)=\Lambda(A_2)$ (cf.~\cite[Section 6.2]{GIT}). 
By this fact and \cite[Lemma 6.2]{Hat23}, we see that $\mu$ induces a bijection of the set of $\mathbbm{k}$-valued points.
Since $A_{g,d,1}$ and $M^{\mathrm{Ab}}_{g,g!\cdot d}$ are normal by \cite[Propositions 2.2.3.7 and 2.2.4.4]{Lan} and Proposition \ref{prop--smoothness--of--Abelian--moduli}, we see that $\mu$ is an isomorphism.
Thus, $\mu$ is the desired isomorphism.
\end{proof}
By applying Theorem \ref{thm--abelian--analytic}, we obtain a description of $M^{\mathrm{Ab}}_{g,g!\cdot d}$ as a quotient of symmetric space by a discrete group.
We use the following remark in the proof of Theorem \ref{thm--Fujino--moduli}.
\begin{rem}\label{rem--abelian--period}
When $\mathbbm{k}=\mathbb{C}$, it is well known that each irreducible component of $A_{g,d,1}$ admits the structure of some quotient of the $g$-dimensional Siegel upper halfspace $\mathfrak{h}_g$ by a certain arithmetic subgroup of the automorphism group of $\mathfrak{h}_g$.
We will explain this here.
Let $\delta_1,\ldots,\delta_g$ be a sequence of positive integers such that $\prod_{j=1}^g\delta_j=d$ and 
\[
\delta_1|\delta_2|\ldots|\delta_g.
\]
Then, for any irreducible component $V$ of $A_{g,d,1}$, there exist $\delta_1,\ldots,\delta_g$ as above such that $V\cong \mathfrak{h}_g/\Gamma_\delta$ by regarding $V$ as a complex analytic space
(cf.~\cite[Appendix to Chap.~7, A]{GIT} and \cite[Section 2]{fujino--canonical-certain}).
Here, 
\begin{align*}
\Gamma_\delta&:=\Bigg\{g\in GL(2g,\mathbb{Z})\Bigg|\text{}^{\mathrm{t}}g\begin{pmatrix}
0 & J_{\delta} \\
-J_{\delta} & 0 \\
\end{pmatrix}g=\begin{pmatrix}
0 & J_{\delta} \\
-J_{\delta} & 0 \\
\end{pmatrix}
\Bigg\}, \text{ and}\\
J_\delta&:=\begin{pmatrix}
\delta_1 & 0 & \cdots & 0 \\
0 & \delta_2 &    0    &    0    \\
\vdots &  0  & \ddots &    \vdots \\
0 &     0   &   \cdots     & \delta_g
\end{pmatrix}.
\end{align*}
It is well known that $\mathfrak{h}_g/\Gamma_\delta$ corresponds to the coarse moduli space whose $\mathbb{C}$-valued points are parameterizing $(X,\lambda,\epsilon)$ such that $\mathrm{Ker}\,\lambda\cong \prod_{j=1}^g(\mathbb{Z}/\delta_j\mathbb{Z})^{\oplus 2}$.

By Theorem \ref{thm--abelian--analytic}, we see that any irreducible component of $M^{\mathrm{Ab}}_{g,g!\cdot d}$ is also expressed as $\mathfrak{h}_g/\Gamma_\delta$ for some $\delta_1,\ldots,\delta_g$.
\end{rem}

\subsubsection{Irreducible holomorphic symplectic varieties}

In this subsection, we work over $\mathbb{C}$.

\begin{defn}[Irreducible holomorphic symplectic manifold]
    Let $X$ be a proper K\"{a}hler manifold such that ${\rm dim}\,X$ is even, say $2n$ for some $n$, and $\pi_1(X)=\{\mathrm{id}\}$.
    We say that $X$ is an {\it irreducible holomorphic symplectic manifold} if $H^{0}(X,\Omega^2_X)$ is generated by an element $\sigma$ such that $\sigma^n\in H^0(X,\omega_X)$ is nowhere vanishing.
    We call this $\sigma$ a {\it holomorphic symplectic form} on $X$.
\end{defn}

\begin{defn}[Symplectic variety]
    Let $Y$ be a projective normal variety with only canonical singularities of dimension $2n$.
    We say that $Y$ is a {\it symplectic variety} if the smooth locus $Y_{\mathrm{reg}}$ of $Y$ admits a holomorphic two-form $\sigma$ such that $\sigma^n$ is nowhere vanishing on $Y_{\mathrm{reg}}$ and for any resolution $f\colon Z\to Y$, $f^*\sigma|_{f^{-1}(Y_{\mathrm{reg}})}$ can be extended to a holomorphic two-form $\tilde{\sigma}$ on $Z$ entirely. 
If $Y$ further admits a resolution $f\colon Z\to Y$ as above such that $\tilde{\sigma}$ defines the structure of $Z$ as an irreducible holomorphic symplectic variety, then we say that $f$ is a {\it symplectic resolution}.
\end{defn}

We recall the following useful result by Koll\'{a}r--Laza--Sacc\`{a}--Voisin.

\begin{thm}[{\cite[Theorem 3.6, Corollary 5.2]{KLSV}}]\label{thm--klsv}
Let $\pi\colon\mathcal{X}\to \Delta$ be a projective morphism of complex analytic spaces where $\Delta$ is the one dimensional unit disk.
Suppose that $\mathcal{X}_0$ is slc, $\mathcal{X}_t$ is irredcuible holomorphic symplectic manifold for any $t\in\Delta\setminus\{0\}$ and $K_{\mathcal{X}}\sim_{\mathbb{Q}}0$.
Then the following are equivalent.
\begin{enumerate}
    \item The monodromy action on $H^2(\mathcal{X}_t)$ is finite for any $t\in\Delta\setminus\{0\}$,
    \item $\mathcal{X}_0$ is klt.
\end{enumerate}
Furthermore, if one of the above conditions holds, then $\mathcal{X}_0$ is a symplectic variety and there exist a morphism $\Delta\ni t\mapsto t^n\in\Delta$ for some $n\in\mathbb{Z}_{>0}$ and a small resolution $f\colon \widetilde{\mathcal{X}}\to \mathcal{X}\times_\Delta\Delta$ such that $f_{0}\colon \widetilde{\mathcal{X}}_{0}\to \mathcal{X}_0$ is a symplectic resolution.
Furthermore, $\widetilde{\mathcal{X}}_0$ is a projective irreducible holomorphic symplectic manifold.
\end{thm}

Note that $f_0^*K_{\mathcal{X}_0}=K_{\widetilde{\mathcal{X}}_0}$ and the projectivity of $\widetilde{\mathcal{X}}_0$ follows from \cite[Proposition 26.13]{GHJ}.

We collect some fundamental properties and results on irreducible holomorphic symplectic varieties.

\begin{defn}
    Let $X$ be an irreducible holomorphic symplectic manifold of dimension $2n$. Then there exists the following integral symmetric bilinear pairing $q_X$ on $H^2(X,\mathbb{Z})$, which is called the {\it Beauville--Bogomolov--Fujiki form}.
    First, we normalize a holomorphic symplectic form $\sigma$ as $\int_X(\sigma\overline{\sigma})^n=1$.
    If $\alpha=\lambda\sigma +\mu\overline{\sigma}+\beta\in H^{2}(X,\mathbb{C})$, where $\lambda$ and $\mu\in\mathbb{C}$ and $\beta\in H^{1,1}(X,\mathbb{C})$, then we set 
    \[
    q_X(\alpha,\alpha)=\lambda\mu+\frac{n}{2}\int_X\beta^2(\sigma\overline{\sigma})^{n-1}.
    \]
    It is well-known that the restriction of $q_X$ to $H^{2}(X,\mathbb{R})$ is a real bilinear form with index $(3,b_2-3)$, where $b_2$ is the second Betti number of $X$.  
    Furthermore, $q_X$ is well-defined as a integral quadratic form on $H^2(X,\mathbb{Z})$ by \cite[Th\'{e}or\`{e}me 5]{beauville}.
    We simply call $q_X$ the BBF form.
    This form is independent from the choice of $\sigma$.
\end{defn}

Now, we explain the results of Verbitsky and Bakker--Lehn on the global Torelli theorem and the period mapping theory on irreducible holomorphic symplectic manifolds.
Fix $b_2$ as an integer satisfying $b_2\ge3$.
Let $\Lambda$ be a free $\mathbb{Z}$-module of rank $b_2$ and suppose that there exists a primitive integral quadratic form $q$ on $\Lambda$. 
Then a {\it marked irreducible holomorphic symplectic manifold} $(X,\alpha)$ consists of an irreducible holomorphic symplectic manifold $X$ and an isomorphism $\alpha\colon H^2(X,\mathbb{Z})\to \Lambda$ such that $q$ coincides with $q_X$ under $\alpha$.
We call this $\alpha$ a {\it marking} of $X$.
Let $\Omega_{\Lambda}$ be a complex submanifold of $\mathbb{P}(\Lambda\otimes\mathbb{C})$ such that
\[
\Omega_{\Lambda}:=\{[p]\in\mathbb{P}(\Lambda\otimes\mathbb{C})|q(p,p)=0, \textrm{ and }q(p,\bar{p})>0\}.
\]
By the Bogomolov--Tian--Todorov theorem (cf.~\cite[Theorem 22.5]{GHJ}), for any compact K\"{a}hler manifold $Y$ with a trivial canonical bundle, the deformation spaces of $Y$ are unobstructed. 
Thus, for any marked irreducible holomorphic symplectic manifold $(X,\alpha)$, the germ of semiversal deformation spaces $\mathrm{Def}(X)$ is smooth.
By gluing $\mathrm{Def}(X)$ in the same way as \cite[\S25.2]{GHJ}, we obtain the coarse moduli space $\mathfrak{M}_{\lambda}$ parametrizing all marked irreducible holomorphic symplectic manifolds as a (possibly non-Hausdorff) complex manifold.
We fix a connected component $\mathfrak{M}_{\lambda}^\circ$ of $\mathfrak{M}_{\lambda}$.  
Then there exists a correspondence $\mathcal{P}\colon \mathfrak{M}_{\lambda}^\circ\to \Omega_{\Lambda}$, called the {\it period map}, which maps an isomorphic class $(X,\alpha)$ to $[\alpha(\sigma_X)]$, where $\sigma_X$ is a non-zero holomorphic symplectic form.
By \cite[Th\'{e}or\`{e}me 5]{beauville} and \cite[Theorem 8.1]{huybrechts}, we know that $\mathcal{P}$ is a surjective and locally biholomorphic map. 

To state the global Torelli theorem, we define the following notion.
\begin{defn}[{\cite[Definition 1.1]{MarkmanSurvey}}]
Let $X_1$ and $X_2$ be irreducible holomorphic symplectic manifolds.
We say that an isomorphism $f\colon H^2(X_1,\mathbb{Z})\to H^2(X_2,\mathbb{Z})$ is a {\it parallel-transport operator} if $f=g_1\circ \ldots\circ g_l$, where $g_j\colon H^2(Y_j,\mathbb{Z})\to H^2(Y_{j+1},\mathbb{Z})$ and $Y_j$'s are irreducible holomorphic symplectic manifolds, such that
\begin{itemize}
\item each $g_{j}$ is induced from an isomorphism $h_{j} \colon Y_{j+1}\to Y_j$, and
    \item there exist a smooth and proper morphism $\pi\colon \mathcal{X}\to B$ of analytic spaces whose fibers are all irreducible holomorphic symplectic manifolds, $b_1$ and $b_2\in B$ such that $Y_{j}=\mathcal{X}_{b_1}$ and $Y_{j+1}=\mathcal{X}_{b_2}$, and a continuous path $\gamma\colon [0,1]\to B$ satisfying $\gamma(0)=b_1$ and $\gamma(0)=b_2$ such that the parallel transport along $\gamma$ induces $g$.
\end{itemize}
We note that parallel-transport operators preserve the BBF forms. 

We define $\mathrm{Mon}(X)\subset O(H^2(X,\mathbb{Z}))$ as the subgroup of parallel-transport operators whose target and source are $H^2(X,\mathbb{Z})$. 
Here, for any free abelian group $L$ of finite rank with a symmetric $\mathbb{Z}$-bilinear form $q$, let $O(L,q_L)$ denote
\[
\{g\in GL(L)\,|\,\text{$q_L(g\cdot x)=q_L(x)$ for any $x\in L$}\}, 
\]
and we will simply write $O(L)$ if there is no fear of confusion.

    Let $\mathfrak{M}_{\lambda}^\circ$ be as above, i.e., a connected component of the coarse moduli space $\mathfrak{M}_{\lambda}$ parametrizing all marked irreducible holomorphic symplectic manifolds. 
    Fix a marked irreducible holomorphic symplectic manifold $(X,\alpha)$ and we regard $(X,\alpha)$ as a point of $\mathfrak{M}_{\lambda}^\circ$.
    We set the monodromy group $\mathrm{Mon}(\mathfrak{M}_{\lambda}^\circ):=\alpha(\mathrm{Mon}(X))$ as a subgroup of $O(\Lambda)$. 
Note that $\mathrm{Mon}(\mathfrak{M}_{\lambda}^\circ)$ is independent from the choice of $(X,\alpha)$.
\end{defn}

\begin{thm}[cf.~{\cite[Proposition 25.14]{GHJ}}]\label{thm--birational--symplectic}
    Let $X_1$ and $X_2$ be irreducible holomorphic symplectic manifolds, and let $f \colon X_{1} \dashrightarrow X_{2}$ be a bimeromorphic map.
    Let $p_1\colon Z\to X_1$ and $p_2\colon Z\to X_2$ be the canonical morphisms, where $Z\subset X_1\times X_2$ is the graph of $f$. 
    We define a map $[Z]_*\colon H^2(X_2,\mathbb{Z})\to H^2(X_1,\mathbb{Z})$ by $[Z]_*(\alpha):=(p_1)_*p_2^*\alpha$ for any $\alpha\in H^2(X_2,\mathbb{Z})$. 
Then $[Z]_*$ is a parallel transport operator.
\end{thm}

\begin{thm}[\cite{Ver}, \cite{BLhk}]\label{thm--verbitsky}
    Fix $\mathfrak{M}_{\lambda}^\circ$ and $\Omega_{\Lambda}$ as above.
    Then the following properties hold.
    \begin{enumerate}
        \item $\mathrm{Mon}(\mathfrak{M}_{\lambda}^\circ)\subset O(\Lambda)$ is of finite index, and 
        \item For any two points $p_1$ and $p_2$ of $\mathfrak{M}_{\lambda}^\circ$ that are mapped to the same point under the period map $\mathcal{P}$, the corresponding marked irreducible holomorphic symplectic manifolds $(X_1,\alpha_1)$ and $(X_2,\alpha_2)$ corresponding to $p_1$ and $p_2$ respectively are bimeromorphic to each other.
Moreover, if $Z\subset X_1\times X_2$ is the graph of the bimeromorphic map as above from $X_1$ to $X_2$, then $\alpha_1\circ [Z]_*=\alpha_2$. 
    \end{enumerate}
\end{thm}

Under the assumption that the second Betti number is at least five, Bakker--Lehn \cite{BLhk} proved the above theorem without using the technique of the hyper-K\"{a}hler rotation. 

In this paper, we mainly deal with projective irreducible holomorphic symplectic manifolds whose second Betti number is greater than or equal to five.

\begin{defn}\label{defn--symp-moduli}
Fix $n\in\mathbb{Z}_{>0}$.
    Let $\mathcal{M}^{\mathrm{symp}}_{2n,v}\subset\mathcal{M}^{\text{klt,CY}}_{2n,v}$ be a substack defined as follows: For any scheme $S$, the collection $\mathcal{M}^{\mathrm{symp}}_{2n,v}(S)$ of objects is defined by
     \[\mathcal{M}^{\mathrm{symp}}_{2n,v}(S)=
\left\{
 f\colon(\mathcal{X},\mathcal{L})\to S
\;\middle|
\begin{array}{l}
\text{$f\in\mathcal{M}^{\text{klt,CY}}_{2n,v}$ and for any geometric point $\bar{s}\in S$,}\\
\text{$\mathcal{X}_{\bar{s}}$ is smoothable to a projective irreducible}\\
\text{holomorphic symplectic manifold $Y$}\\
\text{such that $b_2(Y)\ge5$.}
\end{array}\right\}.
    \]
Note that $\mathcal{X}_{\bar{s}}$ is a symplectic variety by Theorem \ref{thm--klsv}.
\end{defn}
Currently, there are only four known examples of deformation types of irreducible holomorphic symplectic manifolds (so-called the $K3^{[n]}$, generalized Kummer, OG10 and OG6-types by \cite{beauville,OG10,OG6}).
They have the second Betti number at least five (cf.~\cite{beauville,dCRS,MRS}).
This is why the above condition on $b_2(Y)\ge5$ is reasonable.
We make use of this condition to apply Fujino's theory in Theorem \ref{thm--fujino-kim}.

Any polarized smoothable symplectic variety has a smooth semiversal deformation space. 
Although this fact is well known to experts, we give the details for the reader's convenience.

\begin{lem}\label{lem--smoothness--hk--moduli}
    Let $X$ be a symplectic variety smoothable to a projective irreducible holomorphic symplectic manifold. Fix an ample line bundle $L$ on $X$.
    Let $\mathrm{Def}(X,L)$ be a functor on $\mathrm{Art}_{\mathbb{C}}$ such that 
    \[
    \mathrm{Def}(X,L)(A):=
    \left\{
\begin{array}{l}
\!\!\text{isomorphism classes of $(\mathcal{X},\mathcal{L})$ with $\iota\colon X\to \mathcal{X}$, where $\mathcal{X}$ is}\\
\!\!\text{a deformation of $X$ and $\mathcal{L}$ is a line bundle and $\iota$ induces}\\
\!\!\text{an isomorphism of $(\mathcal{X},\mathcal{L})\times_{\mathrm{Spec}\,A}\mathrm{Spec}\,\kappa(A)$ and $(X,L)$.}
\end{array}\right\},
    \]
    where $\kappa(A)$ is the residue field.
 Then, $ \mathrm{Def}(X,L)$ is semiversal and formally smooth.

    Furthermore, a local deformation of a symplectic variety smoothable to a projective irreducible holomorphic symplectic manifold is agian a symplectic variety smoothable to a projective irreducible holomorphic symplectic manifold.
\end{lem}

\begin{proof}
    When $X$ is smooth, we see from \cite[1.14]{huybrechts} that $ \mathrm{Def}(X,L)$ is universal and we can regard it as a smooth hypersurface of $\mathrm{Def}(X)$. 
    We note that this fact holds not only for ample line bundles but also for non-trivial line bundles.
    Thus, it suffices to show Lemma \ref{lem--smoothness--hk--moduli} when $X$ is singular.
    By \cite[Corollary 2.64]{kollar-moduli}, we see that $H^1(X,\mathcal{O}_X)=0$ since $H^1(Y,\mathcal{O}_Y)=0$ for any irreducible holomorphic symplectic manifold $Y$.
    This fact and \cite[III, Theorem 12.11]{Ha} imply $H^1(\mathcal{X},\mathcal{O}_{\mathcal{X}})=0$ for any $\mathcal{X}\in\mathrm{Def}(X)(A)$ and Artinian ring $A$.
    By this fact and \cite[III, Theorem 12.11]{Ha}, for any two line bundles $\mathcal{L}_1$ and $\mathcal{L}_2$ on $\mathcal{X}$, if $\mathcal{L}_1|_{X}\sim\mathcal{L}_2|_{X}$ then $\mathcal{L}_1\sim\mathcal{L}_2$.
    Thus, the natural map
    \[
    \mathrm{Def}(X,L)(A)\longrightarrow\mathrm{Def}(X)(A)
    \]
is injective. 
    Thus, $\mathrm{Def}(X,L)(\mathbb{C}[\epsilon]/(\epsilon^2))$ is a finite dimensional vector space over $\mathbb{C}$.

    By Schlessinger's theorem (see \cite[Theorem 2.11]{Schlessinger}), to show that $\mathrm{Def}(X,L)$ is semiversal, it suffices to check the conditions $\overline{H}$ and $H_{\epsilon}$ of \cite[Theorem 2.3.2]{Sernesi}.
    Thus, it follows that $\mathrm{Def}(X,L)$ is semiversal from the same argument of the proof of \cite[Theorem 3.3.11]{Sernesi} and the fact that $\mathrm{Def}(X,L)(\mathbb{C}[\epsilon]/(\epsilon^2))$ is a finite dimensional vector space over $\mathbb{C}$. 

    In this paragraph, we show the inequality
    \begin{equation*}
    \mathrm{dim}\, \mathrm{Def}(X,L)(\mathbb{C}[\epsilon]/(\epsilon^2))\le \mathrm{dim}\, \mathrm{Def}(X)(\mathbb{C}[\epsilon]/(\epsilon^2))-1.\label{eq--dim--deformation}
    \end{equation*}
    Note that $\mathrm{Def}(X)(\mathbb{C}[\epsilon]/(\epsilon^2))\cong\mathrm{Ext}^1(\Omega_X,\mathcal{O}_X)$ by \cite[Proposition 1.2.9]{Sernesi}.
    By the local-to-global spectral sequence, we have an inclusion $H^1(X,T_X)\subset \mathrm{Ext}^1(\Omega_X,\mathcal{O}_X)$, where $T_X:=\mathscr{H}om_{\mathcal{O}_X}(\Omega_X,\mathcal{O}_X)$.
    If the above inequality does not hold, then we have $H^1(X,T_X)\subset\mathrm{Def}(X,L)(\mathbb{C}[\epsilon]/(\epsilon^2))$.
    This contradicts \cite[Lemma 4.13]{BLhk}.
    Thus, the inequality holds.

    Finally, take a symplectic resolution $\pi\colon \tilde{X}\to X$ by Theorem \ref{thm--klsv}.
    By \cite[Theorem 1]{Nam2}, both $\mathrm{Def}(\tilde{X})$ and $\mathrm{Def}(X)$ are formally smooth of the same dimension and there exists a natural finite morphism $\pi_*\colon \mathrm{Def}(\tilde{X}) \to \mathrm{Def}(X)$.
    We can regard $\mathrm{Def}(\tilde{X},\pi^*L)\subset\mathrm{Def}(\tilde{X})$ as a smooth hypersurface by \cite[1.14]{huybrechts}.
    We see that the restriction of $\pi_*$ induces a finite morphism $\mathrm{Def}(\tilde{X},\pi^*L)\to \mathrm{Def}(X,L)$.
    Indeed, let $A$ be an arbitrary Artinian local ring and $(\widetilde{\mathcal{X}},\widetilde{\mathcal{L}})\in \mathrm{Def}(\tilde{X},\pi^*L)(A)$. 
    We know that $l \pi^*L$ is globally generated and $H^1(\tilde{X},\mathcal{O}_{\tilde{X}}(l\pi^*L))=0$ for any sufficiently large $l\in\mathbb{Z}$ by \cite[Theorem 3.1]{KM}.
    By \cite[III, Theorem 12.11]{Ha}, we see that $l\widetilde{\mathcal{L}}$ is also globally generated.
    If $\mathcal{X}$ is the image of $\widetilde{\mathcal{X}}$ under $\pi_*$, then we see that by the definition of $\pi_*$ that $\mathcal{X}=\mathbf{Proj}_{\mathrm{Spec}(A)}(\bigoplus_{l\ge0}H^0(\widetilde{\mathcal{X}},l\widetilde{\mathcal{L}}))$ (cf.~\cite[Lemma 1.2]{Wahl}).
    Thus, if we let $\Pi\colon \widetilde{\mathcal{X}}\to\mathcal{X}$ be the canonical morphism, then there exists a line bundle $\mathcal{L}$ on $\mathcal{X}$ such that $\Pi^*\mathcal{L}\sim\widetilde{\mathcal{L}}$.
    Therefore, the inequality $\mathrm{dim}\, \mathrm{Def}(X,L)(\mathbb{C}[\epsilon]/(\epsilon^2))\le \mathrm{dim}\, \mathrm{Def}(X)(\mathbb{C}[\epsilon]/(\epsilon^2))-1$
    shows that $\mathrm{Def}(X,L)$ is also smooth.
    We complete the proof of the first assertion.

    The last assertion follows from the proof of the first assertion, Theorem \ref{thm--klsv} and \cite[Proposition 22.7]{GHJ}.
\end{proof}

\begin{cor}\label{cor--hk--moduli--normal}
    $\mathcal{M}^{\mathrm{symp}}_{2d,v}$ is smooth and an open and closed substack of $\mathcal{M}^{\mathrm{klt,CY}}_{2d,v}$.
    In particular, the coarse moduli space $M^{\mathrm{symp}}_{2d,v}$ of $\mathcal{M}^{\mathrm{symp}}_{2d,v}$ is normal with only quotient singularities.
\end{cor}

\begin{proof}
By Lemma \ref{lem--smoothness--hk--moduli}, we see that $\mathcal{M}^{\mathrm{symp}}_{2d,v}$ is an open and closed substack, and other assertions follow from the same argument as in the proof of Proposition \ref{prop--smoothness--of--Abelian--moduli}.
\end{proof}

 We fix a projective irreducible holomorphic symplectic manifold $X$ of dimension $2n$ with the BBF form $q_X$ and an ample line bundle $L$ on $X$. 
 Suppose $b_2(X)\ge 5$.
Let $(\Lambda,q)$ be a lattice isomorphic to $(H^2(X,\mathbb{Z}),q_X)$. 
Fix the canonical isomorphism $\alpha\colon H^2(X,\mathbb{Z})\to \Lambda$ preserving the quadratic forms, and put $h:=\alpha(\mathrm{c}_1(L))$.
We define a complex manifold with two connected components, denoted by $\Omega_{h^{\perp}}$, to be
\[
\Omega_{h^{\perp}}:=\{[p]\in\mathbb{P}(\Lambda\otimes\mathbb{C})|q(p,p)=q(p,h)=0, \textrm{ and }q(p,\bar{p})>0\}.
\]
We denote by $\Omega_{h^{\perp}}^+$ the connected component of $\Omega_{h^{\perp}}$ containing $[\alpha(\sigma)]$, where $\sigma$ is a holomorphic symplectic form on $X$.  
For details, see \cite[\S4]{MarkmanSurvey}.
 
We set $\Gamma$ as the subgroup of $\mathrm{Mon}(\mathfrak{M}_{\Lambda}^\circ)$ stabilizing $h$. 
 Let 
 \[
 \Lambda_h^\perp:=\{\lambda: q(\lambda,h)=0\}
 \]
 a sublattice of $\Lambda$ of index $(2,b_2-3)$.
 Since $\mathrm{Mon}(\mathfrak{M}_{\Lambda}^\circ)$ is a subgroup of $O(\Lambda)$ of finite index, $\Gamma$ is an arithmetic subgroup of $O(\Lambda_h^\perp)$.
 Here, we regard $O(\Lambda_h^\perp)$ as a subgroup of $O(\Lambda)$ of elements fixing $h$.
 We see that $\Omega_{h^{\perp}}^+/\Gamma$ is a normal quasi-projective variety by \cite{BB}. 
 
Let $\mathcal{V}^\circ$ be an irreducible component of $\mathcal{M}^{\mathrm{symp}}_{2n,v}$ and $V^\circ$ its coarse moduli space.
We set $\mathcal{V}^\circ_{\mathrm{sm}}$ as the open substack of $\mathcal{V}^\circ$ parametrizes only smooth varieties and $V^\circ_{\mathrm{sm}}$ as its coarse moduli space.
Suppose that $\mathcal{V}^\circ$ contains a point that corresponds to $(X,L)$.
 Then, we define a natural map 
 \[
\varphi\colon V^\circ_{\mathrm{sm}}\to \Omega_{h^{\perp}}^+/\Gamma
 \]
as follows.  Firstly, we note that $\mathcal{P}^{-1}(\Omega_{h^\perp}^+)$ admits the natural $\Gamma$-action and $\mathcal{P}|_{\mathcal{P}^{-1}(\Omega_{h^\perp}^+)}$ is $\Gamma$-equivariant.
 Since $V^\circ_{\mathrm{sm}}$ is connected, for any point $p$ of $V^\circ_{\mathrm{sm}}$ corresponding to $(Y,M)$, we can find a marking $\beta\colon H^2(Y,\mathbb{Z})\to \Lambda$ and a parallel-transport operator $\gamma\colon H^2(Y,\mathbb{Z})\to H^2(X,\mathbb{Z})$ such that $\beta(\mathrm{c}_1(M))=h$ and $\beta=\alpha \circ \gamma$.
 Then, we can regard $(Y,\beta)$ as an element of $\mathcal{P}^{-1}(\Omega_{h^\perp}^+)$. 
 We set $\varphi(p)$ as the image of $\mathcal{P}(Y,\beta)$ by $\Omega_{h^{\perp}}^+\to\Omega_{h^{\perp}}^+/\Gamma$.
 Then it follows that $\varphi(p)$ is independent of the choice of $(Y,\beta)$ (see \cite[Section 8]{MarkmanSurvey} for details). 
 By \cite[Theorem 8.4]{MarkmanSurvey}, $\varphi$ is an open immersion in the category of algebraic varieties.

\begin{thm}\label{thm--period--mapping--HK}
    There exists an extension $\overline{\varphi}$ of $\varphi$ to $V^\circ$ such that $\overline{\varphi}$ is an isomorphism from $V^\circ$ to $\Omega_{h^{\perp}}^+/\Gamma$ as quasi-projective varieties.
\end{thm} 

\begin{proof}
    For the reader's convenience, we give the proof here.
    Let $W\subset V^\circ\times (\Omega_{h^{\perp}}^+/\Gamma)$ be the Zariski closure of the graph of the birational map defined by $\varphi$.
    Set $\nu\colon\overline{W}\to W$ as the normalization.
    Let $p\colon W\to V^\circ$ be the canonical morphism.
    To see that there exists a morphism $\overline{\varphi}\colon V^\circ\to\Omega_{h^{\perp}}^+/\Gamma$ that is an extension of $\varphi$, it suffices to show that $p\circ\nu$ is an isomorphism.
    Note that $p\circ\nu$ is birational.
    Therefore, by the Zariski's main theorem (cf.~\cite[Theorem 12.73]{gortz-wedhorn}) and Corollary \ref{cor--hk--moduli--normal}, it suffices to show that $p$ is bijective. 
    
    Firstly, we show the surjectivity.
    For any $s\in V^{\circ}$, we can find a morphism $g\colon C\to V^\circ$ such that $g(C)$ containing both of a closed point of $V^\circ_{\mathrm{sm}}$ and $s$.
    By \cite[Theorem 11.4.1]{Ols}, taking a finite covering of $C$, we may assume that there exists a morphism $h\colon C\to \mathcal{V}^\circ$ such that $\pi\circ h=g$, where $\pi\colon\mathcal{V}^\circ\to V^\circ$ is the morphism of the coarse moduli space.
    Take the family of polarized symplectic varieties $(\mathcal{X},\mathcal{L})\to C$ corresponding to $h$.
    By Theorem \ref{thm--klsv} and replacing $C$ with a finite cover, we may assume that there exists a simultaneous resolution $f\colon\widetilde{\mathcal{X}}\to\mathcal{X}$ from a family of irreducible holomorphic symplectic manifolds.
   Then, the period of each fiber of $\widetilde{\mathcal{X}}$ induces a morphism $\mathcal{P}_C\colon C\to \Omega_{h^{\perp}}^+/\Gamma$ in the sense of complex analytic spaces.  
   It is easy to see that if we let $\Gamma_{\mathcal{P}_C}$ be the graph of $\mathcal{P}_C$, then the image of the natural morphism of $\Gamma_{\mathcal{P}_C}\to V^\circ\times (\Omega_{h^{\perp}}^+/\Gamma)$ is contained in $W$. 
   Hence, the image of $p$ contains the point $s$.

   Secondary, we show that $p$ is injective. Assume the contrary.
   Then, there exist a point $s'\in V^\circ$ and two distinct points $d_1,d_2\in\Omega_{h^{\perp}}^+/\Gamma$ such that $(s',d_1),(s',d_2)\in W$.
   It is easy to see that $p$ is quasi-finite over $V^\circ_{\mathrm{sm}}$ and hence $s'\not\in V^\circ_{\mathrm{sm}}$.
Then, we can take two curves $C_1$ and $C_2$ in $W$ such that $C_i$ passes through $(s',d_i)$ for $i=1,2$ and $p(C_i)\cap V^\circ_{\mathrm{sm}}\ne\emptyset$.
Then there exist a finite cover $\nu_i\colon D_i\to C_i$ and a family $(\mathcal{X}_i,\mathcal{L}_i)\to D_i$ which corresponds to the morphisms $p\circ \nu_i$ for each $i=1,2$.
Take points $0_i\in D_i$ such that $\nu_i(0_i)=(s',d_i)$ for $i=1,2$. 
Now, we may assume that $D_i$ is smooth and there exists a simultaneous resolution $f_i\colon \widetilde{\mathcal{X}_i}\to \mathcal{X}_i$ from a family of irreducible holomorphic symplectic varieties.
Note by the definition of $W$ that $(\widetilde{\mathcal{X}}_1)_{0_1}$ and $(\widetilde{\mathcal{X}}_2)_{0_2}$ have markings $\alpha_1$ and $\alpha_2$ such that $((\widetilde{\mathcal{X}}_1)_{0_1},\alpha_1)$ and $((\widetilde{\mathcal{X}}_2)_{0_2},\alpha_2)$ are deformation equivalent as marked irreducible holomorphic symplectic manifolds, and that $d_1$ and $d_2$ are determined by the periods of $((\widetilde{\mathcal{X}}_1)_{0_1},\alpha_1)$ and $((\widetilde{\mathcal{X}}_2)_{0_2},\alpha_2)$ respectively.
On the other hand, $(\widetilde{\mathcal{X}}_1)_{0_1}$ and $(\widetilde{\mathcal{X}}_2)_{0_2}$ are birational to each other since the ample models of them with respect to $(f_1^*\mathcal{L}_{1})_{0_1}$ and $(f_2^*\mathcal{L}_{2})_{0_2}$ respectively are canonically isomorphic.
Under the birational map $\xi\colon (\widetilde{\mathcal{X}}_1)_{0_1}\dashrightarrow (\widetilde{\mathcal{X}}_2)_{0_2}$, we have $\xi_*((f_1^*\mathcal{L}_{1})_{0_1})=(f_2^*\mathcal{L}_{2})_{0_2}$.
By $\xi$ and Theorem \ref{thm--birational--symplectic}, $((\widetilde{\mathcal{X}}_1)_{0_1},\alpha_1)$ and $((\widetilde{\mathcal{X}}_2)_{0_2},\alpha_2)$ have the same period modulo $\Gamma$.
This shows that $d_1$ coincides with $d_2$, and we get a contradiction. 
Therefore, $p$ is injective, and thus $p\circ\nu$ is an isomorphism. 
As argued above, we obtain an extension $\overline{\varphi}\colon V^\circ\to \Omega^{+}_{h^\perp}/\Gamma$ of $\varphi$ as a morphism of algebraic varieties.

In this paragraph, we write down a fundamental property of $\overline{\varphi}$. For any closed point $s\in V^\circ$, take the corresponding symplectic variety $X$.
Then, there exists a symplectic resolution $\psi\colon \tilde{X}\to X$ and $\overline{\varphi}(s)$ is defined by the period of $\tilde{X}$.
This fact immediately follows from the discussion of the previous paragraph.
Furthermore, we note that $\tilde{X}$ is deformation equivalent to $X$ itself by construction and Theorem \ref{thm--klsv}.

From now on, we show that $\overline{\varphi}$ is an isomorphism. By the Zariski's main theorem, it suffices to show that $\overline{\varphi}$ is bijective.
Firstly, we prove the injectivity.
Take two points $s_1$ and $s_2$ of $V^\circ$ such that $\overline{\varphi}(s_1)=\overline{\varphi}(s_2)$. 
Let $(X_1,L_1)$ and $(X_2,L_2)$ be the polarized symplectic varieties corresponding to $s_{1}$ and $s_{2}$, respectively, and let $f_1\colon\tilde{X}_1\to X_1$ and $f_2\colon\tilde{X}_2\to X_2$ be their symplectic resolutions. 
Then there exists a parallel-transport operator 
$\psi\colon H^2(\tilde{X}_1,\mathbb{Z})\to H^2(\tilde{X}_2,\mathbb{Z})$ 
mapping $\mathrm{c}_1(f_1^*L_1)$ to $\mathrm{c}_1(f_2^*L_2)$.
By Theorem \ref{thm--verbitsky}, there exists a birational map $\mu\colon \tilde{X}_2\dashrightarrow\tilde{X}_1$ such that $\mu^*=\psi$.
We note that $\mu$ is small.
Thus, $(X_1,L_1)$ and $(X_2,L_2)$ are isomorphic (cf.~\cite[Theorem 11.39]{kollar-moduli}).
This means $s_1=s_2$.

Secondary, we treat the surjectivity.
Take a point $s''\in \Omega_{h^{\perp}}^+/\Gamma$.
If $s''$ does not belong to the image of $\varphi$, then there exists a morphism $q\colon C\to \Omega_{h^{\perp}}^+/\Gamma$ such that $q(C)$ contains $s''$ and $q(C)\cap \varphi(V^\circ_{\mathrm{sm}})\ne\emptyset$ since we know that $\Omega_{h^{\perp}}^+/\Gamma$ is a quasi-projective variety. 
Let $c\in C$ be a closed point such that $r(c)=s''$.
By taking a finite cover of $C$, we may assume that over $C\setminus\{c\}$, there exists a family of polarized irreducible holomorphic symplectic varieties $f^\circ\colon (\mathcal{X}^\circ,\mathcal{L}^\circ)\to C\setminus\{c\}$ that corresponds to $\varphi^{-1}\circ q|_{C\setminus\{c\}}$.
By \cite[Remark of Thm.~4.11]{Griffiths}, we see that the monodoromy operator around $c$ with respect to the family $f^\circ$ is of finite index.
By \cite[Corollary 5.2]{KLSV} and replacing $C$ with its finite cover, we may assume that there exists an extended family $f\colon (\mathcal{X},\mathcal{L})\to C$ of $f^\circ$ such that $\mathcal{X}_c$ is a symplectic variety, $\mathcal{L}$ is $f$-ample, $f|_{C\setminus\{c\}}$ coincides with $f^\circ$, and $f$ admits a simultaneous resolution $\mu\colon\widetilde{\mathcal{X}}\to\mathcal{X}$ as in Theorem \ref{thm--klsv}. 
We note that $\mathcal{L}$ is also a line bundle by applying \cite[Lemma 3.1]{has-lc-trivial-fib} to $(\mathcal{X}_c,\mathcal{L}_c)$.
By the property of $\overline{\varphi}$ explained in the third paragraph of this proof, we see that $\overline{\varphi}(r(c))=s''$.
We complete the proof.
\end{proof}

\begin{rem}
    Odaka--Oshima showed the above theorem after taking some exponent of $L$ in \cite[Theorem 8.3]{OO}.
    In their terminology, Theorem \ref{thm--period--mapping--HK} states that $M^{(1)}$ has already been isomorphic to $\Omega_{h^{\perp}}^+/\Gamma$.
\end{rem}

\subsubsection{Baily--Borel compactification}\label{subsec--baily--borel}
We work over $\mathbb{C}$ throughout this subsubsection. 
We extend the result of Fujino \cite{fujino--canonical-certain} to the case when general fibers are symplectic varieties smoothable to projective irreducible holomorphic symplectic manifolds. 
We also show Theorem \ref{thm--Fujino--moduli} below, which is useful to correspond a family of klt--trivial fibrations over curves to a family of quasimaps in Section \ref{sec4}.

\begin{defn}\label{defn--fujino--mori}
    Let $f\colon X\to S$ be a contraction from a normal variety with only canonical singularities to a smooth variety.
    Suppose that the geometric generic fiber $X_{\bar{\eta}}$ of $f$ satisfies $\kappa(X_{\bar{\eta}})=0$ and $H^0(X_{\bar{\eta}},\mathcal{O}_{X_{\bar{\eta}}}(K_{X_{\bar{\eta}}}))\ne0$.
By \cite[Proposition 2.2]{FM2}, we have that there exists a unique $\mathbb{Q}$-divisor $D$ on $S$ up to linear equivalence with the canonical  isomorphism
\[
\bigoplus_{i\ge0}\mathcal{O}_S(\lfloor iD\rfloor)\cong \bigoplus_{i\ge0}(f_*\mathcal{O}_X(iK_{X/S}))^{\vee\vee},
\]
as graded algebras, where $(f_*\mathcal{O}_X(iK_{X/S}))^{\vee\vee}$ means the double dual.
 Furthermore, there exist effective $\mathbb{Q}$-divisors $B_{+}$ and $B_{-}$, which have no common components, on $X$ such that
$
 K_X \sim f^*(K_S+D)+B_{+} - B_{-},
$
$f_*\mathcal{O}_X(\lfloor iB_+ \rfloor)\cong\mathcal{O}_S$ for any $i\ge0$, and $\mathrm{codim}_Sf(\mathrm{Supp}\, B_-)\geq 2$. 
Let $B_S$ be the discriminant $\mathbb{Q}$-divisor with respect to $f$ (see Definition \ref{defn--can-bundle-formula}).
 We define $M_{S,f}:=D-B_S$. 
 If $K_{X} \sim_{\mathbb{Q},\,S}0$, in other words, $f\colon (X,0)\to S$ is a klt-trivial fibration, then $M_{S,f}$ coincides with the moduli $\mathbb{Q}$-divisor defined in Definition \ref{defn--can-bundle-formula}. 
 Hence, we also call the divisor $M_{S,f}$ the {\em moduli $\mathbb{Q}$-divisor} with respect to $f$. 
\end{defn}

We say that a contraction $f\colon X\to Y$ of smooth projective varieties is {\it semistable in codimension one} if for any prime divisor $F\subset Y$, then the fiber $f^{-1}(F)$ is simple normal crossing on a neighborhood of $f^{-1}(\eta_F)$, where $\eta_F$ is the generic point of $F$.

\begin{lem}\label{lem--coincidence--moduli}
     Let $f\colon Y\to S$ be a projective morphism from a normal variety with only canonical singularities to a smooth projective variety.
     Suppose that there exist a birational morphism $\pi\colon Y\to X$ to a normal variety $X$ and a contraction $g\colon X\to S$ such that $g\circ \pi=f$.
     Suppose further that there exists an open subset $U\subset S$ such that $K_{f^{-1}(U)} = (\pi|_{f^{-1}(U)})^*K_{g^{-1}(U)} \sim_{U} 0$.
     Let $\Delta$ be a $\mathbb{Q}$-divisor on $X$ such that $\Delta\cap g^{-1}(U)=\emptyset$ and $g\colon (X,\Delta)\to S$ is a subklt--trivial fibration. 
     Let $M_{S,f}$ and $M_{S,g}$ be the $\mathbb{Q}$-divisor defined in Definition \ref{defn--fujino--mori} and in Definition \ref{defn--can-bundle-formula} respectively.
Then, $M_{S,g}\sim M_{S,f}$.
\end{lem}

\begin{proof}
As explained in \cite{A} and \cite{fujino--canonical-certain}, there exists a finite Kummer covering $\mu\colon S'\to S$ with Galois group $G$ such that the normalization $Y'$ of irreducible component of $Y\times_SS'$ dominant to $S'$ admits a resolution of singularities of $Z\to Y'$ such that $h\colon Z\to S'$ is semistable in codimension one.
Let $f'\colon Y'\to S'$ be the canonical morphism.
We note that $Y'$ has only canonical singularities by \cite[Proposition 5.20]{KM}.
Therefore, $h_*\omega_{Z/S'}\cong f'_*\omega_{Y/S'}$ and they admit the natural $G$-action. 
By \cite[Lemmas 5.1, 5.2]{A}, $M_{S',g}:=\mu^*M_{S,g}$ satisfies that there exists a $G$-equivariant isomorphism $\mathcal{O}_{S'}(M_{S',g})\cong h_*\omega_{Z/S'}$.
On the other hand, set $M_{S',f}:=\mu^*M_{S,f}$.
By \cite[Corollary 2.5]{FM2}, we see that there exists a $G$-equivariant isomorphism $\mathcal{O}_{S'}(M_{S',f}):=f'_*\omega_{Y/S'}\cong h_*\omega_{Z/S'}$.
Thus, there exists a $G$-equivariant isomorphism $\mathcal{O}_{S'}(M_{S',f})\cong \mathcal{O}_{S'}(M_{S',g})$.
This isomorphism descends to $M_{S,f}\sim M_{S,g}$. 
\end{proof}

We follow \cite[Section 2]{fujino--canonical-certain} for fundamental notions of polarized variation of Hodge structures and its canonical extension.

\begin{note}\label{setup--modulipart}
Fix $d\in\mathbb{Z}_{>0}$ and $v\in\mathbb{R}_{>0}$. 
We consider $\mathcal{M}_{d,v}^{\mathrm{Ab}}\sqcup\mathcal{M}_{d,v}^{\mathrm{symp}}$, where $\mathcal{M}_{d,v}^{\mathrm{Ab}}$ and $\mathcal{M}_{d,v}^{\mathrm{symp}}$ were defined in Proposition \ref{prop--smoothness--of--Abelian--moduli} and Definition \ref{defn--symp-moduli}, respectively.
It follows from Lemma \ref{lem--deformation--invariant--abelian} and Corollary \ref{cor--hk--moduli--normal} that $\mathcal{M}_{d,v}^{\mathrm{Ab}}\sqcup\mathcal{M}_{d,v}^{\mathrm{symp}}$ is an open and closed substack of $\mathcal{M}_{d,v}^{\mathrm{klt},\mathrm{CY}}$.
Moreover, $M_{d,v}^{\mathrm{Ab}}\sqcup M_{d,v}^{\mathrm{symp}}$ is the coarse moduli space of $\mathcal{M}_{d,v}^{\mathrm{Ab}}\sqcup\mathcal{M}_{d,v}^{\mathrm{symp}}$.
    Fix an irreducible component $\mathcal{V}_{d,v}$ of $\mathcal{M}_{d,v}^{\mathrm{Ab}}\sqcup\mathcal{M}_{d,v}^{\mathrm{symp}}$, and let $V_{d,v}$ be the coarse moduli space of $\mathcal{V}_{d,v}$.  
It is known that $V_{d,v} \cong D/\Gamma$ for some bounded symmetric domain $D$ of type III or IV and arithmetic subgroup $\Gamma$ of the automorphism group of $D$ (cf.~\cite[Section 2]{fujino--canonical-certain}).
The canonical divisor $K_D$ of $D$ admits the natural action of $\Gamma$. 
We say that a section $s\in H^0(D,\mathcal{O}_D(lK_D))$ is an {\it automorphic form} of weight $l$ if $s$ is invariant with respect to $\Gamma$.
For the precise definition of integral automorphic forms, see \cite[Subsection 8.5]{BB}.
By \cite{BB}, there exist finitely many integral automorphic forms of the same weight $l$ for some $l\in\mathbb{Z}_{>0}$ such that they induce a closed embedding $ D/\Gamma\hookrightarrow \mathbb{P}^N$ whose closure $\overline{ D/\Gamma}^{\mathrm{BB}}$ is a normal variety.
As in \cite[Subsection 2.5]{fujino--canonical-certain}, the inclusion $ D/\Gamma\subset\overline{ D/\Gamma}^{\mathrm{BB}}$ is uniquely determined up to isomorphism. 
The closure $\overline{ D/\Gamma}^{\mathrm{BB}}$ is called the {\it Baily--Borel compactification} of $ D/\Gamma$. 

We take such an $l\in\mathbb{Z}_{>0}$ as above which is the smallest, and we denote it by $k$.
Let $\iota\colon \overline{ D/\Gamma}^{\mathrm{BB}}\hookrightarrow \mathbb{P}^N$ be the inclusion for this $k$.
We define positive integers $g$ and $d'$ as follows.
If $\mathcal{V}_{d,v}$ parametrizes Abelian varieties, then we set $g:=d+1$ and $d':=1$.
If $\mathcal{V}_{d,v}$ parametrizes irreducible holomorphic symplectic varieties, then we set $g:=b_2(X)-3$ and $d':=\frac{d}{2}$, where $X$ corresponds to a general point of $\mathcal{V}$.
Note that we have $b_2(X)\ge5$ by our assumption (see Definition \ref{defn--symp-moduli}). 
\end{note}

From now on, we prove several results.

\begin{thm}[cf.~{\cite[Theorem 1.2]{fujino--canonical-certain}, \cite[Theorem 1.1]{kim-can-bundle-formula}}]\label{thm--fujino-kim}
We use the notations in Notation \ref{setup--modulipart}.
Let $f\colon X\to S$ be a contraction of smooth projective varieties, and let $A$ be an $f$-ample $\mathbb{Q}$-line bundle on $X$.
Suppose that there exists a non-empty open subset $U\subset S$ such that $S\setminus U$ is a simple normal crossing divisor, $A|_{U}$ is a line bundle and $f^{-1}(u)$ is a smooth variety for any closed point $u\in U$.
Suppose in addition that $(f^{-1}(u),A|_{f^{-1}(u)})$ belongs to $\mathcal{V}_{d,v}$.
We note that then there exists a natural morphism $\mathcal{P}\colon U\to V_{d,v}(\cong D/\Gamma)$.
Then, there exists a morphism $\overline{\mathcal{P}}\colon S\to \overline{ D/\Gamma}^{\mathrm{BB}}$ that is an extension of $\mathcal{P}$, and 
\begin{equation*}
(\iota\circ\overline{\mathcal{P}})^*\mathcal{O}(d')\sim \mathcal{O}_{S}(gkM_{S,f}),
\end{equation*}
where $M_{S,f}$ is the moduli $\mathbb{Q}$-divisor on $S$ with respect to $f$.
In particular, $gkM_{S,f}$ is a globally generated Cartier divisor.
\end{thm}

\begin{proof}
    We first deal with the case where $A$ is a line bundle. 
    In the case where the geometric generic fiber of $f$ is a K$3$-surface (i.e.~$2$-dimensional irreducible holomorphic symplectic manifold) or an Abelian variety, the assertion immediately follows from \cite[Theorem 1.2]{fujino--canonical-certain}.
    Thus, it suffices to deal with the case where the geometric generic fiber of $f$ is a higher dimensional irreducible holomorphic symplectic manifold of dimension $d=2d'$.

   Here, we further assume that $f$ is semistable in codimension one.
   Then, we can define the natural polarized variations of Hodge structures on $\mathcal{H}:=\mathcal{O}_{U}\otimes(R^{2}(f|_{f^{-1}(U)})_*\mathbb{Z})_{\mathrm{prim}}$ and $\mathcal{H}':=\mathcal{O}_{U}\otimes (R^{2d'}(f|_{f^{-1}(U)})_*\mathbb{Z})$, where ${\rm prim}$ means the primitive part with respect to $A|_U$.
   We note here that we can easily check that $\mathcal{H}'$ indeed has a non-canonical polarization by the Lefschetz decomposition (cf.~\cite[The third comment]{kim-can-bundle-formula}). 
   By \cite[Theorem 2.10]{fujino--canonical-certain} and the assumption that $f$ is semistable in codimension one, there exists an extension $\overline{\mathcal{P}}\colon S\to \overline{ D/\Gamma}^{\mathrm{BB}}$ of $\mathcal{P}$ and 
   \begin{equation*}\tag{$*$}\label{eq--linear--equiv--Fujino-Kim-1}
  (\iota\circ\overline{\mathcal{P}})^*\mathcal{O}(1)\sim \mathscr{F}^2(\overline{\mathcal{H}})^{\otimes gk},
   \end{equation*}
   where $\overline{\mathcal{H}}$ is the canonical extension of $\mathcal{H}$ (see \cite{fujino--canonical-certain}) and $\mathscr{F}^\bullet$ is the Hodge filtration.
   Now we apply the idea of the proof of \cite[Theorem 1.1]{kim-can-bundle-formula}. 
   We see that $\mathrm{Sym}^{d'}\mathcal{H}$ is polarized and there exists an injection of variation of Hodge structures
   \[
  \alpha\colon  \mathrm{Sym}^{d'}\mathcal{H}\hookrightarrow \mathcal{H}'.
   \]
   By \cite[Theorem 10.13]{PS} and the fact that $\mathcal{H}'$ is polarized, there exists a polarized variation of Hodge structures ${\rm Coker}(\alpha)$ such that 
    \[
   \mathrm{Sym}^{d'}\mathcal{H}\oplus {\rm Coker}(\alpha)\cong \mathcal{H}'.
   \]
   Then, it is easy to see that ${\rm Coker}(\alpha)$ has no $(2d',0)$-part.
   By the uniqueness of the canonical extensions (cf.~\cite{fujino--canonical-certain}) and \cite[Lemma 3.1]{kim-can-bundle-formula}, putting $\overline{\mathcal{H}}'$ (resp.~$\overline{\mathrm{Sym}^{d'}\mathcal{H}}$) as the canonical extension of $\mathcal{H}'$ (resp.~$\mathrm{Sym}^{d'}\mathcal{H}$), we see that 
    \[
   \mathscr{F}^{2}(\overline{\mathcal{H}})^{\otimes d'}\cong \mathscr{F}^{2d'}(\overline{\mathrm{Sym}^{d'}\mathcal{H}})\cong \mathscr{F}^{2d'}(\overline{\mathcal{H}}')=f_*\omega_{X/S}.
   \]
  By \cite[Proposition 3.6]{fujino--canonical-certain} and combining \eqref{eq--linear--equiv--Fujino-Kim-1} with the above isomorphism, we have
  \begin{equation*}
(\iota\circ\overline{\mathcal{P}})^*\mathcal{O}(d')\sim \mathcal{O}_{S}(gkM_{S,f}).
\end{equation*}
In particular, $S$ is an Ambro model in this case.

   From now on, we deal with the general case, in other words, the case where $A$ is not necessarily a line bundle and $f$ is not necessarily semistable in codimension one.  
   By definition (see Notation \ref{setup--modulipart}), $V_{d,v}$ is an irreducible component of $M_{d,v}^{\mathrm{Ab}}\sqcup M_{d,v}^{\mathrm{symp}}$. 
   Pick an arbitrary $m\in\mathbb{Z}_{>0}$. 
   Since $V_{d,v}$ and $V_{d,m^{d} v}$ have the same dimension, Remark \ref{rem--abelian--period} and Corollary \ref{thm--period--mapping--HK} imply that
    \begin{equation*}\tag{$**$}\label{eq--mulitiple--isom}
    V_{d,v} \cong V_{d,m^d v}    
    \end{equation*}
    by sending a $\mathbb{C}$-valued point of $V_{d,v}$ corresponding to $(Y,L)$ to the $\mathbb{C}$-valued point corresponding to $(Y,mL)$, where $V_{d,m^d\cdot v}$ is the suitable connected component of $M_{d,m^d\cdot v}^{\mathrm{Ab}}\sqcup M_{d,m^d\cdot v}^{\mathrm{symp}}$. 
     Hence, neither \eqref{eq--linear--equiv--Fujino-Kim-1}, $k$, $d'$ nor $g$ changes by replacing $A$ with $mA$ for any $m\in\mathbb{Z}_{>0}$.
Thus, we may freely replace $A$ with $mA$ for any $m\in\mathbb{Z}_{>0}$.
Take a finite Kummer covering $\pi\colon S'\to S$ with Galois group $G$ such that $S'$ is smooth, $S'\setminus\pi^{-1}(U)$ is a simple normal crossing and there exists an algebraic fiber space $f'\colon X'\to S'$ such that $X'$ is smooth, $f'$ is semistable in codimension one, $X'\times_SU\cong X\times_S\pi^{-1}(U)$, and there exists a morphism $X'\to X\times_SS'$ by \cite[Theorem 4.3]{A}.
We may assume that there exists a $f'$-ample $\mathbb{Q}$-line bundle $A'$ such that $A'|_{X'\times_SU}= A|_{X\times_S\pi^{-1}(U)}$.
Since we may freely replace $A$ with $mA$ for any $m\in\mathbb{Z}_{>0}$, we may further assume that $A'$ is a line bundle.
By what we have shown, there exists a morphism $\widetilde{\mathcal{P}}\colon S'\to \overline{ D/\Gamma}^{\mathrm{BB}}$ that is an extension of $\mathcal{P}\circ\pi$ and $(\iota\circ\widetilde{\mathcal{P}})^*\mathcal{O}(d')\sim \mathcal{O}_{S'}(gkM_{S'})$.
We note that there exists the canonical isomorphism 
\begin{equation*}\tag{$*\!*\!*$}\label{eq--G-equiv--moduli}
    f'_*\omega_{X'/S'}^{\otimes gk}\cong \mathcal{O}_{S'}(gkM_{S'}).
\end{equation*}
Over $X'\times_SU$, this isomorphism is the pullback of $$(f|_{X\times_SU})_*\omega_{X\times_SU/U}^{\otimes gk}\cong \mathcal{O}_{U}(gkM_{S}|_{U}),$$
which is obtained by the same argument in \cite[2.21]{fujino--canonical-certain}.
Thus, \eqref{eq--G-equiv--moduli} is $G$-invariant.
This shows the existence of $\overline{\mathcal{P}}$ as in Theorem \ref{thm--fujino-kim}. 
We obtain the proof.
\end{proof}

The above theorem only deals with the case when the generic fiber is smooth.
The following also deals with the case where the generic fiber is a symplectic variety smoothable to irreducible holomorphic symplectic manifolds.

\begin{thm}\label{thm--Fujino--moduli}
We use the notations in Notation \ref{setup--modulipart}.
Let $f\colon (X,0,A)\to S$ be a polarized klt-trivial fibration between normal projective varieties.
Fix $d\in\mathbb{Z}_{>0}$ and $v\in\mathbb{Z}_{>0}$.
Suppose that there exists a projective birational morphism $\mu\colon \tilde{S}\to S$ such that $\tilde{S}$ is an Ambro model of $f$ and a non-empty open subset $U\subset S$ such that $\mu|_{\mu^{-1}U}$ is isomorphic, $A|_{f^{-1}U}$ is a line bundle and $(f^{-1}(u),A|_{f^{-1}(u)})$ belongs to $\mathcal{V}_{d,v}$ for any closed point $u\in U$.
We note that then there exists the natural morphism $\mathcal{P}\colon U\to V_{d,v}(\cong D/\Gamma)$.

Then, there is a morphism $\widetilde{\mathcal{P}}\colon\tilde{S}\to \overline{D/\Gamma}^{\mathrm{BB}}$ that is an extension of $\mathcal{P}$.
Furthermore, \begin{equation*}
(\iota\circ\widetilde{\mathcal{P}})^*\mathcal{O}(d')\sim \mathcal{O}_{\tilde{S}}(gkM_{\tilde{S}}),
\end{equation*}
where $M_{\tilde{S}}$ is the moduli $\mathbb{Q}$-divisor on $\tilde{S}$ with respect to $f$.
In particular, $d'gkM_{\tilde{S}}$ is a globally generated Cartier divisor.
\end{thm}

Before giving a proof of Theorem \ref{thm--Fujino--moduli}, we put the following lemma.

\begin{lem}[{\cite[Lemma 7]{KX}}]\label{lem--kollar-xu}
Let $S$ be a projective smooth variety and $U\subset S$ a non-empty open subset.
Suppose that there exists a flat contraction $f\colon X\to U$ such that $K_X\sim_{\mathbb{Q},\,U}0$ and $(X_u,0)$ is klt for any $u\in U$. 
Let $H$ be an $f$-ample effective $\mathbb{Q}$-Cartier relative Mumford divisor on $X$.
Then, after shrinking $U$ suitably, there exist a generically finite morphism $h\colon \tilde{S}\to S$ from a projective smooth variety $\tilde{S}$, a flat contraction $g\colon \tilde{X}\to \tilde{S}$ from a normal variety $\tilde{X}$, and a $g$-ample effective $\mathbb{Q}$-Cartier relative Mumford divisor $\tilde{H}$ on $\tilde{X}$ such that $K_{\tilde{X}}\sim_{\mathbb{Q},\,\tilde{S}}0$, $(\tilde{X},\tilde{H})\times_{\tilde{S}}h^{-1}(U)\cong(X,H)\times_Uh^{-1}(U)$, and there exists an $\epsilon>0$ such that $(\tilde{X}_s,\epsilon\tilde{H}_s)$ is slc for any $s\in \tilde{S}$.
\end{lem}

\begin{proof}
We note that there is $\epsilon_{0} \in \mathbb{R}_{>0}$ such that $(X_u,\epsilon_{0} H_u)$ is klt for any $u\in U$. 
Applying \cite[Lemma 7]{KX} to the function field $K$ of $S$, we obtain a projective smooth variety $S'$ whose function field $L$ is finite over $K$, a flat contraction of normal projective varieties $\overline{f}\colon \overline{X}\to S'$, and an effective $\overline{f}$-ample $\mathbb{Q}$-Cartier divisor $\overline{H}$ such that $(\overline{X},\overline{H})\times_{S'}\mathrm{Spec}\,(L)\cong (X,H)\times_{U}\mathrm{Spec}\,(L)$ and there exists an $\epsilon>0$ such that $(\overline{X}_s,\epsilon\overline{H}_s)$ is slc for any $s\in S'$. 
Let $\tilde{S} \to S$ be the resolution of the indeterminacy of the generically finite dominant rational map $S'\dashrightarrow S$. 
We define $(\tilde{X}, \tilde{H})$ to be the pullback of $(\overline{X},\overline{H})$ by $\tilde{S} \to S'$. 
Then $\tilde{X}\to \tilde{S}$ and $\tilde{H}$ are what we wanted. 
\end{proof}

\begin{proof}[Proof of Theorem \ref{thm--Fujino--moduli}]
First, we deal with the case when $f^{-1}(u)$ is smooth for some closed point $u\in U$.
Then, by shrinking $U$, we may assume that $f^{-1}(u)$ is smooth for every closed point $u\in U$.
Take a projective birational morphism $\alpha\colon S'\to \tilde{S}$ such that $(\mu\circ\alpha)^{-1}(U)\to U$ is isomorphic and $S'\setminus(\mu\circ\alpha)^{-1}(U)$ is a simple normal crossing divisor.
Let $X'$ be the normalization of the main component of $X\times_SS'$, and let $\Delta'$ be a $\mathbb{Q}$-divisor on $X'$ such that 
$K_{X'}+\Delta'$ is the pullback of $K_X$ to $X'$.
Let $f' \colon X'\to S'$ be the induced morphism. 
Then $f' \colon(X',\Delta')\to S'$ is a subklt-trivial fibration. 
Since $\tilde{S}$ is an Ambro model of $f$, putting $M_{S'}$ as the moduli $\mathbb{Q}$-divisor with respect to $f'$ then we have $M_{S'}=\alpha^*M_{\tilde{S}}$. 
We take a resolution of singularites $\beta\colon X''\to X'$ such that $\beta$ is an isomorphism over $(\mu\circ\alpha\circ f')^{-1}(U)$.
By Lemma \ref{lem--coincidence--moduli}, $M_{S'}$ is the moduli $\mathbb{Q}$-divisor with respect to $ f'\circ\beta$.
Furthermore, there exists a morphism $\mathcal{P}'\colon S'\to \overline{D/\Gamma}^{\mathrm{BB}}$ that is an extension of $\mathcal{P}$ and 
\[
(\iota\circ\mathcal{P}')^*\mathcal{O}(d')\sim \mathcal{O}_{S'}(gkM_{S'})
\]
by Theorem \ref{thm--fujino-kim}.
Since $M_{S'}=\alpha^*M_{\tilde{S}}$, we have the assertion in this case.

Therefore, we only need to deal with the case where $f^{-1}(u)$ is not smooth for any general closed point $u\in U$.
In this case, $2d'=d$ and $(f^{-1}(u),A|_{f^{-1}(u)})$ belongs to $\mathcal{M}_{2d',v}^{\mathrm{symp}}$ for any general closed point $u\in U$ (cf.~Lemma \ref{lem--deformation--invariant--abelian}).
Since $\mathcal{M}_{2d',v}^{\mathrm{symp}}$ is of finite type, there exists $n\in\mathbb{Z}_{>0}$ such that $nL$ is very ample and $H^1(Y,nL)=0$ for any polarized variety $(Y,L)$ in $\mathcal{M}_{2d',v}^{\mathrm{symp}}(\mathbb{C})$.
We may further assume that there exists an effective relative $\mathbb{Q}$-divisor $\mathcal{D}$ on $X$ such that $\mathcal{D}|_{f^{-1}(U)}$ is relative Mumford over $U$,  $n\mathcal{D}$ is a Cartier divisor and $n\mathcal{D}\sim_{S}nA$.
We fix this $n$ throughout the proof. 
By Theorem \ref{thm--klsv} and shrinking $U$, we see that every fiber of $f^{-1}(U) \to U$ has a symplectic resolution. 
Then we may assume that there exist a finite Galois cover $\zeta\colon U'\to U$ with $G'$ the Galois group and a projective birational morphism $\pi\colon \widetilde{W}\to f^{-1}(U)\times_UU'$ such that $\pi|_{\widetilde{W}\times_{U'}\{u'\}}$ is a symplectic resolution for any closed point $u'\in U'$.
Note that the natural morphism $\tilde{f}\colon\widetilde{W}\to U'$ is smooth and has connected fibers and $K_{\widetilde{W}}\sim_{U'}0$.
We set
$$\mathscr{V}:=R^{2}\tilde{f}_*\mathbb{Z}^\perp\otimes\mathcal{O}_{U'} \quad {\rm and} \quad \mathscr{V}':=R^{2d'}\tilde{f}_*\mathbb{Z}\otimes\mathcal{O}_{U'},$$ where $\perp$ denotes the orthogonal complementary part with respect to $\pi^*A$. 
Then we can define the natural variations of Hodge structures on $\mathscr{V}$ and $\mathscr{V}'$ such that the latter is a polarized variation of Hodge structures.
Let $\mathscr{F}^\bullet \mathscr{V}$ and $\mathscr{F}^\bullet \mathscr{V}'$ denote the filtrations defining the variations of Hodge structures.
It is easy to see that the variation of polarized Hodge structures on $\mathscr{V}$ admits a natural action of $G'$ and defines the period map and this map coincides with $\mathcal{P}|_U\circ \zeta$ by the proof of Theorem \ref{thm--period--mapping--HK} and \cite[Theorem 3.1]{MarkmanSurvey}.
By the same argument as the proof of Theorem \ref{thm--fujino-kim}, we have $\mathscr{F}^{2d'}\mathscr{V}'\cong (\mathscr{F}^2\mathscr{V})^{\otimes d'}$.
On the other hand, by the argument as in \cite[2.21]{fujino--canonical-certain}, there exists a canonical isomorphism $(\mathscr{F}^2\mathscr{V})^{\otimes gk}\cong (\iota\circ \mathcal{P}|_{U}\circ\zeta)^*\mathcal{O}(gk)$.
Thus, there exists the following canonical isomorphism 
\begin{equation*}
(\tilde{f}_*\omega_{\widetilde{W}/U'})^{\otimes gk}=(\mathscr{F}^{2d'}\mathscr{V}')^{\otimes gk}\cong (\iota\circ\mathcal{P}|_{U}\circ\zeta)^*\mathcal{O}(d'gk).
\end{equation*}
It is easy to see that this isomorphism preserves the natural $G'$-actions.
Therefore, we obtain an isomorphism
\begin{equation*}\tag{$\star$}\label{eq--moduli--hodge-1}
((f|_{f^{-1}U})_*\omega_{f^{-1}(U)/U})^{\otimes gk}\cong (\iota\circ\mathcal{P}|_{U})^*\mathcal{O}(d'gk).
\end{equation*}

For any $\epsilon>0$, consider a moduli stack $\mathcal{M}^{\mathrm{SP}}_{d,v,\epsilon}$ such that for any scheme $T$, we attain the following groupoid of pairs: 
\[
    \mathcal{M}^{\mathrm{SP}}_{d,v,\epsilon}(T):=
    \left\{\begin{array}{l}
(\phi\colon\mathcal{X}\to T,\mathcal{H})
\end{array}
\;\middle|
\begin{array}{rl}
\bullet\!\!&\text{$\phi$ is a flat contraction such that $\omega_{\mathcal{X}/T}^{[r]}\sim_T 0$}\\
&\text{for some $r\in\mathbb{Z}_{>0}$,}\\
\bullet\!\!&\text{$\mathcal{H}$ is an effective $\mathbb{Q}$-Cartier relative}\\ 
&\text{Mumford $\mathbb{Q}$-divisor on $\mathcal{X}$ such that $n\mathcal{H}$ is a}\\ 
&\text{relative Mumford divisor, and}\\ 
\bullet\!\!&\text{for any $t\in T$, $(\mathcal{X}_{\bar{t}},\epsilon\mathcal{H}_{\bar{t}})$ is slc, $\mathrm{dim}\, \mathcal{X}_{\bar{t}} = d$,}\\
&\text{and $\mathrm{vol}(\mathcal{H}_{\bar{t}})=v$.}
\end{array}\right\}.
    \]
    We also define a substack $\mathcal{M}^{\mathrm{KX},\circ}_{d,v,\epsilon}$ of $(\mathcal{M}^{\mathrm{SP}}_{d,v,\epsilon})_{\mathrm{red}}$ such that for any scheme $T$, we attain the following groupoid of pairs
     \[
    \mathcal{M}^{\mathrm{KX},\circ}_{d,v,\epsilon}(T):=
    \left\{\begin{array}{l}
(\phi\colon\mathcal{X}\to T,\mathcal{H})
\end{array}
\;\middle|
\begin{array}{rl}
\!\!\!\!&\text{$(\mathcal{X}_{\bar{t}},n\mathcal{H}_{\bar{t}})\in \mathcal{V}_{d,n^d v}(\overline{\kappa(t)})$, where $n\mathcal{H}_{\bar{t}}$}\\
\!\!\!\!&\text{is a Cartier divisor.}
\end{array}\right\},
\]
where $\mathcal{V}_{d,n^d\cdot v}$ is the connected component of $\mathcal{M}_{d,v}^{\mathrm{symp}}$ corresponding to $V_{d,n^d v}$ (\eqref{eq--mulitiple--isom} in the proof of Theorem \ref{thm--fujino-kim}).
It is easy to see that $\mathcal{M}^{\mathrm{KX},\circ}_{d,v,\epsilon}$ is an open substack, and $\mathcal{M}^{\mathrm{KX},\circ}_{d,v,\epsilon}$ is of finite type over $\mathbb{C}$ for any $\epsilon>0$.
Let $\mathcal{M}^{\mathrm{KX}}_{d,v,\epsilon}$ be the Zariski closure of $\mathcal{M}^{\mathrm{KX},\circ}_{d,v,\epsilon}$.
    By \cite[Theorem 8.9]{kollar-moduli} and the proof of \cite[Theorem 2]{KX}, we see that $\mathcal{M}^{\mathrm{KX}}_{d,v,\epsilon}$ is a proper Deligne--Mumford stack over $\mathbb{C}$ for any sufficiently small $\epsilon>0$.

Regarding $n\mathcal{H}$ as a line bundle, there exists a forgetful morphism
$$\xi\colon\mathcal{M}^{\mathrm{KX},\circ}_{d,v,\epsilon}\to \mathcal{V}_{d,n^{d} v}.$$ 
By applying \cite[Proposition 2.6]{vistoli} to $\mathcal{M}^{\mathrm{KX}}_{d,v,\epsilon}$, we see that there exist a flat contraction of projective normal varieties 
$$\varphi\colon Y\to R,$$ a positive real number $\epsilon_1>0$, and an effective $\varphi$-ample $\mathbb{Q}$-Cartier $\mathbb{Q}$-divisor $H$ such that $(Y_{\bar{r}},\epsilon_1 H_{\bar{r}})$ is slc for any geometric point ${\bar{r}}\in R$, $K_Y\sim_{\mathbb{Q},R}0$, and there exists a non-empty open subset $R_0\subset R$ such that $Y_r$ is klt and 
$H_r$ is a Cartier divisor such that $\mathrm{vol}(H_r)=n^dv$ for any $r\in R_0$ and the induced morphism $\psi'\colon R\to \mathcal{M}^{\mathrm{KX}}_{d,v,\epsilon}$ is finite surjective for any $0<\epsilon <\epsilon_1$. 
Let 
$$\psi^\circ\colon R_0\to V_{d,n^dv}(=D/\Gamma)$$
 be the composition of the natural inclusion $R_0\hookrightarrow R$, $\xi\circ\psi' \colon R \to \mathcal{V}_{d,n^dv}$, and the natural morphism $\mathcal{V}_{d,n^dv}\to V_{d,n^dv}$ (\eqref{eq--mulitiple--isom} in the proof of Theorem \ref{thm--fujino-kim}). 
Let $M_{R}$ be the moduli $\mathbb{Q}$-divisor on $R$ with respect to $\varphi$.
We know that $M_R$ is a semiample Cartier divisor such that $\mathcal{O}_R(M_R)\cong \varphi_*\omega_{Y/R}$ by \cite[Lemma 5.2]{A} and this defines a morphism $\psi\colon R\to \overline{D/\Gamma}^{\mathrm{BB}}$ that is an extension of $\psi^\circ$ by Theorem \ref{thm--fujino-kim} such that
\begin{equation*}\tag{$\star\star$}\label{eq--moduli--hodge-2}
    \mathcal{O}_{R}(gkM_R)\sim (\iota\circ\psi)^*\mathcal{O}(d').
\end{equation*}

Now, we construct $\widetilde{\mathcal{P}}$ as asserted.
By applying Lemma \ref{lem--kollar-xu} to $U\subset\tilde{S}$ and  the choice of $n$, after shrinking $U$ we obtain a generically finite morphism $h\colon S'\to \tilde{S}$ from a smooth projective variety $S'$, a flat contraction $g\colon \tilde{X}\to S'$ from a normal variety, and an effective $g$-ample $\mathbb{Q}$-Cartier $\mathbb{Q}$-divisor $\tilde{A}$ on $\tilde{X}$ such that $n\tilde{A}$ is an effective relative Mumford divisor over $S'$, $K_{\tilde{X}}\sim_{\mathbb{Q},\,S'}0$, $(\tilde{X},\tilde{A})\times_{S'}h^{-1}(U)\cong(X,\mathcal{D})\times_Sh^{-1}(U)$, and there exists $0<\epsilon_0<\epsilon_1$ such that $(\tilde{X}_s,\epsilon\tilde{A}_s)$ is slc for any $0<\epsilon<\epsilon_0$ and $s\in S'$. 
Since $K_{f^{-1}(u)}\sim0$ for any $u \in U$, by \cite[Proposition 3.1]{Am} and \cite[proof of Lemma 3.5]{birkar-compl}, we may set $M':=h^*\tilde{M}$ as the moduli $\mathbb{Q}$-divisor on $S'$ with respect to $g\colon \tilde{X}\to S'$. 
By \cite[Lemma 5.2]{A} and its proof, we have $\mathcal{O}_{S'}(M')\cong g_*\omega_{\tilde{X}/S'}$. In particular, $M'$ is a $\mathbb{Z}$-divisor.
Furthermore, by using the data of $(\tilde{X},\tilde{A})$, we obtain the natural morphism $\theta\colon S'\to \mathcal{M}^{\mathrm{KX}}_{d,v,\epsilon}$.
By the proof of Lemma \ref{lem--kollar-xu}, we may replace $S'$ and assume that there exists a morphism $\tilde{\theta}\colon S'\to R$ such that $\theta=\psi'\circ\tilde{\theta}$ and the Stein factorization $\mu\colon S''\to \tilde{S}$ of $h$ is a finite Galois covering with the Galois group $G$.  
Then, we set $\mathcal{P}':=\psi\circ\tilde{\theta}$ that is an extension of $\mathcal{P}$.
Finally, we will show the last assertion.
By \cite[Corollary 2.69]{kollar-moduli}, we see that there exists the following natural isomorphism
\[
\mathcal{O}_{S'}(M_{S'})\cong g_*\omega_{\tilde{X}/S'}\cong \tilde{\theta}^*\varphi_*\omega_{Y/R}\cong \mathcal{O}_R(M_R)
\]
since $\tilde{X}\cong Y\times_RS'$ by construction.
Combining \eqref{eq--moduli--hodge-2} with the above linear equivalence, we have 
\[
\mathcal{O}_{S'}(gkM_{S'})\cong(\iota\circ\mathcal{P}')^*\mathcal{O}(d'). 
\]
On the other hand, we see by \eqref{eq--moduli--hodge-1} that this isomorphism is generically the pullback of 
\[
\mathcal{O}_{\tilde{S}}(gkM_{\tilde{S}})|_U\cong (\iota\circ\mathcal{P}|_{U})^*\mathcal{O}(d').
\]
Since $\iota^*\mathcal{O}(d)$ is ample, we see that there exists a $G$-invariant morphism $\mathcal{P}''\colon S''\to \overline{D/\Gamma}^{\mathrm{BB}}$ such that $\mathcal{P}''\circ h'=\mathcal{P}'$ and there exists a $G$-invariant isomorphism 
\[
\mathcal{O}_{S''}(gkM_{S''})\cong(\iota\circ\mathcal{P}'')^*\mathcal{O}(d'),
\]
where $h'\colon S'\to S''$ is the canonical morphism and $M_{S''}=\mu^*M_{\tilde{S}}$.
Then we obtain $\widetilde{\mathcal{P}}\colon \tilde{S}\to \overline{D/\Gamma}^{\mathrm{BB}}$ and $\mathcal{O}_{\tilde{S}}(gkM_{\tilde{S}})\sim(\iota\circ\widetilde{\mathcal{P}})^*\mathcal{O}(d')$.
It is easy to see that $\widetilde{\mathcal{P}}$ is an extension of $\mathcal{P}$ and we obtain the assertion.
\end{proof}

\begin{defn}\label{defn--Hodge-Q-line-bundle}
We remove the subscript $d,v$ of $V_{d,v}$ for the simplicity of notation. 
    Let $d'$, $k$, and $g$ be as above for the fixed irreducible component $V$.
Let $\overline{V}^{\mathrm{BB}}$ be the Baily--Borel compactification of $V$ and $\iota\colon\overline{V}^{\mathrm{BB}} \hookrightarrow\mathbb{P}^L$ the closed immersion defined by finitely many integral automorphic forms of weight $k$. 
By Theorem \ref{thm--Fujino--moduli}, $\frac{1}{gk}\iota^*\mathcal{O}(d')$ is $\mathbb{Q}$-linearly equivalent to $\Lambda_{\mathrm{Hodge}}$ defined in Theorem \ref{thm--klt--CY--moduli}. 
Hence, we call $\frac{1}{gk}\iota^*\mathcal{O}(d')$ the {\it Hodge $\mathbb{Q}$-line bundle} on $\overline{V}^{\mathrm{BB}}$, and we denote it by $\Lambda_{\mathrm{Hodge}}$. 
This $\Lambda_{\mathrm{Hodge}}$ is the same notion as defined in Theorem \ref{thm--klt--CY--moduli} modulo $\mathbb{Q}$-linear equivalence.
\end{defn}

\subsubsection{K-moduli of log Fano pairs}

For any scheme $S$, we say that $f\colon (X,\Delta)\to S$ is a {\em family of log Fano pairs} if the following hold.
\begin{enumerate}
\item $f$ is a flat projective morphism,
    \item $\Delta=cD$ for some $c$ and $D$ such that $c\in\mathbb{Q}\cap(0,1)$ and $D$ is a K-flat family of divisors on $X$ over $S$ (see \cite[Definition 7.1]{kollar-moduli}),
    \item $-(K_{X/S}+\Delta)$ is $f$-ample, and 
    \item $(X_{\bar{s}},\Delta_{\bar{s}})$ is klt for any geometric point $\bar{s}\in S$.
\end{enumerate}
If $(X_{\bar{s}},\Delta_{\bar{s}},-(K_{X_{\bar{s}}}+\Delta_{\bar{s}}))$ is K-semistable for any geometric point $\bar{s}\in S$, we say that $f$ is a {\em family of K-semistable log Fano pairs}.
Let $f\colon (X,\Delta)\to S$ be a family of log Fano pairs such that $\mathrm{dim}\,X_{s}=d$ for any $s\in S$.
Take $m\in\mathbb{Z}_{>0}$ such that $-m(K_{X/S}+\Delta)$ is ample over $S$.
Then, we have
\[
\mathrm{det}f_*\mathcal{O}_{X}(-mk(K_{X/S}+\Delta))\cong\bigotimes_{i=0}^{d+1}\lambda_i^{\otimes\binom{k}{i}}
\]
for any sufficiently large $k\in\mathbb{Z}$ as the Knudsen--Mumford expansion.
Then, we define the (log) CM line bundle as the $\mathbb{Q}$-line bundle on $S$ by
\[
\lambda_{\mathrm{CM},f, (X,\Delta)}:=\frac{-1}{(d+1)m^{d+1}(-K_{X_s}-\Delta_s)^d}\cdot \lambda_{d+1},
\]
where $s\in S$ is an arbitrary closed point.

\begin{defn}\label{defn--Kss-moduli}
Fix $d\in\mathbb{Z}_{\ge0}$, $v\in\mathbb{Q}_{>0}$ and $c\in\mathbb{Q}\cap(0,1)$.
We define the {\em K-moduli stack} $\mathcal{M}^{\mathrm{Kss}}_{c,d,v}$ as follows: For any scheme $S$, the collection $\mathcal{M}^{\mathrm{Kss}}_{c,d,v}(S)$ is defined to be
\[
\left\{
 f\colon(X,\Delta)\to S
\;\middle|
\begin{array}{rl}
\bullet\!\!&\text{$f$ is a family of K-semistable log Fano pairs,}\\
\bullet\!\!&\text{$c^{-1}\Delta$ is a K-flat family of divisors,}\\
\bullet\!\!&\text{$\mathrm{dim}\,X_{\bar{s}}=d$,}\\
\bullet\!\!&\text{$(-1)^d(K_{X_{\bar{s}}}+\Delta_{\bar{s}})^d=v$ for any geometric point $\bar{s}\in S$.}
\end{array}\right\}.
    \]
    For the definition of K-flatness, see \cite[Definition 7.1]{kollar-moduli}.
By the construction of $\mathcal{M}^{\mathrm{Kss}}_{c,d,v}$ as a quotient stack (cf.~\cite[Theorem 2.21]{XZ1}), it is easy to see that $\mathcal{M}^{\mathrm{Kss}}_{c,d,v}$ also admits the CM line bundle $\lambda_{\mathrm{CM}}$ in a natural way.
\end{defn}
    \begin{thm}[{\cite{ABHLX}, \cite{BLX}, \cite{BHLLX}, \cite{XZ1}, \cite[Theorem 1.3]{LXZ}}]\label{thm--K-moduli--of--log--Fano}
        $\mathcal{M}^{\mathrm{Kss}}_{c,d,v}$ is an Artin stack of finite type over $\mathbbm{k}$ with a good moduli space $\pi\colon \mathcal{M}^{\mathrm{Kss}}_{c,d,v}\to  M^{\mathrm{Kps}}_{c,d,v}$ such that $M^{\mathrm{Kps}}_{c,d,v}$ is projective.
        Moreover, there exists an ample $\mathbb{Q}$-line bundle $\Lambda_{\mathrm{CM}}$ on $M^{\mathrm{Kps}}_{c,d,v}$ such that $\pi^*\Lambda_{\mathrm{CM}}\sim_{\mathbb{Q}}\lambda_{\mathrm{CM}}$.
        Here, such $\Lambda_{\mathrm{CM}}$ is unique up to $\mathbb{Q}$-linear equivalence.
    \end{thm}

    We note that closed points of $\mathcal{M}^{\mathrm{Kss}}_{c,d,v}$ correspond to K-polystable log Fano pairs.
    We also say that $\Lambda_{\mathrm{CM}}$ is the CM line bundle on $M^{\mathrm{Kps}}_{c,d,v}$.

\subsubsection{K-moduli of klt--trivial fibrations over curves}\label{subsec--k-moduli--klt--trivial}

Fix $d\in\mathbb{Z}_{>0}$ and $u,v\in\mathbb{Q}_{>0}$.
We set  $$\mathfrak{Z}_{d, v,u}:=\left\{
\begin{array}{l}
f\colon (X,\Delta=0,A) \to C
\end{array}
\;\middle|
\begin{array}{rl}
(i)&\text{$f$ is a uniformly adiabatically}\\
&\text{K-stable polarized klt-trivial} \\
&\text{fibration over a curve $C$,}\\
(ii)&\text{${\rm dim}X=d$,}\\
(iii)&\text{$K_X\equiv -uf^*H$ for some line bundle} \\
&\text{$H$ on $C$ such that $\mathrm{deg}\,H=1$,} \\
(iv)&\text{$A$ is an $f$-ample line bundle on}\\
&\text{$X$ such that $(H\cdot A^{d-1})=v$.}
\end{array}\right\}.$$
Fix a sufficiently divisible $r\in\mathbb{Z}_{>0}$ as \cite[Lemma 3.1]{HH}.
We recall the following moduli stack $\mathcal{M}_{d,v,u,r}$ in \cite{HH}.

\begin{defn}\label{defn--HH-moduli}
 For any scheme $S$, $\mathcal{M}_{d,v,u,r}(S)$ is the following collection:
 $$\left\{
 \vcenter{
 \xymatrix@C=12pt{
(\mathcal{X},\mathscr{A})\ar[rr]^-{f}\ar[dr]_{\pi_{\mathcal{X}}}&& \mathcal{C} \ar[dl]\\
&S
}
}
\;\middle|
\begin{array}{rl}
(i)&\text{$\pi_{\mathcal{X}}$ is a flat projective morphism and $\mathcal{X}$ is a scheme,}\\
(ii)&\text{$\mathscr{A}\in\mathbf{Pic}_{\mathcal{X}/S}(S)$ such that $\mathscr{A}_{\bar{s}}$ is}\\
&\text{$f_{\bar{s}}$-ample for any geometric point $\bar{s}\in S$,}\\
(iii)&\text{$\omega_{\mathcal{X}/S}^{[r]}$ exists as a line bundle,}\\
(iv)&\text{$\pi_{\mathcal{X}*}\omega_{\mathcal{X}/S}^{[-lr]}$ is locally free and it generates}\\
&\text{$H^0(\mathcal{X}_{s}, \mathcal{O}_{\mathcal{X}_{s}}(-lrK_{\mathcal{X}_{s}}))$ for any point $s\in S$ and any}\\
&\text{$l\in\mathbb{Z}_{>0}$,}\\
(v)&\text{$f$ is the ample model of $\omega_{\mathcal{X}/S}^{[-r]}$ over $S$ and}\\
&\text{$(\mathcal{X}_{\overline{s}},0,\mathscr{A}_{\overline{s}})\to \mathcal{C}_{\overline{s}} \in \mathfrak{Z}_{d, v,u}$ for any geometric}\\
&\text{point $\overline{s}\in S$}
\end{array}\right\}.$$
Here, we define an isomorphism $\alpha\colon (f \colon (\mathcal{X},\mathscr{A})\to\mathcal{C})\to (f' \colon (\mathcal{X}',\mathscr{A}')\to\mathcal{C}')$ of any two objects of $\mathcal{M}_{d,v,u,r}(S)$ to be an $S$-isomorphism $\alpha\colon\mathcal{X}\to\mathcal{X}'$ such that there exists $\mathscr{B}\in\mathbf{Pic}_{\mathcal{C}/S}(S)$ satisfying that $\alpha^*\mathscr{A}'=\mathscr{A}\otimes f^*\mathscr{B}$ as elements of $\mathbf{Pic}_{\mathcal{X}/S}(S)$. 

As noted in \cite[Remark 5.7]{HH}, we see that the reduced structure of $\mathcal{M}_{d,v,u,r}$ is independent of the choice of sufficiently large $r$. 
We denote the reduced structure by ($\mathcal{M}_{d,v,u})_{\mathrm{red}}$. 
\end{defn}

\begin{thm}[{\cite[Theorem 1.3]{HH}}]\label{thm--adiabatic--k-moduli}
    $\mathcal{M}_{d,v,u,r}$ constructed as above is a separated Deligne--Mumford stack of finite type over $\mathbbm{k}$.
    In particular, $\mathcal{M}_{d,v,u,r}$ has a coarse moduli space.
\end{thm}

\begin{defn}
By \cite{KeM} and Theorem \ref{thm--adiabatic--k-moduli}, $\mathcal{M}_{d,v,u,r}$ and $(\mathcal{M}_{d,v,u})_{\mathrm{red}}$ have coarse moduli spaces.
    Let $M_{d,v,u,r}$ and $(M_{d,v,u})_{\mathrm{red}}$ be the coarse moduli spaces of $\mathcal{M}_{d,v,u,r}$ and $(\mathcal{M}_{d,v,u})_{\mathrm{red}}$, respectively.
    For the definition of coarse moduli spaces, see \cite[Definition 11.1.1]{Ols}.
\end{defn}

For an object $(\mathcal{X},\mathcal{A})\to\mathcal{C}\in\mathcal{M}_{d,v,u,r}(S)$ over a scheme $S$, in general $\mathcal{C}$ may not be a projective bundle over $S$ in the Zariski topology. 
However, when $S$ is a smooth curve, $\mathcal{C}$ is a $\mathbb{P}^1$-bundle in the Zariski topology by Tsen's theorem (cf.~\cite[Chapter V.2]{Ha}).
Moreover, we have the following lemma.

\begin{lem}\label{lem--P^1--bundle--trivialize}
    Let $\pi\colon\mathcal{P}\to \mathrm{Spec}\,R$ be a smooth projective morphism with any geometric fiber isomorphic to $\mathbb{P}^1$, where $R$ is a discrete valuation ring essentially of finite type over $\mathbbm{k}$.
Then there exists a finite extension of rings $R\subset R'$ such that $R'$ is regular and there exists an isomorphism
    \[
    \mathbb{P}^1_{R'}\cong \mathcal{P}\times_{\mathrm{Spec}\,R}\mathrm{Spec}\,R'.
    \]
\end{lem}

\begin{proof}
By the assumption, we note that $-K_{\mathcal{P}/\mathrm{Spec}\,R}$ is $\pi$-ample.
Let $r$ be the closed point of $\mathrm{Spec}\,R$.
Since $H^1(\mathbb{P}^1,\mathcal{O}_{\mathbb{P}^1}(-K_{\mathbb{P}^1}))=0$, we can apply \cite[III Theorem 12.11]{Ha} to conclude that there exists an effective divisor $D$ on $\mathcal{P}$ flat over $R$ such that $D|_{\mathcal{P}_{r}}$ is a smooth divisor of degree two.
Since $D$ is finite over $\mathrm{Spec}\,R$ and regular, we see that if we let $R'$ be the coordinate ring of $D$, then $R'$ is finite over $R$.

Now, we deal with the assertion.
By the definition of $R'$, we note that there exists a section $\iota\colon \mathrm{Spec}\,R'\to \mathcal{P}\times_{\mathrm{Spec}\,R}\mathrm{Spec}\,R'$.
Let $\pi'\colon \mathcal{P}\times_{\mathrm{Spec}\,R}\mathrm{Spec}\,R'\to \mathcal{P}\times_{\mathrm{Spec}\,R}\mathrm{Spec}\,R'$.
Then, it is easy to see that there exists an isomorphism 
\[
\mathcal{P}\times_{\mathrm{Spec}\,R}\mathrm{Spec}\,R'\to \mathbb{P}_{\mathrm{Spec}\,R'}(\pi'_*\mathcal{O}_{\mathcal{P}\times_{\mathrm{Spec}\,R}\mathrm{Spec}\,R'}(\iota(\mathrm{Spec}\,R')))\cong \mathbb{P}^1_{\mathrm{Spec}\,R'}
\]
as the first paragraph of the proof of \cite[Proposition 4.2]{HH}.
\end{proof}

\begin{lem}\label{lem--flat--moduli}
    Let $f\colon \mathcal{X}\to \mathcal{C}$ be a proper morphism of schemes flat and projective over $S$ with an $f$-ample line bundle $\mathscr{A}$ on $\mathcal{X}$.
    Suppose that there exists an open subset $U\subset \mathcal{C}$ such that $f|_{f^{-1}(U)}$ is flat and $U\cap \mathcal{C}_s\ne\emptyset$, $f_{s}\colon \mathcal{X}_s\to \mathcal{C}_s$ is flat and $\mathcal{C}_s$ is reduced and connected for any $s\in S$.
Then, there exists a closed subscheme $S'\subset S$ such that for any morphism $g\colon T\to S$ from a scheme, $g$ factors through $S'$ if and only if $f_{\mathcal{C}_T}$ is flat.

Furthermore, we have $S_{\mathrm{red}}=S'_{\mathrm{red}}$.
\end{lem}

\begin{proof}
By applying \cite[Theorem 5.13]{FGA}, there exists a locally closed decomposition $\mathcal{C}'\to \mathcal{C}$ such that for any morphism $h\colon T\to \mathcal{C}$ from a scheme, $\mathcal{X}_T$ is flat over $T$ if and only if $h$ factors through $\mathcal{C}'$. 
Now, we may assume that $S$ is connected and hence $\mathcal{C}$ is connected.
By the assumption and \cite[Theorem 3.20]{kollar-moduli}, we see that the reduced structure $\mathcal{C}_{\mathrm{red}}$ of $\mathcal{C}$ factors through $\mathcal{C}'$.
Thus, we see that $\mathcal{C}'$ is a closed subscheme of $\mathcal{C}$.
Next, we apply \cite[Theorem 5.13]{FGA} to $\mathcal{C}'\to S$ and we obtain a locally closed decomposition $S'\to S$ such that for any morphism $g'\colon T'\to S$ from a scheme, $g'$ factors through $S'$ if and only if $\mathcal{C}'_T$ is flat over $T$.
It is easy to see that $S_{\mathrm{red}}$ also factors through $S'$.
Thus, $S'$ is a closed subscheme of $S$ such that $S_{\mathrm{red}}=S'_{\mathrm{red}}$.
By the construction of $S'$, it is enough to check that for any morphism $g\colon T\to S$ from a scheme, $g$ factors through $S'$ if  $f_{\mathcal{C}_T}$ is flat. 
Take such a morphism $g$. 
Then, $\mathcal{C}_{T}$ factors through $\mathcal{C}'$ by the property of $\mathcal{C}'$.
Furthermore, we see that the natural morphism $\mathcal{C}_T\to \mathcal{C}'_T$ is an isomorphism.
Since $\mathcal{C}_T$ is flat over $T$, $g$ also factors through $S'$.
We obtain the assertion. 
\end{proof}

\begin{lem}\label{lem--stein--factorizaton}
Let $S$ be an arbitrary reduced scheme.
    For any object $f\colon (\mathcal{X},\mathscr{A})\to\mathcal{C}$ of $\mathcal{M}_{d,v,u,r}(S)$, the morphism $f$ is flat and we have $f_*\mathcal{O}_{\mathcal{X}}\cong\mathcal{O}_{\mathcal{C}}$.
    Furthermore, there exists a positive integer $l_0\in\mathbb{Z}_{>0}$, depending only on $d$, $v$, $u$, $r$, satisfying the following.
     There exists a line bundle $\mathscr{L}$ on $\mathcal{C}$ such that $f^*\mathscr{L}\cong \omega_{\mathcal{X}/S}^{[-l_0]}$.
  \end{lem}
  
  \begin{proof}
  First, we will show that $f_*\mathcal{O}_{\mathcal{X}}\cong\mathcal{O}_{\mathcal{C}}$ for any reduced scheme $S$.
  Note that $f$ is flat by Lemma \ref{lem--flat--moduli}.
   For any object $f\colon (\mathcal{X},\mathscr{A})\to\mathcal{C}$ over $S$, we obtain that $f_*\mathcal{O}_{\mathcal{X}}\cong\mathcal{O}_{\mathcal{C}}$.
    Indeed, since $f$ is flat, we obtain a natural injection $\mathcal{O}_{\mathcal{C}}\hookrightarrow f_*\mathcal{O}_{\mathcal{X}}$.
    Let $\mu\colon\mathcal{C}'\to\mathcal{C}$ be the Stein factorization of $f$ and let $\mathfrak{d}$ be the cokernel of $\mathcal{O}_{\mathcal{C}}\hookrightarrow f_*\mathcal{O}_{\mathcal{X}}$.
    Since $f$ is the ample model of $\omega_{\mathcal{X}/S}^{[-l]}$ for any sufficiently large and divisible $l>0$, there exists a $\pi_{\mathcal{C}}$-ample line bundle $\mathscr{L}$ on $\mathcal{C}$ such that $f^*\mathscr{L}\cong \omega_{\mathcal{X}/S}^{[-l_1]}$ for some $l_1\in\mathbb{Z}$ and $(\pi_{\mathcal{C}})_*\mathscr{L}^{\otimes l}\cong (\pi_{\mathcal{X}})_*\omega_{\mathcal{X}/S}^{[-ll_1]}$ for any $l\in\mathbb{Z}_{>0}$.
    This means that $(\pi_{\mathcal{C}})_*\mathscr{L}^{\otimes l}\cong (\pi_{\mathcal{C}}\circ\mu)_*\mu^*\mathscr{L}^{\otimes l}$ for any $l\in\mathbb{Z}_{>0}$.
    On the other hand, $R^1(\pi_{\mathcal{C}})_*\mathscr{L}^{\otimes l}=0$ for any sufficiently large $l$.
    Thus, $(\pi_{\mathcal{C}})_*(\mathscr{L}^{\otimes l}\otimes \mathfrak{d})=0$ by the long exact sequence and hence $\mathfrak{d}=0$.
    This shows that $\mathcal{C}\cong\mathcal{C}'$.
  
  Finally, we will show that there exists $l_0$ as asserted.
  Take an \'etale surjection $g\colon T\to (\mathcal{M}_{d,v,u})_{\mathrm{red}}$ from a reduced scheme of finite type.
  Let $S':=S\times_{(\mathcal{M}_{d,v,u})_{\mathrm{red}}}T$.
     Let $h\colon S'\to S$ be the induced \'etale surjective morphism and $h^*f$ denote the object of $\mathcal{M}_{d,v,u,r}(S')$ defined as the pullback of $f$ under $h$.
   Let $p_1,p_2\colon S'':=S'\times_SS'\to S'$ be the first and the second projection. 
   Let $\mathcal{X}'':=\mathcal{X}\times_SS''$ and $\mathcal{X}':=\mathcal{X}\times_SS'$. 
Let $f_T\colon \mathcal{X}_T\to \mathcal{C}_T\in \mathcal{M}_{d,v,u,r}(T)$ be the associated object to $T\to \mathcal{M}_{d,v,u,r}(T)$. 
      By the same argument in the first paragraph, there exists a line bundle $\mathscr{L}_T$ on $\mathcal{C}_T$ such that $f_T^*\mathscr{L}_T\cong \omega_{\mathcal{X}_T/T}^{[-l_0]}$ for some $l_0\in\mathbb{Z}$.
      Therefore, there exists a $h^*\pi_{\mathcal{C}}$-ample line bundle $\mathscr{L}'$ on $\mathcal{C}\times_SS'$ such that $(h^*f)^*\mathscr{L}'\cong \omega_{\mathcal{X}'/S'}^{[-l_0]}$ for some $l_0\in\mathbb{Z}$. 
      Set $p_1^*\mathscr{L}'$ and $p_2^*\mathscr{L}'$ as the pullbacks of $\mathscr{L}'$ on $\mathcal{C}\times_SS''$ under $p_1$ and $p_2$ respectively.
      Note that $p_1^*h^*f$ and $p_2^*h^*f$ are canonically isomorphic because $h \circ p_{1}= h \circ p_{2}$ by construction of $h \colon S' \to S$ and $S''':=S'\times_SS'\to S'$.
      We note that $(p_1^*h^*f)^*(p_1^*\mathscr{L}')\cong\omega_{\mathcal{X}''/S''}^{[-l_0]} \cong (p_1^*h^*f)^*(p_2^*\mathscr{L}')$. 
      Since $(p_1^*h^*f)_*\mathcal{O}_{\mathcal{X}''}\cong\mathcal{O}_{\mathcal{C}\times_SS''}$, there exists a canonical isomorphism $\sigma\colon p_1^*\mathscr{L}'\to p_2^*\mathscr{L}'$.
      It is easy to see that $(\mathscr{L}',\sigma)$ forms a descent datum (cf.~\cite[Remark 2.10]{HH}).
      Thus, there exists a line bundle $\mathscr{L}$ on $\mathcal{X}$ such that $f^*\mathscr{L}\cong\omega_{\mathcal{X}/S}^{[-l_0]}$.
        \end{proof}

We discuss the adiabatic limit of the CM line bundle on the normalization of $\mathcal{M}_{d,u,v,r}$.

\begin{setup}\label{setup--62}
Let $f\colon (\mathcal{X},\mathcal{A})\to \mathcal{C}$ be a family of uniformly adiabatically K-stable polarized klt--trivial fibrations over curves over a scheme $S$ of dimension $d$.
Let $\pi_{\mathcal{C}}\colon \mathcal{C}\to S$ and $\pi_{\mathcal{X}}\colon \mathcal{X}\to S$ be the canonical morphisms.
Suppose that $f$ is flat and $\mathcal{A}$ is a $\pi_{\mathcal{X}}$-ample line bundle.
Take arbitrary sufficiently large integers $m$ and $q\in\mathbb{Z}_{>0}$.
By Lemma \ref{lem--stein--factorizaton}, there exists $\mathcal{L}$ a $\pi_{\mathcal{C}}$-ample line bundle on $\mathcal{C}$ such that there exists a positive integer $l\in\mathbb{Z}_{>0}$ such that $f^*\mathcal{L}\sim\omega_{\mathcal{X}/S}^{[-l]}$.
We note that we can always take such $\mathcal{L}$.
Since $\mathcal{A}$ is $f$-ample and $f$ is flat, 
$\mathscr{F}_m:=f_*\mathcal{O}_{\mathcal{X}}(m\mathcal{A})$ is a locally free sheaf on $\mathcal{C}$ for any sufficiently large $m$.
By applying the Knudsen--Mumford expansion \cite[Theorem 4]{KnMu} to $\mathscr{F}_m$ for each $m$, we obtain following line bundles $\lambda^{(m)}_i$ uniquely up to isomorphisms for $i=0,1,2$ such that 
\[
\mathrm{det}\left((\pi_{\mathcal{C}})_*(\mathscr{F}_m\otimes\mathcal{O}_{\mathcal{C}}(q\mathcal{L}))\right)=(\lambda^{(m)}_0)^{\otimes \binom{q}{2}}\otimes (\lambda^{(m)}_1)^{\otimes q}\otimes \lambda^{(m)}_2.
\]
We note that if we fix an arbitrary sufficiently large $q$, then by Lemma \ref{lem--stein--factorizaton} and \cite[Theorem 4]{KnMu}, there exist line bundles $\mathcal{M}^{(q)}_{i}$ for $i=0,\ldots,d+1$ that
\begin{align*}
\mathrm{det}\left((\pi_{\mathcal{C}})_*(\mathscr{F}_m\otimes\mathcal{O}_{\mathcal{C}}(q\mathcal{L}))\right)&=\mathrm{det}\left((\pi_{\mathcal{X}})_*(\mathcal{O}_{\mathcal{X}}(m\mathcal{A})\otimes f^*\mathcal{O}_{\mathcal{C}}(q\mathcal{L}))\right)\\
&=\bigotimes_{i=0}^{d+1}(\mathcal{M}^{(q)}_{i})^{\otimes \binom{m}{i}}
\end{align*}
for any sufficiently large $m>0$.
Here, we note that $R^k(\pi_{\mathcal{X}})_*(\mathcal{O}_{\mathcal{X}}(m\mathcal{A})\otimes f^*\mathcal{O}_{\mathcal{C}}(q\mathcal{L}))=0$ for any $k>0$ by the Fujita vanishing theorem \cite[Theorem 1.4.35]{Laz}.
By the two decompositions above, there exist $\mathcal{M}_{j,i}$ for $j=0,1,2$ and $i=0,\ldots,d+1$ uniquely up to isomorphisms such that 
\[
\lambda_j^{(m)}=\bigotimes_{i=0}^{d+1}(\mathcal{M}_{j,i})^{\otimes\binom{m}{i}}
\]
for any sufficiently large $m>0$. 
This means that 
\[
\mathrm{det}\left((\pi_{\mathcal{X}})_*(\mathcal{O}_{\mathcal{X}}(m\mathcal{A})\otimes f^*\mathcal{O}_{\mathcal{C}}(q\mathcal{L}))\right)=\bigotimes_{j=0}^2\bigotimes_{i=0}^{d+1}(\mathcal{M}_{j,i})^{\otimes\binom{m}{i}\binom{q}{j}}.
\]

Consider a polarization $m\mathcal{A}+q'f^*\mathcal{L}$ and the CM line bundle $\lambda_{\mathrm{CM},\pi_{\mathcal{X}},(\mathcal{X},m\mathcal{A}+qf^*\mathcal{L})}$ for any sufficiently large $m\in\mathbb{Z}_{>0}$, and $q':=\frac{q}{m}\in\mathbb{Q}_{>0}$. 
By the definition of the CM line bundle, $\lambda_{\mathrm{CM},\pi_{\mathcal{X}},(\mathcal{X},m\mathcal{A}+mq'f^*\mathcal{L})}=\lambda_{\mathrm{CM},\pi_{\mathcal{X}},(\mathcal{X},\mathcal{A}+q'f^*\mathcal{L})}$.
We note that by the definition of the CM line bundle and the decomposition of $\mathrm{det}\left((\pi_{\mathcal{X}})_*(\mathcal{O}_{\mathcal{X}}(m\mathcal{A})\otimes f^*\mathcal{O}_{\mathcal{C}}(mq'\mathcal{L}))\right)$ as above, there exist $\mathbb{Q}$-divisors $D_0,D_1,D_2,D_3$ and $D_4$ on $S$ such that $$\lambda_{\mathrm{CM},\pi_{\mathcal{X}},(\mathcal{X},\mathcal{A}+q'f^*\mathcal{L})}^{\otimes (\mathcal{A}_t^d+dq'\mathcal{A}_t^{d-1}\cdot \mathcal{L}_t)^2}=q'^4D_4+q'^3D_3+q'^2D_2+q'D_1+D_0,$$ where $t\in S$ is an arbitrary point.
It is easy to see that $D_0,D_1,D_2,D_3$ and $D_4$ are $\mathbb{Q}$-Cartier.\label{setup--1}
\end{setup}

\begin{prop}\label{prop--CM--limit}
    Suppose that $S$ is normal in the situation as in Setup \ref{setup--1}.
    Then, $D_3\sim_{\mathbb{Q}}0$ and $D_4\sim_{\mathbb{Q}}0$.
\end{prop}

\begin{proof}
    Since $D_3$ and $D_4$ are $\mathbb{Q}$-Cartier, we may take a resolution of singularities of $S$ and may assume that $S$ is smooth.
    By applying \cite[Lemma A.2 (a)]{CP} to $\pi_{\mathcal{X}}\colon \mathcal{X}\to S$ and a line bundle $\mathcal{A}+q'f^*\mathcal{L}$, we obtain the following formula.
    \begin{align*}
    &\lambda_{\mathrm{CM},\pi_{\mathcal{X}},(\mathcal{X},m\mathcal{A}+mq'f^*\mathcal{L})}=\frac{1}{\mathcal{A}_t^d+dq'\mathcal{A}_t^{d-1}\cdot \mathcal{L}_t}\\
    &\cdot(\pi_{\mathcal{X}})_*\left(K_{\mathcal{X}/S}\cdot (\mathcal{A}+q'f^*\mathcal{L})^d-\frac{d(\mathcal{A}_t+q'f^*\mathcal{L}_t)^{d-1}\cdot K_{\mathcal{X}_t}}{(d+1)(\mathcal{A}_t^d+dq'\mathcal{A}_t^{d-1}\cdot \mathcal{L}_t)}(\mathcal{A}+q'f^*\mathcal{L})^{d+1}\right).
    \end{align*}
    Here, we note that the proof of \cite[Lemma A.2 (a)]{CP} also works where $X$ is quasi-projective. 
    By Set up \ref{setup--1}, we know that $-lK_{\mathcal{X}/S}\sim\mathcal{L}$.
    Thus, we obtain  
    \[
    \lambda_{\mathrm{CM},\pi_{\mathcal{X}},(\mathcal{X},m\mathcal{A}+mq'f^*\mathcal{L})}=\frac{-d}{2l\mathcal{A}_t^{d-1}\cdot \mathcal{L}_t}(\pi_{\mathcal{X}})_*\left(\mathcal{A}^{d-1}\cdot f^*\mathcal{L}^2\right)+O(q'^{-1}).
    \]
    Thus, $D_3\sim_{\mathbb{Q}}0$ and $D_4\sim_{\mathbb{Q}}0$.
\end{proof}

In the case when $S$ is normal, we see that 
\[\lim_{q'\to\infty}\lambda_{\mathrm{CM},\pi_{\mathcal{X}},(\mathcal{X},\mathcal{A}+q'f^*\mathcal{L})}=\frac{1}{d^2(\mathcal{A}_t^{d-1}\cdot \mathcal{L}_t)^2}D_2
\]
in $\mathrm{Pic}_{\mathbb{Q}}(S)$. 
We also note that $\lim_{q'\to\infty}\lambda_{\mathrm{CM},\pi_{\mathcal{X}},(\mathcal{X},\mathcal{A}+q'f^*\mathcal{L})}$ only depends on the relative linear equivalence class of $\mathcal{A}$ over $\mathcal{C}$.

\begin{defn}[Adiabatic limit over scheme]\label{defn--CM--limit}
    Let $f\colon (\mathcal{X},\mathcal{A})\to \mathcal{C}$ be a family of uniformly adiabatically K-stable polarized klt--trivial fibrations over curves over a scheme $S$ whose fibers are of dimension $d$. 
    Suppose that $S$ is normal.
    Then, $f$ is flat by Lemma \ref{lem--flat--moduli}.
Let $\pi_{\mathcal{C}}\colon \mathcal{C}\to S$ and $\pi_{\mathcal{X}}\colon \mathcal{X}\to S$ be the canonical morphisms.
Then, 
$$\lambda_{\mathrm{CM},f}^\infty:=\lim_{q'\to\infty}\lambda_{\mathrm{CM},\pi_{\mathcal{X}},(\mathcal{X},\mathcal{A}+q'f^*\mathcal{L})}$$
 exists as a $\mathbb{Q}$-line bundle on $S$ by Proposition \ref{prop--CM--limit} and the discussion right after the proposition. 
We call $\lambda_{\mathrm{CM},f}^\infty$ the {\em adiabatic limit} of the CM line bundle.
\end{defn}

Now, we can define the limit of the CM line bundle on the moduli.
By the proof of \cite[Theorem 5.1]{HH}, we know that there exist a quasi-projective scheme $Z$ and a connected semisimple linear algebraic group $G$ acting on $Z$ such that $[Z/G]\cong \mathcal{M}_{d,v,u,r}$.
By the construction, we know that $Z$ is a locally closed subscheme of a certain Hilbert scheme and there exists the universal embedded object $f_Z\colon (\mathcal{U},\mathcal{A})\to \mathcal{C}\in \mathcal{M}_{d,v,u,r}(Z)$, where $\mathcal{A}$ is a $f_Z$-ample line bundle.
Here, $\mathcal{A}$ is unique up to relative linear equivalence over $Z$ and summation of $\omega^{[-tr]}_{\mathcal{U}/Z}$ for $t\in\mathbb{Z}$.

\begin{defn}[Adiabatic limit over moduli]\label{defn--limit--CM} 
Fix one relative linear equivalence class of $\mathcal{A}$ over $Z$.
$G$ also acts on $f_Z$ and $\mathcal{A}$ equivariant over $Z$.
As Lemma \ref{lem--stein--factorizaton}, we take $l_0\in\mathbb{Z}_{>0}$ and a line bundle $\mathscr{L}$ on $\mathcal{C}$ such that $f^*\mathscr{L}\sim_Z \omega_{_{\mathcal{U}/Z}}^{[-l_0]}$. 
There exists $t_0\in\mathbb{Z}_{>0}$ such that for any integer $t'\ge t_0$, $\mathcal{A}\otimes \omega^{[-t'r]}_{\mathcal{U}/Z}$ is also relatively ample over $Z$.
Let $\pi_Z\colon \mathcal{U}\to Z$ be the canonical projection.
Consider the CM line bundle $\lambda_{\mathrm{CM},\pi_{Z},(\mathcal{U},\mathcal{A}^{\otimes m}\otimes f^*\mathscr{L}^{\otimes mt})}$, where $t\in\mathbb{Q}$ such that $t>\frac{t_0r}{l_0}$ and $m\in\mathbb{Z}_{>0}$ is sufficiently divisible such that $mt\in\mathbb{Z}$.
Note that $\lambda_{\mathrm{CM},\pi_{Z},(\mathcal{U},\mathcal{A}^{\otimes m}\otimes f^*\mathscr{L}^{\otimes mt})}$ is independent of the choice of $m$, and hence we write it  $\lambda_{\mathrm{CM},\pi_{Z},(\mathcal{U},\mathcal{A}+tf^*\mathscr{L})}$ for simplicity.
It is easy to see by what we state just before Definition \ref{defn--CM--line--bundle} that some non-zero multiple of $\lambda_{\mathrm{CM},\pi_{Z},(\mathcal{U},\mathcal{A}+tf^*\mathscr{L})}$ admits the natural action of $G$. 
Thus, there exists a $\mathbb{Q}$-line bundle $\lambda_{\mathrm{CM},t}$ on $\mathcal{M}_{d,v,u,r}$ whose pullback to $Z$ is $\mathbb{Q}$-linearly equivalent to $\lambda_{\mathrm{CM},\pi_{Z},(\mathcal{U},\mathcal{A}+tf^*\mathscr{L})}$.
As noted in \cite[Subsection 3.2]{Hat23}, there exists a $\mathbb{Q}$-line bundle $\Lambda_{\mathrm{CM},t}$ on the coarse moduli space $M_{d,v,u,r}$ whose pullback to $\mathcal{M}_{d,v,u,r}$ is $\mathbb{Q}$-linearly equivalent to $\lambda_{\mathrm{CM},t}$.
We call $\lambda_{\mathrm{CM},t}$ and $\Lambda_{\mathrm{CM},t}$  the CM line bundles.
Take the normalization $\nu\colon(\mathcal{M}_{d,v,u})_{\mathrm{red}}^\nu\to\mathcal{M}_{d,v,u,r}$.
We know that if we let $\tilde{\nu}\colon Z^\nu\to Z$ be the normalization, then $(\mathcal{M}_{d,v,u})_{\mathrm{red}}^\nu\cong[Z^\nu/G]$.
By Definition \ref{defn--CM--limit} and the same argument of Proposition \ref{prop--CM--limit}, we see that $\tilde{\nu}^*\lambda_{\mathrm{CM},\pi_{Z},(\mathcal{U},\mathcal{A}+tf^*\mathscr{L})}$ converges to the $G$-linearized $\mathbb{Q}$-line bundle $\lambda^\infty_{\mathrm{CM},\pi_{Z^\nu}}$, where $\pi_{Z^\nu}\colon \mathcal{U}^\nu\to Z^\nu$ is the canonical morphism.
This means that there also exists an element $\lambda^\infty_{\mathrm{CM}}\in \mathrm{Pic}_{\mathbb{Q}}((\mathcal{M}_{d,v,u})_{\mathrm{red}}^\nu)$ such that $\lim_{t\to\infty}\nu^*\lambda_{\mathrm{CM},t}=\lambda^\infty_{\mathrm{CM}}$.
Furthermore, there exists an element $\Lambda^\infty_{\mathrm{CM}}\in \mathrm{Pic}_{\mathbb{Q}}((M_{d,v,u})_{\mathrm{red}}^\nu)$ such that $\lim_{t\to\infty}\nu'^*\Lambda_{\mathrm{CM},t}=\Lambda^\infty_{\mathrm{CM}}$, where $\nu'\colon (M_{d,v,u})_{\mathrm{red}}^\nu \to M_{d,v,u,r}$ is the normalization and the pullback of $\Lambda^\infty_{\mathrm{CM}}$ to $(M_{d,v,u})_{\mathrm{red}}^\nu$ coincides with $\lambda^\infty_{\mathrm{CM}}$.

    We say that the $\mathbb{Q}$-line bundles $\lambda^\infty_{\mathrm{CM}}$ and  $\Lambda^\infty_{\mathrm{CM}}$ constructed as above are the {\it adiabatic limits of the CM line bundles} $\lambda_{\mathrm{CM},t}$ and $\Lambda_{\mathrm{CM},t}$ respectively.
\end{defn}

\section{K-stability of quasimaps}\label{sec3}

In this section, we will collect fundamental notions of quasimaps, define K-stability of log Fano quasimaps and construct K-moduli theory.

\subsection{Notations and definitions}

In this subsection, we collect notations, definitions, and basic properties of quasimaps used in the rest of this paper. 

\begin{note}
    Let $X$ be a scheme and $\mathcal{E}$ a locally free sheaf on $X$.
    We set 
    $$\mathbb{P}_X(\mathcal{E}):=\mathbf{Proj}_X(\bigoplus_{m\ge0}\mathbf{Sym}^m(\mathcal{E})) \quad {\rm and} \quad \mathbb{A}_X(\mathcal{E}):=\mathbf{Spec}_X(\bigoplus_{m\ge0}\mathbf{Sym}^m(\mathcal{E})).$$ 
    The zero section of $\mathbb{A}_X(\mathcal{E})$ is denoted by $\mathbf{0}$.
\end{note}

\begin{defn}[Quasimap]\label{defn-qmaps}
    Let $C$ be a proper connected reduced curve with at worst nodal singularities.
Take $N\in\mathbb{Z}_{>0}$. We call a pair $(C,q)$ a {\it quasimap} to $\mathbb{P}_{\mathbbm{k}}^N$ if the following holds.
\begin{itemize}
    \item $q\colon C\to [\mathbb{A}_{\mathbbm{k}}^{N+1}/\mathbb{G}_{m,\mathbbm{k}}]$ is a morphism such that $q$ maps the generic point of every irreducible component of $C$ to $[\mathbb{A}_{\mathbbm{k}}^{N+1}\setminus\{0\}/\mathbb{G}_{m,\mathbbm{k}}]\cong\mathbb{P}_{\mathbbm{k}}^N$, where $\mathbb{G}_m$ acts on $\mathbb{A}^{N+1}$ in the natural way.
\end{itemize}

For any algebraic space $S$ over $\mathbbm{k}$ and morphism $q\colon \mathcal{C}\to[\mathbb{A}_S^{N+1}/\mathbb{G}_{m,S}]$, we call $(\mathcal{C},q)$ a {\it family of quasimaps to $\mathbb{P}^N$} over $S$ if the morphism satisfies the following. 
\begin{itemize}
\item the composition $\mathcal{C} \to[\mathbb{A}_S^{N+1}/\mathbb{G}_{m,S}] \to S$ is a proper surjective flat morphism, and
    \item for any geometric point $\bar{s}\in S$, $(\mathcal{C}_{\bar{s}},q_{\bar{s}})$ is a quasimap to $\mathbb{P}^{N}$.
\end{itemize}
Let $(\mathcal{C}',q')$ be another family of quasimaps to $\mathbb{P}^N$ over $S$.
A {\it morphism between families of quasimaps}, which we denote by $f\colon (\mathcal{C},q)\to (\mathcal{C}',q')$, is a morphism $f\colon \mathcal{C}\to \mathcal{C}'$ of algebraic spaces over $S$ such that $q'\circ f$ and $q$ are isomorphic as $1$-morphisms.

Let $X$ be a closed subscheme of $\mathbb{P}_{\mathbbm{k}}^N$ and $\mathrm{Cone}(X)\subset \mathbb{A}_{\mathbbm{k}}^{N+1}$ the affine cone of $X$.
Suppose that $X$ is normal and irreducible.
For a family of quasimaps $(\mathcal{C},q)$ to $\mathbb{P}_S^{N}$ over $S$, if $q$ factors through $[\mathrm{Cone}(X)\times S/\mathbb{G}_{m,S}]$, we say that a family of quasimaps $(\mathcal{C},q)$ {\it factors through} $X\subset \mathbb{P}^N$.

We define the category fibered in groupoids $\mathsf{Qmaps}_{X\subset\mathbb{P}^N}$ as follows.
For any scheme $S$, we set $\mathsf{Qmaps}_{X\subset\mathbb{P}^N}(S)$ as the category whose objects are families of quasimaps $(\mathcal{C},q)$ that factor through $X\subset \mathbb{P}^N$ over $S$, and morphisms are morphisms of families of quasimaps.
We can easily check that $\mathsf{Qmaps}_{X\subset\mathbb{P}^N}$ is a category fibered in groupoids by defining the pullback of $(\mathcal{C},q)$ as $(\mathcal{C}_T,q_T)$ for any morphism of schemes $T\to S$, where $q_T$ is the composition of the natural morphism $\mathcal{C}_T\to \mathcal{C}$ and $q$.
If $T$ is the spectrum of $\kappa(s)$ (resp.~$\overline{\kappa(s)}$), then we write $q_s$ (resp.~$q_{\bar{s}}$) as $q_T$.
It is also easy to see that $\mathsf{Qmaps}_{X\subset\mathbb{P}^N}$ is naturally a stack (cf.~\cite[Subsection 2.3]{CFKM}).
\end{defn}

We define the following fundamental notions of quasimaps.

\begin{defn}[Module of sections]\label{defn--linebundle-seq-quasimap}
Let $S$ be a scheme and let $q\colon \mathcal{C}\to [\mathbb{A}_S^{N+1}/\mathbb{G}_{m,S}]$ be a family of quasimaps to $\mathbb{P}^N$ over $S$. 
By the definition of quotient stacks, $q$ corresponds to an isomorphic class of principal $\mathbb{G}_{m, S}$-bundles $\mu\colon\mathcal{P}\to\mathcal{C}$ which has a $\mathbb{G}_{m, S}$-equivariant morphism $p\colon \mathcal{P}\to \mathbb{A}_S^{N+1}$. 
By \cite{Ser}, for any point $c\in\mathcal{C}$, there exists a Zariski open subset $U\subset\mathcal{C}$ such that $c\in U$ and $\mathcal{P}|_U\cong\mathbb{G}_{m,U}$ as principal $\mathbb{G}_{m, U}$-bundles.
Thus, there exists a line bundle $\mathscr{L}$ on $\mathcal{C}$ such that $\mathcal{P}\cong\mathbb{A}_{\mathcal{C}}(\mathscr{L})\setminus\mathbf{0}$ as principal $\mathbb{G}_{m,S}$-bundles. 
We say that such $\mathscr{L}$ is a {\em line bundle associated with $q$}.
$\mathscr{L}$ is uniquely determined up to relative linear equivalence over $S$.

We now define $\mathrm{Sec}_q(\mathscr{L})$ and $|\mathscr{L}|_q$. 
The morphism $p\colon \mathcal{P}\to \mathbb{A}_S^{N+1}$ induces 
$$\bigoplus_{m\in\mathbb{Z}_{\ge0}}\mathbf{Sym}^m\mathcal{O}_{\mathcal{C}}^{\oplus N+1} \to \mu_*\mathcal{O}_{\mathcal{P}}=\bigoplus_{m \in \mathbb{Z}}\mathscr{L}^{\otimes m},$$ where $\mu\colon \mathcal{P}\to \mathcal{C}$ is the canonical morphism and $\mathscr{L}^{\otimes m}$ is the dual of $\mathscr{L}^{\otimes (-m)}$ when $m<0$. 
By the $\mathbb{G}_{m, S}$-equivariance of $p$, the above morphism preserves the grading. 
Hence, we get a morphism $\mathcal{O}_{\mathcal{C}}^{\oplus N+1}\to\mathscr{L}$, which induces 
\[
p^* \colon H^0(\mathcal{O}_{\mathcal{C}})^{\oplus N+1}\to H^0(\mathscr{L}). 
\]
We call ${\rm Im}(p^*) \subset H^0(\mathscr{L})$ the {\it module of sections} and we denote it by $\mathrm{Sec}_q(\mathscr{L})$. 
We also define 
$$|\mathscr{L}|_q:=(\mathrm{Sec}_q(\mathscr{L})\setminus\{0\})/\mathbbm{k}^\times.$$
For any $f\in \mathrm{Sec}_q(\mathscr{L})$, $\mathrm{div}_{\mathscr{L}}(f)$ denotes the zero locus of $f$.
For any morphism $T\to S$ from a scheme, let $f_T\in\mathrm{Sec}_{q_T}(\mathscr{L}_T)$ denote the pullback of $f$.
If $T$ is the spectrum of $\kappa(s)$ (resp.~$\overline{\kappa(s)}$), then we write $f_s$ (resp.~$f_{\bar{s}}$) as $q_T$.
If $f_{s}\ne0$ for any $s\in S$, or $S$ is normal and irreducible and $f\ne0$, then we see that $\mathrm{div}_{\mathscr{L}}(f)$ is a Cartier divisor.
If there is no fear of confusion, we will simply write $\mathrm{div}_{\mathscr{L}}(f)$ as $\mathrm{div}(f)$.
We note that we can regard $|\mathscr{L}|_q$ as a projective space and it has the Zariski topology.

Finally, in this paper we always fix a basis $\{e_i\}_{i=0}^N$ of $\mathbb{A}^{N+1}_{\mathbbm{k}}$ and write $f_i=p^*e_i$. 
We call $\{f_i\}_{i=0}^N$ {\it the sections corresponding to the basis of $\mathbbm{k}^{N+1}$}. 
If $\mathrm{deg}\,\mathscr{L}_{s}$ is independent of the choice of $s$, we define $\mathrm{deg}\,q$ to be
$$\mathrm{deg}\,q :=\mathrm{deg}\,\mathscr{L}_{s}$$
 for a point $s\in S$. 
Note that $s\mapsto \mathrm{deg}\,\mathscr{L}_{s}$ is locally constant because the morphism $\mathcal{C} \to S$ is flat and the fibers are proper curves. 

\end{defn}

\begin{defn}[Constant quasimap]\label{defn--quasi--map--const}
   Let $q\colon \mathbb{P}^1\to [\mathrm{Cone}(X)/\mathbb{G}_{m,\mathbbm{k}}]$ be a  quasimap with a closed embedding $X\hookrightarrow \mathbb{P}^N$.
    We say that $q$ is {\it constant} if there exist an open dense subset $U\subset \mathbb{P}^1$ such that $q(U)$ is a closed point of $X\subset [\mathrm{Cone}(X)/\mathbb{G}_{m,\mathbbm{k}}]$.
\end{defn}

For a quasimap $q\colon \mathbb{P}^1\to [\mathbb{A}^{N+1}/\mathbb{G}_{m,\mathbbm{k}}]$, choose a basis $\{e_i\}_{i=0}^N$ of $\mathbb{A}^{N+1}_{\mathbbm{k}}$ and let $\{f_i\}_{i=0}^N$ be the sections corresponding to the canonical basis of $\mathbbm{k}^{N+1}$ for $k=1,2$.
Since we are working over the field $\mathbbm{k}$ of characteristic $0$, if $q$ is non-constant, then $\mathrm{Sec}_q(\mathscr{L})$ is a vector space of at least dimension two and hence we may replace $e_i$ with a general one such that $f_j\ne0$ for any $j$ and $\mathrm{div}(f_0)\ne\mathrm{div}(f_i)$ for any $i\ge1$.
By the following lemma, if a quasimap $q\colon \mathbb{P}^1\to [\mathbb{A}^{N+1}/\mathbb{G}_{m,\mathbbm{k}}]$ is non-constant then $q$ is completely determined by  $|\mathscr{L}|_q$, where $\mathscr{L}$ is the line bundle associated to $q$.
\begin{lem}\label{lem}
    Let $q_1,\,q_2\colon \mathbb{P}^1\to [\mathbb{A}_{\mathbbm{k}}^{N+1}/\mathbb{G}_{m,\mathbbm{k}}]$ be quasimaps.
    Fix a basis $\{e_i\}_{i=0}^N$ of $\mathbb{A}^{N+1}_{\mathbbm{k}}$ and let $\{f^{(k)}_i\}_{i=0}^N$ be the sections corresponding to the canonical basis of $\mathbbm{k}^{N+1}$ for $k=1,2$.
    Suppose that $q$ is non-constant, the line bundles associated with $q_1$ and $q_2$ coincide, say $\mathscr{L}$, and $f^{(k)}_i\ne0$ and $f^{(k)}_0+f^{(k)}_i\ne0$ for any $i$ and $k=1,2$.
    Put $D_i:=\mathrm{div}(f^{(1)}_i)\in |\mathscr{L}|_q$ for $0\le i\le N$ and $D'_{j}:=\mathrm{div}(f^{(1)}_0+ f^{(1)}_j)$ for $1\le j\le N$.
    Suppose that $D_0\ne D_i$ for any $i\ge1$.
    
    If $D_i=\mathrm{div}(f^{(2)}_i)$ for $0\le i\le N$ and $D'_{j}=\mathrm{div}(f^{(2)}_0+ f^{(2)}_j)$ for $1\le j\le N$, then $q_1=q_2$.
\end{lem}

\begin{proof}
It suffices to show that there exists $c\in\mathbbm{k}\setminus\{0\}$ such that $f^{(1)}_i=cf^{(2)}_i$ for all $0\le i\le N$.
Since $\mathrm{div}(f^{(1)}_i)=\mathrm{div}(f^{(2)}_i)$, we see that there exist non-zero constants $c_i\in\mathbbm{k}\setminus\{0\}$ such that $f^{(1)}_i=c_if^{(2)}_i$.
We may assume $c_0=1$, and it is enough to show $c_i=1$ for all $i\ge1$.
Since $\mathrm{div}(f^{(1)}_0+f^{(1)}_i)=\mathrm{div}(f^{(2)}_0+f^{(2)}_i)$, we see that there exist non-zero constants $d_i\in\mathbbm{k}\setminus\{0\}$ such that $f^{(1)}_0+f^{(1)}_i=d_i(f^{(2)}_0+f^{(2)}_i)=d_i(f^{(1)}_0+c_if^{(1)}_i)$ for all $i\ge1$.
By the relation $D_0\ne D_j$, we see that $f^{(1)}_0$ and $f^{(1)}_j$ are linearly independent over $\mathbbm{k}$, and thus $c_i=d_i=1$.
Therefore, we have the assertion.
\end{proof}

The following lemma is easy but useful. 

\begin{lem}\label{lem--quasi--map--S_2}
    Let $X\hookrightarrow \mathbb{P}^N$ be an inclusion of a closed normal variety, let $\mathcal{C}\to S$ be a projective flat morphism to a normal variety with one-dimensional smooth geometric fibers, and let $q_1,q_2\colon\mathcal{C}\to[\mathrm{Cone}(X)\times S/\mathbb{G}_{m,S}]$ be two families of quasimaps such that $q_1$ is isomorphic to $q_2$ outside a closed subset $W\subset \mathcal{C}$ of codimension at least two.
    Then, $q_1$ and $q_2$ are globally isomorphic. 
\end{lem}

\begin{proof}
 It suffices to show the assertion under the assumption that $X=\mathbb{P}^N$.
    Let $\mathscr{L}_1$ and $\mathscr{L}_2$ be the associated line bundles with $q_1$ and $q_2$, respectively.
    Since $q_1$ and $q_2$ are isomorphic on $\mathcal{C}\setminus W$, we see that 
    \[
    \mathscr{L}_1|_{\mathcal{C}\setminus W}\cong \mathscr{L}_2|_{\mathcal{C}\setminus W}.
    \]
 Since $\mathrm{codim}_{\mathcal{C}}(W)\ge 2$, we have $\mathscr{L}_1\cong \mathscr{L}_2$ globally.
 Take the sections corresponding to the canonical basis $\{ f^{(1)}_i\}_{i=0}^N\subset \mathrm{Sec}_{q_1}(\mathscr{L}_1)$ and $\{ f^{(2)}_i\}_{i=0}^N\subset \mathrm{Sec}_{q_2}(\mathscr{L}_2)$. 
Then we can identify $f^{(1)}_i|_{\mathcal{C}\setminus W}=f^{(2)}_i|_{\mathcal{C}\setminus W}$ via the isomorphism $\mathscr{L}_1\cong \mathscr{L}_2$.
 Since $\mathrm{codim}_{\mathcal{C}}(W)\ge 2$ and $\mathcal{C}$ is normal, $f^{(1)}_i=f^{(2)}_i$ holds globally.
\end{proof}

\begin{defn}[Lc Fano quasimap, log Fano quasimap]\label{defn--lcFanoquasimap}
Let $X$ be a projective normal variety with a closed embedding $X \hookrightarrow \mathbb{P}^N$. 
   Let $q\colon\mathbb{P}^1\to [\mathrm{Cone}(X)/\mathbb{G}_{m, \mathbbm{k}}]$ be a quasimap of degree $m$ (see Definition \ref{defn--linebundle-seq-quasimap}), and let $B$ be an effective $\mathbb{Q}$-divisor on $\mathbb{P}^1$. 
    Let $u$ be a positive rational number.
We say that 
$q\colon(\mathbb{P}^1,B)\to [\mathrm{Cone}(X)/\mathbb{G}_{m, \mathbbm{k}}]$ is a {\it lc Fano quasimap of degree $m$ and of weight $u$} if $(\mathbb{P}^1,B+uD)$ is lc for any general divisor $D\in|\mathscr{L}|_q$ and the inequality $\mathrm{deg}(B)+mu<2$ holds.
If $(\mathbb{P}^1,B+uD)$ is further klt for any general $D\in|\mathscr{L}|_q$, then we say that $q\colon(\mathbb{P}^1,B)\to [\mathrm{Cone}(X)/\mathbb{G}_{m, \mathbbm{k}}]$ is a {\it log Fano quasimap}. 

Let $X$, $q\colon\mathbb{P}^1\to [\mathrm{Cone}(X)/\mathbb{G}_{m, \mathbbm{k}}]$, $B$, and $u$ be as above.
We define the {\it delta invariant} $\delta(q)$ by
  \begin{equation*}
    \delta(q)=\frac{2\min_{p\in\mathbb{P}^1}(1-\mathrm{mult}_p(B+uD))}{\mathrm{deg}(-K_{\mathbb{P}^1}-B-uD)}.\label{eq--delta--quasi--map}
\end{equation*}
    for any general $D\in|\mathscr{L}|_q$ if $q\colon(\mathbb{P}^1,B)\to [\mathrm{Cone}(X)/\mathbb{G}_{m, \mathbbm{k}}]$ is a log Fano quasimap of degree $m$ and of weight $u$, and otherwise we set $\delta(q)=0$. It is easy to see that $\delta(q)$ is independent from the choice of  general $D$.    
\end{defn}

Next, we define the K-stability of log Fano quasimaps.

\begin{defn}[K-stability of log Fano quasimap]\label{de--kst--qmaps}
    Let $q\colon(\mathbb{P}^1,B)\to [\mathrm{Cone}(X)/\mathbb{G}_{m, \mathbbm{k}}]$ be a lc Fano quasimap of degree $m$ and of weight $u$ with respect to a closed embedding $X\subset\mathbb{P}^N$.
Take a (semi)ample test configuration $(\mathcal{C},\mathcal{L})$ for $(\mathbb{P}^1,(2-mu-\mathrm{deg}(B))\mathcal{O}(1))$. 
We say that such $(\mathcal{C},\mathcal{L})$ is a {\it (semi)ample test configuration} for $q$.
    Then, the {\it Donaldson--Futaki invariant} is defined to be
    $$\mathrm{DF}_q(\mathcal{C},\mathcal{L})=\mathrm{DF}_{B+uD}(\mathcal{C},\mathcal{L})$$
    for any  general $D\in|\mathscr{L}|_q$. 
    By the argument of \cite[Proposition 3.8]{BLZ} (cf.~\cite{Hat}), the invariant is independent from the choice of $D$. 
We say that $q$ is
\begin{itemize}
    \item {\it K-semistable} if $\mathrm{DF}_q(\mathcal{C},\mathcal{L})\ge0$ for any ample test configuration for $q$,
    \item {\it K-stable}
 if $\mathrm{DF}_q(\mathcal{C},\mathcal{L})>0$ for any normal nontrivial ample test configuration $(\mathcal{C},\mathcal{L})$ for $q$.
\end{itemize}
\end{defn}

\begin{defn}[Family of K-semistable log Fano quasimap]
Let $X$ be a projective normal variety with a closed embedding $X \hookrightarrow \mathbb{P}^N$, and let $S$ be an algebraic space over $\mathbbm{k}$. 
    Let $q\colon \mathcal{C}\to [\mathrm{Cone}(X)\times S/\mathbb{G}_{m,S}]$ be a family of quasimaps of degree $m$.
    Let $\mathcal{B}$ be a $\mathbb{Q}$-linear combination of Cartier divisors on $\mathcal{C}$ such that every irreducible component of $\mathcal{B}$ is flat over $S$.
    Fix $u\in\mathbb{Q}_{>0}$. 
    We say that $q\colon (\mathcal{C},\mathcal{B})\to[\mathrm{Cone}(X)\times S/\mathbb{G}_{m,S}]$ is a {\it family of K-semistable log Fano quasimap of degree $m$ and of weight $u$} if any geometric point $\bar{s}\in S$ satisfies the conditions that $\mathcal{B}_{\bar{s}}$ is effective and $q_{\bar{s}}\colon (\mathcal{C}_{\bar{s}},\mathcal{B}_{\bar{s}})\to [\mathrm{Cone}(X)/\mathbb{G}_{m}]$ is K-semistable as a log Fano quasimap degree $m$ and of weight $u$. 
\end{defn}

The following notion is closely related to the K-stability of log Fano quasimaps.

\begin{defn}[Log--twisted Fano pair]
    Let $(X,B)$ be a klt log pair. If there exists a semiample $\mathbb{Q}$-line bundle $T$ such that $-(K_X+B+T)$ is ample, then we say that $(X,B,T)$ is a {\it log--twisted Fano pair}.
    For any semiample test configuration $(\mathcal{X},\mathcal{L})$ of $(X,-(K_X+B+T))$, we define the {\it log--twisted Donaldson--Futaki invariant} to be
    \[
    \mathrm{DF}_{(B,T)}(\mathcal{X},\mathcal{L}):=\mathrm{DF}_{B}(\mathcal{X},\mathcal{L})+\mathcal{J}^{T,\mathrm{NA}}(\mathcal{X},\mathcal{L}).
    \]
    We say that $(X,B,T)$ is
\begin{itemize}
    \item {\it K-semistable} if we have $\mathrm{DF}_{(B,T)}(\mathcal{X},\mathcal{L})\ge0$ for any ample test configuration for $(X,-(K_X+B+T))$,
    \item {\it uniformly K-stable} if there exists $\delta>0$ such that $\mathrm{DF}_{(B,T)}(\mathcal{X},\mathcal{L})\ge\delta J^{\mathrm{NA}}(\mathcal{X},\mathcal{L})$ for any normal nontrivial ample test configuration $(\mathcal{X},\mathcal{L})$ for $(X,-(K_X+B+T))$.
    \item {\it K-stable} if $\mathrm{DF}_{(B,T)}(\mathcal{X},\mathcal{L})>0$ for any normal nontrivial ample test configuration $(\mathcal{X},\mathcal{L})$ for $(X,-(K_X+B+T))$.
\end{itemize}
\end{defn}

The log--twisted K-stability is well studied in \cite{BLZ,Hat}.
We construct K-moduli theory of log Fano quasimaps by relating the log twisted K-stability with the K-stability of quasimaps.
To do so, we introduce the following notion.

\begin{defn}[Fixed part and movable part]\label{defn--fixed-movable-part-quasimap}
Let $q\colon(\mathbb{P}^1,B)\to [\mathrm{Cone}(X)/\mathbb{G}_{m, \mathbbm{k}}]$ be a lc Fano quasimap of degree $m$ and of weight $u$ with respect to a closed embedding $X\subset\mathbb{P}^N$. 
Let $\mathscr{L}$ be the line bundle on $\mathbb{P}^{1}$ associated to $q$.  
Then the {\it fixed part} of $q$ is defined to be the fixed divisorial part  of $|\mathscr{L}|_q$. 
Let $B'$ be the fixed part of $q$. 
Then the {\it movable part} of $q$ is defined to be $\mathscr{L}-B'$. 

Let $M$ be the movable part of $q$. 
By definition, $\mathscr{L}=B'+M$, and $(\mathbb{P}^1,B+uB',uM)$ is a log-twisted Fano pair (cf.~\cite{Hat}).
In this paper, $(\mathbb{P}^1,B+uB',uM)$ is called the {\it log--twisted Fano pair associated with $q$}. 
\end{defn}

For the log--twisted Fano pair $(\mathbb{P}^1,B+uB',uM)$, K-stability is completely determined by the following invariant, 
\begin{equation*}
    \delta(\mathbb{P}^1,B+uB',uM):=\frac{2\min_{p\in\mathbb{P}^1}(1-\mathrm{mult}_p(B+uB'))}{\mathrm{deg}(-K_{\mathbb{P}^1}-B)-um}\label{eq--delta--log--twisted--map},
\end{equation*}
by \cite[Theorem A.16]{Hat}.
More precisely, $\delta(\mathbb{P}^1,B+uB',uM)\ge 1$ (resp.~$>1$) if and only if $(\mathbb{P}^1,B+uB',uM)$ is log--twisted K-semistable (resp.~K-stable). We note that the K-stability and the uniform K-stability are equivalent in this case.
We can compare $\delta(q)$ by using $\delta(\mathbb{P}^1,B+uB',uM)$ when $u$ is sufficiently small.

\begin{lem}\label{lem--two--delta--invariants}
    Let $q\colon(\mathbb{P}^1,B)\to [\mathrm{Cone}(X)/\mathbb{G}_{m, \mathbbm{k}}]$ be a log Fano quasimap of degree $m$ and of weight $u$ with respect to a closed embedding $\iota\colon X\subset\mathbb{P}^N$, and let $B'$ and $M$ be the fixed part and the movable part of $q$ as in Definition \ref{defn--fixed-movable-part-quasimap}, respectively.
    We put $v:=\mathrm{deg}(-K_{\mathbb{P}^1}-B)-um$.
Then the inequality 
$$\delta(q)\le \delta(\mathbb{P}^1,B+uB',uM)$$
holds. 
    If further $0<u<1-\frac{v}{2}$ and $\delta(q)\le1$, then $\delta(q)=\delta(\mathbb{P}^1,B+uB',uM)$.

    In particular, $q$ is K-(semi)stable if and only if $(\mathbb{P}^1,B+uD)$ is K-(semi)stable for any general $D\in|\mathscr{L}|_q$, where $\mathscr{L}$ is the line bundle associated with $q$.
\end{lem}

\begin{proof}
    By the definition, it is easy to see $\delta(q)\le \delta(\mathbb{P}^1,B+uB',uM)$.
    Thus, it suffices to show the equality under the assumptions $0<u<1-\frac{v}{2}$ and $\delta(q)\le1$.
    Take a general member $D\in |\mathscr{L}|_{q}$ so that $D-B'$ is reduced and the support of $D-B'$ is disjoint from $\mathrm{Supp}(B+uB')$.
    If $p'\in \mathrm{Supp}(D-B')$, then 
    \[
    \frac{2(1-\mathrm{mult}_{p'}(B+uB'+u(D-B')))}{\mathrm{deg}(-K_{\mathbb{P}^1}-B-uD)}=2\frac{1-u}{v}>1.
    \]
    By the definition and the assumption that $\delta(q)\le 1$,  
    \begin{align*}
        \delta(q)&=\min_{p\in \mathbb{P}^1}\frac{2(1-\mathrm{mult}_{p}(B+uB'+u(D-B')))}{\mathrm{deg}(-K_{\mathbb{P}^1}-B-uD)}\\
        &=\min_{p\not\in \mathrm{Supp}(D-B')}\frac{2(1-\mathrm{mult}_{p}(B+uB'))}{\mathrm{deg}(-K_{\mathbb{P}^1}-B-uD)}\\
        &=\delta(\mathbb{P}^1,B+uB',uM).
    \end{align*}
    We are done.
\end{proof}

\begin{rem}\label{rem--K-stability--coincidence}
It is easy to see by \cite[Lemma 2.22]{Hat} that  
\[
\mathrm{DF}_{(B+uB',uM)}(\mathcal{C},\mathcal{L})=\mathrm{DF}_{q}(\mathcal{C},\mathcal{L})
\]
for any normal nontrivial ample test configuration $(\mathcal{C},\mathcal{L})$ for $q$.
Thus, the log--twisted K-stability of $(\mathbb{P}^1,B+uB',uM)$ and the K-stability of $q$ are equivalent.
   This shows that the K-stability of a quasimap $q\colon(\mathbb{P}^1,B)\to[\mathrm{Cone}(X)/\mathbb{G}_{m, \mathbbm{k}}]$ is equivalent to the K-stability of $[\mathrm{Cone}(\iota)]\circ q\colon(\mathbb{P}^1,B)\to[\mathbb{A}^{N+1}/\mathbb{G}_{m, \mathbbm{k}}]$, where $[\mathrm{Cone}(\iota)]\colon [\mathrm{Cone}(X)/\mathbb{G}_{m, \mathbbm{k}}]\hookrightarrow[\mathbb{A}^{N+1}/\mathbb{G}_{m, \mathbbm{k}}]$ is the morphism induced by a closed immersion $\iota\colon X\hookrightarrow \mathbb{P}^N_{\mathbbm{k}}$.
\end{rem}

By \cite[Theorem A.16]{Hat} and Lemma \ref{lem--two--delta--invariants}, K-(semi)stability of the log--twisted pair $(\mathbb{P}^1,B+uB',uM)$ is equivalent to $\delta(q)>1$ (resp.~$\ge1$) when $u$ is sufficiently small.
Thus, we get the following result.

\begin{thm}\label{thm--delta--quasimap}
    Let $q\colon(\mathbb{P}^1,B)\to [\mathrm{Cone}(X)/\mathbb{G}_{m, \mathbbm{k}}]$ be a lc Fano quasimap of degree $m$ and of weight $u$ with respect to a closed embedding $X\subset\mathbb{P}^N$.
    Suppose that $0<u<1-\frac{v}{2}$, where $v:=\mathrm{deg}(-K_{\mathbb{P}^1}-B)-um$.
    Then, $q$ is K-stable (resp.~K-semistable) if and only if $\delta(q)>1$ (resp.~$\delta(q)\ge1$).
    In particular, if $q$ is K-semistable, then $q$ is a log Fano quasimap.
\end{thm}

\begin{proof}
Let $(\mathbb{P}^1,B+uB',uM)$ be the log--twisted Fano pair associated with $q$.
   As noted in Remark \ref{rem--K-stability--coincidence}, the log--twisted K-stability of $(\mathbb{P}^1,B+uB',uM)$ and the K-stability of $q$ are equivalent.
   By Lemma \ref{lem--two--delta--invariants}, if $0<u<1-\frac{v}{2}$, then we further see that the K-stability of $q$ is equivalent to the K-stability of the pair $(\mathbb{P}^1,B+uD)$ for any general $D\in|\mathscr{L}|_q$.
Thus, the assertion follows from \cite[Corollary 9.6]{BHJ} and \cite[Theorem B]{BlJ}.
\end{proof}

\begin{rem}\label{rem--large--u--1}
    We note that if $B=B'=0$, $m=1$ and $u=1$, then $q$ is always K-stable but $\delta(q)<1$. 
    This means that Theorem \ref{thm--delta--quasimap} does not hold in general without condition $u<1-\frac{v}{2}$. 
\end{rem}

As the K-moduli theory of log Fano pairs, we have to consider K-polystability of log Fano quasimaps.
To define the K-polystability, we recall the following notion.

\begin{defn}[Special test configuration]\label{defn--special-test-config}
    Let $q\colon\mathbb{P}^1 \to [\mathrm{Cone}(X)/\mathbb{G}_{m, \mathbbm{k}}]$, $B$, and $\mathscr{L}$ be as in Definition \ref{defn--lcFanoquasimap}. 
Let $(\mathcal{C},\mathcal{L})$ an ample normal test configuration for $q$.
    We set $\mathcal{D}=\overline{(uD+B)\times\mathbb{G}_{m, \mathbbm{k}}}$ for any  $D\in |\mathscr{L}|_q$. 
    Here, we consider $(uD+B)\times\mathbb{G}_{m, \mathbbm{k}}$ as a subset of $\mathbb{P}^1 \times (\mathbb{A}^1 \setminus \{0\})$ and $\overline{\cdot}$ means the Zariski closure in $\mathcal{C}$.
    We say that $(\mathcal{C},\mathcal{L})$ is a {\it special test configuration} for $q$ if $(\mathcal{C},\mathcal{C}_0+\mathcal{D})$ is plt and $\mathcal{L}\sim_{\mathbb{Q}}-(K_{\mathcal{C}}+\mathcal{D})$ for any general divisor $D\in |\mathscr{L}|_q$.
\end{defn}

We have the following analogous result to \cite[Corollary A.14]{Hat}.

\begin{thm}[{\cite[Corollary 1]{LX}, \cite[Corollary 6.11]{Fjt}, \cite[Theorem 3.12]{BLZ}}]\label{thm--lixu--analog}
    Let $q$ be a log Fano quasimap and $(\mathcal{C},\mathcal{L})$ a normal ample test configuration for $q$.
    Then, there exist $r\in\mathbb{Z}_{>0}$ and a special test configuration $(\mathcal{C}^{\mathrm{s}},\mathcal{L}^{\mathrm{s}})$ such that 
    \[
    \mathrm{DF}_q(\mathcal{C}^{\mathrm{s}},\mathcal{L}^{\mathrm{s}})-\delta J^{\mathrm{NA}}(\mathcal{C}^{\mathrm{s}},\mathcal{L}^{\mathrm{s}})\le r(\mathrm{DF}_q(\mathcal{C},\mathcal{L})-\delta J^{\mathrm{NA}}(\mathcal{C},\mathcal{L}))
    \]
    for any $\delta\in [0,1)$. 
    Furthermore, if $\delta=0$, then equality holds if and only if  $(\mathcal{C},\mathcal{L})$ is a special test configuration.
\end{thm}

\begin{proof}
    Let $(\mathbb{P}^1,B+uB',uM)$ be the log--twisted Fano pair associated with $q$.
     By \cite[Lemma 2.22]{Hat}, we have 
\[
\mathrm{DF}_{(B+uB',uM)}(\mathcal{C},\mathcal{L})=\mathrm{DF}_{q}(\mathcal{C},\mathcal{L})
\]
for any normal nontrivial ample test configuration $(\mathcal{C},\mathcal{L})$ for $q$.
Thus the assertion follows from a similar argument of the proof of \cite[Corollary A.14]{Hat}. 
\end{proof}

By this theorem, to check K-semistability, it is enough to check $\mathrm{DF}_q(\mathcal{C},\mathcal{L})$ for any special test configuration $(\mathcal{C},\mathcal{L})$ for $q$.

\begin{defn}[Canonical quasimap structure] 
In this definition, $\mathbb{P}^{1}_{\mathbbm{k}}$ is denoted by $C$ for the simplicity of representation.
Let $q\colon C\to [\mathrm{Cone}(X)/\mathbb{G}_{m, \mathbbm{k}}]$ be a quasimap with respect to a closed embedding $X \subset \mathbb{P}^{N}_{\mathbbm{k}}$, and let $\mathcal{C}$ be a special test configuration for $C$. 
Note that $\mathbb{G}_m$ acts on $\mathcal{C}$ and $\mathbb{A}^1$ equivariantly by the definition.
We define the {\it canonical quasimap structure}, which is a $\mathbb{G}_{m}$-equivariant family of quasimaps 
$$q_{\mathcal{C}}\colon \mathcal{C}\to[\mathrm{Cone}(X)\times \mathbb{A}^1/\mathbb{G}_{m}]$$
over $\mathbb{A}^1$, as follows:  
Let $q_{\mathbb{A}^1}\colon C_{\mathbb{A}^1}\to [\mathrm{Cone}(X)\times \mathbb{A}^1/\mathbb{G}_{m, \mathbb{A}^1}]$ be the quasimap induced by $q$ such that $q_{\mathbb{A}^1}$ is trivial on the second component $\mathbb{A}^1$.
Note that $q_{\mathbb{A}^1}$ is $\mathbb{G}_{m, \mathbb{A}^1}$-invariant. 
Take a test configuration $\mathcal{C}'$ for $C$ such that there exist $\mathbb{G}_{m}$-equivariant morphisms of test configurations $p\colon\mathcal{C}'\to C_{\mathbb{A}^1}$ and $r\colon\mathcal{C}'\to \mathcal{C}$ which are isomorphic over $\mathbb{A}^1\setminus\{0\}$.
Let $\mathscr{L}$ be the line bundle on $C_{\mathbb{A}^1}$ defined with $q_{\mathbb{A}^1}$, and let $\mathrm{Sec}_{q_{\mathbb{A}^1}}(\mathscr{L})$ be the module of sections (see Definition \ref{defn--linebundle-seq-quasimap}).
We put $\mathscr{L}':=r_*p^*\mathscr{L}$. 
By Definition \ref{defn--special-test-config}, we see that $\mathcal{C}_{0}$ is a normal curve. 
This implies that $\mathcal{C}$ is smooth.  
We also see from Definition \ref{defn--linebundle-seq-quasimap} that $\mathrm{Sec}_{q_{\mathbb{A}^1}}(\mathscr{L}) \neq0$, which implies $H^{0}(C_{\mathbb{A}^1}, \mathscr{L}) \neq 0$. 
Thus, $\mathscr{L}$ is nef over $\mathbb{A}^1$, and therefore $\mathscr{L}'$ is a Cartier divisor on $\mathcal{C}$ and nef over $\mathbb{A}^1$.  
By the negativity lemma, there exists an $r$-exceptional effective divisor $E$ such that $r^*\mathscr{L}'=p^*\mathscr{L}+E$ as Cartier divisors. By this relation, we get a natural inclusion 
$$\iota\colon H^0(C_{\mathbb{A}^1},\mathscr{L})\hookrightarrow H^0(\mathcal{C},\mathscr{L}').$$
We note that $\iota$ is $\mathbb{G}_{m}$-equivariant. 
Hence the subspace $\iota(\mathrm{Sec}_{q_{\mathbb{A}^1}}(\mathscr{L}))\subset  H^0(\mathcal{C},\mathscr{L}')$ is stable under this action.
We put 
$$\mu :=\underset{0\ne f\in\iota(\mathrm{Sec}_{q_{\mathbb{A}^1}}(\mathscr{L}))}{\min}\mathrm{ord}_{\mathcal{C}_0}(f).$$
We have an isomorphism $\mathscr{L}'\cong\mathscr{L}'(-\mu\mathcal{C}_0)$ as line bundles, and this induces an isomorphism 
$$\nu \colon H^0(\mathscr{L}')\to H^0(\mathscr{L}'(-\mu\mathcal{C}_0)).$$
We put the new $\mathbb{G}_m$-action on $\mathscr{L}'(-\mu\mathcal{C}_0)$ via the natural identification $\mathscr{L}'\cong\mathscr{L}'(-\mu\mathcal{C}_0)$.
Then, $\nu$ is $\mathbb{G}_m$-equivariant with respect to this $\mathbb{G}_m$-action.
For the canonical coordinate $e_i$ of $\mathbb{A}^{N+1}$, let $f_i$ be the corresponding section of $\mathrm{Sec}_{q_{\mathbb{A}^1}}(\mathscr{L})$.
Take the corresponding section $f'_i\in \nu(\iota(\mathrm{Sec}_{q_{\mathbb{A}^1}}(\mathscr{L})))$.
Here, note that $f'_i$ is a $\mathbb{G}_m$-invariant section with respect to the action $\mathbb{G}_m$ on $\mathscr{L}'(-\mu\mathcal{C}_0)$ for any $i$.
By the construction of $\nu$ and $\iota$, some $f'_i$ does not vanish along $\mathcal{C}_0$.
Therefore, $\{f'_i\}_{i=0}^N$ define a quasimap
$$q_{\mathcal{C}}\colon \mathcal{C}\to[\mathrm{Cone}(X)/\mathbb{G}_{m, \mathbbm{k}}].$$ 
This is the $\mathbb{G}_m$-equivariant family of quasimaps over $\mathbb{A}^1$ that we wanted to constructed. 

The quasimap $q_{\mathcal{C}}$ is often called the {\it canonical quasimap defined on the special test configuration $\mathcal{C}$}.
\end{defn}

When $q$ is a log Fano quasimap for some effective divisor on $\mathbb{P}^{1}_{\mathbbm{k}}$, there exists only one $\mathbb{G}_{m, \mathbb{A}^1}$-equivariant family of quasimaps on $\mathcal{C}$ over $\mathbb{A}^1$ whose restriction to $\mathcal{C}\setminus\mathcal{C}_0$ coincides with $q_{\mathbb{A}^1}|_{C\times(\mathbb{A}^1\setminus\{0\})}$. In other words, the canonical quasimap structure of $\mathcal{C}$ is uniquely determined. 
See Lemma \ref{lem--canonical-qm--str} below.
This is why the canonical quasimap structure is not an ad hoc notion.

\begin{lem} \label{lem--canonical-qm--str}
    Let $q\colon (\mathbb{P}^1,B)\to [\mathrm{Cone}(X)/\mathbb{G}_{m, \mathbbm{k}}]$ be a log Fano quasimap of degree $m$ and of weight $u$ with respect to a closed embedding $X\subset \mathbb{P}^N$, and let $\mathcal{C}$ be a special test configuration with the canonical quasimap structure $q_{\mathcal{C}}\colon \mathcal{C}\to [\mathrm{Cone}(X)\times\mathbb{A}^1/\mathbb{G}_{m,\mathbb{A}^1}]$.

    Let $q'\colon \mathcal{C}\to [\mathrm{Cone}(X)\times\mathbb{A}^1/\mathbb{G}_{m,\mathbb{A}^1}]$ be another $\mathbb{G}_m$-equivariant family of quasimaps with respect the canonical $\mathbb{G}_m$-action on $\mathcal{C}$ such that $q'_{\mathbb{A}^1\setminus\{0\}}=q_{\mathbb{A}^1\setminus\{0\}}$. Then, $q'=q_{\mathcal{C}}$.
\end{lem}

\begin{proof}
Let $\mathscr{L}$ be the line bundle associated with $q_{\mathcal{C}}$.
Since $\mathcal{C}\cong\mathbb{P}^1\times \mathbb{A}^1$ as a scheme, we see that $\mathscr{L}$ is isomorphic to the line bundle associated with $q'$.
Take a $\mathbb{G}_m$-invariant general divisor $\mathcal{D}\in|\mathscr{L}|_{q'}$.
For $\mathrm{Cone}(X)\subset \mathbb{A}^{N+1}$, let $f'\in\mathrm{Sec}_{q'}(\mathscr{L})$ be the section corresponding to $\mathcal{D}$.
Since we choose $\mathcal{D}$ general, we may assume that $f'_0\ne0$. 
This implies that $\mathcal{D}$ coincides with the Zariski closure $\overline{\mathcal{D}_1\times\mathbb{G}_m}$ of $\mathcal{D}_1\times\mathbb{G}_m$ in $\mathcal{C}$.
Suppose that $f'$ corresponds to an element $e\in\mathbbm{k}^{N+1}$ and $e$ corresponds to a section $f\in\mathrm{Sec}_{q_{\mathcal{C}}}(\mathscr{L})$.
Since we choose $\mathcal{D}$ general, we may assume that $f$ is also general such that $f_0\ne0$.
This means that $\mathrm{div}(f)=\mathcal{D}$.
Thus, if $q_{\mathcal{C}}|_{\mathcal{C}_0}$ is non-constant, then we have that $q_{\mathcal{C}}|_{\mathcal{C}_0}$ coincides with $q'|_{\mathcal{C}_0}$ by Lemma \ref{lem}.
Furthermore, we can see that $q_{\mathcal{C}}=q'$ in this case.

Thus, we may assume that $q_{\mathcal{C}}|_{\mathcal{C}_0}$ is constant.
Let $B'$ be the fixed part of $q$ and $\mathcal{B}'$ the closure of $B'\times\mathbb{G}_m$.
Then, we see that $\mathcal{B}'$ is contained in the fixed part of $q_{\mathcal{C}}$ and $q'$.
Let $s_{\mathcal{B}'}\in\mathcal{O}_{\mathcal{C}}(\mathcal{B}')$ be the section corresponding to $\mathcal{B}'$.
If the functions $\{f_i\}_{i=0}^N\in\mathrm{Sec}_{q_{\mathcal{C}}}(\mathscr{L})$ (resp.~$\{f'_i\}_{i=0}^N\in\mathrm{Sec}_{q'}(\mathscr{L})$) corresponding to the canonical basis define $q_{\mathcal{C}}$ (resp.~$q'$), we set the following new family of quasimaps $\widetilde{q_{\mathcal{C}}}$ (resp.~$\widetilde{q'}$) defined by the sections $\{f_i\cdot s_{\mathcal{B}'}^{-1}\}_{i=0}^N\in\mathrm{Sec}_{q_{\mathcal{C}}}(\mathscr{L}(-\mathcal{B}'))$ (resp.~$\{f'_i\cdot s_{\mathcal{B}'}^{-1}\}_{i=0}^N\in\mathrm{Sec}_{q'}(\mathscr{L}(-\mathcal{B}'))$).
Then, it suffices to compare $\widetilde{q_{\mathcal{C}}}$ with $\widetilde{q'}$ and therefore we may assume that $B'=0$.
It is well-known that the canonical compactification $\overline{\mathcal{C}}$ of $\mathcal{C}$ over $\mathbb{P}^1$ is a Hirzebruch surface $\mathbb{P}_{\mathbb{P}^1}(\mathcal{O}\oplus\mathcal{O}(-r))$ (cf.~\cite[Example 2.8]{BHJ}).
Then, we assert that there exists a $\mathbb{G}_m$-invariant section $\Delta\subset \overline{\mathcal{C}}$ such $\Delta$ is disjoint from any general divisor $\mathcal{D}\in|\mathscr{L}|_{q'}$.
Indeed, it is trivial when $\mathcal{C}$ is trivial and hence we may assume that $\mathcal{C}$ is a non-trivial special test configuration.
Then, $\mathcal{C}_0$ admits the non-trivial $\mathbb{G}_m$-action and $q'|_{\mathcal{C}_0}$ and $q_{\mathcal{C}}|_{\mathcal{C}_0}$ are constant. 
If we set $\Delta$ as the unique section such that $\Delta^2<0$, then any general $\mathcal{D}$ does not intersect with $\Delta$. 
This means that the fixed part of $q_{\mathcal{C}}$ or $q'$ does not contain any point of $\Delta$ either.
Therefore, around $\Delta$, $q_{\mathcal{C}}$ and $q'$ are $\mathbb{G}_m$-equivariant morphisms from $\mathcal{C}$ to $\mathbb{P}^N$.
Now, it is easy to see that the images of $q_{\mathcal{C}}|_{\mathcal{C}_0}$ and $q'|_{\mathcal{C}_0}$ around $\Delta$ are the point $q_{\mathbb{A}^1}(\Delta\setminus\Delta_0)$.
Thus, the assertion holds.
\end{proof}

We note that if $\mathcal{C}$ is not trivial but special, then $q_{\mathcal{C}}|_{\mathcal{C}_0}$ admits the induced $\mathbb{G}_m$-action.
It is easy to see that $|\mathscr{L}_0|_{q_{\mathcal{C}}|_{\mathcal{C}_0}}$ consists of only one divisor $D$.
Thus, we see that K-stability of $q_{\mathcal{C}}|_{\mathcal{C}_0}$ coincides with the K-stability of the log Fano pair $(\mathbb{P}^1,uD+B)$.

We note the following important fact, which can be shown in the same way as the log Fano case.

\begin{lem}[{cf.~\cite[Lemma 3.1]{LWX21}}]\label{lem--DF=0-case}
Let $q$ be a K-semistable log Fano quasimap of weight $u$ and degree $m$ and $(\mathcal{C},\mathcal{L})$ a normal ample test configuration such that $\mathrm{DF}_q(\mathcal{C},\mathcal{L})=0$.
Suppose that $0<u<1-\frac{v}{2}$, where $v:=-\mathrm{deg}(K_{\mathbb{P}^1}+B)-um$.
Then, $(\mathcal{C},\mathcal{L})$ is special and $q_{\mathcal{C}}|_{\mathcal{C}_0}$ is K-semistable.
\end{lem}

\begin{proof}
First, we can easily see that $(\mathcal{C},\mathcal{L})$ is special by Theorem \ref{thm--lixu--analog}.
Next, we treat the assertion that $q_{\mathcal{C}}|_{\mathcal{C}_0}$ is K-semistable.
    If $\mathcal{C}$ is trivial, then we can easily see that the assertion holds.  
    Otherwise, note that then $q_{\mathcal{C}}|_{\mathcal{C}_0}$ admits a non-trivial $\mathbb{G}_m$-action and hence is constant.
    Let $D$ be the fixed part of $q_{\mathcal{C}}|_{\mathcal{C}_0}$ and $\mathcal{B}$ be the closure of $B\times\mathbb{G}_m$ in $\mathcal{C}$.
    Since $q_{\mathcal{C}}|_{\mathcal{C}_0}$ is constant, if we take a general member $D'\in |\mathscr{L}|_q$, where $\mathscr{L}$ is the associated line bundle to $q$, then $D=\mathcal{D}_0$ and $(\mathbb{P}^1,B+uD')$ is K-semistable, where $\mathcal{D}$ is the closure of $D'\times \mathbb{G}_m$ in $\mathcal{C}$ by Lemma \ref{lem--two--delta--invariants}.
   Then, the associated log--twisted pair with $q_{\mathcal{C}}|_{\mathcal{C}_0}$ is the log pair $(\mathcal{C}_0,\mathcal{B}_0+uD,0)$.
    Regarding $(\mathcal{C},\mathcal{L})$ as a test configuration for $(\mathbb{P}^1,B+uD')$, we see that 
    \begin{align*}
        \mathrm{DF}_{B+uD'}(\mathcal{C},\mathcal{L})=\mathrm{DF}_{q}(\mathcal{C},\mathcal{L})=0.
    \end{align*}
   By applying \cite[Lemma 3.1]{LWX21} to $(\mathcal{C},\mathcal{L})$ and a K-semistable log Fano pair $(\mathbb{P}^1,B+uD')$,
    we obtain that $(\mathcal{C}_0,\mathcal{B}_0+uD)$ is K-semistable.
    This is equivalent to K-semistability of $q_{\mathcal{C}}|_{\mathcal{C}_0}$ by Lemma \ref{lem--two--delta--invariants} and hence we complete the proof.
\end{proof}

\begin{defn}[Family of quasimap of product type on test configuration]
 Let $(\mathcal{C},\mathcal{L})$ be a normal ample test configuration for $q$.  Suppose that $(\mathcal{C},\mathcal{L})$ is special.
 Let $q_{\mathcal{C}}$ be the canonical quasimap structure of $\mathcal{C}$. Then, we say that $q_{\mathcal{C}}$ is {\it of product type} if $q_{\mathcal{C}}|_{\mathcal{C}_0}=q$.
 It is easy to see that then $q_{\mathcal{C}}=q_{\mathbb{A}^1}$ as an abstract family of quasimaps. 
 Indeed, let $\{f_0,\ldots,f_N\}\subset\mathrm{Sec}_q(\mathscr{L})$ be the set of sections corresponding to the basis of $\mathbbm{k}^{N+1}$.
 First, note that $(\mathcal{C},\mathcal{B})\cong(\mathbb{P}^1\times\mathbb{A}^1,B\times\mathbb{A}^1)$ as abstract pairs.
 Since $q_{\mathcal{C}}$ admits the natural $\mathbb{G}_m$-action, we see that $f_0,\ldots,f_N$ are all $\mathbb{G}_m$-eigenvectors.
 Let $\lambda_i$ be the weight of $f_i$.
 Since $q_{\mathcal{C}}$ is of product type, we can see that this is isomorphic to $q_{\mathbb{A}_1}$ such that $\mathbb{G}_m$ acts on the quasimap in the way that $(f_0)_1,\ldots,(f_N)_1$ admits the weight $\lambda_0,\ldots,\lambda_N$ respectively.
 Note that this isomorphism does not preserve the $\mathbb{G}_m$-actions in general.

 We say that $q$ is {\it K-polystable} if $q$ is K-semistable and $\mathrm{DF}_q(\mathcal{C},\mathcal{L})=0$ if and only if $q_\mathcal{C}$ is of product type.
\end{defn}

\begin{defn}\label{defn--associated--quasi--map--structure}
    Let $f\colon (W,A)\to\mathbb{P}_{\mathbbm{k}}^{1}$ be a uniformly adiabatically K-stable klt-trivial fibration over $\mathbb{P}_{\mathbbm{k}}^{1}$ of dimension $d$ with general fibers belonging to $\mathcal{M}_{d-1,v}^{\mathrm{klt},\mathrm{CY}}$.
    Let $U$ be a Zariski open subset of $\mathbb{P}_{\mathbbm{k}}^{1}$ such that $f|_{f^{-1}(U)}$ is smooth.
    Let $q^\circ\colon U\to M_{d-1,v}^{\mathrm{klt},\mathrm{CY}}$ be the moduli map corresponding to $f|_{f^{-1}(U)}$. 
    Let $B$ be the discriminant $\mathbb{Q}$-divisor and $M$ the moduli $\mathbb{Q}$-divisor associated to $f$.

    With notation as above, suppose that all the fibers of $f|_{f^{-1}U}$ are Abelian varieties or symplectic varieties smoothable to projective irreducible holomorphic symplectic varieties.
    Then, the image of $q^\circ$ is contained in a normal irreducible component $X^\circ$ of $\mathcal{M}_{d-1,v}^{\mathrm{klt},\mathrm{CY}}$ (Proposition \ref{prop--smoothness--of--Abelian--moduli} and Corollary \ref{cor--hk--moduli--normal}).
    Let $X$ be the Baily--Borel compactification of $X^\circ$ and $\Lambda_{\mathrm{Hodge}}$ the Hodge $\mathbb{Q}$-line bundle on $X$.
Take a sufficiently large and divisible positive integer $l$ such that $\Lambda_{\mathrm{Hodge}}^{\otimes l}$ is a very ample Cartier divisor and $lB$ is an integral divisor.    
Let $\iota\colon X\hookrightarrow \mathbb{P}_{\mathbbm{k}}^N$ be the closed embedding defined by $\Lambda_{\mathrm{Hodge}}^{\otimes l}$.
Then, we can extend $q^\circ$ to a quasimap 
$$q\colon \mathbb{P}_{\mathbbm{k}}^{1} \to [\mathrm{Cone}(X)/\mathbb{G}_{m, \mathbbm{k}}]$$
 of degree $l\cdot\mathrm{deg}(M+B)$ and of weight $\frac{1}{l}$ in the following way: 
Let $\mathrm{Cone}(X)\subset\mathbb{A}_{\mathbbm{k}}^{N+1}$ be the affine cone of $X$ as a subvariety of $\mathbb{P}_{\mathbbm{k}}^N$.
Let $\bar{q}\colon \mathbb{P}_{\mathbbm{k}}^1\to X$ be the unique extension of $q^{\circ}$ by the valuative criterion for properness.
By definition, we have $M\sim_{\mathbb{Q}}\bar{q}^*\Lambda_{\mathrm{Hodge}}$. 
Now take the canonical coordinate sections $x_0,\ldots,x_N$ of $\mathbb{P}^N_{\mathbbm{k}}$ and we set $q$ as the quasimap defined by the sections $\bar{q}^*s_i\cdot s_{lB}\in H^0(\mathbb{P}_{\mathbbm{k}}^{1},\mathcal{O}_{\mathbb{P}_{\mathbbm{k}}^{1}}(l(M+B)))$, where $s_{lB}\in H^0(\mathbb{P}_{\mathbbm{k}}^{1},\mathcal{O}_{\mathbb{P}_{\mathbbm{k}}^{1}}(lB))$ is the section corresponding to $lB$. 

We call the $q\colon \mathbb{P}_{\mathbbm{k}}^{1} \to [\mathrm{Cone}(X)/\mathbb{G}_{m, \mathbbm{k}}]$ a {\it quasimap associated with} $f$.
We note that $q$ depends on the integer $l$, and if we fix $l$ then such $q$ is uniquely determined.
\end{defn}

\subsection{Construction of the moduli stack of K-semistable log Fano quasimaps}\label{subsec3.1}

In this subsection, we first deal with the boundedness of K-semistable log Fano quasimaps.
Let $X$ be a projective normal variety with a closed embedding $\iota\colon X\subset \mathbb{P}^N$, $v,u\in\mathbb{Q}_{>0}$ and $r,m\in\mathbb{Z}_{>0}$.
We set the following groupoid $\mathcal{M}^{\mathrm{Kss,qmaps}}_{m,r,u,v,\iota}(S)$ for any scheme $S$;
$$\left\{
 \vcenter{
 \xymatrix@C=12pt{
(\mathcal{C},\mathcal{B})\ar[rr]^-{q}\ar[dr]_{\pi}&& [\mathrm{Cone}(X)\times S/\mathbb{G}_{m,S}] \ar[dl]\\
&S
}
}
\;\middle|
\begin{array}{rl}
(i)&\text{$q$ is a family of log Fano quasi-}\\
&\text{maps of degree $m$ and of weight $u$,}\\
(ii)&\text{$q_{\bar{s}}$ is K-semistable}\\
&\text{for any geometric point $\bar{s}\in S$,}\\
(iii)&\text{$r\mathcal{B}_{\bar{s}}$ is a $\mathbb{Z}$-divisor}\\
&\text{for any geometric point $\bar{s}\in S$}\\
(iv)&\text{$\mathrm{deg}_{\mathbb{P}^1}(-K_{\mathbb{P}^1}-\mathcal{B}_{\bar{s}})-um=v$}\\
&\text{for any geometric point $\bar{s}\in S$}
\end{array}\right\},$$
where we set its arrows to be isomorphisms of families of quasimaps preserving the divisors.
It is easy to see that the quasi-functor $S\mapsto \mathcal{M}^{\mathrm{Kss,qmaps}}_{m,r,u,v,\iota}(S)$ forms a stack $\mathcal{M}^{\mathrm{Kss,qmaps}}_{m,r,u,v,\iota}$ as $\mathsf{Qmaps}_{X\subset \mathbb{P}^N}$.

Fix a sufficiently divisible and large positive integer $l \in\mathbb{Z}_{>0}$, depending only on $u$, $v$, $r$ and $m$, such that for any object $q\colon(\mathbb{P}^1,B)\to [\mathrm{Cone}(X)/\mathbb{G}_{m, \mathbbm{k}}]$ of $\mathcal{M}^{\mathrm{Kss,qmaps}}_{m,r,u,v,\iota}(\mathbbm{k})$,  the divisor $l(-K_{\mathbb{P}^1}-B)$ is very ample and its degree is even.
Let $P$ be the polynomial defined by
$$P(t):=\chi(\mathbb{P}^1,\mathcal{O}_{\mathbb{P}^1}(tl(-K_{\mathbb{P}^1}-B)))=h^0(\mathbb{P}^1,\mathcal{O}_{\mathbb{P}^1}(tl(-K_{\mathbb{P}^1}-B)))$$
for any $t \in\mathbb{Z}_{>0}$. 
Then $P$ depends only on $l$, $u$, $r$, $v$, and $m$, and therefore $P$ depends only on $u$, $r$, $v$, and $m$.
We note that $l(-K_{\mathbb{P}^1}-B)\sim-\frac{lv+lum}{2}K_{\mathbb{P}^1}$ and $\frac{lv+lum}{2}\in\mathbb{Z}$.

Let $\mathcal{U}$ be the universal family of $\mathbf{Hilb}^{P,\mathcal{O}(1)}_{\mathbb{P}^{P(1)-1}}$, and let $f \colon \mathcal{U} \to \mathbf{Hilb}^{P,\mathcal{O}(1)}_{\mathbb{P}^{P(1)-1}}$ be the structure morphism. 
We regard $rB$ as a $\mathbb{Z}$-divisor of degree $k:=r(2-v-um)$. 
Since we deal with flat projective families of curves, we can consider flat families of effective Cartier divisors of degree $k$ on them.  
We set 
$$\mathcal{W}:=\mathbf{Div}_{\mathcal{U}/\mathbf{Hilb}^{P,\mathcal{O}(1)}_{\mathbb{P}^{P(1)-1}}}^k.$$
By \cite[Corollary 7.53 and Theorem 7.40]{kollar-moduli} (see also \cite[Definition 7.37]{kollar-moduli}), it follows that $\mathcal{W}$ is the fine moduli space of all K-flat divisors of degree $k$. 
Let 
$$\mathcal{B} \subset \mathcal{U}_{\mathcal{W}}$$
be the universal family of $\mathcal{W}$. 

By the standard argument and \cite[Proposition 9.42]{kollar-moduli}, we can take $Z_1\subset \mathcal{W}$ as a locally closed subscheme such that a morphism $S \to \mathcal{W}$ factors through $Z_{1} \hookrightarrow \mathcal{W}$ if and only if the following conditions hold:
\begin{itemize}
\item 
$l(-K_{\mathcal{U}_{S}/S}-\frac{1}{r}\mathcal{B}_{S})\sim_{S}\mathcal{O}_{\mathcal{U}_{S}}(1)|_{\mathcal{U}_{S}}$, and
\end{itemize}
for any geometric point $s \in S$ with the geometric point $\bar{s}$, 
\begin{itemize}
\item
$(\mathcal{U}_{\bar{s}},\mathcal{B}_{\bar{s}})$ is a klt pair with the Hilbert polynomial $P$, and
\item
$f_*(\mathcal{O}_{\mathcal{U}}(1)|_{\mathcal{U}_{S}})\otimes \kappa(s)\to H^0(\mathcal{U}_{s},\mathcal{O}_{\mathcal{U}}(1)|_{\mathcal{U}_{s}})$ is an isomorphism. 
\end{itemize}
 By these conditions, any geometric fiber of $\mathcal{U}_{Z_1}\to Z_1$ coincides with $\mathbb{P}^1$. 
 
By an argument of the Quot functor, we get an open subscheme $Z_2\subset \mathbf{Quot}_{\mathcal{U}_{Z_1}/Z_1,\mathcal{O}_{\mathcal{U}_{Z_1}}^{\oplus m+1}}$ such that putting $\mathscr{L}$ as the universal coherent sheaf on $\mathbf{Quot}_{\mathcal{U}_{Z_1}/Z_1,\mathcal{O}_{\mathcal{U}_{Z_1}}^{\oplus m+1}}\times_{Z_1} \mathcal{U}_{Z_1}$, then for any geometric point $\bar{s}\in \mathbf{Quot}_{\mathcal{U}_{Z_1}/Z_1,\mathcal{O}_{\mathcal{U}_{Z_1}}^{\oplus m+1}}$, the restriction $\mathscr{L}_{\bar{s}}$ is isomorphic to the line bundle $\mathcal{O}(m)$ on $\mathbb{P}^1_{\bar{s}}$ if and only if $\bar{s}\in Z_2$. 
Then, the Hilbert polynomial of $\mathscr{L}_{\bar{s}}$ is determined.
Thus, by construction, $Z_2$ is of finite type over $Z_1$ (see Subsection \ref{Subsection--Hilb}).

Let $\xi\colon \mathcal{O}_{\mathcal{U}_{Z_2}}^{\oplus m+1}\to\mathscr{L}_{Z_{2}}$ be a surjective morphism of coherent sheaves on $\mathcal{U}_{Z_{2}}$ that corresponds to the natural morphism $Z_2\to \mathbf{Quot}_{\mathcal{U}_{Z_1}/Z_1,\mathcal{O}_{\mathcal{U}_{Z_1}}^{\oplus m+1}}$. 
We note that $\mathrm{Ker}\,\xi$ is uniquely determined for such $\xi$.
Next, we consider the scheme 
$$H:=\mathbf{Hom}_{Z_2}(\mathcal{O}^{N+1}_{\mathcal{U}_{Z_2}},\mathscr{L}_{Z_{2}})\cong \mathbb{A}_{\mathcal{U}_{Z_2}}((f_{Z_2*}\mathscr{L}_{Z_2})^\vee)^{N+1}.$$
Let $g\in \mathfrak{Hom}(\mathcal{O}^{N+1}_{\mathcal{U}_H},\mathscr{L}_{H})(H)$ be the universal element. 
Then, $g$ induces the following $H$-morphism
\[
g'\colon\mathbb{A}_{\mathcal{U}_H}(\mathscr{L}_{H})\to \mathbb{A}_{\mathcal{U}_H}^{N+1}. 
\]
Let $\mu\colon H\to Z_2$ be the canonical morphism. 
We note that each geometric point $\bar{s}\in H$ corresponds to an $(N+1)$-tuple $(f_0,f_1,\ldots,f_N)\in H^0(\mathcal{O}_{\mathcal{U}_{\bar{s}}},\mathscr{L}_{\bar{s}}^{N+1})$.
There exists an open subset $Z_3\subset H$ such that a geometric point $\bar{s}\in H$ is contained in $Z_3$ if and only if for any point $p\in\mathcal{U}_{\overline{s}}$, there exists $f'\in \mathrm{Span}(f_0,f_1,\ldots,f_N)\setminus\{0\}$ such that $\mathrm{mult}_p(\frac{1}{r}\mathcal{B}_{\bar{s}}+u\cdot \mathrm{div}(f'))\le 1-\frac{v}{2}$ (\eqref{eq--delta--quasi--map} in Definition \ref{defn--lcFanoquasimap}). 
.

Finally, there exists a closed subscheme $Z_4\subset Z_3$ such that a morphism $h\colon T\to Z_3$ from a scheme $T$ factors through $Z_4$ if and only if $g'_{T}$ factors through $\mathrm{Cone}(X)\times \mathcal{U}_T$, where $g'_{T}$ is the base change of $g'\colon\mathbb{A}_{\mathcal{U}_H}(\mathscr{L}_{H})\to \mathbb{A}_{\mathcal{U}_H}^{N+1}$, which was defined in the previous paragraph. 
Indeed, the ideal corresponding to $\mathrm{Cone}(X)\subset \mathbb{A}^{N+1}$ is generated by finitely many homogeneous polynomials, and thus we can apply Lemma \ref{lem--closed--cone} to obtain $Z_4$. 

Note that $\mathbf{Hilb}^{P,\mathcal{O}(1)}_{\mathbb{P}^{P(1)-1}}$ admits the natural action of $PGL(P(1))$ and $\mathbf{Quot}_{\mathcal{U}_{Z_1}/Z_1,\mathcal{O}_{\mathcal{U}_{Z_1}}^{\oplus m+1}}$ admits the natural action of $PGL(m+1)$.
By this fact and construction, it is easy to see that $PGL(P(1))\times PGL(m+1)$ acts on $Z_4$ and $\mathcal{U}_{Z_4}$ in a natural way. 
Now we prove the following theorem.

\begin{thm}\label{thm--quot--stack}
    Fix an embedding of a projective normal variety $\iota\colon X\subset \mathbb{P}^N$, $v,u\in\mathbb{Q}_{>0}$ and $r,m\in\mathbb{Z}_{>0}$.
We also assume that $0<u<1-\frac{v}{2}$. Then, we have
$$\mathcal{M}^{\mathrm{Kss,qmaps}}_{m,r,u,v,\iota}\cong [Z_4/PGL(P(1))\times PGL(m+1)].$$
    In particular, $\mathcal{M}^{\mathrm{Kss,qmaps}}_{m,r,u,v,\iota}$ is an Artin stack of finite type over $\mathbbm{k}$.
\end{thm}

\begin{proof}
We use the notations in the above construction. 
We consider the composition of $g'_{Z_4}\colon\mathbb{A}_{\mathcal{U}_{Z_4}}(\mathscr{L}_{Z_4})\to \mathbb{A}_{\mathcal{U}_{Z_4}}^{N+1}$ and the natural morphism $\mathbb{A}_{\mathcal{U}_{Z_4}}^{N+1}\to \mathbb{A}_{Z_4}^{N+1}$. 
Here, $g'_{Z_4}$ is the base change of $g'\colon\mathbb{A}_{\mathcal{U}_H}(\mathscr{L}_{H})\to \mathbb{A}_{\mathcal{U}_H}^{N+1}$ in the above construction by $Z_{4} \hookrightarrow H$. 
This morphism factors through the inclusion $\mathrm{Cone}(X)\times Z_4\hookrightarrow \mathbb{A}_{Z_4}^{N+1}$.
Thus, we obtain the following family of quasimaps over $Z_4$:
\[
q_{Z_4}\colon \mathcal{U}_{Z_4}\to [\mathrm{Cone}(X)\times Z_4/\mathbb{G}_{m,Z_{4}}].
\]
Consider $\mathcal{B}_{Z_4}$.
By Lemma \ref{lem--two--delta--invariants} and the assumption $0<u<1-\frac{v}{2}$, we can regard 
\[
q_{Z_4}\colon (\mathcal{U}_{Z_4},\mathcal{B}_{Z_4})\to [\mathrm{Cone}(X)\times Z_4/\mathbb{G}_{m,Z_{4}}]
\]
as a family of K-semistable log Fano quasimaps of degree $m$ and of weight $u$.
It is easy to see that $q_{Z_4}$ is $PGL(P(1))\times PGL(m+1)$-invariant.
Furthermore, it is also easy to check that $PGL(m+1)$ acts on $Z_4$ freely since the automorphism group of the line bundle $\mathcal{O}_{\mathbb{P}^1}(m)$ is isomorphic to $\mathbb{G}_{m}$. Hence, $Z_4/PGL(m+1)$ is defined as an algebraic space such that $Z_4\to Z_4/PGL(m+1)$ is a principal bundle.
Let \[
\tilde{q}\colon(\mathcal{U}_{Z_4}/PGL(m+1),\mathcal{B}_{Z_4}/PGL(m+1))\to [\mathrm{Cone}(X)\times Z_4/\mathbb{G}_m]\]
be the induced family of log Fano quasimaps over $Z_4/PGL(m+1)$.
The above family of log Fano quasimaps is $PGL(P(1))$-equivariant over $Z_4/PGL(m+1)$ and induces the morphism 
$$\varphi\colon [Z_4/PGL(P(1))\times PGL(m+1)]\to \mathcal{M}^{\mathrm{Kss,qmaps}}_{m,r,u,v,\iota}.$$ 

To show that $\varphi$ induces a categorical equivalence, we will construct an inverse morphism of $\psi$ as follows. 
We will construct a morphism
$$\psi \colon \mathcal{M}^{\mathrm{Kss,qmaps}}_{m,r,u,v,\iota} \to  [Z_4/PGL(P(1))\times PGL(m+1)]$$
such that $\varphi \circ \psi$ is the identity. 
Fix a scheme $S$. 
Let $q_{\mathcal{C}}\colon (\mathcal{C},\mathcal{D})\to [\mathrm{Cone}(X)\times S/\mathbb{G}_{m,S}]$ be a family of K-semistable log Fano quasimaps of degree $m$ and of weight $u$, and let $\mathscr{L}_{\mathcal{C}}$ be the associated line bundle on $\mathcal{C}$.
As in \cite[Proof of Theorem 4.2]{HH}, there is an \'{e}tale cover $\nu\colon S'\to S$ such that $\mathcal{C}_{S'}\cong \mathbb{P}^1\times S'$, $\pi_{\mathcal{C}_{S'},*}(\mathcal{O}_{\mathcal{C}_{S'}}(l(-K_{\mathcal{C}_{S'}/S'}-\mathcal{D}_{S'})))\cong\mathcal{O}_{S'}^{\oplus P(1)}$, where $\pi_{\mathcal{C}_{S'}}\colon \mathcal{C}_{S'}\to S'$ is the canonical morphism and $l \in \mathbb{Z}_{>0}$ is as in the construction of $Z_{4}$, and $\mathscr{L}_{\mathcal{C},S'}\cong \mathcal{O}_{\mathbb{P}^1_{S'}}(m)$. 
Then we get a closed embedding $(\mathbb{P}^1\times S',\mathcal{D}_{S'})\hookrightarrow\mathbb{P}^{P(1)-1}_{S'}$ defined by the complete linear system $|\mathcal{O}_{\mathcal{C}_{S'}}(l(-K_{\mathcal{C}_{S'}/S'}-\mathcal{D}_{S'}))|$, a surjective homomorphism $\xi'\colon\mathcal{O}_{\mathcal{C}_{S'}}^{\oplus m+1}\to \mathscr{L}_{\mathcal{C},S'}$, and the homomorphism $g_{\mathcal{C}_{S'}}\colon \mathcal{O}_{\mathcal{C}_{S'}}^{\oplus N+1}\to \mathscr{L}_{\mathcal{C},S'}$ induced by $q_{\mathcal{C},S'}$.
Using the above data, we obtain a morphism $h'\colon S'\to Z_4$ such that $q_{\mathcal{C},S'}$ coincides with the base change of $q_{Z_4}$ by $h'$.
Consider the two projections $p_1,p_2\colon S'\times_SS'\to S'$ and the natural morphism $\mu\colon Z_4\to [Z_4/PGL(P(1))\times PGL(m+1)]$.
Then, it is easy to see that $\mu\circ h'\circ p_1\cong\mu\circ h'\circ p_2$ as a 1-morphism.
Thus, there exists a unique morphism $h\colon S\to [Z_4/PGL(P(1))\times PGL(m+1)]$ such that $h' \circ \mu =h\circ \nu$. 
By considering the element of $[Z_4/PGL(P(1))\times PGL(m+1)](S)$ corresponding to $h$, we can define a morphism $\psi \colon \mathcal{M}^{\mathrm{Kss,qmaps}}_{m,r,u,v,\iota} \to [Z_4/PGL(P(1))\times PGL(m+1)]$, and it is easy to check that $\varphi \circ \psi$ and $\psi\circ\varphi$ are isomorphic to the identity morphisms. 
Thus, we obtain the proof.
\end{proof}

\begin{rem}\label{rem--moduli--embed}
    We can see in the above proof that $[Z_3/PGL(P(1))\times PGL(m+1)]\cong \mathcal{M}^{\mathrm{Kss,qmaps}}_{m,r,u,v,\mathrm{id}_{\mathbb{P}_{\mathbbm{k}}^N}}$.
    Let $[\mathrm{Cone}(\iota)]\colon [\mathrm{Cone}(X)/\mathbb{G}_m]\hookrightarrow[\mathbb{A}^{N+1}_{\mathbbm{k}}/\mathbb{G}_m]$ be the canonical closed immersion induced by $\iota\colon X\hookrightarrow\mathbb{P}^N_{\mathbbm{k}}$.
    Then, the correspondence of a family of K-semistable log Fano quasimaps $q\in \mathcal{M}^{\mathrm{Kss,qmaps}}_{m,r,u,v,\mathrm{id}_{\mathbb{P}_{\mathbbm{k}}^N}}(S)$ to $[\mathrm{Cone}(\iota)]\circ q\in \mathcal{M}^{\mathrm{Kss,qmaps}}_{m,r,u,v,\mathrm{id}_{\mathbb{P}_{\mathbbm{k}}^N}}(S)$ gives the closed immersion $\mathcal{M}^{\mathrm{Kss,qmaps}}_{m,r,u,v,\iota}\hookrightarrow \mathcal{M}^{\mathrm{Kss,qmaps}}_{m,r,u,v,\mathrm{id}_{\mathbb{P}_{\mathbbm{k}}^N}}$, which is induced by $Z_4\subset Z_3$. 
\end{rem}

We set the following substack $\mathcal{M}^{\mathrm{Kst,qmaps}}_{m,r,u,v,\iota}$ of $\mathcal{M}^{\mathrm{Kss,qmaps}}_{m,r,u,v,\iota}$ such that the collection $\mathcal{M}^{\mathrm{Kss,qmaps}}_{m,r,u,v,\iota}(S)$ of objects for any scheme $S$;
$$\left\{
 \vcenter{
 \xymatrix@C=12pt{
(\mathcal{C},\mathcal{B})\ar[rr]^-{q}\ar[dr]_{\pi}&& [\mathrm{Cone}(X)\times S/\mathbb{G}_{m,S}] \ar[dl]\\
&S
}
}
\;\middle|
\begin{array}{rl}
(i)&\text{$q$ is a family of log Fano quasi-}\\
&\text{maps of degree $m$ and of weight $u$,}\\
(ii)&\text{$q_{\bar{s}}$ is strictly K-stable}\\
&\text{for any geometric point $\bar{s}\in S$,}\\
(iii)&\text{$r\mathcal{B}_{\bar{s}}$ is a $\mathbb{Z}$-divisor}\\
&\text{for any geometric point $\bar{s}\in S$}\\
(iii)&\text{$\mathrm{deg}_{\mathbb{P}^1}(-K_{\mathbb{P}^1}-\mathcal{B}_{\bar{s}})-um=v$}\\
&\text{for any geometric point $\bar{s}\in S$}
\end{array}\right\}.$$

\begin{cor}\label{cor--moduli--kst--quasi--maps}
    Fix an embedding of a projective normal variety $\iota\colon X\subset \mathbb{P}^N$, $v,u\in\mathbb{Q}_{>0}$ and $r,m\in\mathbb{Z}_{>0}$.
We also assume that $0<u<1-\frac{v}{2}$.
Then, $\mathcal{M}^{\mathrm{Kst,qmaps}}_{m,r,u,v,\iota}$ is a Deligne--Mumford stack and an open substack fo $\mathcal{M}^{\mathrm{Kss,qmaps}}_{m,r,u,v,\iota}$.
\end{cor}

\begin{proof}
    We first show that $\mathcal{M}^{\mathrm{Kst,qmaps}}_{m,r,u,v,\iota}$ is an open substack of $\mathcal{M}^{\mathrm{Kss,qmaps}}_{m,r,u,v,\iota}$.
    Take $Z_4$ as Theorem \ref{thm--quot--stack}.
    As the construction of $Z_4$, we can take an open subset $Z_5\subset Z_4$ such that  for any geometric point $\bar{s}\in Z_4$, $\bar{s}\in Z_5$ if and only if there exists $0\ne f'\in \mathrm{Span}(f_0,f_1,\ldots,f_N)$ such that $(\mathcal{U}_{\overline{s}},\frac{1}{r}\mathcal{B}_{\bar{s}}+u\cdot \mathrm{div}(f'))$ is K-stable, i.e., $\mathrm{mult}_p(\frac{1}{r}\mathcal{B}_{\bar{s}}+u\cdot \mathrm{div}(f'))< 1-\frac{v}{2}$ for any closed point $p\in \mathcal{U}_{\bar{s}}$.
    By Lemma \ref{lem--two--delta--invariants}, this condition is also equivalent to that $q_{Z_4,\bar{s}}$ is K-stable.
    Thus,
    $$\mathcal{M}^{\mathrm{Kst,qmaps}}_{m,r,u,v,\iota}\cong [Z_5/PGL(P(1))\times PGL(m+1)]$$
    and it is obvious that $\mathcal{M}^{\mathrm{Kst,qmaps}}_{m,r,u,v,\iota}$ is an open substack of $\mathcal{M}^{\mathrm{Kss,qmaps}}_{m,r,u,v,\iota}$.

To show that $\mathcal{M}^{\mathrm{Kst,qmaps}}_{m,r,u,v,\iota}$ is a Deligne--Mumford stack, it suffices to show that for any K-stable log Fano quasimap $q\colon (\mathbb{P}^1,B)\to [\mathrm{Cone}(X)/\mathbb{G}_m]$, the automorphism group $\mathrm{Aut}(q)$ of $q$ is finite by \cite[Corollary 8.4.2]{Ols}.
Note that $\mathrm{Aut}(q)$ is an algebraic group contained in $PGL(2)$ as a Zariski closed subgroup.
Assume that $\mathrm{Aut}(q)$ is not finite. 
Then $q$ is constant.
By taking general $D\in|\mathscr{L}|_q$, where $\mathscr{L}$ is the line bundle associated with $q$, $(\mathbb{P}^1,B+uD)$ is K-stable due to Lemma \ref{lem--two--delta--invariants}.
Then, $\mathrm{Aut}(q)$ is contained in $\mathrm{Aut}(\mathbb{P}^1,B+uD)$. 
This contradicts to \cite[Corollary 1.3]{BX}.
Hence, $\mathrm{Aut}(q)$ is finite and we have the assertion.
\end{proof}

We need to prepare the following lemma to show the positivity of the CM line bundle.

\begin{lem}
Fix an embedding of a projective normal variety $\iota\colon X\subset \mathbb{P}^N$, $v,u\in\mathbb{Q}_{>0}$ and $r,m\in\mathbb{Z}_{>0}$.
We also assume that $0<u<1-\frac{v}{2}$.
    Set the following substack $\mathcal{M}^{\mathrm{Kss,nonconst.qmaps}}_{m,r,u,v,\iota}$ of $\mathcal{M}^{\mathrm{Kss,qmaps}}_{m,r,u,v,\iota}$ so that the collection of objects of $\mathcal{M}^{\mathrm{Kss,nonconst.qmaps}}_{m,r,u,v,\iota}(S)$ for any scheme $S$ is
    $$\left\{
 \vcenter{
 \xymatrix@C=12pt{
(\mathcal{C},\mathcal{B})\ar[rr]^-{q}\ar[dr]_{\pi}&& [\mathrm{Cone}(X\times S)/\mathbb{G}_{m,S}] \ar[dl]\\
&S
}
}
\;\middle|
\begin{array}{rl}
(i)&\text{$q$ is a family of log Fano quasi-}\\
&\text{maps of degree $m$ and of weight $u$,}\!\!\\
(ii)&\text{$q_{\bar{s}}$ is K-semistable}\\
&\text{for any geometric point $\bar{s}\in S$,}\\
\!(iii)&\text{$r\mathcal{B}_{\bar{s}}$ is a $\mathbb{Z}$-divisor}\\
&\text{for any geometric point $\bar{s}\in S$}\\
(iv)&\text{$\mathrm{deg}_{\mathbb{P}^1}(-K_{\mathbb{P}^1}-\mathcal{B}_{\bar{s}})-um=v$}\\
&\text{for any geometric point $\bar{s}\in S$}\\
(v)&\text{$q_{\bar{s}}$ is non-constant for any}\\
&\text{geometric point $\bar{s}\in S$}
\end{array}\right\}.$$
Then, $\mathcal{M}^{\mathrm{Kss,nonconst.qmaps}}_{m,r,u,v,\iota}$ is an open substack of $\mathcal{M}^{\mathrm{Kss,qmaps}}_{m,r,u,v,\iota}$.
\end{lem}

\begin{proof}
    It suffices to show the following by Theorem \ref{thm--quot--stack}.
    For any scheme $S$ of finite type over $\mathbbm{k}$ and object $q\colon(\mathcal{C},\mathcal{B})\to[\mathrm{Cone}(X)\times S/\mathbb{G}_{m,S}]$ in $\mathcal{M}^{\mathrm{Kss,qmaps}}_{m,r,u,v,\iota}(S)$, there exists an open subset $S'\subset S$ such that for any geometric point $\bar{s}\in S$, $q_{\bar{s}}$ is non-constant if and only if $\bar{s}\in S'$. 
    We will show this as follows.
    By Lemma \ref{lem--openness--criterion}, it suffices to check the following two statements for $q\colon(\mathcal{C},\mathcal{B})\to[\mathrm{Cone}(X)\times S/\mathbb{G}_{m,S}]$ over a scheme essentially of finite type over $\mathbbm{k}$.
    \begin{enumerate}
        \item If $S$ is smooth and irreducible and $q_{\bar{\eta}}$ is non-constant where $\eta$ is the generic point of $S$, then there exists a non-empty open subset $U\subset S$ such that $q_{\bar{s}}$ is non-constant for any $\bar{s}\in U$, and
        \item If $S$ is a spectrum of discrete valuation ring $R$ essentially of finite type over $\mathbbm{k}$ and $q_{\bar{s}}$ is non-constant, where $s\in S$ is the closed point, then $q_{\bar{\eta}}$ is non-constant, where $\eta$ is the generic point of $S$. 
    \end{enumerate}
    Let $\mathscr{L}$ be the line bundle induced by $q$.
    For (1), if $q_{\bar{\eta}}$ is non-constant, then we can choose two sections $f_1,f_2\in\mathrm{Sec}_{q}(\mathscr{L})$ such that $\mathrm{div}(f_{1,\bar{\eta}})\ne\mathrm{div}(f_{2,\bar{\eta}})$.
    By shrinking $S$, we may assume that $\mathrm{div}(f_1)$ and $\mathrm{div}(f_2)$ are flat over $S$.
    It is not hard to see that $\mathrm{div}(f_{1,\bar{s}})\ne\mathrm{div}(f_{2,\bar{s}})$ is an open condition in $S$.
    Therefore, there exists an open subset $U\subset S$ such that $\mathrm{div}(f_{1,\bar{s}})\ne\mathrm{div}(f_{2,\bar{s}})$ for any $\bar{s}\in U$.
    This implies that $q_{\bar{s}}$ is non-constant for any $\bar{s}\in U$.

    For (2), if $q_{\bar{s}}$ is non-constant, then we can choose two sections $f_1,f_2\in\mathrm{Sec}_{q}(\mathscr{L})$ such that $\mathrm{div}(f_{1,\bar{s}})\ne\mathrm{div}(f_{2,\bar{s}})$.
    The condition $\mathrm{div}(f_{1,\bar{s}})\ne\mathrm{div}(f_{2,\bar{s}})$ is open and hence $\mathrm{div}(f_{1,\bar{\eta}})\ne\mathrm{div}(f_{2,\bar{\eta}})$.
    This implies that $q_{\bar{\eta}}$ is non-constant.

    Since (1) and (2) hold, we obtain the assertion.
\end{proof}

\subsection{$\Theta$-reductivity and $\mathsf{S}$-completeness}

In this subsection, we deal with the $\Theta$-reductivity and $\mathsf{S}$-completeness for $\mathcal{M}^{\mathrm{Kss,qmaps}}_{m,r,u,v,\iota}$. 

We prepare some notations. 
For a scheme $S$, we set 
$$\Theta_S:=[\mathbb{A}^1_{S}/\mathbb{G}_{m,S}].$$
    Let $R$ be a discrete valuation ring with the fractional field $K$, and let $\pi\in R$ be a generator of the maximal ideal.
    For the simplicity of notation, $\Theta_{\mathrm{Spec}R}$ is denoted by $\Theta_R$.
 We also define an open subset $\Theta_R^\circ \subset \Theta_R$ by
$$\Theta_R^\circ:=\Theta_{\mathrm{Spec}\,K}\cup [\mathbb{G}_{m,R}/\mathbb{G}_{m,R}].$$
    We set 
    $$\mathrm{ST}(R):=[\mathrm{Spec}(R[s,t]/(st-\pi))/\mathbb{G}_{m,R}],$$
     where $\mathbb{G}_{m,R}$ acts on $R[s,t]/(st-\pi)$ in the way that 
    \[
    R[s,t]/(st-\pi)\ni s,t\mapsto s\otimes\mu,t\otimes\mu^{-1}\in  R[s,t]/(st-\pi)\otimes_RR[\mu,\mu^{-1}].
    \]
    Let $\mathbf{0}\in\mathrm{Spec}(R[s,t]/(st-\pi))$ be the point corresponding to the maximal ideal generated by $\pi$, $s$, and $t$. 
    We put $$\mathrm{ST}(R)^\circ:=[(\mathrm{Spec}(R[s,t]/(st-\pi))\setminus\{\mathbf{0}\})/\mathbb{G}_{m,R}].$$

\begin{defn}[$\Theta$-reductivity and $\mathsf{S}$-completeness]
   Let $\mathcal{M}$ be a stack over $\mathbbm{k}$ such that $\mathcal{M}=[X/G]$ for some scheme $X$ of finite type over $\mathbbm{k}$ and linear algebraic group $G$ over $\mathbbm{k}$.
    We say that 
    \begin{itemize}
        \item $\mathcal{M}$ satisfies {\it $\Theta$-reductivity} if for any discrete valuation ring $R$ and morphism $\varphi^\circ\colon\Theta_R^\circ\to \mathcal{M}$, there exists a unique extension $\varphi\colon \Theta_R\to \mathcal{M}$ of $\varphi^\circ$, and
        \item $\mathcal{M}$ satisfies {\it $\mathsf{S}$-completeness} if for any discrete valuation ring $R$ and morphism $\psi^\circ\colon\mathrm{ST}(R)^\circ\to \mathcal{M}$, there exists a unique extension $\psi\colon \mathrm{ST}(R)\to \mathcal{M}$ of $\psi^\circ$.
    \end{itemize}
    By \cite[Propositions 3.18 and 3.42]{AHLH}, if the above conditions hold at least for any discrete valuation ring $R$ essentially of finite type over $\mathbbm{k}$, then both conditions hold.
\end{defn}

\begin{thm}\label{thm--s-complete}
Fix a closed embedding of a projective normal variety $\iota\colon X\subset \mathbb{P}^N$, $v,u\in\mathbb{Q}_{>0}$ and $r,m\in\mathbb{Z}_{>0}$.
We assume $0<u<1-\frac{v}{2}$.  
Then $\mathcal{M}^{\mathrm{Kss,qmaps}}_{m,r,u,v,\iota}$ satisfies $\Theta$-reductivity and $\mathsf{S}$-completeness.
\end{thm}

\begin{proof}
    We first deal with $\mathsf{S}$-completeness.
    Let $R$ be a discrete valuation ring essentially of finite type over $\mathbbm{k}$, and let $\psi^\circ\colon \mathrm{ST}(R)^\circ\to \mathcal{M}^{\mathrm{Kss,qmaps}}_{m,r,u,v,\iota}$ be an arbitrary morphism. 
    Then the datum of $\psi^\circ$ corresponds to $q^\circ\colon (Y^{\circ},B^\circ)\to[\mathrm{Cone}(X)/\mathbb{G}_m]$ a $\mathbb{G}_m$-equivariant family of K-semistable log Fano quasimaps over $\mathrm{Spec}(R[s,t]/(st-\pi))\setminus\{\mathbf{0}\}$ of weight $u$ and degree $m$.
    Thus, it suffices to show that there exists a $\mathbb{G}_m$-equivariant family of K-semistable log Fano quasimaps $q\colon (Y,B)\to[\mathrm{Cone}(X)/\mathbb{G}_m]$ over $\mathrm{Spec}(R[s,t]/(st-\pi))$ such that $q|_{\mathrm{Spec}(R[s,t]/(st-\pi))\setminus\{\mathbf{0}\}}=q^\circ$ unique up to isomorphism.
    Take a general vector $(a_i)_{i=0}^{N}\in\mathbbm{k}^{N+1}$ and its corresponding section $f^\circ\in H^0(Y^\circ,\mathscr{L}^\circ)$, where $\mathscr{L}^\circ$ is the line bundle induced by $q^\circ$.
    We note that $f^\circ$ is $\mathbb{G}_m$-invariant.
    By the choice of $f^\circ$, we see that $g^\circ\colon(Y^{\circ},B^\circ+u\mathrm{div}(f^\circ))\to \mathrm{Spec}(R[s,t]/(st-\pi))\setminus\{\mathbf{0}\}$ is a $\mathbb{G}_m$-equivariant family of K-semistable log pairs by Lemma \ref{lem--two--delta--invariants}. 
    By \cite[Theorem 3.3]{ABHLX} and its proof, $g^\circ$ is extended to a $\mathbb{G}_m$-equivariant family of K-semistable log Fano pairs
    \[
    g\colon(Y,B+uD)\to \mathrm{Spec}(R[s,t]/(st-\pi))
    \]
     unique up to isomorphism, where $B$ is the closure of $B^\circ$ and $D$ is the closure of $\mathrm{div}(f^\circ)$.
    Since the codimension of $Y\setminus Y^\circ$ is two and any fiber of $g$ is a normal curve, there exists a unique extension $\mathscr{L}$ of $\mathscr{L}^\circ$ as a line bundle.
    Then we obtain the canonical isomorphism $H^0(Y,\mathscr{L})\cong H^0(Y^\circ,\mathscr{L}^\circ)$.
    Therefore, we see that there exists the unique extension $f\in H^0(Y,\mathscr{L})$ of $f^\circ$ and $\mathrm{div}(f)$ does not contain $Y\setminus Y^\circ$.
    Furthermore, the homomorphism $\mathcal{O}_{Y^\circ}^{N+1}\to \mathscr{L}^\circ$ induced by $q^\circ$ is uniquely extended to a homomorphism $\mathcal{O}_{Y}^{N+1}\to \mathscr{L}$.
    This shows that we have an extension $q\colon (Y,B)\to[\mathrm{Cone}(X)/\mathbb{G}_m]$ as a family of K-semistable log Fano quasimaps.

    Next, we show that the extension $q$ of $q^\circ$ is unique. Take  $q'\colon(Y',B')\to[\mathrm{Cone}(X)/\mathbb{G}_m]$ another extension.
    Let $\mathscr{L}'$ be the line bundle associated to $q'$. Since the codimension of $Y'\setminus Y^\circ$ is two, we see that $H^0(Y',\mathscr{L}')\cong H^0(Y^\circ,\mathscr{L}^\circ)$.
   We choose $f^\circ\in H^0(Y^\circ,\mathscr{L}^\circ)$ corresponding to a general vector $(a_i)_{i=0}^{N}\in\mathbbm{k}^{N+1}$ such that $(Y_{\mathbf{0}},B_{\mathbf{0}}+u\mathrm{div}(f)_{\mathbf{0}})$ and $(Y'_{\mathbf{0}},B'_{\mathbf{0}}+u\mathrm{div}(f')_{\mathbf{0}})$ are K-semistable, where $f$ and $f'$ correspond to the extensions of $f^\circ$. 
    By \cite[Theorem 3.3]{ABHLX}, the families $g\colon(Y,B+u\mathrm{div}(f))\to \mathrm{Spec}(R[s,t]/(st-\pi))$ and $g'\colon(Y',B'+u\mathrm{div}(f'))\to \mathrm{Spec}(R[s,t]/(st-\pi))$ are isomorphic to each other.
    By $H^0(Y',\mathscr{L}')\cong H^0(Y^\circ,\mathscr{L}^\circ)$, the quasimap structures $q$ and $q'$ are also isomorphic to each other in the natural way.
    These arguments imply the $\mathsf{S}$-completeness.
    
On the other hand, the $\Theta$-reductivity follows from a similar argument as above by reducing to the log Fano pair case \cite[Theorem 5.2]{ABHLX}.
Thus, we complete the proof.
\end{proof}
It is easy to see that Theorem \ref{thm--s-complete} does not hold for log--twisted K-stability by the following easy example.
Example \ref{ex--S-compl--log--twisted} states that to construct ``a moduli of log--twisted K-semistable Fano pairs'' via the good moduli theory, we need some additional information, for example the quasimap structure.
\begin{ex}\label{ex--S-compl--log--twisted}
    Consider a log Fano pair $(\mathbb{P}^1,\frac{1}{2}p)$ for some closed point $p\in \mathbb{P}^1$.
    Note that if $(\mathbb{P}^1,\frac{1}{2}p)$ degenerates to a log Fano pair, then the degeneration is isomorphic to $(\mathbb{P}^1,\frac{1}{2}p)$.
    Furthermore, the log--twisted Fano pair $(\mathbb{P}^1,\frac{1}{2}p,-\frac{1}{4}K_{\mathbb{P}^1})$ is twisted K-polystable in the sense of \cite[Definition 3.2]{BLZ} even though $(\mathbb{P}^1,\frac{1}{2}p)$ is K-unstable. 
    Here, the stack $[\mathrm{Spec}\mathbbm{k}/\mathbb{G}_a]$ is isomorphic to the moduli stack of klt log Fano pairs degenerated from $(\mathbb{P}^1,\frac{1}{2}p)$.
    If the family of $(\mathbb{P}^1,\frac{1}{2}p)$ satisfies the $\mathsf{S}$-completeness, this contradicts \cite[Proposition 3.47]{AHLH}.
    This means that a family of log--twisted K-semistable Fano pairs does not satisfy the $\mathsf{S}$-completeness in general without any quasimap structure.
\end{ex}

\subsection{Properness}
In this subsection, we deal with the properness of the good moduli space of $\mathcal{M}^{\mathrm{Kss,qmaps}}_{m,r,u,v,\iota}$.
We will show the properness in a different way from \cite{BHLLX}, i.e., without using the $\Theta$-stratification.

\begin{lem}\label{lem--properness}
    Let $R$ be a discrete valuation ring essentially of finite type over $\mathbbm{k}$ with the fractional field $K$.
    Let $q\colon (Y,B)\to[\mathbb{A}^{N+1}_K/\mathbb{G}_{m,K}]$ be a family of log Fano quasimaps over $\mathrm{Spec}\,K$ of degree $m$ and of weight $u$ such that $Y_{\overline{K}}\cong\mathbb{P}^1_{\overline{K}}$, where $\overline{K}$ is the algebraic closure of $K$.

Then there exists a finite field extension $K\subset K'$ and the integral closure $R'$ of $R$ in $K'$ satisfying the following.

    \begin{itemize}
    \item there exists a family of log Fano quasimaps $\bar{q}\colon (Y_{R'},B_{R'})\to [\mathbb{A}^{N+1}_{R'}/\mathbb{G}_{m,R'}]$,
    \item an arbitrary closed point $r'\in\mathrm{Spec}\,R'$ is mapped to the closed point $r\in\mathrm{Spec}\,R$, and
    \item the restriction $\bar{q}|_{\mathrm{Spec}(K')}\colon (Y_{K'},B_{K'})\to [\mathbb{A}^{N+1}_{K'}/\mathbb{G}_{m,K'}]$ coincides with the base change of $q\colon Y\to[\mathbb{A}^{N+1}_K/\mathbb{G}_{m,K}]$. 
    \end{itemize}
\end{lem}

\begin{proof}
    By the assumption and Lemma \ref{lem--P^1--bundle--trivialize}, there exists a
   finite field extension $K''$ of $K$ such that $Y\times_{\mathrm{Spec}K}\mathrm{Spec}K''\cong\mathbb{P}^1_{K''}$. 
    Let $R''$ be the integral closure of $R$ in $K''$ and let $\bar{r}$ be a geometric point of $\mathrm{Spec} R''$ supported on a closed point.
    Let $\mathscr{L}_{K''}$ be the associated line bundle with $q_{K''}$.
    Take the corresponding section $f_j\in\mathrm{Sec}_{q_{K''}}(\mathscr{L}_{K''})$ to the $j$-th coordinate function on the canonical basis of $\mathbb{A}_{\mathbbm{k}}^{N+1}$ for $1\le j\le N+1$.
    By an open immersion $Y_{K''}\hookrightarrow\mathbb{P}^1_{R''}$, we can regard $f_j$ as a rational section of $\mathcal{O}_{\mathbb{P}^1_{R''}}(m)$.
    Let $r''$ be a closed point of $\mathrm{Spec}\,R''$, $\pi''_{r''}$ a generator of the maximal ideal of $R''$ corresponding to $r''$ and $\mu_{r''}:=\min_{f_j}\mathrm{coeff}_{\mathbb{P}^1_{r''}}(\mathrm{div}(f_j))$.
    Here, we may regard all $f_j$ as regular sections of $\mathcal{O}_{\mathbb{P}^1_{R''}}(m)$ and assume that there exists $f_j$ that has neither any zero nor pole along $\mathbb{P}^1_{r''}$ by replacing $f_j$'s with $f_j\cdot\prod_{r''}(\pi''_{r''})^{-\mu_{r''}}$.
    Let $B''$ be the Zariski closure of $B\times_{\mathrm{Spec}K}\mathrm{Spec}K''$ in $\mathbb{P}^1_{R''}$.
    Now, $f_j$'s define a quasimap structure $q_{R''}\colon(\mathbb{P}^1_{R''},B'')\to[\mathbb{A}^{N+1}_{R''}/\mathbb{G}_{m,R''}]$ but note that $q_{R''}$ is not a family of log Fano quasimaps in general.    
    Now we prove the following claim.
    \begin{claim*}
\label{claim--1}    There exists a finite extended field $K'$ of $K''$ satisfying the following for any general $f=\sum a_if_i$, where $a_i\in\mathbbm{k}$. 
\begin{enumerate}
    \item Let $R'$ be the integral closure of $R''$ in $K'$. Then, there exists a family $\pi'\colon (Y',B'+D')\to \mathrm{Spec}R'$ such that $Y'\cong\mathbb{P}^1_{R'}$, $(Y',B'+D'+Y'_{r'})$ is plt and $-(K_{Y'}+B'+D')$ is $\pi'$-ample, where $r'$ is an arbitrary closed point of $\mathrm{Spec}\,R'$, and
    \item there exists a birational map $\xi\colon(\mathbb{P}^1_{R''},B''+u\mathrm{div}(f))\times_{\mathrm{Spec}\,R''}\mathrm{Spec}\,R' \dashrightarrow (Y',B'+D')$, where $D'=u\xi_*(\mathrm{div}(f))$ and $B'=\xi_*(B''\times_{\mathrm{Spec}\,R''}\mathrm{Spec}\,R')$.
\end{enumerate}
Furthermore, the construction of $(Y',B')$ does not depend on the choice of general $f$. 
\end{claim*}
\begin{proof}[Proof]
    By the semistable reduction theorem (cf.~\cite{KKMS}), there exists a finite surjective morphism $\mathrm{Spec}\,R' \to \mathrm{Spec}\,R''$ for some $R'$ such that the base change $(\mathbb{P}^1_{R''},B''+u\mathrm{div}(f))\times_{\mathrm{Spec}\,R''}\mathrm{Spec}\,R'$ admits a semistable resolution 
    $$\varphi\colon\mathcal{X}\to\mathbb{P}^1_{R'}\times_{\mathrm{Spec}\,R''}\mathrm{Spec}\,R'.$$ 
    We note that this semistable reduction can be taken as the same one for any general $f$.
    By defining $K'$ to be the fractional field of $K''$, we obtain a finite field extension $K' \supset K''$. 
    By definition, $\varphi$ is projective and birational and 
    \[
    (\mathcal{X},\varphi^{-1}_*(B''\times_{\mathrm{Spec}\,R''}\mathrm{Spec}\,R'+u\mathrm{div}(f))+\mathcal{X}_{r'})
    \] 
    is log smooth and dlt, where $r'$ is an arbitrary closed point of $\mathrm{Spec}\,R'$. 
    Replacing $R'$ by the integral closure of $R''$ in $K'$, we may assume that $R'$ is the integral closure of $R''$ in $K'$. 
    Next, 
    we run the minimal model program of $(\mathcal{X},\varphi^{-1}_*(B'''+u\mathrm{div}(f))+\mathcal{X}_{r'})$ over $\mathrm{Spec}\,R'$ and it ends with a Mori fiber space $$\psi\colon(\mathcal{X},\varphi^{-1}_*(B'''+u\mathrm{div}(f))+\mathcal{X}_{r'})\dashrightarrow (Y',\psi_*(\varphi^{-1}_*(B'''+u\mathrm{div}(f)))+Y'_{r'})$$
   for any closed point $r'\in\mathrm{Spec}\,R'$, where $B'''$ denotes $B''\times_{\mathrm{Spec}\,R''}\mathrm{Spec}\,R'$ by \cite{BCHM} and Remark \ref{rem-mmp-over-local-rings}.
    Since $Y'$ is $\mathbb{Q}$-factorial and the relative Picard number $\rho(Y'/\mathrm{Spec}\,R')=1$, $Y'_{r'}$ is irreducible.
    Hence, $(Y',\psi_*(\varphi^{-1}_*(B'''+u\mathrm{div}(f)))+Y'_{r'})$ is plt and $-K_{Y'}-\psi_*(\varphi^{-1}_*(B'''+u\mathrm{div}(f)))$ is relatively ample for any $r'$.
    Since the minimal model program as above depends only on the relative $\mathbb{R}$-linear equivalence class of $\varphi^{-1}_*(B'''+u\mathrm{div}(f))+\mathcal{X}_{r'}$, $\psi$ does not depend on the choice of $r'$ or $f$.
    Moreover, $\mathbb{P}^1_{\overline{r'}}\cong Y'_{\overline{r'}}$ since $Y'_{r'}$ is one dimensional and normal, and $\mathbb{P}^1_{K'}\cong Y'_{K'}$ by assumption.
Note that the canonical morphism $\pi_{Y'}\colon Y'\to\mathrm{Spec}\,R'$ is smooth and projective.
    Therefore, we can take a line bundle $\mathcal{L}$ on $Y'$ such that $\mathcal{L}|_{Y'_{K'}}\cong\mathcal{O}_{\mathbb{P}^1_{K'}}(1)$. 
    It is easy to see that $h^0(Y'_{r'},\mathcal{L}_{r'})=2$ and $h^1(Y'_{r'},\mathcal{L}_{r'})=0$ for any closed point $r'\in\mathrm{Spec}\,R'$ and hence $Y'\cong\mathbb{P}_{R'}(\pi_{Y',*}\mathcal{L})$.
    Thus, we obtain the assertion.
\end{proof}
\noindent {\it Proof of Lemma \ref{lem--properness} continued. }
Let $f=\sum a_if_i$ be a general section and take $Y'$ as in the claim.
We may assume that $\mathrm{coeff}_{Y'_{r'}}(\mathrm{div}(f))=\min_{h\in\mathrm{Sec}_{q_{K'}\setminus\{0\}}(\mathscr{L}_{K'})}\mathrm{coeff}_{Y'_{r'}}(\mathrm{div}(h)) $ by perturbing $f$.
By taking a section $g\in H^0(\mathbb{P}^1_{R'},\mathcal{O}_{\mathbb{P}^1_{R'}}(m))$ such that $u\mathrm{div}(g)=D'$ and $g_j\in H^0(\mathbb{P}^1_{R'},\mathcal{O}(m))$ such that $g_j=g\cdot \frac{f_j}{f}$.
We see that $g_j$ is a well-defined section. 
Then, it is easy to see that $g_j$'s define a quasimap structure $q'\colon Y'\to[\mathbb{A}^{N+1}_{R'}/\mathbb{G}_{m,R'}]$ as desired.
\end{proof}

\begin{thm}\label{thm--properness}
    Let $R$ be a discrete valuation ring essentially of finite type over $\mathbbm{k}$ with the fractional field $K$.
    Let $q\colon Y\to[\mathbb{A}^{N+1}_K/\mathbb{G}_{m,K}]$ be a quasimap over $\mathrm{Spec}\,K$ such that $Y_{\overline{K}}\cong\mathbb{P}^1_{\overline{K}}$ of degree $m$ and of weight $u$, where $\overline{K}$ is the algebraic closure of $K$.
    Let $B$ be an effective $\mathbb{Q}$-divisor on $Y$ and suppose that $q_{\overline{K}}\colon (Y_{\overline{K}},B_{\overline{K}})\to[\mathbb{A}^{N+1}_{\overline{K}}/\mathbb{G}_{m,\overline{K}}]$ is a K-semistable log Fano quasimap.
    Suppose that $0<u<\frac{1}{2}\left(\mathrm{deg}_{\mathbb{P}_K^1}(B_{K})+um\right)$.

    Then there exists a finite field extension $K'$ of $K$ and the integral closure $R'$ of $R$ in $K'$ such that there exists a family of quasimap $\bar{q}\colon Y_{R'}\to [\mathbb{A}^{N+1}_{R'}/\mathbb{G}_{m,R'}]$ and an effective $\mathbb{Q}$-divisor $B_{R'}$ on $Y_{R'}$ satisfying the following:
    \begin{itemize}
    \item the restriction $\bar{q}|_{\mathrm{Spec}(K')}\colon (Y_{K'},B_{K'})\to [\mathbb{A}^{N+1}_{K'}/\mathbb{G}_{m,K'}]$ coincides with the base change of $q\colon Y\to[\mathbb{A}^{N+1}_K/\mathbb{G}_{m,K}]$, 
        \item $\bar{q}\colon (Y_{R'},B_{R'})\to [\mathbb{A}^{N+1}_{R'}/\mathbb{G}_{m,R'}]$ is a family of K-semistable log Fano quasimaps.
    \end{itemize}
\end{thm}

\begin{proof}
    By Lemma \ref{lem--properness} and replacing $R$, we may assume that $Y_R\cong\mathbb{P}^1_{R}$ and there exists a family of klt log Fano quasimaps 
    \[
    q_R\colon (Y_R,B_R)\to [\mathbb{A}^{N+1}_{R}/\mathbb{G}_{m,R}].
    \]
    Note that $R$ may not be a discrete valuation ring anymore.
    Let $\bar{r}$ be a geometric point in $\mathrm{Spec}\,R$ supported on a closed point. 
     Suppose that the fiber $q_{\bar{r}}\colon (Y_{\bar{r}},B_{\bar{r}})\to [\mathbb{A}^{N+1}_{\bar{r}}/\mathbb{G}_{m,{\bar{r}}}]$ of $q_R$ over $\bar{r}$ is K-unstable. 
    Let $\mathscr{L}$ be the line bundle on $Y_R=\mathbb{P}^1_{R}$ associated with $q_{R}$.
    Set $\mu:=\mathrm{deg}(B_{{\bar{r}}})+um$ and take a general section $f\in \mathrm{Sec}_{q_{R}}(\mathscr{L})$, and put $D_f:=\mathrm{div}_{\mathscr{L}}(f)$.
    Since $q_{\bar{r}}$ is K-unstable, Lemma \ref{lem--two--delta--invariants} implies that $(\mathbb{P}^1_{\bar{r}},B_{{\bar{r}}}+u(D_f)_{\bar{r}})$ is K-unstable and there exists the unique closed point $s\in \mathbb{P}^1_{\bar{r}}$ such that $\mathrm{mult}_s(B_{{\bar{r}}}+u(D_f)_{\bar{r}})>\frac{\mu}{2}$.
Note that $\mathrm{mult}_s(B_{{\bar{r}}}+u(D_f)_{\bar{r}})$ is independent from the choice of general $f$.
    Without loss of generality, we may assume that any irreducible component $S$ of $\mathrm{Supp}(B_R+uD_f)$ is isomorphic to $\mathrm{Spec}\, R$. Indeed, if there exists an irreducible component $S$ non-isomorphic to $\mathrm{Spec}\, R$, then we may replace $\mathrm{Spec}\, R$ with the normalization of $S$.

From now, we show that we can replace the family $q_R$ of log Fano quasimaps by a family of K-semistable log Fano quasimaps.
We construct a sequence of blow-ups $h_i\colon Y_i\to Y_{i-1}$ at the smooth closed point of $Y_{i-1}$, where $Y_0:=Y_R$. 
Put $B_0^{(f)}:=B_R+uD_f$. 
First, take the blow-up $h_1\colon Y_1\to Y_R$ at the point $s$.
Let $E_1$ be the exceptional divisor, and set $B_1^{(f)}:=(h_1)_{*}^{-1}B_0^{(f)}$.
Suppose that we have constructed $h_k\colon Y_k\to Y_{k-1}$ with the exceptional divisor $E_k$. 
Set $B_k^{(f)}:=(h_{k})_{*}^{-1}B^{(f)}_{k-1}$.
If $E_k$ and $B_k^{(f)}$ intersect at a closed point $s_k$ with multiplicity bigger than $\frac{\mu}{2}$ for any general $f$, then let $h_{k+1}\colon Y_{k+1}\to Y_k$ be the blow-up at $s_k$.
Here, we note that the closed point $s_{k}$ and the intersection multiplicity of $E_k$ and $B_k^{(f)}$ at $s_k$ is independent from the choice of general $f$. 
It is not hard to see that if $s_k$ exists, then $E_k$ and $B_k^{(f)}$ intersect  with multiplicity less than $\frac{\mu}{2}$ at other points since any irreducible component of $B_k^{(f)}$ is a section over $\mathrm{Spec}\, R$. 
We note that all irreducible components of $B_k^{(f)}$ have multiplicity at most $\frac{\mu}{2}$ since $q$ is a family of K-semistable log Fano quasimaps.
By the theorem of resolution of singularities of a curve embedded in a regular surface, this procedure stops at some $k$. 
At this $k$, we contract all irreducible components $Y_{k,\bar{r}}$ but $E_k$ and construct the family of log Fano pairs $(Y'_R,B'_R+u(D^{(f)})')$, where $B'_R$ and $(D^{(f)})'$ are the strict transforms of $B_R$ and $D^{(f)}$ for any general $f$ respectively.
We check that $(Y'_{\bar{r}},B'_{\bar{r}}+u(D^{(f)})'_{\bar{r}})$ is K-semistable for any general $f$.
For this, it is enough to show that $B'_{\bar{r}}+u(D^{(f)})'_{\bar{r}}$ has multiplicity at most $\frac{\mu}{2}$ for any point.
Let $h\colon Y_k\to Y'_R$ be the contraction.
By the definition of $h_k$, we see that $B'_{\bar{r}}+u(D^{(f)})'_{\bar{r}}$ has the multiplicity at most $\frac{\mu}{2}$ for any point over which $h$ is isomorphic.
On the other hand, we see that $h_{*}^{-1}(B'+u(D^{(f)})')\cdot (Y_{k,\bar{r}}-E_k)\le \frac{\mu}{2}$ by the definition of $h_{k-1}$.
This shows that $(Y'_{\bar{r}},B'_{\bar{r}}+u(D^{(f)})'_{\bar{r}})$ is K-semistable.

We know that we can take a general $f\in \mathrm{Sec}_{q_R}(\mathscr{L})$ such that for any $g\in\mathrm{Sec}_{q_R}(\mathscr{L})$, $\frac{g}{f}$ has no pole at $E_k$ as a rational function.
Fix such $f$.
We take $f'\in H^0(\mathcal{O}_{Y'_R}((D^{(f)})'))$ such that $\mathrm{div}(f')=(D^{(f)})'$. 
Let $f_i\in\mathrm{Sec}_{q_R}(\mathscr{L}_R)$ be the section corresponding to the canonical coordinate sections of $\mathbb{A}^{N+1}$.
Then, we see that $g_i:=\frac{f_i}{f}f'\in H^0(\mathcal{O}_{Y'_R}((D^{(f)})'))$ and $g_i$'s  define a quasimap structure $q'_R\colon Y'_R\to[\mathbb{A}_R^{N+1}/\mathbb{G}_{m,R}]$.
By the choice of $q'_R$, $q'_{\bar{r}}$ is K-semistable and $q'_R|_{\mathrm{Spec}\,K}=q$. 

By repeating this procedure for all closed points of $\mathrm{Spec}\,R$, we complete the proof.
\end{proof}

\begin{cor}
Fix a closed embedding of a projective normal variety $\iota\colon X\subset \mathbb{P}^N$, $v,u\in\mathbb{Q}_{>0}$ and $r,m\in\mathbb{Z}_{>0}$. 
We also assume that $0<u<1-\frac{v}{2}$.  
Then,     $\mathcal{M}^{\mathrm{Kss,qmaps}}_{m,r,u,v,\iota}$ admits a good moduli space $M^{\mathrm{Kps,qmaps}}_{m,r,u,v,\iota}$ proper over $\mathbbm{k}$. 
Furthermore, there is a natural bijection of the set of $\mathbbm{k}$-valued points of $M^{\mathrm{Kps,qmaps}}_{m,r,u,v,\iota}$ and the set of the K-polystable quasimaps.
\end{cor}

\begin{proof}
To show the assertions, we may replace $M^{\mathrm{Kps,qmaps}}_{m,r,u,v,\iota}$ with $M^{\mathrm{Kps,qmaps}}_{m,r,u,v,\mathrm{id}_{\mathbb{P}^N}}$ and assume that $X=\mathbb{P}^N$ by Remark \ref{rem--moduli--embed} and  \cite[Lemma 4.14]{alper}.
   By Theorem \ref{thm--quot--stack}, the first assertion in this case immediately follows from \cite[Theorem A]{AHLH}, Theorems \ref{thm--s-complete} and \ref{thm--properness}.
   The last assertion follows from Lemma \ref{lem--DF=0-case}.
\end{proof}

\subsection{Positivity of the CM line bundle}\label{subsection--projectivity}

Here, we discuss on the definition of CM line bundles of families of log Fano quasimaps and their positivity. 
We fix a normal closed subvariety $\iota\colon X\hookrightarrow \mathbb{P}^N$.

\begin{defn}[Maximal valuation on a family of quasimaps]
    Let $S$ be a normal variety and $q\colon (Y,B)\to [\mathrm{Cone}(X)/\mathbb{G}_{m,S}]$ a family of K-semistable log Fano quasimaps of degree $m$ and of weight $u$ over $S$.
    We say that the family $q$ has {\em maximal variation} if for any very general point $p_0 \in S$, irreducible curve $C \subset S$ passing through $p_0$ and two general distinct points $p_1$ and $p_2\in C$, $q_{p_1}$ and $q_{p_2}$ are not isomorphic.
\end{defn}

\begin{defn}[CM line bundle]\label{defn--qmaps--CM--linebundle}
   Let $q\colon (Y,B)\to [\mathrm{Cone}(X)\times S/\mathbb{G}_{m,S}]$ be a family of log Fano quasimaps of degree $m$ and of weight $u$ over a scheme $S$ of finite type over $\mathbbm{k}$.
   Let $\pi\colon Y\to S$ be the canonical morphism and $\mathscr{L}$ the line bundle on $Y$ induced by $q$.
   Take a sufficiently divisible $\ell\in\mathbb{Z}_{>0}$ such that $-\ell(K_{Y/S}+B)-\ell u\mathscr{L}$ is a $\pi$-ample Cartier divisor.
   By the Knudsen--Mumford expansion \cite{KM}, there exist line bundles $\lambda_0$, $\lambda_1$, and $\lambda_2$ such that for any $m>0$, we have
   \[
   \mathrm{det}\pi_*\mathcal{O}_Y(-m(\ell(K_{Y/S}+B)+\ell u\mathscr{L}))\sim\lambda_0^{\otimes\binom{m}{2}}\otimes\lambda_1^{\otimes m}\otimes \lambda_2.
   \]
   We note that $\lambda_0$, $\lambda_1$, and $\lambda_2$ are uniquely determined up to linear equivalence.
   Moreover, for any morphism $g\colon T\to S$ from a scheme, we have
  \[
   \mathrm{det}\pi_{T,*}\mathcal{O}_{Y_T}(-m(\ell(K_{Y_T/T}+B_T)+\ell u\mathscr{L}_T))\sim g^*\lambda_{0}^{\otimes\binom{m}{2}}\otimes g^*\lambda_1^{\otimes m}\otimes g^*\lambda_2.
   \]
   Then, we define the {\em CM line bundle} of the log Fano quasimap $q$ by
   \[
   \lambda_{\mathrm{CM},\pi,q}:=\frac{-2}{\ell^2}\lambda_0.
   \]
   This is a $\mathbb{Q}$-line bundle over $S$ and independent from the choice of $\ell$.
   We note that for any morphism $h\colon C\to S$ from a proper smooth curve, we have 
   $$\mathrm{deg}_C(h^*\lambda_{\mathrm{CM},\pi,q})=-(K_{Y_C/C}+B_C+u\mathscr{L}_C)^2.$$
\end{defn}

We show the following property of non-constant quasimaps (cf.~Definition \ref{defn--quasi--map--const}).

\begin{lem}\label{lem--maximal-variation-from-qmaps-to-pairs}
    Let $q\colon(Y,B)\to[\mathrm{Cone}(X)\times C/\mathbb{G}_{m,C}]$ be a family of log Fano quasimaps of weight $u$ and degree $m$ with the maximal variation over a projective curve $C$. 
    Let $a_0,\ldots,a_{N+1}\in\mathbbm{k}^{N+1}$ be vectors such that $\{a_0,\ldots,a_N\}$ is a basis, and let $\mathscr{L}$ be the line bundle associated to $q$. 
    Let $f_{j}\in\mathrm{Sec}_{q}(\mathscr{L})$ be the section corresponding to $a_j$ for each $j=0,\ldots,N+1$.
    Suppose that for a point $c\in C$, we have $f_{j,c}\ne0$ for any $j=0,\ldots,N+1$ and $\mathrm{div}(f_{j,c})\ne\mathrm{div}(f_{N+1,c})$ for any $j=0,\ldots,N$.

Then, the family of pairs $$\sigma\colon\left(Y,\mathrm{Supp}(B)+\sum_{j=1}^{N+1}\mathrm{div}(f_j)_{\mathrm{hor}}+\sum_{j=1}^N\mathrm{div}(f_j+f_{N+1})_{\mathrm{hor}}\right)\to C$$ has maximal variation (cf.~\cite[Definition 2.20]{Hat23}).
    \end{lem}

    \begin{proof}
    Assume that $\sigma$ does not have the maximal variation.
    Then all general fibers of $\sigma$ are isomorphic to each other.
    Put $D:=\mathrm{Supp}(B)+\sum_{j=1}^{N+1}\mathrm{div}(f_j)+\sum_{j=1}^N\mathrm{div}(f_j+f_{N+1})$. 
    We may assume that $Y=\mathbb{P}^1\times C$ and $D$ is horizontal by shrinking $C$ if necessary.
    Consider the following scheme 
    $$I:=\mathrm{Isom}_{C\times C}((Y,D)\times C,C\times (Y,D)),$$
    which is of finite type over $C\times C$ since for any point $p\in C\times C$, there exists an open neighborhood $p\in U\subset C\times C$, $I|_U$ is a closed subscheme of $ PGL(2)\times U$.

    By the assumption, the canonical morphism $\mu\colon I\to C\times C$ has an dense and constructible image $\mu(I)$. 
    Thus, there exists a non-empty open subset $U\subset C\times C$ such that $U\subset \mu(I)$ and $I|_U$ is a closed subscheme of $ PGL(2)\times U$.
    Let $p_1\colon C\times C\to C$ and $p_2\colon C\times C\to C$ be the first and the second projections, respectively.
    The condition for $c\in C$ that $f_{j,c}\ne0$ for any $j=0,\ldots,N+1$ and $\mathrm{div}(f_{j,c})\ne\mathrm{div}(f_{N+1,c})$ for any $j=0,\ldots,N$ is open.
    Thus, we can take a closed point $c\in p_1(U)$ such that $f_{j,c}\ne0$ for any $j=0,\ldots,N+1$ and $\mathrm{div}(f_{j,c})\ne\mathrm{div}(f_{N+1,c})$ for any $j=0,\ldots,N$.
    By regarding $I$ as a $C$-scheme under $p_1$, we can consider the fiber $I_c$ of $I$ over $c\in p_1(U)$.
    Let $\nu\colon I_c\to C$ be the composition of the canonical morphism $I_c\to C\times C$ and $p_2$.
    Since $U\cap p_1^{-1}(c)\subset p_2^{-1}(\nu(I_c))$, there exists a morphism $\psi\colon V\to C$ from an affine smooth curve such that $\xi\colon V\times_CI_c\to V$ has a section. 
    Here, $\xi$ is induced by $p_2$.
    Let $\tau\colon (Y_c,D_c)\times V\to (Y,D)\times_{C}V$ be a $V$-isomorphism corresponding to $\xi$. 
    We note that $I_c=\mathrm{Isom}_C((Y_c,D_c)\times C,(Y,D))$.
    Take $c'\in V$ such that $\psi(c')=c$.
    By replacing $\tau$ with the composition $\tau\circ(\tau_{c'}\times\mathrm{id}_{V})^{-1}$, we may assume that $\tau_{c'}$ is the identity.
    Furthermore, by replacing $V$ with its finite cover, we may assume that any irreducible component of $D_V$ is a section over $V$.
    Thus, for any $v\in V$, $\tau_{v}\colon (Y_c,D_c)\to (Y_v,D_v)$ and irreducible component $E$ of $D_V$, we may assume that $\tau_v$ maps $E_c$ to $E_v$.
    Therefore, $\tau_v$ maps $\mathrm{div}(f_{j,c})$ and $\mathrm{div}(f_{j,c}+f_{N+1,c})$ to $\mathrm{div}(f_{j,v})$ and $\mathrm{div}(f_{j,v}+f_{N+1,v})$ respectively for any $0\le j\le N+1$.
This means 
that $q_c$ and $q_v$ are isomorphic as log Fano quasimaps by Lemma \ref{lem}.
This contradicts the assumption that $q$ has maximal variation.
We complete the proof.
    \end{proof}

Let $\pi\colon \mathcal{U}_{Z_4}\to Z_4$ be the canonical morphism, which was defined in Subsection \ref{subsec3.1}.
Then $Z_4$ admits the CM line bundle $\lambda_{\mathrm{CM},\pi,q}$ and this is $PGL(P(1))\times PGL(m+1)$-equivariant. 
Thus, $\lambda_{\mathrm{CM},\pi,q}$ descends to a line bundle $\lambda^{\mathrm{qmaps}}_{\mathrm{CM}}$ on $\mathcal{M}^{\mathrm{Kss,qmaps}}_{m,r,u,v,\iota}$.
By the same argument of \cite[Lemma 10.2]{CP}, we have the following.
\begin{lem}\label{lem--CM--descent}
    $\lambda^{\mathrm{qmaps}}_{\mathrm{CM}}$ descends to a $\mathbb{Q}$-line bundle $\Lambda^{\mathrm{qmaps}}_{\mathrm{CM}}$ on $M^{\mathrm{Kps,qmaps}}_{m,r,u,v,\iota}$ that is uniquely determined up to $\mathbb{Q}$-linear equivalence.
\end{lem}

\begin{proof}
First, note that for any closed point $s\in \mathcal{M}^{\mathrm{Kss,qmaps}}_{m,r,u,v,\iota}$, the stabilizer group $G_s$ is reductive (by \cite[Proposition 12.14]{alper}).
    By the proof of Theorem \ref{thm--quot--stack}, $PGL(m+1)$ acts on $Z_4$ without any non-trivial stabilizer. 
    Hence, for any closed point $s\in \mathcal{M}^{\mathrm{Kss,qmaps}}_{m,r,u,v,\iota}$, the stabilizer group $G_s$ is contained in $PGL(P(1))$.
    Here, $G_s$ is a fiber of 
    $$\mathcal{I}:=\mathcal{M}^{\mathrm{Kss,qmaps}}_{m,r,u,v,\iota}\times_{\Delta,\mathcal{M}^{\mathrm{Kss,qmaps}}_{m,r,u,v,\iota}\times\mathcal{M}^{\mathrm{Kss,qmaps}}_{m,r,u,v,\iota},\Delta}\mathcal{M}^{\mathrm{Kss,qmaps}}_{m,r,u,v,\iota}\to \mathcal{M}^{\mathrm{Kss,qmaps}}_{m,r,u,v,\iota}$$ over $s$.
    Let $G_{s}^\circ$ be the identity component of $G_s$.
    Since $\mathcal{I}$ is an Artin stack of finite type, there exists a positive integer $N$ such that the order of $G_s/G_{s}^\circ$ divides $N$ for any closed point $s$.
    Let $l$ be a positive integer such that $(\lambda^{\mathrm{qmaps}}_{\mathrm{CM}})^{\otimes l}$ is a line bundle.
    By the same argument of \cite[Lemma 10.2]{CP}, we see that $G_{s}^\circ$ acts on $(\lambda^{\mathrm{qmaps}}_{\mathrm{CM}})^{\otimes l}_s$ trivially.
    Furthermore, $G_s/G_{s}^\circ$ acts on $(\lambda^{\mathrm{qmaps}}_{\mathrm{CM}})^{\otimes Nl}_s$ trivially.
    Thus, \cite[Theorem 10.3]{alper} shows that there exists a unique line bundle  $L$ on $M^{\mathrm{Kss,qmaps}}_{m,r,u,v,\iota}$ such that its pullback to $\mathcal{M}^{\mathrm{Kss,qmaps}}_{m,r,u,v,\iota}$ coincides with $(\lambda^{\mathrm{qmaps}}_{\mathrm{CM}})^{\otimes Nl}$ up to isomorphism.
    By letting $\Lambda^{\mathrm{qmaps}}_{\mathrm{CM}}:=\frac{1}{Nl}L$, we obtain the assertion.
\end{proof}

We call $\Lambda^{\mathrm{qmaps}}_{\mathrm{CM}}$ the {\em CM line bundle} on $M^{\mathrm{Kps,qmaps}}_{m,r,u,v,\iota}$. 
The following theorem is the main result of this subsection.
\begin{thm}\label{thm--projectivity}
Fix $v,u\in\mathbb{Q}_{>0}$ and $r,m\in\mathbb{Z}_{>0}$. 
We assume $0<u<1-\frac{v}{2}$. 
Then, the CM line bundle $\Lambda^{\mathrm{qmaps}}_{\mathrm{CM}}$ on $M^{\mathrm{Kps,qmaps}}_{m,r,u,v,\iota}$ is ample.
    In particular, $M^{\mathrm{Kps,qmaps}}_{m,r,u,v,\iota}$ is a projective scheme.
\end{thm}

We first show the nefness of the CM line bundle.

\begin{prop}\label{lem--CM-nefness}
Let $q\colon (Y,B)\to [\mathrm{Cone}(X)\times C/\mathbb{G}_{m,C}]$ be a family of K-semistable log Fano quasimaps of degree $m$ and of weight $u$ over a smooth proper curve $C$. Let $\pi\colon Y\to C$ be the canonical morphism.
Suppose that $0<u<\frac{1}{2}\left(\mathrm{deg}_{Y_t}(B_t)+um \right)$ for any general $t\in C$. 

Then, $\mathrm{deg}\,\lambda_{\mathrm{CM},\pi,q}\ge0$.
\end{prop}

To show this proposition, we prepare the following lemma, which deals with a special case.

\begin{lem}\label{lem--CM--minimization}
    Let $q\colon (Y,B)\to [\mathrm{Cone}(X)\times C/\mathbb{G}_{m,C}]$ be a family of K-semistable log Fano quasimaps of degree $m$ and of weight $u$ over a smooth proper curve $C$. Let $\pi\colon Y\to C$ be the canonical morphism.
Let $\mathscr{L}$ be the line bundle associated with $q$.
Suppose that there exist an effective $\mathbb{Q}$-divisor $D\sim_{\mathbb{Q}}\mathscr{L}$ and a family of K-semistable  log Fano pairs $\pi'\colon (Y',B')\to C$ such that $(Y',B')|_{U}\cong (Y,B+D)|_U$ over some non-empty open subset $U$ of $C$.

Then, $\mathrm{deg}\,\lambda_{\mathrm{CM},\pi,q}\ge\mathrm{deg}\,\lambda_{\mathrm{CM},\pi',(Y',B')}$.
\end{lem}

\begin{proof}
     We decompose $D$ as $D_{\mathrm{vert}}+D_{\mathrm{hor}}$, where $\pi$ maps $D_{\mathrm{vert}}$ to finite points on $C$ and every irreducible component of $D_{\mathrm{hor}}$ is flat over $C$.
    Since $D\equiv D_{\mathrm{hor}}+\mu \pi^{-1}(c)$ for some $\mu\ge0$, we have by an easy calculation that
    \[
    \mathrm{deg}\,\lambda_{\mathrm{CM},\pi,q}\ge\mathrm{deg}\,\lambda_{\mathrm{CM},\pi,(Y,B+D_{\mathrm{hor}})}.
    \]
    By \cite[Theorem 3.22]{CM}, we obtain that
    \[
   \mathrm{deg}\,\lambda_{\mathrm{CM},\pi,(Y,B+D_{\mathrm{hor}})}\ge\mathrm{deg}\,\lambda_{\mathrm{CM},\pi',(Y',B')}.
    \]
   Thus, we obtain the assertion.
\end{proof}

\begin{proof}[Proof of Proposition \ref{lem--CM-nefness}]
    Choose a closed point $c\in C$. 
By Lemma \ref{lem--two--delta--invariants}, we can find a general vector $(a_{0},\,\dots,\,a_{N})\in\mathbbm{k}^{N+1}$ corresponding to a section $f\in \mathrm{Sec}_{q_c}(\mathscr{L}_c)$ such that $(Y_c,B_c+u\mathrm{div}(f_c))$ is K-semistable.
    We decompose $\mathrm{div}(f)$ as $D_{\mathrm{vert}}+D_{\mathrm{hor}}$, where $\pi$ maps $D_{\mathrm{vert}}$ to finite points on $C$ and every irreducible component of $D_{\mathrm{hor}}$ is flat over $C$.
    Since $D_{\mathrm{hor},c}=\mathrm{div}(f_c)$ and $(Y_c,B_c+u\mathrm{div}(f_c))$ is K-semistable, we see by \cite[Corollary 1.2]{BLX} and \cite[Corollary 7.4]{BHLLX} that there exist a finite morphism $g\colon C'\to C$ from a smooth curve and a family $\pi'\colon(Y',B')\to C'$ of K-semistable log Fano pairs such that $(Y',B')|_U$ and $(Y,B+uD_{\mathrm{hor}})\times_CC'|_U$ are isomorphic over some non-empty open subset $U\subset C'$.
    By Lemma \ref{lem--CM--minimization}, we have
    \[
   \mathrm{deg}(g)\cdot \mathrm{deg}\,\lambda_{\mathrm{CM},\pi,q}\ge\mathrm{deg}\,\lambda_{\mathrm{CM},\pi',(Y',B')}.
    \]
    By \cite[Corollary 4.7]{XZ1}, we have
    \[
    \mathrm{deg}\,\lambda_{\mathrm{CM},\pi',(Y',B')}\ge0.
    \]
    We complete the proof.
\end{proof}

We quickly recall the nef and bigness for $\mathbb{Q}$-line bundles on proper algebraic spaces. 
\begin{defn}
    Let $X$ be a proper algebraic space with a $\mathbb{Q}$-line bundle $L$.
    We say that $L$ is {\it nef} if for any morphism $\varphi\colon C\to X$ from a projective curve, $\mathrm{deg}_C(\varphi^*L)\ge0$.
    Furthermore, we say that a nef $\mathbb{Q}$-line bundle $L$ is {\it big} if for any proper generically finite morphism $\psi\colon V\to X$ from a scheme, we have $(\psi^*L)^{\mathrm{dim}\,V}>0$.
\end{defn}

To show Theorem \ref{thm--projectivity}, it suffices to show the following proposition by the Nakai--Moishezon's criterion for ampleness for algebraic spaces \cite[Theorem 3.11]{kollar-moduli-stable-surface-proj}.

\begin{prop}\label{prop--key--to--positivity}
Fix $v,u\in\mathbb{Q}_{>0}$ and $r,m\in\mathbb{Z}_{>0}$. 
We also assume that $0<u<1-\frac{v}{2}$. 
    Let $S$ be an arbitrary proper irreducible and reduced subspace of $M^{\mathrm{Kps,qmaps}}_{m,r,u,v,\iota}$. Then, $\Lambda^{\mathrm{qmaps}}_{\mathrm{CM}}|_S$ is big and nef.
\end{prop}

For this, we first deal with the following special case.

\begin{lem}\label{lem--bigness-for-const-maps}
    With notation as in Proposition \ref{prop--key--to--positivity}, suppose that for any geometric point $\bar{s}\in S$, the K-polystable log Fano quasimap $(\mathbb{P}^1_{\bar{s}},B)\to [\mathrm{Cone}(X)/\mathbb{G}_{m,\bar{s}}]$ corresponding to $\bar{s}$ is constant.

    Then, $\Lambda^{\mathrm{qmaps}}_{\mathrm{CM}}|_S$ is ample.
\end{lem}

\begin{proof}
   We define a reduced closed substack $\mathcal{M}^{\mathrm{Kss,const.qmaps}}_{m,r,u,v,\iota}$ of $\mathcal{M}^{\mathrm{Kss,qmaps}}_{m,r,u,v,\iota}$ as the complement of an open substack $\mathcal{M}^{\mathrm{Kss,nonconst.qmaps}}_{m,r,u,v,\iota}$.
   By \cite[Lemma 4.14]{alper}, $\mathcal{M}^{\mathrm{Kss,const.qmaps}}_{m,r,u,v,\iota}$ has the good moduli space $M^{\mathrm{Kps,const.qmaps}}_{m,r,u,v,\iota}$, which is the scheme-theoretic image of $\mathcal{M}^{\mathrm{Kss,const.qmaps}}_{m,r,u,v,\iota}$ under the canonical morphism $\mathcal{M}^{\mathrm{Kss,qmaps}}_{m,r,u,v,\iota}\to M^{\mathrm{Kps,qmaps}}_{m,r,u,v,\iota}$.
   By assumption, we have $S\subset M^{\mathrm{Kps,const.qmaps}}_{m,r,u,v,\iota}$.
   Therefore, it suffices to show that $\Lambda^{\mathrm{qmaps}}_{\mathrm{CM}}|_{M^{\mathrm{Kps,const.qmaps}}_{m,r,u,v,\iota}}$ is ample. 

   Let $r'$ be the denominator of $u$ and consider $\mathcal{M}^{\mathrm{Kss}}_{1,rr',v}$ the K-moduli stack of log Fano pairs.
   We will show that there exists a canonical morphism 
   $$\xi\colon \mathcal{M}^{\mathrm{Kss,const.qmaps}}_{m,r,u,v,\iota}\to \mathcal{M}^{\mathrm{Kss}}_{1,rr',v}\times X.$$
   Take a surjective smooth morphism $S\to \mathcal{M}^{\mathrm{Kss,const.qmaps}}_{m,r,u,v,\iota}$ and the corresponding object $q\colon(\mathcal{C},\mathcal{B})\to [\mathrm{Cone}(X)\times S/\mathbb{G}_{m,S}]$ of $\mathcal{M}^{\mathrm{Kss,const.qmaps}}_{m,r,u,v,\iota}(S)$.
   Then, it follows from Lemma \ref{lem--two--delta--invariants} that there exists a smooth surjection $S'\to S$ such that for any connected component $S_i$ of $S'$, there exists a general section $f_i\in\mathrm{Sec}_{q_{S_i}}(\mathscr{L}_{S_i})$ with which the pair $(\mathcal{C}_{S_i},\mathcal{B}_{S_i}+u\mathrm{div}(f_{i}))$ is a family of K-semistable log pairs.
   Since $q_{\bar{s}}$ is constant for any geometric point $\bar{s}\in S$, the family $(\mathcal{C}_{S_i},\mathcal{B}_{S_i}+u\mathrm{div}(f_{i}))$ is independent from the choice of $f_i$.
   This shows that $(\mathcal{C}_{S_i},\mathcal{B}_{S_i}+u\mathrm{div}(f_{i}))$ over $S_i$ defines a descent datum and this descends to an object of $\mathcal{M}^{\mathrm{Kss}}_{1,rr',v}(S)$ (for the definition of descent data, see \cite[Remark 2.10]{HH}). 
   On the other hand, by replacing $S_{i}$ with an \'etale covering, we may assume that there exists a section $\mu_i\colon S_i\to \mathcal{C}_{S_i}$ for any $i$ and we can take a morphism $\theta_i=q_{S_i}\circ \mu_i$ since $\mathcal{C}\to S$ is smooth.
   Since $q_{\bar{s}}$ is constant for any geometric point $\bar{s}\in S$, we see that $q_{S_i}$ factors through an open subscheme $X\times S\cong[\mathrm{Cone}(X)\setminus\{0\}/\mathbb{G}_m]\times S\subset [\mathrm{Cone}(X)\times S/\mathbb{G}_{m,S}]$ and $\theta_i$ is determined independently from the choice of $\mu_i$ since $S$ is reduced.
   Therefore, $\{\theta_i\}_{i}$ defines an effective descent datum and this descends to a morphism $\theta\colon S\to X$. 
   It is easy to see that $\theta$ descends to a canonical morphism $\theta'\colon \mathcal{M}^{\mathrm{Kss,const.qmaps}}_{m,r,u,v,\iota}\to X$.
   By using $\theta'$, we obtain the desired morphism $\xi$. 

   By \cite[Theorem 6.6]{alper}, $\xi$ induces a map of algebraic spaces \[\xi'\colon M^{\mathrm{Kps,const.qmaps}}_{m,r,u,v,\iota} \to M^{\mathrm{Kps}}_{1,rr',v}\times X.\]
   By comparing $\mathbbm{k}$-valued points of them, we see that $\xi'$ is a quasi-finite morphism.
   This fact and \cite[Theorem 7.2.10]{Ols} show that $M^{\mathrm{Kps,const.qmaps}}_{m,r,u,v,\iota}$ is a scheme.
   Since $M^{\mathrm{Kps,const.qmaps}}_{m,r,u,v,\iota}$ and $M^{\mathrm{Kps}}_{1,rr',v}\times X$ are proper, $\xi'$ is also finite.
   Take an arbitrary morphism $g\colon C\to M^{\mathrm{Kps,const.qmaps}}_{m,r,u,v,\iota}$.
   Let $\Lambda_{\mathrm{CM}}$ be the CM line bundle on $M^{\mathrm{Kps}}_{1,rr',v}$, which is ample by \cite[Theorem 7.9]{XZ1}.
   Let $p_1\colon M^{\mathrm{Kps}}_{1,rr',v}\times X\to M^{\mathrm{Kps}}_{1,rr',v}$ and $p_2\colon M^{\mathrm{Kps}}_{1,rr',v}\times X\to X$ be the projections.
   To show the ampleness of $\Lambda^{\mathrm{qmaps}}_{\mathrm{CM}}|_{M^{\mathrm{Kps,const.qmaps}}_{m,r,u,v,\iota}}$, it suffices to show that $$\mathrm{deg}g^*\xi'^*(p_1^*\Lambda_{\mathrm{CM}}+2vp_2^*\iota^*\mathcal{O}_{\mathbb{P}^N}(1))= \mathrm{deg}g^*\Lambda^{\mathrm{qmaps}}_{\mathrm{CM}},$$
   which shows that $\xi'^*(p_1^*\Lambda_{\mathrm{CM}}+2vp_2^*\iota^*\mathcal{O}_{\mathbb{P}^N}(1))\equiv \Lambda^{\mathrm{qmaps}}_{\mathrm{CM}}|_{M^{\mathrm{Kps,const.qmaps}}_{m,r,u,v,\iota}}$.

   By Lemma \ref{lem--good-moduli-lift} and replacing $C$ with its finite cover, we may assume that $g$ lifts up to $\tilde{g}\colon C\to\mathcal{M}^{\mathrm{Kss,const.qmaps}}_{m,r,u,v,\iota}$.
   Take the family $q\colon (Y,B)\to [\mathrm{Cone}(X)\times C/\mathbb{G}_{m,C}]$ of K-semistable log Fano quasimaps over $C$ corresponding to $\tilde{g}$. 
   Let $\mathscr{L}$ be the line bundle induced by $q$ and take a general section $f\in\mathrm{Sec}_{q}(\mathscr{L})$ such that the support of $\mathrm{div}(f)$ does not contain all fibers of $\pi\colon Y\to C$.
   We decompose $u\mathrm{div}(f)$ as $D_{\mathrm{hor}}+D_{\mathrm{vert}}$ into the horizontal part and the vertical part.
   Note that $D_{\mathrm{vert}}$ is numerically equivalent to $\mu F$, where $\mu\in\mathbb{Q}_{>0}$ and $F$ is a fiber of $\pi$.
   Under $\xi$, $q$ assigns a morphism $\psi\colon C\to X$ and the family $(Y,B+D_{\mathrm{hor}})$ of K-semistable log Fano pairs over $C$.
   Since we can regard $q$ is a family of constant maps from $Y$ to $X$ over $C$ and $\pi$ is a contraction, we see that the line bundle $\mathscr{L}-u^{-1}D_{\mathrm{hor}}$ defines a morphism $q=\psi\circ\pi\colon Y\to X$ and hence $u\mathrm{deg}\psi^*\iota^*\mathcal{O}_{\mathbb{P}^N}(1)=\mu$.
   In this case, 
   \begin{align*}
       \mathrm{deg}g^*\Lambda^{\mathrm{qmaps}}_{\mathrm{CM}}&=(-(K_{Y/C}+B+u\mathscr{L}))^2\\
       &=(-(K_{Y/C}+B+D_{\mathrm{hor}}))^2 +2v\mu\\
       &=\mathrm{deg}g^*\xi'^*(p_1^*\Lambda_{\mathrm{CM}}+2vp_2^*\iota^*\mathcal{O}_{\mathbb{P}^N}(1)).
   \end{align*}
   Thus, we complete the proof.
\end{proof}

\begin{lem}\label{lem--bigness-for-nonconstant-maps}
    With notation as in Proposition \ref{prop--key--to--positivity}, suppose that the geometric generic point $\bar{\eta}\in S$ corresponds to a K-polystable non constant log Fano quasimap.

    Then, $\Lambda^{\mathrm{qmaps}}_{\mathrm{CM}}|_S$ is big and nef.
\end{lem}

\begin{proof}
    By Chow's lemma for algebraic spaces (cf.~\cite[Theorem 7.4.1]{Ols}), there exists a projective smooth variety $T$ and a generically finite and proper morphism $g\colon T\to S$.
    By Lemma \ref{lem--good-moduli-lift} and replacing $T$ with a proper generically finite cover of $T$, we may assume that there exist an open subscheme $T^\circ\subset T$ and a lift $h\colon T^\circ\to\mathcal{M}^{\mathrm{Kss,qmaps}}_{m,r,u,v,\iota}\times_{M^{\mathrm{Kps,qmaps}}_{m,r,u,v,\iota}}S$ of $g|_{T^\circ}$ such that for any point $t\in T^\circ$, $h(t)$ is the unique closed point of $\mathrm{Spec}\,\kappa(g(t))\times_{M^{\mathrm{Kps,qmaps}}_{m,r,u,v,\iota}} \mathcal{M}^{\mathrm{Kss,qmaps}}_{m,r,u,v,\iota}$.
    Fix an ample line bundle $H$ on $T$.
    Now, it suffices to show that $g^*\Lambda^{\mathrm{qmaps}}_{\mathrm{CM}}$ is big since $\Lambda^{\mathrm{qmaps}}_{\mathrm{CM}}$ is nef by Proposition \ref{lem--CM-nefness}.

Let $\eta_T$ be the generic point of $T$ and $q^\circ\colon(Y^\circ,B^\circ)\to[\mathrm{Cone}(X)\times T^\circ/\mathbb{G}_{m,T^\circ}]$ the family of K-polystable log Fano quasimaps corresponding to $h$.
Let $\mathscr{L}^\circ$ be the line bundle on $Y^\circ$ associated with $q^\circ$.
Since $q^\circ_{\overline{\eta_T}}$ is non-constant by the construction of $h$, there exist $N+2$ sections $f_0,f_1,\ldots,f_{N+1}\in\mathrm{Sec}_{q^\circ}(\mathscr{L}^\circ)\setminus\{0\}$ such that $\mathrm{div}(f_{i,\overline{\eta_T}})\ne \mathrm{div}(f_{N+1,\overline{\eta_T}})$ and $f_0,\ldots,f_N$ correspond to a basis of $\mathbbm{k}^{N+1}$. 
By shrinking $T^\circ$ if necessary, we may assume that any $\mathrm{div}(f_i)$ or $\mathrm{div}(f_i+f_{N+1})$ does not contain $Y^\circ_{t}$ for any $t\in T^\circ$.
By applying Lemma \ref{lem--maximal-variation-from-qmaps-to-pairs}, the family of pairs $(Y^\circ,B^\circ+\sum_{j=0}^{N+1}\mathrm{div}(f_j)+\sum_{j=0}^{N}(\mathrm{div}(f_j+f_{N+1})))$ has maximal variation.
Furthermore, by choosing $f_j$ general enough, we may assume that the pair $$\left(Y^\circ_{\overline{\eta_T}},B^\circ_{\overline{\eta_T}}+\frac{u}{2(N+2)}\sum^{N+1}_{j=0}\mathrm{div}(f_{j,\overline{\eta_T}})+\frac{u}{2(N+1)}\sum^{N}_{j=0}(\mathrm{div}(f_{j,\overline{\eta_T}}+f_{N+1,\overline{\eta_T}}))\right)$$
is K-stable.
Indeed, since $q^\circ_{\overline{\eta_T}}$ is K-stable, this follows from Lemma \ref{lem--two--delta--invariants}.
Furthermore, by shrinking $T^\circ$ if necessary, we may assume that 
$$\pi^\circ\colon\left(Y^\circ,B^\circ+\frac{u}{2(N+2)}\sum^{N+1}_{j=0}\mathrm{div}(f_{j})+\frac{u}{2(N+1)}\sum^{N}_{j=0}(\mathrm{div}(f_{j}+f_{N+1}))\right)\to T^\circ$$
is a family of K-stable log Fano pairs with maximal variation (\cite[Theorem A]{BL}).
We note that all fibers of $\pi^\circ$ belongs to the collection of objects of $\mathcal{M}^{\mathrm{Kss}}_{1,2N(N+1)r,v}$.
    
    For any strongly movable curve $C\subset T$ (for the definition, see \cite[Definition 1.3]{BDPP}), we take the canonical morphism $\varphi\colon \tilde{C}\to T$ from the normalization $\tilde{C}$ of $C$.
    By the definition of movable curves, it is easy to see that there exist a projective morphism of smooth quasi-projective varieties $\gamma\colon U\to V$ such that every geometric fiber of $r$ is a connected smooth curve and a dominant morphism $\mu\colon U\to T$ such that $\mu|{U_{p}}\colon U_p\to T$ coincides with $\varphi\colon \tilde{C}\to T$ for some closed point $p\in V$.
    Taking hyperplane sections of $V$, we may assume that $\mu$ is generically finite.
By Lemma \ref{lem--good-moduli-lift} and \cite[Theorem A.8]{AHLH} applied to $\mathcal{M}^{\mathrm{Kss}}_{1,2N(N+1)r,v}$, taking a non-empty open subset $W\subset V$ and a finite cover $\beta\colon \tilde{U}\to U\times_VW$ such that $\tilde{U}\to W$ is smooth with connected fibers, we may assume that there exists a family
\[
\tilde{\pi}\colon (\tilde{Y},\tilde{B})\to \tilde{U}
\]
of K-polystable log Fano pairs such that
\begin{align*}
&(\tilde{Y},\tilde{B})|_{(\mu\circ\beta)^{-1}(T^\circ)}\cong \\
&\left(Y^\circ,B^\circ+\frac{u}{2(N+2)}\sum^{N+1}_{j=0}\mathrm{div}(f_{j})+\frac{u}{2(N+1)}\sum^{N}_{j=0}(\mathrm{div}(f_{j}+f_{N+1}))\right)\times_{T^\circ}(\mu\circ\beta)^{-1}(T^\circ).
\end{align*}
Here, we note that $W$ might not contain $p$ any longer.
    
    On the other hand, let $H$ be an ample line bundle on $T$.
    Take a sufficiently large and divisible $M\in\mathbb{Z}_{>0}$ and a positive integer $d\in\mathbb{Z}_{>0}$ such that the following conditions hold (cf.~\cite[Definition 6.1]{XZ1}) for any scheme $S'$ and $f'\colon (Y',B')\to S'\in \mathcal{M}^{\mathrm{Kss}}_{1,2N(N+1)r,v}(S')$.
\begin{enumerate}
\item $L':=M(-K_{Y'/S'}-B')$ is relatively $\pi^\circ$-very ample line bundle,
\item if $D':=\mathrm{Supp}(B'),$
then $H^j(Y'_s,\mathcal{O}_{Y'_s}(mL'_s))=H^j(D'_s,\mathcal{O}_{D'_s}(mL_s'|_{D'_s}))=0$ for any $s\in S'$, $m>0$ and $j>0$,
\item for any closed point $s\in S'$, the embeddings of $D'_s$ and $Y'_s$ into $\mathbb{P}^{h^0(Y'_s,\mathcal{O}_{Y'_s}(L'_s))-1}$ defined by $L'_s$ are cut out by homogeneous polynomials of at most degree $d$ at least set theoretically,
\item for any closed point $s\in S'$, $\mathrm{Sym}^dH^0(Y'_s,\mathcal{O}_{Y'_s}(L'_s))\to H^0(Y'_s,\mathcal{O}_{Y'_s}(dL'_s))$ and $H^0(Y'_s,\mathcal{O}_{Y'_s}(dL'_s))\to H^0(D'_s,\mathcal{O}_{D'_s}(dL'_s|_{D'_s}))$ are surjective.
\end{enumerate}
Indeed, we can take such $M$ and $d$ by the fact that $\mathcal{M}^{\mathrm{Kss}}_{1,2N(N+1)r,v}$ is of finite type over $\mathbbm{k}$ (cf.~Theorem \ref{thm--K-moduli--of--log--Fano}).
We take $r'\in\mathbb{Z}_{>0}$ such that $r'\lambda_{\mathrm
{CM},f'}$ is a line bundle for any $f'$.
By \cite[Theorem 7.2]{XZ1}, there exists $c_0>0$ depending only on the choice of $H$, $r'$ and $d$ but independent from the choice of $\tilde{C}$ and $U$ such that 
\[
\mathrm{deg}(\lambda_{\mathrm
{CM},\tilde{\pi},(\tilde{Y},\tilde{B})}|_{\tilde{U}_{p'}})\ge c_0\mathrm{deg}(\beta^*\mu^*H|_{\tilde{U}_{p'}})
\] 
for any general point $p'\in W$.

In this paragraph, we show the following inequality 
\begin{equation*}\label{eq--positivity-of-CM-quasimaps}
\mathrm{deg}(\beta^*\mu^*g^*\Lambda^{\mathrm{qmaps}}_{\mathrm
{CM}}|_{\tilde{U}_{p'}})\ge\mathrm{deg}(\lambda_{\mathrm
{CM},\tilde{\pi},(\tilde{Y},\tilde{B})}|_{\tilde{U}_{p'}})
\end{equation*}
for any general point $p'\in W$.

By \cite[Theorem A.8]{AHLH}, after possibly replacing $\tilde{U}_{p'}$ with its finite cover, we may assume that there exist 
    a family of K-polystable log Fano quasimaps
    $$
    q_2\colon (Y_2,B_2)\to [\mathrm{Cone}(X)\times \tilde{U}_{p'}/\mathbb{G}_{m,\tilde{U}_{p'}}]
    $$
    such that $q_2|_{(\mu\circ\beta)^{-1}(T^\circ)\cap \tilde{U}_{p'}}$ coincides with $q^\circ_{(\mu\circ\beta)^{-1}(T^\circ)\cap \tilde{U}_{p'}}$.
Then, Lemma \ref{lem--CM--minimization} shows the desired inequality.
Indeed, the CM line bundle of $q_2$ coincides with $\beta^*\mu^*g^*\Lambda^{\mathrm{qmaps}}_{\mathrm
{CM}}|_{\tilde{U}_{p'}}$ in this case.

Finally, since $\mathrm{deg}_{U_{p'}}(\mu^*g^*\Lambda^{\mathrm{qmaps}}_{\mathrm{CM}}|_{U_{p'}})=\mathrm{deg}_{\tilde{C}}(\mu^*g^*\Lambda^{\mathrm{qmaps}}_{\mathrm{CM}}|_{\tilde{C}})$ and $\mathrm{deg}_{U_{p'}}(\mu^*H|_{U_{p'}})=\mathrm{deg}_{\tilde{C}}(\mu^*H|_{\tilde{C}})$, we have the inequality
\[
\mathrm{deg}_{\tilde{C}}(\mu^*g^*\Lambda^{\mathrm{qmaps}}_{\mathrm{CM}}|_{\tilde{C}})\ge c_0\cdot \mathrm{deg}_{\tilde{C}}(\mu^*H|_{\tilde{C}}).
\]
This inequality holds for any strongly movable curve $C\subset T$ and we see that $g^*\Lambda^{\mathrm{qmaps}}_{\mathrm{CM}}$ is big by \cite[Theorem 2.2]{BDPP} since $c_0$ does not depend on the choice of $C$.
\end{proof}

\begin{proof}[Proof of Proposition \ref{prop--key--to--positivity}]
    This immediately follows from Lemmas \ref{lem--bigness-for-const-maps} and \ref{lem--bigness-for-nonconstant-maps}.
\end{proof}

\begin{proof}[Proof of Theorem \ref{thm--projectivity}]
    This immediately follows from Proposition \ref{prop--key--to--positivity} and \cite[Theorem 3.11]{kollar-moduli-stable-surface-proj}.
\end{proof}
\subsection{K-moduli theory for an arbitrary weight}
As we saw, we established the K-moduli theory for log Fano quasimaps when the weight is sufficiently small.
From now, we extend this theory to an arbitrary weight case.

Fix $v,u\in\mathbb{Q}_{>0}$ and $r,m\in\mathbb{Z}_{>0}$.
Take an arbitrary $l\in\mathbb{Z}_{>0}$ and consider the $l$-th Veronese embedding $v_l\colon \mathbb{P}^N\hookrightarrow\mathbb{P}^{N'}$, where $N':=\binom{N+l}{l}-1$. 
Let $\mathrm{Cone}'(X)\subset\mathbb{A}^{N'+1}$ be the affine cone with respect to the embedding $v_l\circ\iota\colon X\hookrightarrow \mathbb{P}^{N'}$.
By construction, there is a canonical morphism $v'_l\colon\mathrm{Cone}(X)\to \mathrm{Cone}'(X)$.
    Note that $v'_l$ is not $\mathbb{G}_m$-equivariant in the natural way but if we consider the composition of the natural action of $\mathbb{G}_{m}$ on $\mathrm{Cone}'(X)$ and 
    \[
    \mathbb{G}_{m}\ni t\mapsto t^l\in\mathbb{G}_m
    \]
   instead of the canonical action, then $v'_l$ is $\mathbb{G}_m$-equivariant. 
Therefore, we obtain the canonical morphism of stacks $[v_l]\colon[\mathrm{Cone}(X)/\mathbb{G}_m]\to [\mathrm{Cone}'(X)/\mathbb{G}_m]$ induced by $v'_l$, where $[\mathrm{Cone}'(X)/\mathbb{G}_m]$ is the quotient stack with respect to the canonical action.
By using this morphism, we can obtain the natural morphism of stacks
\[
\tilde{v}_l\colon\mathcal{M}^{\mathrm{Kss,qmaps}}_{m,r,u,v,\iota}\to\mathcal{M}^{\mathrm{Kss,qmaps}}_{lm,r,\frac{u}{l},v,v_l\circ\iota},
\]
which maps an arbitrary object $q\colon (\mathcal{C},\mathcal{B})\to [\mathrm{Cone}(X)\times S/\mathbb{G}_{m,S}]$ of $\mathcal{M}^{\mathrm{Kss,qmaps}}_{m,r,u,v,\iota}(S)$ to $[v_l]\circ q\colon (\mathcal{C},\mathcal{B})\to [\mathrm{Cone}(X)\times S/\mathbb{G}_{m,S}]$ as an object of $\mathcal{M}^{\mathrm{Kss,qmaps}}_{lm,r,\frac{u}{l},v,v_l\circ\iota}(S)$.
Indeed, if $\mathscr{L}$ is the line bundle associated with $q$, then we can obtain a family of quasimaps $[v_l]\circ q$ such that $\mathscr{L}^{\otimes l}$ is associated with $[v_l]\circ q$. 
Moreover, we can see that the fixed part of $[v_l]\circ q$ coincides with the fixed part of $q$.
By this, it is easy to see that $q$ is K-(semi)stable if and only if so is $[v_l]\circ q$ (cf.~Remark \ref{rem--K-stability--coincidence}).
Therefore, we can construct $\tilde{v}_l$ as claimed.
Similarly, we can easily deduce that $q$ is K-polystable if and only if so is $[v_l]\circ q$ since we know that if $q$ is K-polystable but not K-stable then $q$ is constant.

\begin{lem}\label{lem--stack--closed--immersion}
Fix 
$v,u\in\mathbb{Q}_{>0}$ and $r,m\in\mathbb{Z}_{>0}$.
Take an arbitrary $l\in\mathbb{Z}_{>0}$ and consider the $l$-th Veronese embedding $v_l\colon \mathbb{P}^N\hookrightarrow\mathbb{P}^{N'}$, where $N':=\binom{N+l}{l}-1$.

Then, $\tilde{v}_l\colon\mathcal{M}^{\mathrm{Kss,qmaps}}_{m,r,u,v,\iota}\to\mathcal{M}^{\mathrm{Kss,qmaps}}_{lm,r,\frac{u}{l},v,v_l\circ\iota}$ defined as above is a closed immersion.
\end{lem}

\begin{proof}
    Let $q'\colon (\mathcal{C},\mathcal{B})\to [\mathrm{Cone}'(X)\times S/\mathbb{G}_{m,S}]$ be an object of $\mathcal{M}^{\mathrm{Kss,qmaps}}_{lm,r,\frac{u}{l},v,v_l\circ\iota}(S)$, where $\mathrm{Cone}'(X)\subset\mathbb{A}^{N'+1}$ is the affine cone with respect to the embedding $v_l\circ\iota\colon X\hookrightarrow \mathbb{P}^{N'}$.
To show this lemma, it suffices to show in the above situation that there exists a closed subscheme $S'\subset S$ such that for any morphism $g\colon T\to S$ from a scheme, then $g$ factors through $S'$ if and only if there exists a unique object $q\in \mathcal{M}^{\mathrm{Kss,qmaps}}_{m,r,u,v,\iota}(T)$ such that $q'_T\cong\tilde{v}_l(q)$.
We note that if $S'$ exists, then such $S'$ is unique by the property as we stated above.
From now, we will show the existence of such $S'$.
    Let $\mathscr{L}'$ be the line bundle associated with $q'$.
    Take an arbitrary point $s_0\in S$.
    Let $e_0,\ldots,e_{N}$ be the canonical coordinates of $\mathbb{A}^{N+1}_{\mathbbm{k}}$.
Suppose that there exists a suitable choice of coordinates $e'_0,\ldots,e'_{N'}$ of $\mathbb{A}^{N'+1}_{\mathbbm{k}}$ such that $e'_0,\ldots,e'_{N}$ are the images of $e_0^l,\ldots, e_N^l$ via the Veronese embedding.
Since the assertion is local, we may shrink $S$ and assume that $\mathcal{C}\cong\mathbb{P}^1_S$.
Let $f'_0,\ldots,f'_{N'}$ be the canonical sections corresponding to $e'_0,\ldots,e'_{N'}$.
Then, we see that for any $s\in S$, some $f'_i$ for $0\le i\le N$ satisfies that $f'_{i,s}\not=0$ since $q'$ is a quasimap to $[\mathrm{Cone}'(X)\times S/\mathbb{G}_{m,S}]$ and $X\subset \mathbb{P}^{N'}$ is contained in the image of the Veronese embedding $v_l$. 
We may assume that $f'_{0,s}\not=0$ for any $s\in S$ by shrinking $S$ and changing the numbering.
By $\mathrm{div}(f'_{0,s})$, we can take the corresponding morphism $\mu\colon S\to \mathbf{Div}^{lm}_{\mathcal{C}/S}$, where $\mathbf{Div}^{lm}_{\mathcal{C}/S}$ denotes the open and closed subscheme of $\mathbf{Div}_{\mathcal{C}/S}$ parametrizing relative Cartier divisors of degree $lm$.

Here, we note that $\mathbf{Div}^{d}_{\mathcal{C}/S}$ is isomorphic to the $d$-th symmetric product of $\mathbb{P}^1_S$ over $S$ for any integer $d>0$ (cf.~\cite[Remark 9.3.9]{FGA}).
Now, we will show that the natural morphism $\xi\colon\mathbf{Div}^{m}_{\mathcal{C}/S}\to\mathbf{Div}^{lm}_{\mathcal{C}/S}$, which is induced by multiplying $l$ to divisors, is a closed immersion in this case.
Since $\mathbf{Div}^{m}_{\mathcal{C}/S}$ and $\mathbf{Div}^{lm}_{\mathcal{C}/S}$ are projective over $S$, we may assume that $S=\mathrm{Spec}\,\mathbbm{k}$ by \cite[Proposition 12.93]{gortz-wedhorn}.
It is well-known that the $d$-th symmetric product $(\mathbb{P}^1)^{(d)}\cong\mathbb{P}^d$ as \cite[Example 7.1.3 (1)]{FGA}.
Fix a natural open immersion $\mathbb{A}^1\hookrightarrow \mathbb{P}^1$.
For any point $p\in \mathbf{Div}^{m}_{\mathbb{P}^1/\mathbbm{k}}\cong\mathbb{P}^m$, by changing the projective coordinate of $\mathbb{P}^1$, we may assume that $p\in (\mathbb{A}^1)^{(m)}\subset \mathbf{Div}^{m}_{\mathbb{P}^1/\mathbbm{k}}$.
Let $\pi_d\colon \mathbb{A}^d\to (\mathbb{A}^1)^{(d)}\cong \mathbb{A}^d/\mathfrak{S}_d$ be the canonical morphism of the geometric quotient, where $\mathfrak{S}_d$ is the $d$-th symmetric group for any $d\in\mathbb{Z}_{>0}$.
Then, $\xi|_{(\mathbb{A}^1)^{(m)}}$ is induced by the composition of a morphism $\mathbb{A}^m\to \mathbb{A}^{lm}$ that is the $m$-th product of the diagonal morphism $\Delta\colon \mathbb{A}^1\to \mathbb{A}^l$ and the canonical morphism $\pi_{lm}\colon\mathbb{A}^{lm}\to (\mathbb{A}^1)^{(lm)}$.
It is easy to check that $\xi|_{(\mathbb{A}^1)^{(m)}}$ is a closed immersion.
Indeed, let $\xi^\sharp\colon\mathbbm{k}[x_1,\ldots,x_{ml}]^{(\mathfrak{S}_{ml})}\to \mathbbm{k}[y_1,\ldots,y_{m}]^{(\mathfrak{S}_{m})}$ be the morphism of invariant rings corresponding to $\xi|_{(\mathbb{A}^1)^{(m)}}$.
Then, $\mathrm{Im}(\xi^\sharp)$ contains $y_1+\ldots +y_m$.
By induction, we can show that $\mathrm{Im}(\xi^\sharp)$ contains the elementary symmetric polynomials of degree $1\le k\le m$.
Indeed, we may assume that $\mathrm{Im}(\xi^\sharp)$ contains all elementary symmetric polynomials of degree at most $k$.
Then, $(y_1+\ldots +y_m)^{k+1}-(y_1^{k+1}+\ldots +y_m^{k+1})$ is generated by elementary symmetric polynomials of degree at most $k+1$.
Thus, it is easy to see that $\mathrm{Im}(\xi^\sharp)$ also contains the elementary symmetric polynomial of degree $k+1$.
Therefore, $\xi^\sharp$ is surjective.
Since $\xi$ is proper and injective, we see that $\xi$ is a closed immersion.

Finally, set $S':=\mu^{-1}(\xi(\mathbf{Div}_{\mathcal{C}/S}^m))$ and we will check that this $S'$ is the desired closed subscheme.
Indeed, there exists a unique relative Cartier divisor $\mathscr{D}$ over $S'$ such that $l\mathscr{D}=\mathrm{div}(f'_0)_{S'}$.
Let $\mathscr{L}=\mathcal{O}_{\mathcal{C}}(\mathscr{D})$.
Then, we have that $\mathscr{L}^{\otimes l}\sim\mathscr{L}'$.
By $H^1(\mathcal{C}_s,\mathcal{O}_{\mathcal{C}_s}(\mathscr{D}_s))=0$ for any $s\in S$, we can choose a section $f_0\in H^0(\mathcal{C},\mathscr{L})$ by shrinking $S$ if necessary.
Then, there exists a section $c\in H^0(S,\mathcal{O}_{S}^{\times})$ such that $f_0^l=cf'_0$.
Let $g'_{ij}\in\mathrm{Sec}_{q'}(\mathscr{L})$ be the sections corresponding to the sections of $H^0(\mathbb{P}^{N'},\mathcal{O}_{\mathbb{P}^{N'}}(1))$, which is mapped to $e_0^{l-j}\cdot e^j_{i}$ under the Veronese embedding $v_l$.
Then, we can set $f_i:=\frac{g'_{i1}}{f'_0}f_0$ as well-defined sections in $H^0(\mathcal{C},\mathscr{L})$.
Indeed, $\mathrm{div}\left(\frac{(g'_{i1})^l}{(f'_0)^l}\right)+\mathrm{div}(f'_0)$ is an effective relative Cartier divisor linearly equivalent to $\mathscr{L}'$ since 
\[
\frac{g'_{0j}}{g'_{1j}}=\frac{g'_{1j}}{g'_{2j}}=\ldots=\frac{g'_{l-1j}}{g'_{lj}}
\]
by the property of $v_l$, where $g'_{0j}=f'_0$.
By using the functions $f_0,\ldots, f_{N}$, we can define the quasimap structure $q\colon \mathcal{C}_{S'}\to [\mathrm{Cone}(X)\times S'/\mathbb{G}_{m,S'}]$.
Furthermore, we see that $q\colon (\mathcal{C},\mathcal{B})\to [\mathrm{Cone}(X)\times S'/\mathbb{G}_{m,S'}]$ is a family of K-semistable log Fano quasimaps and $q'_{S'}=[v_l]\circ q$.
It is easy to check the uniqueness of such $q$.
On the other hand, let $g\colon T\to S$ be a morphism from a scheme such that there exists an object $q\in \mathcal{M}^{\mathrm{Kss,qmaps}}_{m,r,u,v,\iota}(T)$ such that $q'_T\cong\tilde{v}_l(q)$ up to isomorphism.
By the property of $q$, if we let $\mathscr{L}$ be the associated line bundle with $q$, then there exists a section $f\in\mathrm{Sec}_q(\mathscr{L})$ such that $l\mathrm{div}(f)=\mathrm{div}(f'_0)$.
Then, we see that $g$ factors through $S'$ by the property of $S'$.
Thus, we obtain the assertion.
\end{proof}

\begin{thm}\label{thm--K-moduli-of-log-Fano-quasimaps}
Fix a closed embedding of a projective normal variety $\iota\colon X\subset \mathbb{P}^N$, $v,u\in\mathbb{Q}_{>0}$ and $r,m\in\mathbb{Z}_{>0}$. 
Then, $\mathcal{M}^{\mathrm{Kss,qmaps}}_{m,r,u,v,\iota}$ is an Artin stack of finite type over $\mathbbm{k}$ with a good moduli space $M^{\mathrm{Kps,qmaps}}_{m,r,u,v,\iota}$ that is projective and the set of $\mathbbm{k}$-valued points of $M^{\mathrm{Kps,qmaps}}_{m,r,u,v,\iota}$ naturally corresponds to the set of isomorphic classes of the K-polystable quasimaps bijectively.  
Furthermore, $\Lambda^{\mathrm{qmaps}}_{\mathrm{CM}}$ is ample on $M^{\mathrm{Kps,qmaps}}_{m,r,u,v,\iota}$.
\end{thm}

\begin{proof}
    If $u<1-\frac{v}{2}$, the assertion follows from Theorems \ref{thm--quot--stack}, \ref{thm--s-complete}, \ref{thm--properness} and \ref{thm--projectivity}.
We deal with the general case.
By Lemma \ref{lem--stack--closed--immersion}, take a sufficiently large integer $l\in\mathbb{Z}_{>0}$ such that  $\tilde{v}_l\colon\mathcal{M}^{\mathrm{Kss,qmaps}}_{m,r,u,v,\iota}\to\mathcal{M}^{\mathrm{Kss,qmaps}}_{lm,r,\frac{u}{l},v,v_l\circ\iota}$ is a closed immersion and $\frac{u}{l}<1-\frac{v}{2}$.
Then \cite[Lemma 4.14]{alper} shows the assertion since $\tilde{v}_l^*\lambda_{\mathrm{CM}}^{\mathrm{qmaps}}$ coincides with the CM line bundle defined on $\mathcal{M}^{\mathrm{Kss,qmaps}}_{m,r,u,v,\iota}$.
\end{proof}

Finally, we remark that $\mathcal{M}^{\mathrm{Kst,qmaps}}_{m,r,u,v,\iota}$ is a separated Deligne--Mumford stack.
Indeed, to see this, we may assume that $u<1-\frac{v}{2}$ by Lemma \ref{lem--stack--closed--immersion}.
In this case, it is easy to see that the natural morphism $\mathcal{M}^{\mathrm{Kst,qmaps}}_{m,r,u,v,\iota}\to M^{\mathrm{Kps,qmaps}}_{m,r,u,v,\iota}$ is $\mathsf{S}$-complete by Lemmas \ref{lem--canonical-qm--str}, \ref{lem--DF=0-case} and \cite[Proposition 3.43 (3)]{AHLH}.
Thus, Corollary \ref{cor--moduli--kst--quasi--maps} and \cite[Proposition 3.46]{AHLH} show that $\mathcal{M}^{\mathrm{Kst,qmaps}}_{m,r,u,v,\iota}$ is separated.

\section{Morphism between two moduli}
\label{sec4}
Throughout this section, we work over $\mathbb{C}$.
Fix $d\in\mathbb{Z}_{>0}$ and $v\in\mathbb{Z}_{>0}$. 
We consider $\mathcal{M}_{d-1,v}^{\mathrm{klt},\mathrm{CY}}$, $\mathcal{M}_{d-1,v}^{\mathrm{Ab}}$, and $\mathcal{M}_{d-1,v}^{\mathrm{symp}}$.
Note that $\mathcal{M}_{d-1,v}^{\mathrm{Ab}}\sqcup\mathcal{M}_{d-1,v}^{\mathrm{symp}}$ is an open and closed smooth substack of $\mathcal{M}_{d-1,v}^{\mathrm{klt},\mathrm{CY}}$ (cf.~Lemma \ref{lem--deformation--invariant--abelian} and Corollary \ref{cor--hk--moduli--normal}).
Let $\mathcal{V}$ be an arbitrary irreducible component and $V$ the coarse moduli space.
We have seen that $V$ admits the Baily--Borel compactification $\overline{V}^{\mathrm{BB}}$. 

We consider a substack $\mathcal{M}_{d,v,u,r,V}^{\mathrm{CYFib}}\subset \mathcal{M}_{d,v,u,r}$ such that for any scheme $S$, an object  $f\colon(\mathcal{X},\mathcal{A})\to \mathcal{C}\in\mathcal{M}_{d,v,u,r}(S)$ is an object of $\mathcal{M}_{d,v,u,r,V}^{\mathrm{CYFib}}(S)$ if and only if for any point $s\in S$ with the generic point $\eta$ of $\mathcal{C}_s$, we have $(\mathcal{X}_{\overline{\eta}},\mathcal{L}_{\overline{\eta}})\in\mathcal{V}(\mathrm{Spec}\,\overline{\kappa(\eta)})$. 
Then $\mathcal{M}_{d,v,u,r,V}^{\mathrm{CYFib}}$ is an open and closed substack of $\mathcal{M}_{d,v,u,r}$ since $\mathcal{V}$ is a connected component of $\mathcal{M}_{d-1,v}^{\mathrm{klt,CY}}$. 
Since $\mathcal{M}_{d,v,u,r,V}^{\mathrm{CYFib}}$ is a separated Deligne--Mumford stack of finite type, there exists a coarse moduli space $M_{d,v,u,r,V}^{\mathrm{CYFib}}$. 
Let  \[\nu\colon(\mathcal{M}_{d,v,u,V}^{\mathrm{CYFib}})^{\nu}\to\mathcal{M}_{d,v,u,r,V}^{\mathrm{CYFib}}\] be the normalization.
We omit $r$ here because the normalization of $\mathcal{M}_{d,v,u,r,V}^{\mathrm{CYFib}}$ does not depend on the choice of $r$ as we noted in Definition \ref{defn--HH-moduli}.
Let $(M_{d,v,u,V}^{\mathrm{CYFib}})^{\nu}$ denote the coarse moduli space of $(\mathcal{M}_{d,v,u,V}^{\mathrm{CYFib}})^{\nu}$.

In this section, we discuss the positivity of the CM line bundle $\Lambda_{\mathrm{CM},t}|_{M_{d,v,u,r,V}^{\mathrm{CYFib}}}$ for any sufficiently large $t>0$ in the situation of Definition \ref{defn--CM--limit}.

\begin{setup}\label{ass--Hodge--compactification}
Set $k\in\mathbb{Z}_{>0}$ such that $k$ is the smallest positive integer among $l$ such that there exist finitely many integral automorphic forms of weight $l$ that induce a closed immersion $V\hookrightarrow \mathbb{P}^{L}$ for some $L\in\mathbb{Z}_{>0}$.
If $\mathcal{V}$ generically parametrizes irreducible holomorphic symplectic manifolds with the second Betti number $b_2\ge5$, then we set $g:=b_2-3$ and $d':=\frac{d-1}{2}$.
Otherwise, we set $g:=d$ and $d'=1$.
Note that $\Lambda_{\mathrm{Hodge}}^{\otimes gk}$ (see Definition \ref{defn--Hodge-Q-line-bundle}) is a very ample line bundle on $\overline{V}^{\mathrm{BB}}$ by Theorem \ref{thm--Fujino--moduli}. 
  \end{setup}
  
  We will use the notation in Setup \ref{ass--Hodge--compactification} throughout this section.
  Theorem \ref{thm--Fujino's--period--mapping--theory} is necessary to correspond a family of uniformly adiabatically K-stable klt--trivial fibrations over curves to a family of quasimaps.

\begin{thm}\label{thm--Fujino's--period--mapping--theory}
    We put $l=gkl_0$, where $l_0$ as in Lemma \ref{lem--stein--factorizaton}. 
    Note that $l$ is a positive integer depending only on $d$, $v$, $u$, $r$, and $V$. 
    Then the the following holds.
    Let $S$ be a smooth affine variety and $f\colon (\mathcal{X},\mathcal{A})\to\mathcal{P}$ a family belonging to $\mathcal{M}_{d,v,u,r,V}^{\mathrm{CYFib}}(S)$ such that $\mathcal{A}$ is represented by a line bundle on $\mathcal{X}$.
    Let $\mathcal{P}_{\mathrm{klt}}$ be the largest open subset of $\mathcal{P}$ such that all geometric fibers of $f$ over $\mathcal{P}_{\mathrm{klt}}$ are klt. 
    Let $\mu^\circ\colon \mathcal{P}_{\mathrm{klt}}\to V$ be the morphism induced by the family $f|_{f^{-1}(\mathcal{P}_{\mathrm{klt}})}$ of klt Calabi--Yau varieties. 
    Take an Ambro model $\rho\colon\widetilde{\mathcal{P}}\to \mathcal{P}$ with respect to $f$, and let $\widetilde{\mathcal{M}}$ and $\widetilde{\mathcal{B}}$ be the moduli $\mathbb{Q}$-divisor and the discriminant $\mathbb{Q}$-divisor on $\mathcal{P}$, respectively.
    Then, after replacing $\widetilde{\mathcal{P}}$ with a higher birational model suitably, we have:
    \begin{enumerate}
       \item \label{thm--Fujino's--period--mapping--theory-(1)} $l\widetilde{\mathcal{M}}$ and $l\widetilde{\mathcal{B}}$ are both $\mathbb{Z}$-divisors, and
        \item \label{thm--Fujino's--period--mapping--theory-(2)} putting $L:=h^0(\overline{V}^{\mathrm{BB}},\Lambda^{\otimes l}_{\mathrm{Hodge}})-1$, then there exist $L+1$ sections $$\widetilde{\varphi}_0,\ldots,\widetilde{\varphi}_L\in H^0(\widetilde{\mathcal{P}},\mathcal{O}_{\widetilde{\mathcal{P}}}(l\widetilde{\mathcal{M}}))$$ that generate $l\widetilde{\mathcal{M}}$ and define a morphism $\widetilde{\mu}\colon\widetilde{\mathcal{P}}\to \mathbb{P}^L$ such that $\widetilde{\mu}^*\mathcal{O}(1)\sim l\widetilde{\mathcal{M}}$, $\widetilde{\mu}$ factors $\iota\colon \overline{V}^{\mathrm{BB}}\hookrightarrow \mathbb{P}^L$ and this is equivalent to the composition of  $\rho$, $\mu^\circ$ and $\iota$ as rational maps.  
        Here, $\iota\colon \overline{V}^{\mathrm{BB}}\to \mathbb{P}^L$ is induced by the complete linear system $|\Lambda^{\otimes l}_{\mathrm{Hodge}}|$ $=|\iota^*\mathcal{O}_{\mathbb{P}^L}(d'l_{0})|$. 
    \end{enumerate}
\end{thm}

\begin{proof}
Since $S$ is a smooth affine variety and $f\colon (\mathcal{X},\mathcal{A})\to\mathcal{P}$ is in particular an object of $\mathcal{M}_{d,v,u,r}$, we see that $\mathcal{P}$ is a smooth quasi-projective variety and $K_{\mathcal{X}}$ is $\mathbb{Q}$-Cartier. 
By our assumption on $\mathcal{A}$, we may regard $\mathcal{A}$ as a Cartier divisor on $\mathcal{X}$. 
Let $\mathcal{P} \hookrightarrow \mathcal{P}^{c}$ be an open immersion to a normal projective variety $\mathcal{P}^{c}$. 
By applying Corollary \ref{cor--compactification-lc-trivial-fib-2} to $f\colon (\mathcal{X},0)\to\mathcal{P}$, $\mathcal{A}$, and $\mathcal{P} \hookrightarrow \mathcal{P}^{c}$, we get a diagram
$$
\xymatrix
{
\mathcal{X} \ar[d]_{f} \ar@{^{(}->}[r] & \mathcal{X}^{c} \ar[d]^{f^{c}} 
\\
\mathcal{P} \ar@{^{(}->}[r] & \mathcal{P}^{c}  
}
$$
and an $f^{c}$-ample $\mathbb{Q}$-divisor $\mathcal{A}^{c}$ on $\mathcal{X}^{c}$ such that $f^{c} \colon (\mathcal{X}^{c},0) \to \mathcal{P}^{c}$ is a klt-trivial fibration and $\mathcal{A}^{c}|_{\mathcal{X}}=\mathcal{A}$. 
Now we may apply Theorem \ref{thm--Fujino--moduli} to the polarized klt-trivial fibration $f^{c} \colon (\mathcal{X}^{c},\mathcal{A}^{c}) \to \mathcal{P}^{c}$ and a projective birational morphism $\widetilde{\rho} \colon \widetilde{\mathcal{P}}^{c} \to \mathcal{P}^{c}$ from an Ambro model $\widetilde{\mathcal{P}}^{c}$. 
By fixing $\widetilde{\mathcal{P}}^{c}$ and replacing $\widetilde{\mathcal{P}}$ if necessary, we may assume that $\widetilde{\mathcal{P}}$ is an open subscheme of $\widetilde{\mathcal{P}}^{c}$. 
By Theorem \ref{thm--Fujino--moduli} and Definition \ref{defn--Hodge-Q-line-bundle}, 
there exists a morphism $\widetilde{\mathcal{P}}^{c} \to  \overline{V}^{\mathrm{BB}}$, where $\overline{V}^{\mathrm{BB}}$ is as in Setup \ref{ass--Hodge--compactification}, such that 
\begin{equation*}
\mu^*\Lambda^{\otimes gk}_{\mathrm{Hodge}}=\mu^*\mathcal{O}_{\overline{V}^{\mathrm{BB}}}(d')\sim \mathcal{O}_{\widetilde{\mathcal{P}}^{c}}(gk\widetilde{\mathcal{M}}^{c}),
\end{equation*}
where $\widetilde{\mathcal{M}}^{c}$ is the trace of the moduli $\mathbb{Q}$-b-divisor on $\widetilde{\mathcal{P}}^{c}$. 
Put $L:=h^0(\overline{V}^{\mathrm{BB}},\Lambda^{\otimes l}_{\mathrm{Hodge}})-1$. 
By Definition \ref{defn--Hodge-Q-line-bundle} and embedding $\overline{V}^{\mathrm{BB}}$ into $\mathbb{P}^{L}$ appropriately, we obtain $L+1$ sections $$\widetilde{\varphi}_0,\ldots,\widetilde{\varphi}_L\in H^0(\widetilde{\mathcal{P}},\mathcal{O}_{\widetilde{\mathcal{P}}}(l\widetilde{\mathcal{M}}))$$ that generate $l\widetilde{\mathcal{M}}$ and define a morphism $\widetilde{\mu}\colon\widetilde{\mathcal{P}}\to \mathbb{P}^L$ such that the linear equivalence $\widetilde{\mu}^*\mathcal{O}_{\mathbb{P}^{L}}(1)\sim l\widetilde{\mathcal{M}}$ holds, $\widetilde{\mu}$ factors through the closed embedding $\iota\colon \overline{V}^{\mathrm{BB}}\hookrightarrow \mathbb{P}^L$, and $\widetilde{\mu}$ is equivalent to the composition $\iota \circ \mu^\circ \circ \rho \colon\widetilde{\mathcal{P}}\to \mathcal{P} \dashrightarrow \overline{V}^{\mathrm{BB}}\hookrightarrow \mathbb{P}^L$ as a rational map. 
This implies (\ref{thm--Fujino's--period--mapping--theory-(2)}) of Theorem \ref{thm--Fujino's--period--mapping--theory}. 

The above argument also implies that $l\widetilde{\mathcal{M}}$ is a $\mathbb{Z}$-divisor. 
By Lemma \ref{lem--stein--factorizaton}, we have $lK_{\mathcal{X}/\mathcal{P}} \sim f^{*}D$ for some Cartier divisor $D$ on $\mathcal{P}$. 
Then $lK_{\mathcal{X}_{\eta}}\sim0$, where $\mathcal{X}_{\eta}$ is the generic fiber of $f$. 
Note that $X_{\eta}$ is also the generic fiber of $f^{c}$. 
By construction of the moduli $\mathbb{Q}$-divisor (\cite{FM2}), we have
$l\widetilde{\mathcal{B}}=\rho^{*}D-l\widetilde{\mathcal{M}}$. 
Hence, $l\widetilde{\mathcal{B}}$ is also a $\mathbb{Z}$-divisor. 
From this argument, we see that (\ref{thm--Fujino's--period--mapping--theory-(1)}) of Theorem \ref{thm--Fujino's--period--mapping--theory} holds. 
We finish the proof.  
\end{proof}

We recall the definition of seminormal schemes and seminormalization.

\begin{defn}[{\cite[Tag 0EUL]{stacksproject-chap98}, \cite{GT}}]\label{def-seminormal}
Let $A$ be a reduced algebra of finite type over a field.
Then we say that $A$ is {\it seminormal} if for any $x,y\in A$ such that $x^2=y^3$, then there exists a unique element $a\in A$ such that $a^3=x$ and $a^2=y$.
As \cite[Tag 0EUR]{stacksproject-chap98}, there exists a unique ring $A'$ up to isomorphism that is integral over $A$ such that $\mathrm{Spec}\,A'\to \mathrm{Spec}\,A$ is a homeomorphism inducing isomorphisms of residue fields.
We note that an integral homeomorphism inducing isomorphisms of residue fields is a universally homeomorphism.
We say that $A'$ is the {\it seminormalization} of $A$.
Note that $A'$ is contained in the normalization of $A$ and hence finite over $A$.

By \cite[Theorems 1.6 and 4.1]{GT}, we can define the {\it seminormalization} morphism $\pi\colon \mathscr{X}^{\mathrm{sn}}\to \mathscr{X}$ for an arbitrary Artin stack $\mathscr{X}$ of finite type over a field as the unique morphism such that $\mathscr{X}^{\mathrm{sn}}$ is reduced and for any \'etale morphism $\mathrm{Spec}\,A\to (\mathscr{X})_{\mathrm{red}}$ from a reduced algebra of finite type over the field, $\mathscr{X}_{\mathrm{sn}}\times_{(\mathscr{X})_{\mathrm{red}}}\mathrm{Spec}\,A$ is isomorphic to the spectrum of the seminormalization of $A$. 
We say that $\mathscr{X}^{\mathrm{sn}}$ as above is the {\it seminormalization} of $\mathscr{X}$, and $\mathscr{X}$ is {\it seminormal} if $\pi$ is an isomorphism. 
\end{defn}

By applying Theorem \ref{thm--Fujino's--period--mapping--theory}, we can deduce the following technical core in this section.
We fix $\iota\colon \overline{V}^{\mathrm{BB}}\hookrightarrow\mathbb{P}^L$, which is induced by $|\Lambda_{\mathrm{Hodge}}^{\otimes l}|$.

\begin{thm}\label{thm--quasi--map--construction--moduli--map}
    Set $l:=gkl_{0}$ as in Theorem \ref{thm--Fujino's--period--mapping--theory}.
    Then, there exists a morphism of separated Deligne--Mumford stacks
    \[
    \alpha\colon(\mathcal{M}_{d,v,u,V}^{\mathrm{CYFib}})^{\mathrm{sn}}\to \mathcal{M}^{\mathrm{Kst.qmaps}}_{l(2-u),1,\frac{1}{l},u,\iota}.
    \]
\end{thm}

To show this, we have to show the following.

\begin{lem}\label{lemma--description--of--base--locus}
     Let $S$ be a smooth quasi-projective variety, and let $f\colon (\mathcal{X},\mathcal{A})\to\mathcal{P}$ be a family that belongs to $\mathcal{M}_{d,v,u,r,V}^{\mathrm{CYFib}}(S)$ such that $\mathcal{A}$ is represented by a line bundle.
    Take an arbitrary Ambro model $\rho\colon\tilde{\mathcal{P}}\to \mathcal{P}$ with respect to $f$, and let $\widetilde{\mathcal{M}}$ and $\widetilde{\mathcal{B}}$ be the moduli $\mathbb{Q}$-divisor and the discriminant $\mathbb{Q}$-divisor, respectively.
    Let $E$ be the $\rho$-exceptional divisor $\mathbb{Q}$-linearly equivalent to $K_{\widetilde{\mathcal{P}}}-\rho^*K_{\mathcal{P}}$. 
    Then, 
    \[
E+\widetilde{\mathcal{B}}+\widetilde{\mathcal{M}}=\rho^*(\mathcal{B}+\mathcal{M})
\]
    holds and $\widetilde{\mathcal{B}}+E$ is effective. 
\end{lem}

\begin{proof}
    We first note that $E$ is effective since $\mathcal{P}$ is smooth by the assumption that $S$ is smooth.
Let $\mathcal{B}:=\rho_*\widetilde{\mathcal{B}}$ and $\mathcal{M}:=\rho_*\widetilde{\mathcal{M}}$. 
Then, we have
\[
K_{\widetilde{\mathcal{P}}}+\widetilde{\mathcal{B}}+\widetilde{\mathcal{M}}=\rho^*(K_{\mathcal{P}}+\mathcal{B}+\mathcal{M}).
\]
By this formula, we also have that 
\[
E+\widetilde{\mathcal{B}}+\widetilde{\mathcal{M}}=\rho^*(\mathcal{B}+\mathcal{M}).
\]
Since $\widetilde{\mathcal{M}}$ is $\rho$-nef, we see that $\rho^*\mathcal{M}-\widetilde{\mathcal{M}}$ is effective and $\rho$-exceptional by \cite[Lemma 3.39]{KM}.
By the definition of the generalized log pair, $\mathcal{B}$ is also effective.
Thus, $E+\widetilde{\mathcal{B}}=\rho^*(\mathcal{B}+\mathcal{M})-\widetilde{\mathcal{M}}$ is effective.
\end{proof}

\begin{thm}\label{thm--quasi--map--const}
        Let $S$ be a quasi-projective normal variety and let $f\colon(\mathcal{X},\mathcal{A})\to\mathcal{P}$ be an object of $\mathcal{M}_{d,v,u,r,V}^{\mathrm{CYFib}}(S)$.
        Suppose that $\mathcal{A}$ is a line bundle.
Then, there exists the following object $q\colon \mathcal{P}\to [\mathrm{Cone}(\overline{V}^{\mathrm{BB}})\times S/\mathbb{G}_{m,S}]$ of $\mathcal{M}_{l(2-u),1,\frac{1}{l},u,\iota}^{\mathrm{Kst.qmaps}}(S)$ uniquely up to isomorphism satisfying the following. 
\begin{enumerate}
\item Let $\mathscr{L}$ be the associated line bundle on $\mathcal{P}$ with $q$. Then, $f^*\mathscr{L}\sim_S l(K_{\mathcal{X}/S}-f^*K_{\mathcal{P}/S})$.
\item For any geometric point $\bar{s}\in S$, let $B_{\bar{s}}$ be the fixed part of $q_{\bar{s}}$.
Then, $\frac{1}{l}B_{\bar{s}}$ is the discriminant $\mathbb{Q}$-divisor with respect to $f_{\bar{s}}$ and $q_{\bar{s}}$ defines the moduli map associated to $f_{\bar{s}}$ (cf.~Definition \ref{defn--associated--quasi--map--structure}).
\end{enumerate}
Furthermore, $f^*\mathscr{L}\sim l(K_{\mathcal{X}/S}-f^*K_{\mathcal{P}/S})$.
    \end{thm}
    \begin{proof}
        We first deal with the uniqueness of $q$.
        Assume that there exists another family of quasimaps $q'$ with the same property as $q$.
        Then, $q'$ and $q$ define two line bundles $\mathscr{L}'$ and $\mathscr{L}$.
        By the assumption, we see that $\mathscr{L}_{\bar{s}}\sim\mathscr{L}'_{\bar{s}}$ for any geometric point $\overline{s}\in S$.
Since it is enough to show that $q$ and $q'$ are isomorphic to each other Zariski locally on $S$ and all geometric fibers of $\mathcal{P}\to S$ are irreducible and smooth, we may assume that $\mathscr{L}\sim\mathscr{L}'$ by shrinking $S$.
Let $\varphi_0,\ldots,\varphi_L$ (resp.~$\varphi'_0,\ldots,\varphi'_L$) be the sections corresponding to the canonical coordinates of $\mathbb{A}^{L+1}$ of $\mathscr{L}$ (resp.~$\mathscr{L}'$).
By the assumption, we see that for any geometric point $\bar{s}\in S$, $\varphi_{0,\bar{s}},\ldots,\varphi_{L,\bar{s}}$ and $\varphi'_{0,\bar{s}},\ldots,\varphi'_{L,\bar{s}}$ define the same quasimap structures.
This is equivalent to that there exists $\alpha_{\bar{s}}\in\overline{\kappa(s)}^{\times}$ such that $\alpha_{\bar{s}}=\frac{\varphi_{j,\bar{s}}}{\varphi'_{j,\bar{s}}}$ for any $j=0,\ldots,L$ such that $\varphi_{j,\bar{s}}\cdot \varphi'_{j,\bar{s}}\ne0$.
This means that shrinking $S$ if necessary, we may assume that there exists $\alpha\in H^0(S,\mathcal{O}_S^{\times})$ such that $\varphi_i=\alpha \varphi'_i$ for any $i\in\{0,\ldots,L\}$.
Hence, $q$ and $q'$ are isomorphic families of quasimaps to each other.

Next, we deal with the existence of such $q$ as in the assertion.
To show this, we first show that we may replace $S$ with any log resolution of $S$ freely.
More precisely, we show the following claim.
\begin{claim*}
    Let $\varphi\colon S'\to S$ be an arbitrary projective birational morphism from a smooth variety.
    Suppose that there exists a family of quasimaps $q'\colon \mathcal{P}_{S'}\to [\mathrm{Cone}(\overline{V}^{\mathrm{BB}})\times S'/\mathbb{G}_{m,S'}]$ as the assertion for $f_{S'}\colon (\mathcal{X}\times_SS',\mathcal{A}_{S'})\to \mathcal{P}_{S'}$.
    Then, there exists a family of quasimaps $q$ as the assertion for $S$ such that $q_{S'}$ is isomorphic to $q'$.
\end{claim*}

\begin{proof}[Proof of Claim]   
Let $\mathscr{L}'$ be the associated line bundle on $\mathcal{P}_{S'}$ with $q'$ and $\varphi'_0,\ldots,\varphi'_{L}\in\mathrm{Sec}_{q'}(\mathscr{L}')$ the sections corresponding to the canonical basis.
Note that by uniqueness of $q$ as the assertion, it suffices to show the existence of $q$ locally over $S$.
    By the property of $q'$, we see that $f_{S'}^*\mathscr{L}'\sim_{S'}l(K_{\mathcal{X}_{S'}/S'}-f_{S'}^*K_{\mathcal{P}_{S'}/S'})$.
   Since $f\in \mathcal{M}_{d,v,u,r}(S)$, $\omega_{\mathcal{X}/S}^{[-l]}$ is a pullback of a certain relatively ample line bundle on $\mathcal{P}$ over $S$. 
    By the choice of $l$ and Lemma \ref{lem--stein--factorizaton}, there exists a line bundle $\mathscr{L}$ on $\mathcal{P}$ such that $f^*\mathscr{L}=l(K_{\mathcal{X}/S}-f^*K_{\mathcal{P}/S})$. 
    By Lemma \ref{lem--stein--factorizaton} and the assumption that $S'$ is smooth, we know that $f_{S'*}\mathcal{O}_{\mathcal{X}_{S'}}\cong \mathcal{O}_{\mathcal{P}_{S'}}$.
    We also have that $\varphi_{\mathcal{P}*}\mathcal{O}_{\mathcal{P}_{S'}}\cong\mathcal{O}_{\mathcal{P}}$ for $\varphi_{\mathcal{P}}\colon\mathcal{P}_{S'}\to \mathcal{P}$. 
    Thus, $(\varphi_{\mathcal{P}}\circ f_{S'})_*\mathcal{O}_{\mathcal{X}_{S'}}\cong \mathcal{O}_{\mathcal{P}}$ holds and $\mathscr{L}'\sim_{S'} \varphi_{\mathcal{P}}^*\mathscr{L}$ by \cite[Lemma 9.2.7]{FGA}.
    Here, we show that $\mathscr{L}'\sim_S\varphi_{\mathcal{P}}^*\mathscr{L}$.
    Indeed, let $\mathscr{M}$ be a line bundle on $S'$ such that $\mu'^*\mathscr{M}\sim\mathscr{L}'\otimes\varphi_{\mathcal{P}}^*\mathscr{L}^{\otimes -1}$, where $\mu'\colon \mathcal{P}_{S'}\to S'$ is the canonical morphism.
    Take an arbitrary closed point $s\in S$ and an arbitrary closed point $p\in \mathcal{P}_s$.
    Let $U\subset \mathcal{P}$ be an open affine neighborhood of $p$ such that $\mathscr{L}|_U\sim\mathcal{O}_U$ and $\mu'^*\mathscr{M}|_{\varphi_{\mathcal{P}}^{-1}(U)}\sim\mathscr{L}'|_{\varphi_{\mathcal{P}}^{-1}(U)}$.
    By shrinking $U$ and renumbering $i$ if necessary, we may assume that for any $s_1\in U$, there exists a point $p_1\in \varphi_{\mathcal{P}}^{-1}(s_1)$ such that $\varphi'_{0,p_1}\ne0$.
    Here, by the definition of $q'$, the family $q'|_{\varphi^{-1}(\varphi(\mu'(s_1)))}$ is isotrivial (see Definition \ref{defn--isotrivial} below).
    Therefore, it is easy to see that $\varphi'_0$ does not vanish at any point of $\varphi_{\mathcal{P}}^{-1}(U)$.
    This means that $\mu'^*\mathscr{M}\sim_{\mathcal{P}}0$.
    By this, it is not hard to see that $\mathscr{M}'\sim_S0$.
   Thus, we may assume that $\mathscr{L}'\sim\varphi_{\mathcal{P}}^*\mathscr{L}$ by shrinking $S$. 
    Via the identification $H^0(\mathcal{P}',\mathscr{L}')\cong H^0(\mathcal{P},\mathscr{L})$, we can consider $\varphi'_i$'s to be the sections in $H^0(\mathcal{P},\mathscr{L})$.
    Then, $\varphi'_i\in H^0(\mathcal{P},\mathscr{L})$ define a family of quasimaps $q$ such that $q_{S'}=q'$.
    By this property, it is easy to check that $q$ has the desired properties.
    We complete the proof. 
\end{proof}

\noindent{\it Proof of Theorem \ref{thm--quasi--map--const} continued.} 
Thus, replacing $S$ with a suitable log resolution $S'$ of $S$ if necessary, we may assume that there exist $\rho\colon\tilde{\mathcal{P}}\to\mathcal{P}$ and $L+1$ sections $\widetilde{\varphi_0},\ldots,\widetilde{\varphi_L}\in H^0(\tilde{\mathcal{P}},\mathcal{O}_{\widetilde{\mathcal{P}}}(l\widetilde{\mathcal{M}}))$ as in Theorem \ref{thm--Fujino's--period--mapping--theory}.
Here, $\widetilde{\mathcal{M}}$ and $\widetilde{\mathcal{B}}$ are the moduli $\mathbb{Q}$-divisor and the discriminant $\mathbb{Q}$-divisor on $\tilde{\mathcal{P}}$ respectively.
Let $E$ be the $\rho$-exceptional divisor $\mathbb{Q}$-linearly equivalent to $K_{\widetilde{\mathcal{P}}}-\rho^*K_{\mathcal{P}}$. 
By the choice of $l$ and Lemma \ref{lem--stein--factorizaton}, there exists a line bundle $\mathscr{L}$ on $\mathcal{P}$ such that $f^*\mathscr{L}=l(K_{\mathcal{X}/S}-f^*K_{\mathcal{P}/S})$.
In this section, we construct the desired family of quasimaps by using the $L+1$ sections.
By Lemma \ref{lemma--description--of--base--locus}, there exists the natural injection
\begin{equation}\tag{$\clubsuit$}\label{eq--injection-of--function}
H^0(\tilde{\mathcal{P}},\mathcal{O}_{\widetilde{\mathcal{P}}}(l\widetilde{\mathcal{M}}))\hookrightarrow H^0(\tilde{\mathcal{P}},\mathcal{O}_{\widetilde{\mathcal{P}}}(l(\widetilde{\mathcal{M}}+\widetilde{\mathcal{B}}+E)))\cong H^0(\mathcal{P},\mathscr{L}).
\end{equation}
Via this identification, we identify $\widetilde{\varphi_i}$ with a section $\varphi_i\in H^0(\mathcal{P},\mathscr{L})$ for any $i$.
We claim that $\varphi_i\in H^0(\mathcal{P},\mathscr{L})$ define a structure of family of quasimaps $q\colon \mathcal{P}\to[\mathrm{Cone}(\overline{V}^{\mathrm{BB}})\times S/\mathbb{G}_{m,S}]$.
Indeed, let $\mathcal{P}_{\mathrm{klt}}$ be the open subset of $\mathcal{P}$ such that any geometric fiber of $f$ over $\mathcal{P}_{\mathrm{klt}}$ is klt and $\mu^\circ\colon \mathcal{P}_{\mathrm{klt}}\to \overline{V}^{\mathrm{BB}}$ the morphism induced by the family $f|_{f^{-1}(\mathcal{P}_{\mathrm{klt}})}$ of klt Calabi--Yau varieties.
Let $D$ be a prime divisor on $S$.
We note that the canonical morphism $\mathcal{P}_{\mathrm{klt}}\to S$ is surjective and at any general point of $s\in D$, some $\varphi_{i}|_{\mathcal{P}_{\mathrm{klt},s}}$ does not vanish since $\mathcal{P}_{\mathrm{klt},s}$ is not contained in $\rho(\mathrm{Exc}(\rho))$.
This shows that $q$ is a family of quasimaps over $S$ in codimension one.
For any closed point $s\in S$, take the one point blow-up $\xi\colon \mathrm{Bl}_sS\to S$.
As Theorem \ref{thm--Fujino's--period--mapping--theory}, there exist $\rho'\colon \widetilde{\mathcal{P}}'\to\mathcal{P}_{\mathrm{Bl}_sS}$ and $L+1$ sections $\widetilde{\varphi_0}',\ldots,\widetilde{\varphi_L}'\in H^0(\tilde{\mathcal{P}}',\mathcal{O}_{\widetilde{\mathcal{P}}'}(l\widetilde{\mathcal{M}}'))$ corresponding to the canonical coordinate of $\mathbb{A}^{L+1}$.
Here, $\widetilde{\mathcal{M}}'$ and $\widetilde{\mathcal{B}}'$ are the moduli $\mathbb{Q}$-divisor and the discriminant $\mathbb{Q}$-divisor on $\tilde{\mathcal{P}}'$ respectively.
Let $E'$ be the $\rho'$-exceptional divisor $\mathbb{Q}$-linearly equivalent to $K_{\widetilde{\mathcal{P}}'}-\rho'^*K_{\mathcal{P}_{\mathrm{Bl}_sS}}$. 
Let $\widetilde{\mathcal{W}}$ be a smooth variety such that there exist birational projective morphisms $p_{\tilde{\mathcal{P}}}\colon \widetilde{\mathcal{W}}\to \tilde{\mathcal{P}}$ and $p_{\widetilde{\mathcal{P}}'}\colon \widetilde{\mathcal{W}}\to \widetilde{\mathcal{P}}'$.
Note that $p_{\widetilde{\mathcal{P}}'}^*\widetilde{\mathcal{M}}'\sim p_{\tilde{\mathcal{P}}}^*\widetilde{\mathcal{M}}$ by the definition of Ambro models.
Therefore, we can show that the natural injection
\begin{align*}
   H^0(\widetilde{\mathcal{P}}',\mathcal{O}_{\widetilde{\mathcal{P}}'}(l\widetilde{\mathcal{M}}'))\hookrightarrow H^0(\widetilde{\mathcal{P}}',\mathcal{O}_{\widetilde{\mathcal{P}}'}(l(\widetilde{\mathcal{M}}'+\widetilde{\mathcal{B}}'+E')))\cong H^0(\mathcal{P},\mathscr{L})
\end{align*}
obtained as \eqref{eq--injection-of--function} is naturally isomorphic to \eqref{eq--injection-of--function} and by this identification, $\widetilde{\varphi_i}'$ coincides with $\widetilde{\varphi_i}$ for any $i$.
By this fact, since $\widetilde{\varphi_i}'$ does not vanish over the generic point of $\xi^{-1}(s)$ for some $i$ as we have shown, $\widetilde{\varphi_i}_s$ does not vanish.
Thus, $q$ is a family of quasimaps over $S$.
Note that if $q$ satisfies conditions (1) and (2), then $f^*\mathscr{L}\sim l(K_{\mathcal{X}/S}-f^*K_{\mathcal{P}/S})$ holds by construction.

 It suffices to show that $q$ satisfies the condition (2) for any geometric point $\bar{s}\in S$.
 In this section, we reduce this to the case when $S$ is a smooth curve and the condition (2) holds for all but one closed points of $S$.
We first note that it is enough to check condition (2) only for closed points for $S$.
Indeed, it is easy to check that for any subvariety $W\subset S$, if an arbitrary closed point of $W$ satisfies condition (2), then condition (2) also holds for the geometric generic point of $W$ using the same argument as the proof of \cite[Lemma 4.1]{HH}. 
Indeed, by the proof of \cite[Lemma 4.1]{HH}, we see that there exists an \'etale morphism $W'\to W$ such that for  $\mathcal{P}_{W'}$, we can take an Ambro model $\rho\colon\widetilde{\mathcal{P}}\to \mathcal{P}_{W'}$ such that $(\widetilde{\mathcal{P}},\widetilde{\mathcal{M}}+\widetilde{\mathcal{B}}+\mathrm{Exc}(\rho))$ is relatively simple normal crossing over $W'$, where $\widetilde{\mathcal{M}}$ and $\widetilde{\mathcal{B}}$ are the moduli $\mathbb{Q}$-divisor and the discriminant $\mathbb{Q}$-divisor respectively.
We also note that for any general closed point $s\in S$, we can easily check that condition (2) holds using the same argument as the proof of \cite[Lemma 4.1]{HH}.
For any closed point $s\in S$,  we take a morphism $g\colon C\to S$ from a smooth affine curve such that a general point of $C$ is mapped to a general point of $S$ and a closed point $0\in C$ is mapped to $s$ under $g$.
Consider $g^*q$ and $q_C$.
Here, the latter is the quasimap constructed for $C$ in the same way as $q$ for $S$.
We claim that $g^*q$ and $q_C$ are isomorphic.
Indeed, there exists a non-empty open subset $U\subset C$ such that $g^*q$ and $q_C$ are isomorphic over $\mathcal{P}_{\mathrm{klt}}\times_SC\cup \mathcal{P}\times_SU$.
Note that $$\mathrm{codim}_{\mathcal{P}\times_SC}(\mathcal{P}\times_SC\setminus(\mathcal{P}_{\mathrm{klt}}\times_SC\cup \mathcal{P}\times_SU))\ge2.$$ 
Thus, $g^*q$ and $q_C$ are canonically isomorphic to each other by Lemma \ref{lem--quasi--map--S_2}.
Therefore, replacing $S$ with $C$ and shrinking $C$ if necessary, we may assume that $C=S$ is a curve with a closed point $0\in C$ such that at least the condition (2) holds for any closed point of $C\setminus\{0\}$.

Finally, we deal with that under the assumption that that $S$ is a curve with a closed point $0\in S$ such that at least the condition (2) holds for any closed point of $S\setminus\{0\}$, 
$\varphi_0|_0,\ldots,\varphi_L|_0$ define the canonical quasimap structure associated with $(\mathcal{X}_0,\mathcal{A}_0)\to \mathbb{P}^1$ (cf.~Definition \ref{defn--associated--quasi--map--structure}).
Since $\mathcal{P}$ is a ruled surface over $S$, we may assume that $\mathcal{P}\cong\mathbb{P}^1\times S$ by shrinking $S$. 
 Take an Ambro model $\rho\colon\widetilde{\mathcal{P}}\to \mathbb{P}^1\times S$ of $\mathcal{X}\to\mathbb{P}^1\times S$, where $\rho\colon\widetilde{\mathcal{P}}\to \mathbb{P}^1\times S$ is a projective birational morphism from a smooth variety as in Theorem \ref{thm--Fujino's--period--mapping--theory}.
    Let $D\cong\mathbb{P}^1$ be the strict transform of $\mathbb{P}^1\times\{0\}$ in $\mathcal{P}$. 
    Note that $D$ is isomorphic to $\mathbb{P}^1\times\{0\}$ via $\rho$.
    Let $(\widetilde{\mathcal{X}},\widetilde{\mathcal{D}})\to\widetilde{\mathcal{P}}$ be a projective contraction of smooth varieties that is birational to $\mathcal{X}\to \mathbb{P}^1\times S$ such that $\mathcal{X}$ and $(\widetilde{\mathcal{X}},\widetilde{\mathcal{D}})$ are log crepant.
    Let $\widetilde{\mathcal{B}}$ and $\widetilde{\mathcal{M}}$ be the discriminant and the moduli $\mathbb{Q}$-divisors associated with $(\widetilde{\mathcal{X}},\widetilde{\mathcal{D}})\to\widetilde{\mathcal{P}}$ respectively. 
    We note that $E:=K_{\widetilde{\mathcal{P}}}-\rho^*K_{\mathbb{P}^1\times S}$ is $\rho$-exceptional and effective since $S\times \mathbb{P}^1$ is smooth. 
    By replacing $\rho$ with a further blow up, we may assume that $(\widetilde{\mathcal{P}},\widetilde{\mathcal{B}}+\widetilde{\mathcal{M}}+\widetilde{\mathcal{P}}_0 + \mathrm{Exc}(\rho))$ is simple normal crossing.
    Here, the sections $\varphi_0,\ldots,\varphi_L$ define the morphism $\mu^\circ\colon (\mathbb{P}^1\times S)_{\mathrm{klt}}\to V$ by the construction of $q$.
    Furthermore, if $\tilde{\mu}\colon \widetilde{\mathcal{P}}\to \overline{V}^{\mathrm{BB}}$ denotes the morphism that is equivalent to $\mu^\circ$ as a rational map, then via the identification 
    \[ H^0(\tilde{\mathcal{P}},\mathcal{O}_{\widetilde{\mathcal{P}}}(l\widetilde{\mathcal{M}}))\ni\rho^*\varphi\hookrightarrow  \varphi\in H^0(\mathcal{P},\mathscr{L}),
    \]
    $\rho^*\varphi_i$'s generate $l\widetilde{\mathcal{M}}$ and define $\tilde{\mu}$.
    Thus, if we regard $\rho^*\varphi_i$ as a section of $H^0(\widetilde{\mathcal{P}},\rho^*\mathscr{L})$, then $\rho^*\varphi_i$'s generate the ideal sheaf $\mathcal{O}_{\widetilde{\mathcal{P}}}(-l(\widetilde{\mathcal{B}}+E))$ (cf.~Lemma \ref{lemma--description--of--base--locus}).
    Let $(\widehat{\mathcal{X}},\widehat{\mathcal{D}})$ be the main component of $(\widetilde{\mathcal{X}},\widetilde{\mathcal{D}}+\widetilde{\mathcal{X}}_0-\widehat{\mathcal{X}})\times_{\widetilde{\mathcal{P}}}D$.
    Here, we may assume that $\widehat{\mathcal{X}}$ is a smooth divisor in $\widetilde{\mathcal{X}}$.
    By restricting $\tilde{\mu}$ to $D$, we have that $\widetilde{\mathcal{M}}|_D$ is the moduli $\mathbb{Q}$-divisor associated with the subklt--trivial fibration $(\widehat{\mathcal{X}},\widehat{\mathcal{D}})\to D$. 
    It is easy to see that $(\widehat{\mathcal{X}},\widehat{\mathcal{D}}+(\widetilde{\mathcal{X}_0}-\widehat{\mathcal{X}})|_{\widehat{\mathcal{X}}})$ is log crepant to $\mathcal{X}_0$ by the adjunction formula.
    Thus, $(D,(\widetilde{\mathcal{M}}+\widetilde{\mathcal{B}}+\widetilde{\mathcal{P}}_0-D)|_D)$ is isomorphic to the generalized log pair defined by $\mathcal{X}_0\to\mathbb{P}^1$.
    Since $\widetilde{\mathcal{M}}|_D$ is the moduli $\mathbb{Q}$-divisor, it is not hard to see that $(\widetilde{\mathcal{B}}+\widetilde{\mathcal{P}}_0-D)|_D$ is the discriminant divisor with respect to $\mathcal{X}_0\to\mathbb{P}^1$.
    Furthermore, we note that $E|_D=(\widetilde{\mathcal{P}}_0-D)|_D$ as Cartier divisors.
    Indeed, consider the following equation of log canonical divisors
    \[
    K_{\widetilde{\mathcal{P}}}+\widetilde{\mathcal{P}}_{0}=\rho^*(K_{\mathbb{P}^1\times S}+\mathbb{P}^1\times\{0\})+E.
    \]
    By this, we have
    \[
    K_{D}+(\widetilde{\mathcal{P}}_{0}-D)|_D=\rho|_D^*K_{\mathbb{P}^1}+E|_D.
    \]
    Hence, $E|_D=(\widetilde{\mathcal{P}}_0-D)|_D$ as Cartier divisors.
    Since $\rho^*\varphi_0|_D,\ldots,\rho^*\varphi_L|_D$ generate the ideal sheaf $\mathcal{O}_{D}(-l(\widetilde{\mathcal{B}}+E)|_D)$, we see that $\varphi_0|_0=\rho^*\varphi_0|_D,\ldots,\varphi_L|_0=\rho^*\varphi_L|_D$ define the associated quasimap structure to the klt--trivial fibration $f_0\colon \mathcal{X}_0\to\mathbb{P}^1$.
    This shows that condition (2) holds in this case.
    We complete the proof.
    \end{proof}

    \begin{cor}
Let $S$ be a quasi-projective seminormal scheme of finite type over $\mathbbm{k}$ and let $f\colon(\mathcal{X},\mathcal{A})\to\mathcal{P}$ be an object of $\mathcal{M}_{d,v,u,r,V}^{\mathrm{CYFib}}(S)$.
        Suppose that $\mathcal{A}$ is a line bundle.
Then, there exists the following object $q\colon \mathcal{P}\to [\mathrm{Cone}(\overline{V}^{\mathrm{BB}})\times S/\mathbb{G}_{m,S}]$ of $\mathcal{M}_{l(2-u),1,\frac{1}{l},u,\iota}^{\mathrm{Kst.qmaps}}(S)$ uniquely up to isomorphism satisfying the following. 
\begin{enumerate}
\item Let $\mathscr{L}$ be the associated line bundle on $\mathcal{P}$ with $q$. Then, $f^*\mathscr{L}\sim l(K_{\mathcal{X}/S}-f^*K_{\mathcal{P}/S})$.
\item For any geometric point $\bar{s}\in S$, let $B_{\bar{s}}$ be the fixed part of $q_{\bar{s}}$.
Then, $\frac{1}{l}B_{\bar{s}}$ is the discriminant $\mathbb{Q}$-divisor with respect to $f_{\bar{s}}$ and $q_{\bar{s}}$ defines the moduli map associated to $f_{\bar{s}}$ (cf.~Definition \ref{defn--associated--quasi--map--structure}).
\end{enumerate}       
    \end{cor}
 
\begin{proof}
The uniqueness of $q$ also follows from the same argument of the proof of Theorem \ref{thm--quasi--map--const}.
Therefore, to show the existence part, we may freely shrink $S$.
Take the normalization $\nu\colon S^\nu\to S$.
As Theorem \ref{thm--quasi--map--const}, we get the object $q^\nu\colon \mathcal{P}^\nu\to [\mathrm{Cone}(\overline{V}^{\mathrm{BB}})\times S^\nu/\mathbb{G}_{m,S^\nu}]$ associated with $(\mathcal{X}\times_SS^\nu,\mathcal{A}^\nu)$, where $\mathcal{A}^\nu$ is defined as the pullback of $\mathcal{A}$.
Note that $\mathcal{P}^\nu\cong \mathcal{P}\times_SS^\nu$.
Let $s\in S$ be a closed point and take sections $\varphi_0,\ldots,\varphi_L\in \mathrm{Sec}_{q^\nu}(\mathscr{L}^\nu)$ that correspond to the canonical basis of $\mathbbm{k}^{L+1}$ such that $(\varphi_i)_{p}\ne0$ for any $i\in\{0,\ldots,L\}$, where $\mathscr{L}^\nu$ is the line bundle associated with $q^\nu$.
Shrinking $S$, we may assume that $(\varphi_i)_{s'}\ne0$ for any $i\in\{0,\ldots,L\}$ and $s'\in S^\nu$.
Set $R:=S^\nu\times_SS^\nu$ and $S^\nu/R\cong S$ by \cite[Lemma 9.8]{kollar-mmp} (see also \cite[Definition 9.4]{kollar-mmp} for the definition of $S^\nu/R$).
Note that $\varphi_{0}$ defines a section $\psi^\nu\colon S^\nu\to \mathbf{Div}_{\mathcal{P}^\nu/S^\nu}$.
Note that for any $s',s''\in S^\nu$ that are mapped to the same point of $S$, $(\varphi_{0})_{s'}$ and $(\varphi_{0})_{s''}$ are identified under the isomorphism $\mathbb{P}^1_{s'}\cong\mathbb{P}^1_{s''}$.
This means that $\psi^\nu(s')=\psi^\nu(s'')$ and there exists a unique morphism $\psi\colon S\to \mathbf{Div}_{\mathcal{P}/S}$ such that $\psi^\nu$ coincides with the morphism induced by $\psi$.
Let $\mathcal{D}_0$ be the relative Cartier divisor on $\mathcal{P}$ over $S$, which is induced by $\psi$.
Set $\mathscr{L}\cong\mathcal{O}_{\mathcal{P}}(\mathcal{D}_0)$.
We note that $\mathscr{L}^\nu\cong \mathcal{O}_{\mathcal{P}^\nu}(\mathrm{div}_{\mathscr{L}^\nu}(\varphi_0))$ and $\varphi_i$ can be regarded as sections $\eta^\nu_i\colon\mathcal{P}^\nu\to \mathbb{A}_{\mathcal{P}^\nu}((\mathscr{L}^\nu)^\vee)$, where $(\mathscr{L}^\nu)^\vee$ is the dual of $\mathscr{L}^\nu$.
Then, there exist sections $\eta_i\colon \mathcal{P}\to \mathbb{A}_{\mathcal{P}}(\mathscr{L}^\vee)$ that induce $\eta^\nu_i$.
Such $\eta_i$ define a family of quasimaps $q\colon \mathcal{P}\to [\mathrm{Cone}(\overline{V}^{\mathrm{BB}})\times S/\mathbb{G}_{m,S}]$.
By construction of $q$, we complete the existence part of the assertion.
\end{proof}

 \begin{proof}[Proof of Theorem \ref{thm--quasi--map--construction--moduli--map}]
We first recall that $\mathcal{M}^{\mathrm{CYfib}}_{d,v,u,r,V}$ is obtained as a quotient stack $[Z/G]$ for some scheme $Z$ of finite type over $\mathbbm{k}$ and algebraic group $G$ acting on $Z$ by the proof of \cite[Theorem 1.3, 5.1]{HH}.
Note that $(\mathcal{M}^{\mathrm{CYfib}}_{d,v,u,V})^{\mathrm{sn}}\cong [Z^{\mathrm{sn}}/G]$, where $Z^{\mathrm{sn}}$ is the seminormalization of $Z$.
Let $f\colon(\mathcal{X},\mathcal{A})\to \mathcal{P}$ be the object corresponding to $Z^{\mathrm{sn}}\to \mathcal{M}^{\mathrm{CYfib}}_{d,v,u,r,V}$.
Recall by the proof of \cite[Theorems 1.3, 5.1]{HH} that we can choose $Z$ and $G$ so that $\mathcal{A}$ is a line bundle and $f\colon(\mathcal{X},\mathcal{A})\to \mathcal{P}$ is $G$-equivariant over $Z^\nu$.
Then, we can construct a family of quasimaps $q\colon\mathcal{P}\to[\mathrm{Cone}(\overline{V}^{\mathrm{BB}})\times Z^{\mathrm{sn}}/\mathbb{G}_{m,Z^{\mathrm{sn}}}]\in \mathcal{M}^{\mathrm{Kst.qmaps}}_{l(2-u),1,\frac{1}{l},u,\iota}(Z^{\mathrm{sn}})$ with respect to $f$ as Theorem \ref{thm--quasi--map--const}. 
By the uniqueness of $q$, it is easy to check that $q$ admits a natural $G$-action and is $G$-equivariant over $Z^{\mathrm{sn}}$.
This induces a $G$-invariant morphism $\alpha'\colon Z^{\mathrm{sn}}\to \mathcal{M}^{\mathrm{Kst.qmaps}}_{l(2-u),1,\frac{1}{l},u,\iota}$.
Now, $\alpha'$ descends to $[Z^{\mathrm{sn}}/G]$ and hence we obtain $\alpha$ as the assertion.
\end{proof}

Next, we observe the following property of $\alpha$.

\begin{thm}\label{thm--quasi-finiteness--of--two--moduli}
The morphism $\alpha$ in Theorem \ref{thm--quasi--map--construction--moduli--map} is a quasi-finite morphism.
In particular, $(M_{d,v,u,V}^{\mathrm{CYFib}})^{\mathrm{sn}}$ is quasi-projective.
\end{thm}

To prove the theorem, we prepare the following two results.

\begin{lem}\label{lem--principal-g-bundle-isotrivial}
    Let $U$ and $T$ be smooth affine curves and $t\in T$ a closed point.
    Let $G$ be a finite reduced group scheme and $\pi\colon\mathscr{G}\to U\times T$ a principal $G$-bundle.

    Then, there exist a finite morphism $T'\to T$ from a smooth curve and an isomorphism $\mathscr{G}_{T'}\cong\mathscr{G}_{t}\times T'$ as principal $G$-bundles over $U\times T'$.
\end{lem}

\begin{proof}
    By replacing $\mathscr{G}$ with $\mathbf{Isom}_{U\times T}^G(\mathscr{G},\mathscr{G}_{t}\times T)$, we may assume that $\mathscr{G}_{t}$ has a section over $U\times\{t\}$. Equivalently, $\mathscr{G}_{t}\cong G\times U$ as a principal $G$-bundle over $U$.
    To show the assertion under the assumption that $\mathscr{G}_{t}\cong G\times U$, it suffices to prove that there exist a finite morphism $T'\to T$ from a smooth curve and a section of $\mathscr{G}_{T'}$ over $U\times T'$.

Let $f\colon T'\to T$ be the Stein factorization of the natural morphism $\mathscr{G}\to T$.
$f$ and the natural morphism $\mathscr{G}\to U$ induce the morphism $h\colon \mathscr{G}\to U\times T'$.
Let $p_2\colon U\times T'\to T'$ be the second projection and $t'\in T'$ a closed point mapped to $t$ via $f$.
We set $\mathscr{G}_{t'}$ be the fiber of $p_2\circ h$ over $t'$.
Now, $\mathscr{G}_{t'}\subset \mathscr{G}_t\cong G\times U$ is a connected component by the property of the Stein factorization and $G\times U$ is smooth.
In particular, $\mathscr{G}_{t'}$ is isomorphic to $U$ via $h_{t'}$.
Note that $h$ is finite and $\mathscr{G}$ is flat over $T'$. 
Now, we claim that $h$ is isomorphic around $U\times \{t'\}$.
Indeed, $\mathscr{G}$, $U$ and $T'$ are all affine and let $A$, $B$ and $C$ be the corresponding coordinate rings to $\mathscr{G}$, $U\times T'$ and $T'$.
Let $\mathfrak{m}$ be the maximal ideal of $C$ corresponding to $t'$.
Then the morphism $h^{\#}\colon B\to A$ corresponding to $h$ induces the isomorphism $\overline{h^{\#}}\colon B/\mathfrak{m}B\to A/\mathfrak{m}A$ corresponding to $h_{t'}$.
By Nakayama's lemma (cf.~\cite[(1.M)]{Mat}), we know that the support of $\mathrm{Coker}(h^{\#})$ is disjoint from $U\times \{t'\}$.
This means that $h$ is a closed immersion in a neighborhood of $U\times \{t'\}$.
Furthermore, the support of $\mathrm{Ker}(h^{\#})$ is also disjoint from $U\times \{t'\}$ by Nakayama's lemma.
Indeed, since $A$ is flat over $C$, we have $\mathrm{Ker}(h^{\#})/\mathfrak{m}\mathrm{Ker}(h^{\#})\cong \mathrm{Ker}(\overline{h^{\#}})$.
Let $\pi_{T'}\colon \mathscr{G}_{T'}\to U\times T'$ be the base change of $\pi$ under $f$.
This shows that there exist a closed subset $W\subset U\times T'$ disjoint from $U\times \{t'\}$ and a section $s'\colon U\times T'\setminus W\to \mathscr{G}$.
We note that $W$ might be empty here.
There exists a section $s''\colon\mathscr{G}\to\mathscr{G}_{T'}$ that is induced by the universal property of the cartesian diagram and the morphism $p_2\circ h$.
Set $s:=s''\circ s'$. 
Then $s$ is a section of $\mathscr{G}_{T'}$ over $U\times T'\setminus W$.
Since $U\times T'$ is smooth and $\mathscr{G}_{T'}$ is finite over $U\times T'$, we see that $s$ can be extended to all codimension one points of $U\times T'$ by the valuative criterion for properness \cite[II, Theorem 4.7]{Ha}.
Let $W'\subset W$ be a closed subset such that $\mathrm{codim}_{U\times T'}W'\ge2$ and $\bar{s}\colon U\times T'\setminus W'\to \mathscr{G}_{T'}$ an extended section of $s$.
Since $\pi_{T'}$ is finite and $\mathscr{G}_{T'}$ is smooth and affine, we can extend $\bar{s}$ entirely to $U\times T'$ by the $S_2$-property of the coordinate ring of $\mathscr{G}_{T'}$.
Thus, we complete the proof.
\end{proof}
Before stating the next result, we briefly recall the definition of isotriviality.
\begin{defn}\label{defn--isotrivial}
    Let $f\colon X\to S$ be a projective morphism of schemes of finite type over $\mathbb{C}$.
    We say that $X$ is {\it isotrivial} over $S$ if $X_s$ and $X_{s'}$ are isomorphic for any closed points $s,s'\in S$.

    Let $q\colon (C,B)\to [\mathbb{A}^{N+1}_S/\mathbb{G}_{m,S}]$ be a family of log Fano quasimaps over $S$.
    We say that $q$ is {\it isotrivial} over $S$ if $q_s$ and $q_{s'}$ are isomorphic for any closed points $s,s'\in S$.
\end{defn}

\begin{prop}\label{prop--isotrivial--fibers}
Let $C$ be a smooth affine curve and consider $f\colon(\mathcal{X},\mathcal{A})\to\mathbb{P}^1\times C$ a family of $\mathcal{M}^{\mathrm{CYfib}}_{d,u,v,r,V}(C)$.
If the induced family of quasimaps $\mathbb{P}^1\times C\to[\mathrm{Cone}(\overline{V}^{\mathrm{BB}})/\mathbb{G}_m]$ is isotrivial over $C$, then there exists a non-empty open subset $W\subset C$ such that $\mathcal{X}\times_CW$ is also isotrivial over $W$.\label{prop--quasi-finiteness}
\end{prop}
\begin{proof}
Since the induced family of quasimaps is isotrivial, there exists an open subset $U \subset \mathbb{P}^{1}$ 
 such that $f$ has only klt fibers over $U\times C$. 
By using $f|_{U\times C}$, we obtain the following morphism $F\colon U\times C\to \mathcal{M}^{\mathrm{CY,klt}}_{v,r}$.
Fix arbitrary general closed points $c\in C$ and $u\in U$.
In this paragraph, we will show that $F_{C'}=F|_{U\times \{c\}}\circ p_{1,U\times C'}$ for some finite morphism from a smooth curve $C'$, where $F_{C'}$ is induced by the base change of $F$ under $U\times C'\to U\times C$ and $p_{1,U\times C'}\colon U\times C'\to U$ is the first projection.
Let $\pi\colon \mathcal{M}^{\mathrm{CY,klt}}_{v,r}\to M^{\mathrm{CY,klt}}_{v,r}$ be the morphism of the coarse moduli space.
By \cite[Theorem 11.3.1]{Ols}, if $G$ is the stabilizer group of $F((u,c))$ with respect to $\mathcal{M}^{\mathrm{CY,klt}}_{v,r}$, then there exists an affine scheme $V'$ admitting an action of $G$ such that we can take an \'etale neighborhood $V'/G\to M^{\mathrm{CY,klt}}_{v,r}$ of $\pi\circ F((u,c))$ such that $[V'/G]\cong\mathcal{M}^{\mathrm{CY,klt}}_{v,r}\times_{M^{\mathrm{CY,klt}}_{v,r}}V'/G$. 
By assumption, $\pi\circ F$ is a constant morphism and hence this factors through $V'/G$ naturally.
Thus, $F$ also factors through $[V'/G]$.
Let $F'\colon U\times C\to [V'/G]$ be the induced morphism, $\mathscr{G}\to U\times C$ the principal $G$-bundle associated with $F'$ and $\gamma\colon \mathscr{G}\to V'$ the associated $G$-equivariant morphism.
By Lemma \ref{lem--principal-g-bundle-isotrivial}, replacing $C$ with its finite cover if necessary, we may assume that $\mathscr{G}\cong\mathscr{G}_c\times C$ as principal $G$-bundles.
For any closed point $z\in \mathscr{G}_c$, $\gamma$ maps $\{z\}\times C$ to a point of $V'$ since the induced family of quasimaps is isotrivial and $V'\to V'/G$ is finite.
Since $\mathscr{G}_c\times C$ is reduced and $V'$ is affine, there exists a morphism $\gamma'\colon \mathscr{G}_c\to V'$ such that $\gamma=\gamma'\circ p_1$, where $p_1\colon \mathscr{G}_c\times C\to \mathscr{G}_c$ is the first projection.
By the definition of $[V'/G]$ and $\gamma=\gamma'\circ p_1$, we obtain that $F'=F'|_{U\times\{c\}}\circ p_{1,U\times C}$, where $p_{1,U\times C}\colon U\times C\to C$ is the first projection. 
Then, the assertion of this paragraph immediately follows from $F'=F'|_{U\times\{c\}}\circ p_{1,U\times C}$.

By the previous paragraph, replacing $\mathcal{X} \to C$ with the base change by a finite cover $C' \to C$, we can find a closed point $0 \in C$ such that $f^{-1}(U \times C)\cong \mathcal{X}|_{U\times\{0\}}\times C$ as schemes over $C$.
We put $F_{U_{0}}:=\mathcal{X}|_{U\times\{0\}}$ and set $F_{0}$ as the closure of $F_{U_{0}}$ in $\mathcal{X}$. 
Note that $F_{0}=f^{-1}(\mathbb{P}^{1}\times\{0\})$ and this is the fiber of $\mathcal{X} \to C$ over $0 \in C$. 
By construction, we get a birational map $\psi \colon F_{0} \times C \dashrightarrow \mathcal{X}$ over $\mathbb{P}^{1}\times C$ induced by $f^{-1}(U \times C)\cong \mathcal{X}|_{U\times\{0\}}\times C$. 
$$
\xymatrix@C=12pt{
F_{0} \times C\ar[rd]_-{f_{0}\times {\rm id}_{C}}\ar@{-->}[rr]^-{\psi}&&\mathcal{X}\ar[ld]^-{f}\\
&\mathbb{P}^{1}\times C
}
\qquad
\xymatrix@C=12pt{
F_{U_{0}} \times C\ar[rd]\ar[rr]^-{\psi|_{F_{U_{0}} \times C}}_-{\cong}&&f^{-1}(U \times C)\ar[ld]^-{f}\\
&U\times C
}
$$
In the rest of this paragraph, we construct a projective birational morphism $Z \to F_{0}$, klt pairs $(Z \times C,\Delta)$ and $(\mathcal{X},\Delta_{\mathcal{X}})$, and a birational contraction $(Z \times C,\Delta) \dashrightarrow (\mathcal{X},\Delta_{\mathcal{X}})$ over $C$ with good properties. 
Let $h \colon \widetilde{\mathcal{X}} \to \mathcal{X}$ and $h' \colon \widetilde{\mathcal{X}} \to F_{0} \times C$ be a common log resolution of $\psi^{-1}$. 
We consider the divisor $h^{*}K_{\mathcal{X}}-h'^{*}K_{F_{0} \times C}$. 
Since $\psi$ is an isomorphism over $U \times C$ and we have $K_{F_{0} \times C} \sim_{\mathbb{Q},\, \mathbb{P}^{1} \times C}0$ and $K_{\mathcal{X}} \sim_{\mathbb{Q},\,\mathbb{P}^{1} \times C}0$, we can find a $\mathbb{Q}$-divisor $D$ on $\mathbb{P}^{1}\times C$, whose support is contained in $(\mathbb{P}^{1}\setminus U)\times C$, such that  
$$h^{*}K_{\mathcal{X}}-h'^{*}K_{F_{0} \times C}=h'^{*}(f_{0}\times {\rm id}_{C})^{*}D.$$
We may write $D=D_{+}-D_{-}$, where $D_{+}$ and $D_{-}$ are effective $\mathbb{Q}$-divisors having no common components. 
We define 
$$\Gamma_{0}:=f^{*}_{0}(D_{+}|_{\mathbb{P}^{1}\times \{0\}})\qquad {\rm and} \qquad \Delta_{\mathcal{X}}:=f^{*}D_{-}.$$
By definition, $\Gamma_{0}$ is an effective $\mathbb{Q}$-divisor on $F_{0}$ such that $\Gamma_{0}\times C=(f_{0}\times {\rm id}_{C})^{*}D_{+}$. 
We have 
\begin{equation*}
\begin{split}
&h^{*}(K_{\mathcal{X}}+\Delta_{\mathcal{X}})-h'^{*}(K_{F_{0} \times C}+(\Gamma_{0}\times C))\\
=&h'^{*}(f_{0}\times {\rm id}_{C})^{*}D+h^{*}f^{*}D_{-}-h'^{*}(f_{0}\times {\rm id}_{C})^{*}D_{+}=0.
\end{split}
\end{equation*}
Therefore, $A_{(F_{0} \times C,\, \Gamma_{0}\times C)}(P)=A_{(\mathcal{X},\, \Delta_{\mathcal{X}})}(P)$ for any prime divisor $P$ over $\mathcal{X}$. 
Since $D_{+}$ and $D_{-}$ have no common component, any prime divisor $P$ over $\mathcal{X}$ satisfies either $A_{(F_{0} \times C,\, \Gamma_{0}\times C)}(P)=A_{(F_{0} \times C,\, 0)}(P)>0$ or $A_{(\mathcal{X},\, \Delta_{\mathcal{X}})}(P)=A_{(\mathcal{X},\, 0)}(P)>0$.  
These facts show that $(F_{0}\times C,\Gamma_{0}\times C)$ and $(\mathcal{X},\Delta_{\mathcal{X}})$ are klt pairs.
By \cite[Proposition 2.36]{KM}, there is a log resolution $\mu \colon Z \to F_{0}$ of $(F_{0},\Gamma_{0})$ such that if we write 
$$K_{Z}+\Delta_{0} = \mu^{*}(K_{F_{0}}+\Gamma_{0})+E_{0}$$
for some effective $\mathbb{Q}$-divisors $\Delta_{0}$ and $E_{0}$ on $Z$ which have no common components, then ${\rm Supp}\,\Delta_{0}$ is smooth. 
We define $\Delta$ to be the $\mathbb{Q}$-divisor on $Z$ by $\Delta:=\Delta_{0}\times C$. 
We consider the birational map over $\mathbb{P}^{1}\times C$
$$\varphi \colon Z \times C \dashrightarrow \mathcal{X}$$
defined by the composition of $Z \times C \to F_{0} \times C$ and $\psi \colon F_{0} \times C \dashrightarrow \mathcal{X}$. 
By construction, $A_{(Z \times C, \, \Delta)}(P) \leq A_{(Z \times C, \, \Delta-(E_{0}\times C))}(P)= A_{(\mathcal{X},\, \Delta_{\mathcal{X}})}(P)$ for any prime divisor $P$ over $\mathcal{X}$. 
Because $(Z \times C, \Delta)$ is terminal, $\varphi$ is a birational contraction over $\mathbb{P}^{1}\times C$, and we have $\varphi_{*} \Delta=\Delta_{\mathcal{X}}$.

By the argument in the previous paragraph, we get a birational contraction 
$$\varphi \colon (Z \times C, \Delta) \dashrightarrow (\mathcal{X},\Delta_{\mathcal{X}})$$
over $\mathbb{P}^{1}\times C$ of klt pairs $(Z \times C, \Delta)$ and $(\mathcal{X},\Delta_{\mathcal{X}})$ such that 
\begin{itemize}
\item
$Z$ is smooth and $(Z \times C, \Delta)$ is terminal, and
\item
$A_{(Z \times C, \,\Delta)}(P) \leq A_{(\mathcal{X}, \, \Delta_{\mathcal{X}})}(P)$ for any prime divisor $P$ over $\mathcal{X}$. 
\end{itemize}  
Since $K_{\mathcal{X}}+\Delta_{\mathcal{X}} \sim_{\mathbb{Q},\,\mathbb{P}^{1}\times C}0$, there is an ample $\mathbb{Q}$-divisor $D'$ on $\mathbb{P}^{1}\times C$ such that putting $B$ (resp.~$B_{\mathcal{X}}$) as the pullback of $D'$ to $Z\times C$ (resp.~$\mathcal{X}$), then both $(Z \times C, \Delta+B)$ and $(\mathcal{X},\Delta_{\mathcal{X}}+B_{\mathcal{X}})$
 are klt pairs and the divisor $K_{\mathcal{X}}+\Delta_{\mathcal{X}} +B_{\mathcal{X}}$ is the pullback of an ample $\mathbb{Q}$-divisor on $\mathbb{P}^{1}\times C$. 
Let $H_{\mathcal{X}}$ be an effective $\mathbb{Q}$-Cartier divisor on $\mathcal{X}$ which is ample over $\mathbb{P}^{1}\times C$, and let $H$ be the pullback of $H_{\mathcal{X}}$ to $Z\times C$. 
Since $Z \times C$ is smooth, $H$ is $\mathbb{Q}$-Cartier. 
By rescaling $H_{\mathcal{X}}$, we may assume that $K_{\mathcal{X}}+\Delta_{\mathcal{X}} +B_{\mathcal{X}}+H_{\mathcal{X}}$ is ample over $C$ and both $(Z \times C, \Delta+B+H)$ and $(\mathcal{X},\Delta_{\mathcal{X}}+B_{\mathcal{X}}+H_{\mathcal{X}})$ are klt pairs. 
By construction, $B_{\mathcal{X}}+H_{\mathcal{X}}$ is nef over $\mathbb{P}^{1}\times C$ and $\varphi_{*}(B+H)=B_{\mathcal{X}}+H_{\mathcal{X}}$. 
Combining the two facts and $A_{(Z \times C, \,\Delta)}(P) \leq A_{(\mathcal{X}, \, \Delta_{\mathcal{X}})}(P)$ for any prime divisor $P$ over $\mathcal{X}$, by taking a common resolution of $\varphi$ and using the negativity lemma, we have 
$$A_{(Z \times C, \,\Delta+B+H)}(P) \leq A_{(\mathcal{X}, \, \Delta_{\mathcal{X}}+B_{\mathcal{X}}+H_{\mathcal{X}})}(P)$$
for any prime divisor $P$ over $\mathcal{X}$. 
From this discussion, we see that 
$$\varphi \colon (Z \times C, \Delta+B+H) \dashrightarrow (\mathcal{X},\Delta_{\mathcal{X}}+B_{\mathcal{X}}+H_{\mathcal{X}})$$
is a birational contraction over $C$ from a klt pair to the log canonical model over $C$. 
By shrinking $C$, we may assume that for any closed point $c \in C$, the restriction 
$$\varphi_{c} \colon (Z, \Delta_{c}+B_{c}+H_{c}) \dashrightarrow (\mathcal{X}_{c},\Delta_{\mathcal{X}_{c}}+B_{\mathcal{X}_{c}}+H_{\mathcal{X}_{c}})$$
is a birational contraction  from a klt pair to the log canonical model. 
By construction, $B_{c}+H_{c}$ is big for any $c \in C$. 

Finally, we prove that $\mathcal{X}_{c}\cong \mathcal{X}_{c'}$ for any closed points $c,\,c' \in C$, which implies the conclusion of Proposition \ref{prop--quasi-finiteness}. For any $c,\,c' \in C$, the divisors $\Delta_{c}+B_{c}+H_{c}$ and $\Delta_{c'}+B_{c'}+H_{c'}$ are algebraically equivalent, and therefore $$K_{Z}+\Delta_{c}+B_{c}+H_{c} \equiv K_{Z}+\Delta_{c'}+B_{c'}+H_{c'}.$$
We recall that $\mathcal{X}_{c}$ (resp.~$\mathcal{X}_{c'}$) is the log canonical model of $(Z, \Delta_{c}+B_{c}+H_{c})$ (resp.~$(Z, \Delta_{c'}+B_{c'}+H_{c'})$). 
Hence, Lemma \ref{lem--lc-model} shows $\mathcal{X}_{c}\cong \mathcal{X}_{c'}$. 
Thus, Proposition \ref{prop--quasi-finiteness} holds. 
\end{proof}   

Now we are ready to prove Theorem \ref{thm--quasi-finiteness--of--two--moduli}. 

\begin{proof}[Proof of Theorem \ref{thm--quasi-finiteness--of--two--moduli}]
To show $\alpha$ is quasi-finite, take an arbitrary morphism $\mathrm{Spec}(\Omega)\to \mathcal{M}^{\mathrm{Kst.qmaps}}_{n,u,v,r,\iota}$, where $\Omega$ is an algebraically closed field over $\mathbb{C}$.
Suppose that there exists a morphism $\pi\colon C\to \mathrm{Spec}(\Omega)\times_{\mathcal{M}^{\mathrm{Kst.qmaps}}_{n,u,v,r,\iota}}(\mathcal{M}^{\mathrm{CYfib}}_{d,u,v,V})^{\mathrm{sn}}$ from a smooth curve.
Let $\gamma\colon \mathrm{Spec}(\Omega)\times_{\mathcal{M}^{\mathrm{Kst.qmaps}}_{n,u,v,r,\iota}}(\mathcal{M}^{\mathrm{CYfib}}_{d,u,v,V})^{\mathrm{sn}}\to M^{\mathrm{CYfib}}_{d,u,v,r,V}$ be the canonical morphism.
Then, by Proposition \ref{prop--isotrivial--fibers} and \cite[Proposition 6.3]{Hat23}, we see that the image of $\gamma\circ\pi$ is a point.
This shows that $\mathrm{Spec}(\Omega)\times_{\mathcal{M}^{\mathrm{Kst.qmaps}}_{n,u,v,r,\iota}}(\mathcal{M}^{\mathrm{CYfib}}_{d,u,v,V})^{\mathrm{sn}}$ is a Deligne--Mumford stack quasi-finite over $\mathrm{Spec}(\Omega)$.
Thus, $\alpha$ is quasi-finite.

Let $\overline{\alpha}\colon(M^{\mathrm{CYfib}}_{d,u,v,V})^{\mathrm{sn}}\to M^{\mathrm{Kst.qmaps}}_{n,u,v,r,\iota}$ be the induced morphism of the coarse moduli spaces by $\alpha$.
Then we see that $\overline{\alpha}$ is separated and quasi-finite.
Since $M^{\mathrm{Kst.qmaps}}_{n,u,v,r,\iota}$ is projective by Theorem \ref{thm--projectivity}, $(M^{\mathrm{CYfib}}_{d,u,v,V})^{\mathrm{sn}}$ is quasi-projective by \cite[Theorem 7.2.10]{Ols}.
\end{proof}

\begin{lem}\label{lem--linear--equiv--two--CM}
    Let $S$ be a normal variety and take $f\colon (\mathcal{X},\mathcal{A})\to \mathcal{C}\in\mathcal{M}^{\mathrm{CYfib}}_{d,v,u,r,V}(S)$.
    Let $q\colon\mathcal{C}\to [\mathrm{Cone}(\overline{V}^{\mathrm{BB}})\times S/\mathbb{G}_{m,S}]$ be the associated family of K-semistable log Fano quasimaps as Theorem \ref{thm--quasi--map--const}.
Then, there exist positive integers $m,m'\in\mathbb{Z}_{>0}$ such that $(\lambda_{\mathrm{CM},f}^{\infty})^{\otimes m}\sim\lambda_{\mathrm{CM},q}^{\otimes m'}$.

Furthermore, let $G$ be a semisimple linear algebraic group acting on $S$ and if $(\mathcal{X},\mathcal{A})$ admits an action of $G$ equivariant over $S$, then the above linear equivalence preserves the $G$-linearizations.
\end{lem}

\begin{proof}
To show the first assertion, we may assume that $S$ is smooth by taking a resolution of singularities.
Let $\pi_{\mathcal{X}}\colon \mathcal{X}\to S$ and $\pi_{\mathcal{C}}\colon \mathcal{C}\to S$ be the canonical morphisms.
By replacing $\mathcal{A}$ with $\mathcal{A}+a f^*\mathcal{L}$ for any sufficiently large $a>0$, we may also assume that $\mathcal{A}$ is $\pi_{\mathcal{X}}$-ample, where $\mathcal{L}$ is a $\pi_{\mathcal{C}}$-ample line bundle on $\mathcal{C}$.
In the proof of Proposition \ref{prop--CM--limit}, we deduced that 
\[
\lambda_{\mathrm{CM},f}^{\infty}=\frac{d}{2\mathcal{A}_t^{d-1}\cdot K_{\mathcal{X}_t}}(\pi_{\mathcal{X}})_*\left(\mathcal{A}^{d-1}\cdot K_{\mathcal{X}/S}^2\right)
\]
as $\mathbb{Q}$-divisors, where $t\in S$ is an arbitrary point.
Here, we note that $\frac{-d}{2\mathcal{A}_t^{d-1}\cdot K_{\mathcal{X}_t}}\in\mathbb{Q}_{>0}$.
Let $\mathscr{L}$ be the associated line bundle with $q$.
Since $f$ is a contraction (see Lemma \ref{lem--stein--factorizaton}), $lK_{\mathcal{X}/S}\sim lf^*(K_{\mathcal{C}/S}+l^{-1}\mathscr{L})$ by the first paragraph in the proof of Theorem \ref{thm--quasi--map--const} continued. 
Then, it is easy to see that 
\begin{equation*}
l^2(\pi_{\mathcal{X}})_*\left(\mathcal{A}^{d-1}\cdot K_{\mathcal{X}/S}^2\right)\sim l^2\mathcal{A}^{d-1}_c(\pi_{\mathcal{C}})_*((K_{\mathcal{C}/S}+l^{-1}\mathscr{L})^2)\label{eq--linearlization}
\end{equation*}
as Weil divisors, where $c\in\mathcal{C}$ is an arbitrary closed point.
Thus, the first assertion follows from the definitions of the CM line bundles.

If $(\mathcal{X},\mathcal{A})$ admits an action of $G$ equivariant over $S$, then $(\lambda_{\mathrm{CM},f}^{\infty})^{\otimes m}\sim\lambda_{\mathrm{CM},q}^{\otimes m'}$ preserves the $G$-linearizations by \cite[Proposition 1.4]{GIT} and we obtain the last assertion.
\end{proof}

\begin{cor}\label{cor--final}
 Notation as in Definition \ref{defn--limit--CM}.  
 Then, $\Lambda_{\mathrm{CM},\infty}|_{(M^{\mathrm{CYfib}}_{d,v,u,V})^\nu}$ is ample.
 Furthermore, $M^{\mathrm{CYfib}}_{d,v,u,r,V}$ is quasi-projective and $\Lambda_{\mathrm{CM},t}|_{(M^{\mathrm{CYfib}}_{d,v,u,V})^\nu}$ is ample for any sufficiently large $t>0$.
\end{cor}

\begin{proof}
By \cite[Tag 0GFB]{stacksproject-chap98}, to show the ampleness of $\Lambda_{\mathrm{CM},t}|_{M^{\mathrm{CYfib}}_{d,v,u,r,V}}$ for any sufficiently large $t>0$, it suffices to show the ampleness of $\nu^*\Lambda_{\mathrm{CM},t}|_{(M^{\mathrm{CYfib}}_{d,v,u,V})^\nu}$.
Since \[\lim_{t\to\infty}\nu^*\Lambda_{\mathrm{CM},t}|_{(M^{\mathrm{CYfib}}_{d,v,u,V})^\nu}=\Lambda_{\mathrm{CM},\infty}|_{(M^{\mathrm{CYfib}}_{d,v,u,V})^\nu},\]
as $\mathbb{Q}$-linear equivalence classes, it suffices to show that $\Lambda_{\mathrm{CM},\infty}|_{(M^{\mathrm{CYfib}}_{d,v,u,V})^\nu}$ is ample.

By the proof of \cite[Theorem 5.1]{HH}, there exists a quasi projective scheme $Z$ over $\mathbb{C}$ with an action of a semisimple linear algebraic group $G$ such that $\mathcal{M}^{\mathrm{CYfib}}_{d,v,u,r,V}\cong[Z/G]$.
If we let $Z^\nu$ be the normalization of $Z$, then $(\mathcal{M}^{\mathrm{CYfib}}_{d,v,u,V})^\nu\cong[Z^\nu/G]$.
By the canonical morphism $Z^\nu\to [Z^\nu/G]$ and the composition $Z^\nu\to \mathcal{M}_{l(2-u),1,\frac{1}{l},u,\iota}^{\mathrm{Kst.qmaps}}$ of this morphism and $\alpha$, we obtain the corresponding family $f\colon (\mathcal{X},\mathcal{A})\to \mathcal{C}\in\mathcal{M}^{\mathrm{CYfib}}_{d,v,u,r,V}(Z^\nu)$ of uniformly adiabatically K-stable klt--trivial fibrations over curves and its associated family $q$ of quasimaps.
Let $\mathcal{L}$ be the line bundle on $\mathcal{C}$ induced by $q$. 
Applying Lemma \ref{lem--linear--equiv--two--CM} to $Z^\nu\to\mathcal{M}_{l(2-u),1,\frac{1}{l},u,\iota}^{\mathrm{Kss.qmaps}}$ and the two families $f$ and $q$, we obtain a $G$-equivariant linear equivalence $(\lambda_{\mathrm{CM},f}^{\infty})^{\otimes m}\sim\lambda_{\mathrm{CM},q}^{\otimes m'}$ for some $m,m'\in\mathbb{Z}_{>0}$. 
This linear equivalence descents to $[Z^\nu/G]$ and we have that $\nu^*(\lambda_{\mathrm{CM},\infty})^{\otimes m}\cong \alpha^*(\lambda^{\mathrm{qmaps}}_{\mathrm{CM}})^{\otimes m'}$.

By Theorem \ref{thm--quasi-finiteness--of--two--moduli} and Corollary \ref{cor--moduli--kst--quasi--maps}, we see that the morphism $\beta\colon (M^{\mathrm{CYfib}}_{d,v,u,V})^\nu\to M_{l(2-u),1,\frac{1}{l},u,\iota}^{\mathrm{Kst.qmaps}}$ of the coarse moduli spaces induced by $\alpha$ is also quasi-finite.
Note also that $\beta$ is separated since $(M^{\mathrm{CYfib}}_{d,v,u,V})^\nu$ and $M_{l(2-u),1,\frac{1}{l},u,\iota}^{\mathrm{Kst.qmaps}}$ are separated.
By \cite[Theorem 7.2.10]{Ols} and Theorem \ref{thm--projectivity}, we see that $(M^{\mathrm{CYfib}}_{d,v,u,V})^\nu$ is a scheme and hence $\beta^*\Lambda^{\mathrm{qmaps}}_{\mathrm{CM}}$ is ample.
On the other hand, since $\nu^*(\lambda_{\mathrm{CM},\infty})^{\otimes m}\cong \alpha^*(\lambda^{\mathrm{qmaps}}_{\mathrm{CM}})^{\otimes m'}$, we have $\frac{m'}{m}\beta^*\Lambda^{\mathrm{qmaps}}_{\mathrm{CM}}\sim_{\mathbb{Q}}\Lambda_{\mathrm{CM},\infty}|_{(M^{\mathrm{CYfib}}_{d,v,u,V})^\nu}$ (cf.~\cite[Theorem 10.3]{alper}).
Thus, $\Lambda_{\mathrm{CM},\infty}|_{(M^{\mathrm{CYfib}}_{d,v,u,V})^\nu}$ is ample.

On the other hand, there are two $\mathbb{Q}$-Cartier $\mathbb{Q}$-divisors $D_1$ and $D_0$ on $(M^{\mathrm{CYfib}}_{d,v,u,V})^\nu$ such that
\[
\Lambda_{\mathrm{CM},q'}^{\otimes (\mathcal{A}_z^d+dq'\mathcal{A}_z^{d-1}\cdot \mathcal{L}_z)^2}=(dq'\mathcal{A}_z^{d-1}\cdot \mathcal{L}_z)^2\Lambda_{\mathrm{CM},\infty}+q'D_1+D_0
\]
due to Set up \ref{setup--62} and Proposition \ref{prop--CM--limit} applied to $Z^\nu$, where $z\in Z^\nu$ is a general closed point.
Therefore, the last assertion follows from \cite[II, Exercise 7.5]{Ha}. 
\end{proof}

\begin{rem}\label{rem-Koex}
We note that Corollary \ref{cor--final} does not directly imply the ampleness of $\Lambda_{\mathrm{CM},t}$ either on $\mathcal{M}_{d,v,u,r,V}$ or $(\mathcal{M}_{d,v,u,V})^{\mathrm{sn}}$ for any sufficiently large $t\in\mathbb{Q}_{>0}$ because even when the pullback of a line bundle $L$ on $X$ to the normalization $X^\nu$ is ample, $L$ itself is not ample in general as \cite[Proposition 2]{Koex}. 
For a similar reason, the quasi-projectivity of $(\mathcal{M}_{d,v,u,V})^{\mathrm{sn}}$ (Theorem \ref{thm--quasi-finiteness--of--two--moduli}) does not imply the quasi-projectivity of $\mathcal{M}_{d,v,u,r,V}$.
\end{rem}

Finally, we observe that $\mathcal{M}_{d,v,u,w,r}$, which is defined in \cite[Theorem 5.1]{HH}, does not satisfy Theorem \ref{thm--quasi-finiteness--of--two--moduli}, which is the technical core of this paper, when $\mathcal{M}_{d,v,u,w,r}$ parameterizes klt--trivial fibrations of Kodaira dimension one.
\begin{rem}\label{rem--final}
Let $\pi\colon \mathcal{M}_{d,v,u,w,r}\to M_{d,v,u,w,r}$ be the coarse moduli space. 
Now, we will construct a nontrivial curve family of $\mathcal{M}_{d,v,u,w,r}$ but the associated family of quasimaps is isotrivial.
Let $f\colon (X,A)\to C$ be an object of $\mathcal{M}_{d,v,u,w,r}(\mathbb{C})$ such that $X$ is a product $C\times Y$ of a smooth curve $C$ with genus at least two and a smooth variety $Y$ such that $K_Y\sim0$, where $f$ is the first projection.
Suppose further that there exists an ample line bundle $H$ on $Y$ such that $A=K_X+g^*H$, where $g\colon X\to Y$ is the second projection.
Let $\mathrm{Aut}(X)$ be the automorphism group of $X$.
Since $f$ is induced by $|lK_X|$ for some sufficiently large $l$, $\mathrm{Aut}(X)$ preserves $f$ and there exists a natural homomorphism $\mathrm{Aut}(X)\to \mathrm{Aut}(C)$.
Take a general smooth proper curve $D\subset \mathbf{Pic}^0_{C/\mathbb{C}}$ and let $\mathscr{L}$ be the line bundle on $C\times D$ associated to $D\subset \mathbf{Pic}^0_{C/\mathbb{C}}$.
Now consider $f_D\colon (X\times D,p_1^*A\otimes f_D^*\mathscr{L})\to C\times D\in \mathcal{M}_{d,v,u,w,r}(D)$, where $p_1\colon X\times D\to X$ is the first projection.
Then the corresponding morphism $D\to M_{d,v,u,w,r}$ is not constant.
Indeed, assume contrary that the above morphism is constant.
Then, for any closed points $d_1,d_2\in D$, there exists $\varphi_{d_1,d_2}\in \mathrm{Aut}(X)$ such that $\varphi_{d_1,d_2}^*(A\otimes \mathscr{L}_{d_2})=A\otimes \mathscr{L}_{d_1}$.
By the choice of $D$, if we fix an arbitrary $d_1$, there are infinitely many distinct $\varphi_{d_1,d_2}$.
Since $\mathrm{Aut}(C)$ is finite, there exists some $d_1$ such that $\varphi_{d_1,d_2}\in \mathrm{Ker}(\mathrm{Aut}(X)\to \mathrm{Aut}(C))$ for infinitely many $d_2$.
However, $\varphi_{d_1,d_2}$ acts on $Y\cong f^{-1}(c)$ for any closed point $c\in C$ and preserves $H\cong A|_{f^{-1}(c)}$.
Thus, there are only finitely many choices of $\varphi_{d_1,d_2}$ if we fix some $d_1$ (cf.~\cite[Corollary 3.23]{CM}).
This is a contradiction.
This means that $D\to M_{d,v,u,w,r}$ is not a constant morphism.
We note that if $Y$ is an Abelian variety or an irreducible holomorphic symplectic manifold, then we can define the associated family of quasimaps similarly to Theorem \ref{thm--quasi--map--const}.
It is easy to see that the family of quasimaps is isotrivial.
Therefore, we cannot expect a quasi-finiteness of the corresponding for $\mathcal{M}_{d,v,u,w,r}$ as Theorem \ref{thm--quasi-finiteness--of--two--moduli}.

\end{rem}

%%%%%%%%%%%%%%%

\end{document}